\documentclass{amsbook}

\usepackage{amsfonts}
\usepackage{bbm}
\usepackage{amsmath}
\usepackage{amssymb}
\usepackage{inputenc}
\usepackage{mathptmx}
\usepackage[left=2cm,right=2cm,bottom=3cm,top=3cm]{geometry}
\usepackage{helvet}
\usepackage{amscd}
\usepackage{courier}
\usepackage{type1cm}         

\usepackage{graphicx}        
\usepackage{multicol}        
\usepackage[bottom]{footmisc}

\newtheorem{theorem}{Theorem}[chapter]
\theoremstyle{plain}

\newtheorem{conjecturenew}[theorem]{Conjecture}
\newtheorem{corollary}[theorem]{Corollary}

\newtheorem{definition}[theorem]{Definition}
\newtheorem{example}[theorem]{Example}
\newtheorem{exercise}[theorem]{Exercise}
\newtheorem{lemma}[theorem]{Lemma}

\newtheorem{question}[theorem]{Question}

\newtheorem{proposition}[theorem]{Proposition}
\newtheorem{remark}[theorem]{Remark}

\numberwithin{equation}{section}

\newcommand{\concat}{
  \mathord{
    \mathchoice
    {\raisebox{1ex}{\scalebox{.7}{$\frown$}}}
    {\raisebox{1ex}{\scalebox{.7}{$\frown$}}}
    {\raisebox{.7ex}{\scalebox{.5}{$\frown$}}}
    {\raisebox{.7ex}{\scalebox{.5}{$\frown$}}}
  }
}

\makeindex             

\newtheorem{fact}[theorem]{Fact}

\newcommand{\N}{\mathbb{N}}
\newcommand{\Z}{\mathbb{Z}}
\newcommand{\Q}{\mathbb{Q}}
\newcommand{\R}{\mathbb{R}}
\newcommand{\C}{\mathbb{C}}
\renewcommand{\l}{\mathcal{L}}

\newcommand{\Pow}{\mathcal{P}}

\newcommand{\starN}{{}^{\ast}\N}
\newcommand{\starZ}{{}^{\ast}\Z}
\newcommand{\starQ}{{}^{\ast}\Q}
\newcommand{\starR}{{}^{\ast}\R}

\newcommand{\staralpha}{{}^{\ast}\alpha}
\newcommand{\st}{\operatorname{st}}
\newcommand{\BD}{\operatorname{BD}}

\renewcommand{\u}{\mathcal{U}}

\newcommand{\starA}{{}^{\ast}A}
\newcommand{\starB}{{}^{\ast}B}
\newcommand{\starC}{{}^{\ast}C}
\newcommand{\starS}{{}^{\ast}S}
\newcommand{\starX}{{}^{\ast}X}
\newcommand{\starY}{{}^{\ast}Y}

\newcommand{\starf}{{}^{\ast}f}
\newcommand{\starg}{{}^{\ast}g}

\newcommand{\F}{\mathcal{F}}
\newcommand{\G}{\mathcal{G}}
\newcommand{\U}{\mathcal{U}}
\newcommand{\s}{\mathcal{S}}

\newcommand{\UU}{\mathbb{U}}
\newcommand{\VV}{\mathbb{V}}
\newcommand{\V}{\mathcal{V}}

\newcommand{\mgoth}{\mathfrak{m}}
\newcommand{\igoth}{\mathfrak{i}}

\newcommand{\zfc}{\textsf{ZFC}}
\newcommand{\FS}{\operatorname{FS}}

\newcommand{\Fin}{\text{Fin}}
\newcommand{\Fun}{\text{Fun}}
\newcommand{\B}{\operatorname{B}}
\newcommand{\stem}{\operatorname{stem}}

\begin{document}

\author{Mauro Di Nasso, Isaac Goldbring, and Martino Lupini}
\title{Nonstandard Methods in Ramsey Theory and Combinatorial Number Theory}
\maketitle

\frontmatter

\chapter*{Acknowledgements}

The collaboration between the authors first began when they participated in an American Institute of Mathematics (AIM) Structured Quartet Research Ensemble (or SQuaRE) program together with Renling Jin, Steven Leth, and Karl Mahlburg.  We thus want to thank AIM for all of their support during our three year participation in the SQuaRE program as well as their encouragement to organize a larger workshop on the subject.  A preliminary version of this manuscript was distributed during that workshop and we want to thank the participants for their valuable comments.  In particular, Steven Leth and Terence Tao gave us a tremendous amount of feedback and for that we want to give them an extra expression of gratitude.

\chapter*{Introduction}

Generally speaking, Ramsey theory studies which combinatorial configurations of a structure can always be found in one of the pieces of a given finite partition. More generally, it considers the problem of which combinatorial configurations can be found in sets that are ``large'' in some suitable sense. Dating back to the foundational results of van der Waerden, Ramsey, Erd\H{o}s, Tur\'{a}n and others from the 1920s and 1930s, Ramsey theory has since then had an extraordinary development. On the one had, many applications of Ramsey theory have been found to numerous other areas of mathematics, ranging from functional analysis, topology, and dynamics, to set theory, model theory, and computer science. On the other hand, results and methods  from other areas of mathematics have been successfully applied to establish new results in Ramsey theory. For instance, ergodic theory and the theory of recurrence in measurable dynamics has had a huge impact on Ramsey theory, to the point of giving rise to the research area of ``ergodic Ramsey theory.'' Perhaps the best known achievement of this approach is the ergodic-theoretic proof of Szemer\'{e}di's theorem due to Furstenberg in the 1980s. In a different (but intimately related) direction, the theory of ultrafilters has been an important source of methods and ideas for Ramsey theory. In particular, the study of topological and algebraic properties of the space of ultrafilters has been used to give short and elegant proofs of deep combinatorial pigeonhole principles. Paradigmatic in this direction is the Galvin--Glazer ultrafilter proof of Hindman's theorem on sets of finite sums, previously established by Hindman in 1974 via a delicate, purely combinatorial argument. 

Recently, a new thread of research has emerged, where problems in Ramsey theory are studied from the perspective of nonstandard analysis and nonstandard methods. Developed by Abraham Robinson in the 1960s and based on first order logic and model theory, nonstandard analysis provided a formal and rigorous treatment of calculus and classical analysis via infinitesimals, an approach more similar in spirit to the approach originally taken in the development of calculus in the 17th and 18th century, and avoids the epsilon-delta arguments that are inherent in its later formalization due to Weierstrass. While this is perhaps its most well known application, nonstandard analysis is actually much more versatile. The foundations of nonstandard analysis provide an approach, which we shall call the nonstandard method, that is applicable to virtually any area of mathematics. The nonstandard method has thus far been used in numerous areas of mathematics, including functional analysis, measure theory, ergodic theory, differential equations, and stochastic analysis, just to name a few such areas.

In a nutshell, the nonstandard method allows one to extend the given mathematical universe and thus regard it as contained in a much richer nonstandard universe. Such a nonstandard universe satisfies strong saturation properties which in particular allow one to consider limiting objects which do not exist in the standard universe. This procedure is similar to passing to an ultrapower, and in fact the nonstandard method can  also be seen as a way to axiomatize the ultrapower construction in a way that distillates its essential features and benefits, but avoids being bogged down by the irrelevant details of its concrete implementation. This limiting process allows one to reformulate a given problem involving finite (but arbitrarily large) structures or configurations into a problem involving a single structure or configuration which is infinite but for all purposes behaves as though it were finite (in the precise sense that it is hyperfinite in the nonstandard universe). This reformulation can then be  tackled directly using finitary methods, ranging from combinatorial counting arguments to recurrence theorems for measurable dynamics, recast in the nonstandard universe. 

In the setting of Ramsey theory and combinatorics, the application of nonstandard methods had been pioneered by the work of Keisler, Leth, and Jin from the 1980s and 1990s. These applications had focused on density problems in combinatorial number theory. The general goal in this area is to establish the existence of combinatorial configurations in sets that are large in that sense that they have positive asymptotic density. For example, the aforementioned celebrated theorem of Szemer\'{e}di from 1970 asserts that a set of integers of positive density contains arbitrarily long finite arithmetic progressions. One of the contributions of the nonstandard approach is to translate the notion of asymptotic density on the integers, which does not satisfies all the properties of a measure, into an actual measure in the nonstandard universe. This translation then makes methods from measure theory and ergodic theory, such as the ergodic theorem or other recurrence theorems, available for the study of density problems. In a sense, this can be seen as a version of Furstenberg's correspondence (between sets of integers and measurable sets in a dynamical system), with the extra feature that the dynamical system obtained perfectly reflects \emph{all} the combinatorial properties of the set that one started with. The achievements of the nonstandard approach in this area include the work of Leth on arithmetic progressions in sparse sets, Jin's theorem on sumsets, as well as Jin's Freiman-type results on inverse problems for sumsets. More recently, these methods have also been used by  Jin, Leth, Mahlburg, and the present authors to tackle a conjecture of Erd\H{o}s concerning sums of infinite sets (the so-called $B+C$ conjecture), leading to its eventual solution by Moreira, Richter, and Robertson.

Nonstandard methods are also tightly connected with ultrafilter methods. This has been made precise and successfully applied in recent work of Di Nasso, where he observed that there is a perfect correspondence between ultrafilters and elements of the nonstandard universe up to a natural notion of equivalence. On the one hand, this allows one to manipulate ultrafilters as nonstandard points, and to use ultrafilter methods to prove the existence of certain combinatorial configurations in the nonstandard universe. One the other hand, this gives an intuitive and direct way to infer, from the existence of certain ultrafilter configurations, the existence of corresponding standard combinatorial configurations via the fundamental principle of transfer in the nonstandard method. This perspective has successfully been applied by Di Nasso, Luperi Baglini, and co-authors to the study of partition regularity problems for Diophantine equations over the integers, providing in particular a far-reaching generalization of the classical theorem of Rado on partition regularity of systems of linear equations. Unlike Rado's theorem, this recent generalization also includes equations that are \emph{not} linear. 

Finally, it is worth mentioning that many other results in combinatorics can be seen, directly or indirectly, as applications of the nonstandard method. For instance, the groundbreaking work of Hrushovski and Breuillard--Green--Tao on approximate groups, although not originally presented in this way, admit a natural nonstandard treatment. The same applies to the work of Bergelson and Tao on recurrence in quasirandom groups.

The goal of this manuscript is to introduce the uninitiated reader to the nonstandard method and to provide an overview of its most prominent applications in Ramsey theory and combinatorial number theory. In particular, no previous knowledge of nonstandard analysis will be assumed. Instead, we will provide a complete and self-contained introduction to the nonstandard method in the first part of this book. Novel to our introduction is a treatment of the topic of iterated hyperextensions, which is crucial for some applications and has thus far appeared only in specialized research articles. The intended audience for this book include researchers in combinatorics that desire to get acquainted with the nonstandard approach, as well as logicians and experts of nonstandard analysis who have been working in this or other areas of research. The list of applications of the nonstandard method to combinatorics and Ramsey theory presented here is quite extensive, including cornerstone results of Ramsey theory such as Ramsey's theorem, Hindman's theorem on sets of finite sums, the Hales--Jewett theorem on variable words, and Gowers' theorem on $\mathrm{FIN}_k$. It then proceeds with results on partition regularity of Diophantine equations and with density problems in combinatorial number theory. A nonstandard treatment of the triangle removal lemma, the Szemer\'{e}di regularity lemma, and of the already mentioned work of Hrushovski and Breuillard--Green--Tao on approximate groups conclude the book. We hope that such a complete list of examples will help the reader unfamiliar with the nonstandard method get a good grasp on how the approach works and can be applied. At the same time, we believe that collecting these results together, and providing a unified presentation and approach, will provide a useful reference for researchers in the field and will further stimulate the research in this currently very active area.

\

\noindent Pisa, Italy

\noindent Irvine, California

\noindent Wellington, New Zealand

\chapter*{Notation and Conventions}

We set $\N:=\{1,2,3,\ldots\}$ to denote the set of \emph{positive} natural numbers and $\N_0:=\{0,1,2,3,\ldots\}$ to denote the set of natural numbers.

We use the following conventions for elements of particular sets:
\begin{itemize}
\item $m$ and $n$ range over $\N$;
\item $h,k$, and $l$ range over $\Z$;
\item $H,K,M$, and $N$ range over elements of $\starN$;
\item $\delta$ and $\epsilon$ always denote (small) positive real numbers, while $\varepsilon$ denotes a positive infinitesimal element of $\starR$;
\item Given any set $S$, we let $\alpha$, $\beta$, and $\gamma$ denote arbitrary (possibly standard) elements of $\starS$.
\end{itemize}

For any set $S$, we set $\Fin(S):=\{F\subseteq X\mid X\ \text{is finite}\}$.

For any $n$, we write $[n]:=\{1,\ldots,n\}$.  Similarly, we write $[N]:=\{1,\ldots,N\}$.

Given any nonempty finite set $I$ and any set $A$, we write $\delta(A,I):=\frac{|A\cap I|}{|I|}$.  We extend this to the nonstandard situation:  if $I$ is a nonempty hyperfinite set and $A$ is an internal set, we set $\delta(A,I):=\frac{|A\cap I|}{|I|}$.  We also write $\delta(A,n):=\delta(A,[n])$ and $\delta(A,N):=\delta(A,[N])$.

Given a hyperfinite set $X$, we let $\l_X$ denote the $\sigma$-algebra of Loeb measurable subsets of $X$ and we let $\mu_X$ denote the Loeb measure on $\l_X$ that extends the normalized counting measure on $X$.  (See Chapter \ref{Loeb}.)  When $X=\{1,\ldots,N\}$, we write $\l_N$ and $\mu_N$ instead of $\l_X$ and $\mu_X$.  If $A$ is internal, we write $\mu_X(A):=\mu_X(A\cap X)$.

Suppose that $A\subseteq \Z$ and $k\in \N$.  We write $$k\cdot A:=\{x_1+\cdots+x_k \ : \ x_1,\ldots,x_k\in A\}$$ and $$ k\dot A:=\{kx \ : \ x\in A\}.$$  Of course $kA\subseteq k\cdot A$.

Throughout this book, $\log$ always denotes the logarithm base $2$.

\tableofcontents

\mainmatter

\chapter{Ultrafilters}

\section{Basics on ultrafilters}
Throughout this chapter, we let $S$ denote an infinite set.
\begin{definition}
A (proper) \emph{filter} on $S$ is a set $\mathcal{F}$ of subsets of $S$ (that is, $\mathcal{F}\subseteq \mathcal P(S)$) such that:
\begin{itemize}
\item $\emptyset\notin \mathcal{F}$, $S\in \mathcal{F}$;
\item if $A,B\in \mathcal{F}$, then $A\cap B\in \mathcal{F}$;
\item if $A\in \mathcal{F}$ and $A\subseteq B$, then $B\in \mathcal{F}$.
\end{itemize}
\end{definition}

We think of elements of $\mathcal{F}$ as ``big'' sets (because that is what filters do, they catch the big objects).  The first and third axioms are (hopefully) intuitive properties of big sets.  Perhaps the second axiom is not as intuitive, but if one thinks of the complement of a big set as a ``small'' set, then the second axiom asserts that the union of two small sets is small (which is hopefully more intuitive).

\begin{exercise}
Set $\mathcal{F}:=\{A\subseteq S\  | \ S\setminus A \text{ is finite}\}$.  Prove that $\mathcal{F}$ is a filter on $S$, called the \emph{Frech\'et} or \emph{cofinite} filter on $S$.
\end{exercise}

\begin{exercise}\label{generate}
Suppose that $\mathcal D$ is a set of subsets of $S$ with the \emph{finite intersection property}:  whenever $D_1,\ldots,D_n\in \mathcal D$, we have $D_1\cap\cdots \cap D_n\not=\emptyset$.  Set 
$$\langle \mathcal D\rangle:=\{E\subseteq S \ | \ D_1\cap\cdots\cap D_n\subseteq E \text{ for some }D_1,\ldots,D_n\in \mathcal D\}.$$  Show that $\langle \mathcal D\rangle$ is the smallest filter on $S$ containing $\mathcal D$, called the \emph{filter generated by $\mathcal D$}.
\end{exercise}

If $\mathcal{F}$ is a filter on $S$, then a subset of $S$ cannot be simultaneously big and small (that is, both it and its complement belong to $\mathcal{F}$), but there is no requirement that it be one of the two.  It will be desirable (for reasons that will become clear in a moment) to add this as an additional property:

\begin{definition}
If $\mathcal{F}$ is a filter on $S$, then $\mathcal{F}$ is an \emph{ultrafilter}\index{ultrafilter} if, for any $A\subseteq S$, either $A\in \mathcal{F}$ or $S\setminus A\in \mathcal{F}$ (but not both!).
\end{definition}

Ultrafilters are usually denoted by $\u$.  Observe that the Frech\'et filter on $S$ is not an ultrafilter since there are sets $A\subseteq S$ such that $A$ and $S\setminus A$ are both infinite.  

The following exercise illustrates one of the most important properties of ultrafilters \index{ultrafilter}.

\begin{exercise}\label{Exercise:part}
A filter $\mathcal{F}$ on $S$ is an ultrafilter if and only if whenever
$A_1,\ldots,A_n$ are pairwise disjoint subsets of $S$ such that 
$A_1\cup \cdots \cup A_n\in \mathcal{F}$, there is a unique $i\in \{1,\ldots,n\}$ such that 
$A_i\in \mathcal{F}$.
\end{exercise}

We have yet to see an example of an ultrafilter\index{ultrafilter}.  Here is a ``trivial'' source of ultrafilters:

\begin{definition}
Given $s\in S$, set $\u_s:=\{A\subseteq S\  | \ s\in A\}$.
\end{definition}

\begin{exercise}
For $s\in S$, prove that $\u_s$ is an ultrafilter\index{ultrafilter} on $S$, called the \emph{principal ultrafilter generated by $s$}.
\end{exercise}

We say that an ultrafilter\index{ultrafilter} $\u$ on $S$ is \emph{principal} if $\u=\u_s$ for some $s\in S$.  Although principal ultrafilters settle the question of the existence of ultrafilters, they will turn out to be useless for our purposes, as we will soon see.  From a philosophical viewpoint, principal ultrafilters fail to capture the idea that sets belonging to the ultrafilter are large, for $\{s\}$ belongs to the ultrafilter $\u_s$ and yet hardly anyone would dare say that the set $\{s\}$ is large!

\begin{exercise}\label{ultrafin}
Prove that an ultrafilter\index{ultrafilter} $\u$ on $S$ is nonprincipal if and only if there is no \emph{finite} set $A\subseteq S$ such that $A\in \u$ if and only if
$\u$ extends the Frech\'et filter.
\end{exercise}

We now would like to prove the existence of nonprincipal ultrafilters\index{ultrafilter}.  The following exercise will be the key to doing this.

\begin{exercise}
Suppose that $\mathcal{F}$ is a filter on $S$.  Then $\mathcal{F}$ is an ultrafilter on $S$ if and only if it is a maximal filter, that is, if and only if, whenever $\mathcal{F}'$ is a filter on $S$ such that $\mathcal{F}\subseteq \mathcal{F}'$, we have $\mathcal{F}=\mathcal{F}'$.
\end{exercise}

Since it is readily verified that the union of an increasing chain of filters on $S$ containing
a filter $\mathcal{F}$ is once again a filter on $S$ containing $\mathcal{F}$,
the previous exercises and Zorn's lemma yield the following:

\begin{theorem}
Nonprincipal ultrafilters on $S$ exist\index{ultrafilter}.
\end{theorem}

\begin{exercise}\label{imageultrafilter}
Suppose that $f:S\to T$ is a function between sets.  Then given any ultrafilter $\u$ on $S$, the set
$$f(\u):=\{A\subseteq T \ : \ f^{-1}(A)\in \u\}$$ is an ultrafilter on $T$, called the \emph{image ultrafilter of $\u$ under $f$}.
\end{exercise} 

\section{The space of ultrafilters $\beta S$}

In this section, $S$ continues to denote an infinite set.  Since topological matters are the subject of this subsection, we will also treat $S$ as a topological space equipped with the discrete topology.

The set of ultrafilters\index{ultrafilter} on $S$ is denoted $\beta S$.  There is a natural topology on $\beta S$ obtained by declaring, for $A\subseteq S$, 
the following sets as basic open sets:
$$U_A:=\{\u \in\beta S \ : \ A\in \u\}.$$ 
(Note that the $U_A$'s are indeed a base for a topology as $U_A\cap U_B=U_{A\cap B}$.)  Since the complement of $U_A$ in $\beta S$ is $U_{S\setminus A}$, we see that the basic open sets are in fact clopen.  Note also that $\beta S$ is Hausdorff:  if $\u,\mathcal{V}\in \beta S$ are distinct, take $A\subseteq S$ with $A\in \u$ and $S\setminus A\in \mathcal{V}$. Then $\u\in U_A$ and $\mathcal{V}\in U_{S\setminus A}$ and clearly $U_A$ and $U_{S\setminus A}$ are disjoint.  

\begin{exercise}\label{ex-Phiclosed}
Let $\Phi(X)$ be any property of subsets of $S$.
Then the following (possibly empty) set is closed:
$$\mathcal{C}_\Phi:=\{\u\in\beta S\ :\ \forall A\in\u\ \Phi(A)\}.$$
\end{exercise}

\begin{theorem}
$\beta S$ is a compact space.
\end{theorem}

\begin{proof}
It is enough to show that every covering of $\beta S$ by basic open sets has a finite subcover.  Let $(A_i)$ be a family of subsets of $S$ such that $(U_{A_i})$ covers $\beta S$.  Suppose, towards a contradiction, that this cover of $\beta S$ has no finite subcover.  We claim then that $(S\setminus A_i)$ has the finite intersection property.  Indeed, given $J\subseteq I$ finite, there is $\u \in \beta S\setminus \bigcup_{i\in J}U_{A_i}$, whence $S\setminus A_i\in \u$ for each $i\in J$, and hence $\bigcap_{i\in J}(S\setminus A_i)\not=\emptyset$.  It follows that there is a $\u\in \beta S$ such that $S\setminus A_i\in \u$ for all $i\in I$, contradicting the fact that $\u\in U_{A_i}$ for some $i\in I$.
\end{proof}

We identify $S$ with the set of principal ultrafilters\index{ultrafilter} on $S$. Under this identification, $S$ is dense in $\beta S$:  if $A\subseteq S$ is nonemtpy and $s\in A$, then the principal ultrafilter $\u_s\in U_A$.  Thus, $\beta S$ is a compactification of $S$.  In fact,  we have:

\begin{theorem}\label{sc}
$\beta S$ is the Stone-\v{C}ech compactification of $S$.
\end{theorem}

We remind the reader that the Stone-\v{C}ech compactification of $S$ is the unique compactifiction $X$ of $S$ with the following property:  any function $f:S\to Y$ with $Y$ compact Hausdorff has a unique continuous extension $\tilde{f}:X\to Y$.   In order to prove the previous theorem, we will first need the following lemma, which is important in its own right:

\begin{lemma}\label{ultralimit}
Suppose that $Y$ is a compact Hausdorff space and $(y_s)_{s\in S}$ is a familiy of elements of $Y$ indexed by $S$.  Then for any $\u\in \beta S$, there is a unique element $y\in Y$ with the property that, for any open neighborhood $U$ of $y$, we have $\{s\in S\ : \ y_s\in U\}\in \u$. 
\end{lemma}

\begin{proof}
Suppose, towards a contradiction, that no such $y$ exists.  Then for every $y\in Y$, there is an open neighborhood $U_y$ of $y$ such that $\{s\in S \ : \ y_s\in U_y\}\notin \u$.  By compactness, there are $y_1,\ldots,y_n\in Y$ such that $Y=U_{y_1}\cup\cdots\cup U_{y_n}$.  There is then a unique $i\in\{1,\ldots,n\}$ such that $\{s\in S \ : \ y_s\in U_{y_i}\}\in \u$, yielding the desired contradiction.

The uniqueness of $y$ follows from the fact that $Y$ is Hausdorff together with the fact that $\u$ does not contain two disjoint sets.
\end{proof}

\begin{definition}
In the context of the previous lemma, we call the unique $y$ the \emph{ultralimit of $(y_s)$ with respect to $\u$}, denoted $\lim_{s,\u} y_s$ or simply just $\lim_\u y_s$.
\end{definition}

\begin{proof}[of Theorem \ref{sc}]
Suppose that $f:S\to Y$ is a function into a compact Hausdorff space.  Define $\tilde{f}:\beta S\to Y$ by $\tilde{f}(\u):=\lim_\u f(s)$, which exists by Lemma \ref{ultralimit}.  It is clear that $\tilde{f}(\u_s)=f(s)$, so $\tilde{f}$ extends $f$.  We must show that $\tilde{f}$ is continuous.  Fix $\u\in \beta S$ and let $U$ be an open neighborhood of $\tilde{f}(\u)$ in $Y$.  Let $V\subseteq U$ be an open neighborhood of $\tilde{f}(\u)$ in $Y$ such that $\overline{V}\subseteq U$.  
(This is possible because every compact Hausdorff space is \emph{regular}, that is,
every point has a base of closed neighborhoods.) Take $A\in \u$ such that $f(s)\in V$ for $s\in A$.  Suppose $\mathcal{V}\in U_A$, so $A\in \mathcal{V}$. Then $\lim_{\mathcal{V}} f(s)\in \overline{V}\subseteq U$, so $U_A\subseteq \tilde{f}^{-1}(U)$.

Recall that the set of points where two continuous functions with values in a Hausdroff space agree is closed. As $S$ is dense in $\beta S$, it follows that such a continuous extension is unique.
\end{proof}

Now that we have shown that $\beta S$ is the Stone-\v{C}ech compactification of $S$, given $f:S\to Y$ where $Y$ is a compact Hausdorff space, we will let $\beta f:\beta S\to Y$ denote the unique continuous extension of $f$.

\begin{definition}\label{def-ku}
Fix $k\in \N$.  Let $m_k:\N\to \N$ be defined by $m_k(n):=kn$.  Then for $\u\in \beta\N$, we set $k\u:=(\beta m_k)(\u)$.
\end{definition}

Note that $A\in k\u\Leftrightarrow A/k:=\{n\in\N\mid nk\in A\}\in\u$.
The ultrafilters\index{ultrafilter} $k\u$ will play an important role in Chapter \ref{partreg}.

\begin{exercise}
Given $A\subseteq S$, show that $\overline{A}=U_A$, where $\overline{A}$ denotes the closure of $A$ in $\beta S$.
\end{exercise}

Let $\ell ^{\infty }(S)$ denote the space of bounded real-valued functions on $S$.  Given $f\in \ell ^{\infty }(S)$, take $r\in \R^{>0}$ such that $f(S)\subseteq [-r,r]$, whence we may consider its unique continuous extension $\beta f:\beta S\to [-r,r]$.  Note that the function $\beta f$ does not depend on the choice of $r$.  The following exercise will be useful in Chapter \ref{sumset}.

\begin{exercise}
The function $f\mapsto \beta f$ is an isomorphism between $\ell ^{\infty }(S)$ and $C(\beta S)$ as Banach spaces.
\end{exercise}

\section{The case of a semigroup}

We now suppose that $S$ is the underlying set of a semigroup $(S,\cdot)$.  Then one can extend the semigroup operation $\cdot$ to a semigroup operation $\odot$ on $\beta S$ by declaring, for $\u,\mathcal{V}\in \beta S$ and $A\subseteq S$, that
$$A\in \u\odot \mathcal{V} \Leftrightarrow \{s\in S \ : \ s^{-1}A\in \mathcal{V}\}\in \u.$$  
Here, $s^{-1}A:=\{t\in S \ : \ st\in A\}$.  In other words, $\u\odot \mathcal{V}=\lim_{s,\u}(\lim_{t,\mathcal{V}}s\cdot t)$, where these limits are taken in the compact space $\beta S$.  In particular, note that $\u_s\odot \u_t=\u_{s\cdot t}$, so this operation on $\beta S$ does indeed extend the original operation on $S$.  It is also important to note that, in general, ultralimits do not commutate and thus, in general, $\u\odot \mathcal{V}\not=\mathcal{V}\odot \u$, even if $(S,\cdot)$ is commutative.  (See Chapter \ref{hypgenerator} for more on this lack of commutativity.)

The following theorem is the key to many applications of ultrafilter/nonstandard methods in Ramsey theory.

\begin{theorem}\label{betaSsemitop}
$(\beta S, \odot)$ is a compact, \emph{right topological semigroup}, that is, $\odot $ is a semigroup operation on  the compact space $\beta S$ such that, for each $\mathcal{V}\in \beta S$, the map $\u\mapsto \u\odot \mathcal{V}: \beta S \to \beta S$ is continuous. 
\end{theorem}

\begin{proof}
Fix $\mathcal{V}\in \beta S$ and let $\rho_{\mathcal{V}}:\beta S\to \beta S$ be defined by $\rho_{\mathcal{V}}(\u):=\u\odot \mathcal{V}$.  We need to show that $\rho_{\mathcal{V}}$ is continuous.  Towards this end, fix $A\subseteq S$. We must show that $\rho_{\mathcal{V}}^{-1}(U_A)$ is open.  Let $B:=\{s\in S \ : \ s^{-1}A\in \mathcal{V}\}$.  It remains to note that $\rho_{\mathcal{V}}^{-1}(U_A)=U_B$.
\end{proof}

One can introduce some notation to express more succinctly the semigroup operation in $\beta S$. Given $A\subseteq S$ and $t\in S$, one defines $At^{-1}$ to be the set $\{ s\in S: s+t\in A\}$. Similarly, for $A\subseteq S$ and $\mathcal{V}\in \beta S$, one defines $A \mathcal{V}^{-1}$ to be $\{ x\in S: Ax^{-1}\in \mathcal{V}\}$. Then for $\mathcal{U},\mathcal{V}\in \beta S$ and $A\subseteq S$, one has that $S\in \mathcal{U}\odot \mathcal{V}$ if and only if $A\mathcal{V}^{-1}\in \mathcal{U}$.


\section*{Notes and references} 
The notion of ultrafilter was introduced by H.\ Cartan \cite{cartan_theorie_1937,cartan_filtres_1937}
in 1937 to study convergence in topological spaces.
Ultrafilters and the corresponding construction of ultraproduct 
are common tools in mathematical logic, but they also
found many applications in other fields of mathematics, 
especially in topology, algebra, and functional analysis.
A classic reference on ultrafilters is the book \emph{``Ultrafilters''}
by Comfort and Negrepontis \cite{comfort_ultrafilters_1977}.
See also the more recent \cite{bergelson_ultrafilters_2010} for a review of
ultrafilters across mathematics.
The extension of the operation on a semigroup to the space of ultrafilters can be seen as a particular instance of the notion of Arens product on the bidual of a Banach algebra \cite{arens_adjoint_1951}. Indeed, one can regard the space of ultrafilters over a semigroup $S$ as a susbpace of the second dual of the Banach algebra $\ell _1{S}$ endowed with the usual convolution product. This was the initial approach taken in the study of the Stone-\v{C}ech compactification since the 1950s \cite{day_amenable_1957,civin_second_1961}. Its realization as a space of ultrafilters was first explicitly considered by Ellis \cite{ellis_lectures_1969}.


\chapter{Nonstandard analysis}\label{ch-nsa}

If one wants to present the methods of nonstandard analysis
in their full generality and with full rigor, then notions and tools
from mathematical logic such as ``first-order formula'' or 
``elementary extension'' are definitely needed.
However, we believe that a gentle introduction to the basics of 
nonstandard methods and their use in combinatorics does not directly
require any technical machinery from logic. Only at a later stage,
when advanced nonstandard techniques are applied
and their use must be put on firm foundations, detailed knowledge of notions from logic will be necessary.

We will begin with presenting the main properties of the nonstandard versions of the
natural, integer, rational, and real numbers, that will be named by
adding the prefix ``hyper''. Then we will introduce the fundamental principle
of nonstandard analysis, namely the \emph{transfer principle}
of the \emph{star map}.  While at this stage the treatment will still be informal, it will still be sufficient for the reader to gain a first idea of how nonstandard methods can be used in applications.

In the appendix, we give sound and rigorous foundations
to nonstandard analysis in full generality
by formally introducing first order logic. The reader already familiar with nonstandard methods
can proceed directly to the next chapter.

\section{Warming-up}\label{sec-warmingup}

To begin with, let us recall the following notions, which are at the 
very base of nonstandard analysis. 


\begin{definition}\label{definition:infinitesimal}
An element $\varepsilon$ of an ordered field $\mathbb{F}$ is \emph{infinitesimal}\index{hyperreal number!infinitesimal} (or \emph{infinitely small})
if $|\varepsilon|<\frac{1}{n}$ for every $n\in\N$.
A number $\Omega$ is \emph{infinite}\index{hyperreal number!infinite} if $|\Omega|>n$ for every $n\in\N$.
\end{definition}

In Definition \ref{definition:infinitesimal} we identify a natural number $n$ with the element of $\mathbb{F}$ obtained as the $n$-fold sum of $1$ by itself. 
Clearly, a nonzero number is infinite if and only if its
reciprocal is infinitesimal. We say that a
number is \emph{finite} or \emph{bounded}
if it is not infinite. 

\begin{exercise}
\

\begin{enumerate}
\item
If $\xi$ and $\zeta$ are finite, then
$\xi+\zeta$ and $\xi \cdot \zeta$ are finite.
\item
If $\xi$ and $\zeta$ are infinitesimal\index{hyperreal number!infinitesimal}, then
$\xi +\zeta $ is infinitesimal.
\item
If $\xi$ is infinitesimal and $\zeta$ is finite, then
$\xi\cdot\zeta$ is infinitesimal.
\item
If $\xi$ is infinite and $\zeta$ is not infinitesimal, then
$\xi\cdot\zeta$ is infinite.
\item
If $\xi\neq 0$ is infinitesimal and $\zeta$ is not infinitesimal, then $\xi/\zeta$ is infinitesimal.
\item
If $\xi$ is infinite\index{hyperreal number!infinite} and $\zeta$ is finite, then $\xi/\zeta$ is infinite.

\end{enumerate}
\end{exercise}

Recall that an ordered field $\mathbb{F}$ is \emph{Archimedean}
if for every positive $x\in\mathbb{F}$ there exists $n\in\N$\index{ordered field!non-Archimedean}
such that $n x>1$.

\begin{exercise}
The following properties are equivalent for an ordered field $\mathbb{F}$\index{ordered field}:
\begin{enumerate}
\item
$\mathbb{F}$ is non-Archimedean;
\item
There are nonzero infinitesimal numbers\index{hyperreal number!infinitesimal} in $\mathbb{F}$;
\item
The set of natural numbers is bounded in $\mathbb{F}$.
\end{enumerate}
\end{exercise}

We are now ready to introduce the nonstandard reals.

\begin{definition}
The \emph{hyperreal field} $\starR$ \index{hyperreal number} is a proper extension
of the ordered field $\R$ that satisfies additional properties (to be 
specified further on). The element of $\starR$ are called \emph{hyperreal numbers}.
\end{definition}

By just using the above incomplete definition, the following is proved.

\begin{proposition}\label{nonArchimedean}
The hyperreal field $\starR$\index{hyperreal number} is non-Archimedean, and hence
it contains nonzero infinitesimals and infinite numbers\index{hyperreal number!infinitesimal}.
\end{proposition}

\begin{proof}
Since $\starR$ is a proper extension of the real field, we can pick
a number $\xi\in\starR\setminus\R$. Without loss of generality,
let us assume $\xi>0$. If $\xi$ is infinite, then we are done.
Otherwise, by the completeness property of $\R$, we can consider
the number $r=\inf\{x\in\R\mid x>\xi\}$. (Notice that it
may be $r<\xi$.)
It is readily checked that $\xi-r$ is a nonzero infinitesimal number.
\end{proof}

We remark that, as a non-Archimedean field\index{ordered field!non-Archimedean}, $\starR$ is \emph{not} complete
(e.g., the set of infinitesimals is bounded but has no least upper bound).
We say that two hyperreal numbers are \emph{infinitely close} if their difference is infinitesimal. The nonstandard counterpart of completeness is given by
the following property. 

\begin{theorem}[Standard Part]
Every finite hyperreal number\index{hyperreal number} $\xi\in\starR$ is 
infinitely close to a unique real number $r\in\R$,
called the \emph{standard part}\index{standard part} of $\xi$. 
In this case, we use the notation $r=\st(\xi)$.
\end{theorem}

\begin{proof}
By the completeness of $\R$, we can set
$\st(\xi):=\inf\{x\in\R\mid x>\xi\}=\sup\{y\in\R\mid y<r\}$.
By the supremum (or infimum) property, 
it directly follows that $\st(\xi)$ is infinitely close to $\xi$.
Moreover, $\st(\xi)$ is the unique real number with that property, 
since infinitely close real numbers are necessarily equal.
\end{proof}

It follows that every finite hyperreal number\index{hyperreal number} $\xi$ has a unique representation
in the form $\xi=r+\varepsilon$ where $r=\st(\xi)\in\R$
and $\varepsilon $ is infinitesimal\index{hyperreal number!infinitesimal}.  Given finite hyperreals $\xi$ and $\zeta$, it is sometimes convenient to write $\xi\gtrsim \zeta$ to mean $\st(\xi)\geq \st(\zeta)$.

The following are the counterparts in the nonstandard
setting of the familiar properties of limits of real sequences.

\begin{exercise}
For all finite hyperreal numbers\index{hyperreal number} $\xi,\zeta$:

\begin{enumerate}
\item
$\st(\xi)<\st(\zeta)\Rightarrow\xi<\zeta\Rightarrow\st(\xi)\le\st(\zeta)$;
\item
$\st(\xi+\zeta)=\st(\xi)+\st(\zeta)$;
\item
$\st(\xi\cdot\zeta)=\st(\xi)\cdot\st(\zeta)$;
\item
$\st(\frac{\xi}{\zeta})=\frac{\st(\xi)}{\st(\zeta)}$ whenever
$\zeta$ is not infinitesimal\index{hyperreal number!infinitesimal}.
\end{enumerate}
\end{exercise}

\begin{definition}
The \emph{ring of hyperinteger numbers} $\starZ$ is an unbounded
discretely ordered subring of $\starR$ that
satisfies special properties (to be specified further on), including the following:
\begin{itemize}
\item
For every $\xi\in\starR$ there exists $\zeta\in\starZ$ 
with $\zeta\le \xi<\zeta+1$. Such a $\zeta$ is called the 
\emph{hyperinteger part} of $\xi$, denoted $\zeta=\lfloor\xi\rfloor$.
\end{itemize}
\end{definition}

Since $\starZ$ is discretely ordered, notice that its finite part
coincides with $\Z$. This means that for every $z\in\Z$
there are \emph{no} hyperintegers $\zeta\in\starZ$ such that $z<\zeta<z+1$.

\begin{definition}
The \emph{hypernatural numbers} \index{hypernatural number} $\starN$ are the positive
part of $\starZ$. Thus $\starZ=-\starN\cup\{0\}\cup\starN$, where 
$-\starN=\{-\xi\mid\xi\in\starN\}$
are the negative hyperintegers.
\end{definition}

\begin{definition}
The field of \emph{hyperrational numbers}
$\starQ$ is the quotient field of $\starZ$. Thus hyperrational numbers $\zeta\in\starQ$ can be represented
as fractions $\zeta=\frac{\xi}{\nu}$ where $\xi\in\starZ$ and $\nu\in\starN$.
\end{definition}

\begin{exercise}
The hyperrational numbers $\starQ$ are dense in $\starR$, that is,
for every pair $\xi<\xi'$ in $\starR$ there exists $\eta\in\starQ$ such that
$\xi<\eta<\xi'$.
\end{exercise}

We remark that, although still incomplete, our definitions suffice to get
a clear picture of the order-structure of the 
two main nonstandard objects that we will consider here,
namely the hypernatural numbers \index{hypernatural number}
$\starN$ and the hyperreal line\index{hyperreal number} $\starR$.
In particular, let us focus on the nonstandard natural numbers.
One possible way (but certainly not the only possible way)
to visualize them is the following:

\begin{itemize}
\item
The \emph{hypernatural numbers} $\starN$ are the extended version
of the natural numbers that is obtained by allowing the use of a 
``mental telescope" 
to also see infinite numbers beyond the finite ones.
\end{itemize}

So, beyond the usual finite numbers $\N=\{1,2,3,\ldots\}$,
one finds infinite numbers\index{hypernatural number!infinite} $\xi>n$ for all $n\in\N$.
Every $\xi\in\starN$ has a successor $\xi+1$,
and every non-zero $\xi\in\starN$ has a predecessor $\xi-1$.
$$\starN\ =\ \big\{\underbrace{1,2,3,\ldots,n,\ldots}_{\text{\small finite numbers}}
\quad \underbrace{\ldots, N-2, N-1, N, N+1, N+2, \ldots}_{\text{\small infinite numbers}}\
\big\}$$

Thus the set of finite numbers $\N$ does not have a
greatest element and the set of infinite numbers\index{hypernatural number!infinite}
$\starN\setminus\N$ does not have a least element,
whence $\starN$ is \emph{not} well-ordered. Sometimes we will write
$\nu>\N$ to mean that $\nu\in\starN$ is infinite.

\begin{exercise}
Consider the equivalence relation $\sim_f$ on $\starN$
defined by setting
$\xi\sim_f\zeta$ if $\xi-\zeta$ is finite. The corresponding equivalence classes are called \emph{galaxies}.
The quotient set $\starN/{\sim_f}$ inherits an order structure, which turns it into a dense linearly ordered set with least
element $[1]=\N$ and with no greatest element.
\end{exercise}

\section{The star map and the transfer principle\index{transfer principle}}

As we have seen in the previous section, corresponding
to each of the sets $\N, \Z, \Q, \R$, one
has a \emph{nonstandard extension}, namely the 
sets $\starN, \starZ, \starQ, \starR$, respectively.
A defining feature of nonstandard analysis
is that one has a canonical way of extending
\emph{every} mathematical object $A$ under study
to an object $\starA$ which
inherits all ``elementary'' properties
of the initial object. 

\begin{definition}\label{starmap}
The \emph{star map} is a function that associates to each
``mathematical object" $A$ under study its \emph{hyper-extension}
(or \emph{nonstandard extension}) $\starA$
in such a way that the following holds:
\begin{itemize}
\item
\emph{Transfer principle:}\ 
Let $P(A_{1},\ldots,A_{n})$ be an ``elementary property"
of the mathematical objects $A_1,\ldots,A_n$.
Then $P(A_{1},\ldots,A_{n})$ is true if and only if
$P(\starA_{1},\ldots,\starA_{n})$ is true:
$$\ \quad P(A_{1},\ldots,A_{n})\ \Longleftrightarrow\ 
P(\starA_{1},\ldots,\starA_{n}).$$
\end{itemize}
\end{definition}

One can think of hyper-extensions as
a sort of weakly isomorphic copy of the initial objects.
Indeed, by the \emph{transfer principle}\index{transfer principle},
an object $A$ and its hyper-extension $\starA$ are 
indistinguishable as far as
their ``elementary properties" are concerned.
Of course, the crucial point here
is to precisely determine which properties are ``elementary"
and which are not.

Let us remark that the above
definition is indeed incomplete in that the 
notions of ``mathematical object'' and of ``elementary property''
are still to be made precise and rigorous.
As anticipated in the introduction, we will do this gradually.

To begin with, it will be enough to include 
in our considered ``mathematical objects" the following:

\begin{enumerate}
\item
Real numbers
and tuples of real numbers;
\item
All sets $A\subseteq\R^k$ of real tuples, and
all functions $f:A\to B$ between them;
\item
All sets made up of objects in (1) and (2), including, \emph{e.g.},
the families $\F\subseteq\bigcup_k\mathcal{P}(\R^k)$ of sets
of real $k$-tuples, and the families of functions
$\G\subseteq\Fun(\R^k,\R^h)$.
\end{enumerate}

More generally,
every structure under study could be safely taken as one of
our ``mathematical objects".\footnote
{~According to the usual set-theoretic foundational framework,
every mathematical object is identified with a set
(see Remark \ref{reductionism} in the Appendix).
However, here we will stick to the common perception that
considers numbers, ordered pairs, relations, functions, and sets as
mathematical objects of distinct nature.}

As for the notion of ``elementary property'',
we will start working with a semi-formal definition. 
Although not fully rigorous from a logical point of view, 
it may nevertheless looks perfectly fine to many,
and we believe that it can be safely 
adopted to get introduced to nonstandard analysis
and to familiarize oneself with its basic notions and tools.

\begin{definition}
A property $P$ is \emph{elementary}
if it can be expressed by an \emph{elementary formula},
that is, by a formula where:
\begin{enumerate}
\item
Besides the usual logical connectives (``not'', ``and'', ``or'',
``if $\ldots$ then'', ``if and only if'')
and the quantifiers (``there exists'', ``for every'')
only the basic notions of equality, membership, set,
ordered $k$-tuple, 
$k$-ary relation, domain, range, function, value of a function 
at a given point, are involved;
\item
The scope of every quantifier is \emph{bounded}, 
that is, quantifiers always occur
in the form ``there exists $x\in X$'' or ``for every $y\in Y$''
for specified sets $X, Y$. More generally,
also nested quantifiers 
``$Q\,x_1\in x_2\ \text{and}\ Q\,x_2\in x_3\ \ldots\ \text{and}\ 
Q\,x_n\in X$'' 
are allowed, where $Q$ is either ``there exists" or ``for every",
$x_1,\ldots,x_n$ are variables, and $X$ is a specified set. 
\end{enumerate}
\end{definition}

An immediate consequence of the \emph{transfer} principle is that
all fundamental mathematical constructions are preserved 
under the star map, with the only two
relevant exceptions being \emph{powersets} and 
\emph{function sets} (see Proposition \ref{properinclusions}). 
Below we give three comprehensive lists in three distinct propositions, 
the first one about
sets and ordered tuples, the second one about relations,
and the third one about functions. Until the notion of ``elementary property" has been made 
 precise, one can take those properties
as axioms for the star map.

\begin{proposition}\label{starpropsets}
\
\begin{enumerate}
\item\label{starequality}
$a=b\Leftrightarrow{}^*a={}^*b$.
\item\label{starmembership}
$a\in A\Leftrightarrow{}^*a\in\starA$.
\item\label{starset}
$A$ is a set if and only if $\starA$ is a set.
\item
${}^*\emptyset=\emptyset$.
\end{enumerate}

If $A,A_1,\ldots,A_k,B$ are sets:
\begin{enumerate}
\setcounter{enumi}{4}
\item
$A\subseteq B\Leftrightarrow \starA\subseteq\starB$.
\item
${}^*(A\cup B)=\starA\cup\starB$.
\item
${}^*(A\cap B)=\starA\cap\starB$.
\item\label{starsetminus}
${}^*(A\setminus B)=\starA\setminus\starB$.
\item\label{startuple}
${}^*\{a_1,\ldots,a_k\}=\{{}^*a_1,\ldots,{}^*a_k\}$.
\item\label{starorderedtuple}
${}^*(a_1,\ldots,a_k)=({}^*a_1,\ldots,{}^*a_k)$.
\item
${}^*(A_1\times\ldots\times A_k)=\starA_1\times\ldots\times\starA_k$.
\item
${}^*\{(a,a)\mid a\in A\}=\{(\xi,\xi)\mid \xi\in \starA\}$.
\end{enumerate}

If $\F$ is a family of sets:
\begin{enumerate}
\setcounter{enumi}{12}
\item\label{starfamilymembership}
${}^*\{(x,y)\mid x\in y\in \F\}=\{(\xi,\zeta)\mid \xi\in\zeta\in{}^*\F\}$.
\item\label{starfamilyunion}
${}^*(\bigcup_{F\in\F}F)=\bigcup_{G\in{}^*\F}G$.
\end{enumerate}
\end{proposition}

\begin{proof}
Recall that by our definition, the notions of equality, membership,
set, and ordered $k$-tuple are elementary;
thus by direct applications of \emph{transfer} one obtains 
(\ref{starequality}), (\ref{starmembership}), (\ref{starset}),
and (\ref{starorderedtuple}), respectively.
All other properties are easily proved by considering
suitable elementary formulas. 
As examples, we will consider here only three of them.

(\ref{starsetminus}). The property ``$C=A\setminus B$"
is elementary, because it is formalized by the
elementary formula: $$\text{``}\forall x\in C\ (x\in A\ \text{or}\ x\notin B)\ \text{and}\ 
\forall x\in A\ (x\notin B\Rightarrow x\in C)\text{''}.$$
So, by \emph{transfer}, we have that $C=A\setminus B$ holds if and only if
$$\text{``}\forall x\in {}^{\ast}C\ (x\in \starA\ \text{or}\ x\notin {}^{\ast}B)\ \text{and}\ 
\forall x\in \starA\ (x\notin {}^{\ast}B\Rightarrow x\in {}^{\ast}C)\text{''},$$
that is, if and only if $\starC=\starA\setminus\starB$.

(\ref{startuple})
The property ``$C=\{a_1,\ldots,a_k\}$" is formalized by the elementary
formula: ``$a_1\in C\ \text{and}\ \ldots$
$\ldots\ \text{and}\ a_k\in C\ \text{and}\ 
\forall x\in C\ (x=a_1\ \text{or}\ \ldots\ \text{or}\ x=a_k)$''. So, we can
apply \emph{transfer} and obtain that ${}^*C=\{{}^*a_1,\ldots,{}^*a_k\}$.

(\ref{starfamilyunion}). The property ``$A=\bigcup_{F\in\F}F$'' is formalized by the
elementary formula: ``$\forall x\in A\ (\exists y\in\F\ \text{with}\ x\in y)\ \text{and}\ 
\forall y\in\F\ \forall x\in y\ (x\in A)$.''
Then by \emph{transfer} one gets ``$\starA=\bigcup_{y\in{}^*\F}y$."
\end{proof}

\begin{proposition}\label{starproprelations}
\

\begin{enumerate}
\item\label{starrelation}
$R$ is a $k$-ary relation if and only if ${}^*R$ is a $k$-ary relation.
\end{enumerate}

If $R$ is a binary relation:
\begin{enumerate}
\setcounter{enumi}{1}
\item\label{stardomain}
${}^*\{a\mid \exists b\ R(a,b)\}=\{\xi\mid \exists \zeta\ {}^*R(\xi,\zeta)\}$,
that is, ${}^*\text{domain}(R)=\text{domain}({}^*R)$.
\item\label{starrange}
${}^*\{b\mid \exists a\ R(a,b)\}=\{\zeta\mid \exists \xi\ {}^*R(\xi,\zeta)\}$,
that is, ${}^*\text{range}(R)=\text{range}({}^*R)$.
\item\label{starbinaryexchange}
${}^*\{(a,b)\mid R(b,a)\}=\{(\xi,\zeta)\mid {}^*R(\zeta,\xi)\}$.
\end{enumerate}

If $S$ is a ternary relation:
\begin{enumerate}
\setcounter{enumi}{4}
\item
${}^*\{(a,b,c)\mid S(c,a,b)\}=\{(\xi,\zeta,\eta)\mid {}^*S(\xi,\eta,\zeta)\}$.
\item
${}^*\{(a,b,c)\mid S(a,c,b)\}=\{(\xi,\zeta,\eta)\mid {}^*S(\xi,\eta,\zeta)\}$.
\end{enumerate}
\end{proposition}

\begin{proof}
(\ref{starrelation}), (\ref{stardomain}), and (\ref{starrange})
are proved by direct applications of \emph{transfer},
because the notions of $k$-ary relation, domain, and range
are elementary by definition.

(\ref{starbinaryexchange}). The property ``$C=\{(a,b)\mid R(b,a)\}$'' is formalized by the
the conjunction of the elementary formula
 ``$\forall z\in C\ \exists x\in\text{domain}(R)\ 
\exists y\in\text{range}(R)\ \text{s.t.}\ R(x,y)\ \text{and}\ z=(y,x)$''
and the elementary
formula ``$\forall x\in\text{domain}(R)\ \forall y\in\text{range}(R)\ (y,x)\in C$".
Thus \emph{transfer}\index{transfer principle} applies and one obtains 
$\starC=\{(\xi,\zeta)\mid(\zeta,\xi)\in{}^{\ast} R\}$.

(5) and (6) are proved by considering similar elementary formulas as in 
(\ref{starbinaryexchange}).
\end{proof}

\begin{proposition}\label{starpropfunctions}
\

\begin{enumerate}
\item\label{starfunc}
$f$ is a function if and only if $\starf$ is a function.
\end{enumerate}

If $f,g$ are functions and $A,B$ are sets:
\begin{enumerate}
\setcounter{enumi}{1}
\item\label{starfdomain}
${}^*\text{domain}(f)=\text{domain}(\starf)$.
\item\label{starfrange}
${}^*\text{range}(f)=\text{range}(\starf)$.
\item\label{starfAB}
$f:A\to B$ if and only if $\starf:\starA\to\starB$.\footnote
{~Recall that notation $f:A\to B$ means that $f$ is a function
with $\text{domain}(f)=A$ and $\text{range}(f)\subseteq B$.}
\item\label{stargraph}
${}^*\text{graph}(f)=\text{graph}(\starf)$.
\item\label{starvalue}
${}^*(f(a)=(\starf)({}^*a)$ for every $a\in\text{domain}(f)$.
\item 
If $f:A\to A$ is the identity, then $\starf:\starA\to\starA$ 
is the identity, that is ${}^*(1_A)=1_{{}^*\!A}$.
\item
${}^*\{f(\alpha )\mid a\in A\}=\{\starf(\xi)\mid\xi\in\starA\}$,
that is ${}^*(f(A))=\starf(\starA)$.
\item
${}^*\{a\mid f(\alpha )\in B\}=\{\xi\mid\starf(\xi)\in\starB\}$,
that is ${}^*(f^{-1}(B))=(\starf)^{-1}(\starB)$.
\item\label{starcomposition}
${}^*(f\circ g)=\starf\circ\starg$.
\item
${}^*\{(a,b)\in A\times B\mid f(a)=g(b)\}=
\{(\xi,\zeta)\in\starA\times\starB\mid \starf(\xi)=\starg(\zeta)\}$.
\end{enumerate}
\end{proposition}

\begin{proof}
(\ref{starfunc}), (\ref{starfdomain}), (\ref{starfrange}),
and (\ref{starvalue}) are proved by direct applications of \emph{transfer},
because the notions of function, value of a function at a given point,
domain, and range, are elementary. (\ref{starfAB}) is a direct
corollary of the previous properties.
We only prove two of the remaining properties as all of the proofs are similar to one another.

(\ref{stargraph}). The property ``$C=\text{graph}(f)$" is formalized
by the elementary formula obtained as the conjunction of the formula
``$\forall z\in C\ \exists x\in\text{domain}(f)\ \exists y\in\text{range}(f)\
\text{such that }\ y=f(x)\ \text{and}\ (x,y)\in C$''
with the formula ``$\forall x\in\text{domain}(f)\ \forall y\in\text{range}(f)\ 
(y=f(x)\Rightarrow (x,y)\in C)$''. The desired equality
follows by \emph{transfer} and by the previous properties.

(\ref{starcomposition}).
If $f:A\to B$ and $g:B\to C$, 
then the property ``$h=g\circ f$'' is formalized by the formula
``$h:A\to C\ \text{and}\ \forall x\in A\ \forall y\in C\ 
(h(x)=y\Leftrightarrow \exists z\in B\ f(x)=z\ \text{and}\ g(z)=y)$''.
\end{proof}

\begin{exercise}
Prove that a function $f:A\to B$ is 1-1 if and only if $\starf:\starA\to\starB$ is 1-1.
\end{exercise}

We now discuss a general result about the star map 
that is really useful in practice (and, in fact, several 
particular cases have already been included
in the previous propositions):
If a set is defined by means of an elementary property,
then its hyper-extension is defined by the same property 
where one puts stars in front of the parameters. 

\begin{proposition}\label{applyingstar}
Let $\varphi(x,y_1,\ldots,y_n)$ be an elementary formula.
For all objects $B,A_1,\ldots,A_n$ one has
$${}^*\{x\in B\mid \varphi(x,A_1,\ldots,A_n)\}\ =\ 
\{x\in\starB\mid \varphi(x,\starA_1,\ldots,\starA_n)\}.$$
\end{proposition}

\begin{proof}
Let us denote by $C=\{x\in B\mid \varphi(x,A_1,\ldots,A_n)\}$.
Then the following property holds:
$$P(A_1,\ldots,A_n,B,C):\ 
\forall x\left(x\in C \Leftrightarrow
(x\in B\ \text{and}\ \varphi(x,A_1,\ldots,A_n)\right).$$

The above formula is elementary. Indeed, it is an abbreviation for
the conjuction of the  two formulas:
``$\forall x\in C\ (x\in B\ \text{and}\ \varphi(x,A_1,\ldots,A_n)$"
and ``$\forall x\in B\ (\varphi(x,A_1,\ldots,A_n)\Rightarrow x\in C)$'',
where all quantifiers are bounded, and where $\varphi$ is elementary
by hypothesis. Then we can apply \emph{transfer} and obtain the validity of
$P(\starA_1,\ldots,\starA_n,\starB,\starC)$, that is
$\starC=\{x\in\starB\mid\varphi(x,\starA_1,\ldots,\starA_n)\}$.
\end{proof}

An immediate corollary is the following.

\begin{proposition}
If $(a,b)=\{x\in\R\mid a<x<b\}$ is an open interval
of real numbers then ${}^*(a,b)=\{\xi\in\starR\mid a<\xi<b\}$,
and similarly for intervals of the form
$[a,b)$, $(a,b]$, $(a,b)$, $(-\infty,b]$ and $[a,+\infty)$.
Analogous properties hold for
intervals of natural, integer, or rational numbers.
\end{proposition}

\subsection{Additional assumptions}

By property Proposition \ref{starpropsets} (1) and (2), 
the hyper-extension $\starA$ of a set $A$
contains a copy of $A$ given by the hyper-extensions
of its elements 
$${}^\sigma A\ =\ \{{}^*a\mid a\in A\}\ \subseteq\starA.$$

Notice that, by \emph{transfer},
an hyper-extension ${}^*x$ belongs to $\starA$ if and only if $x\in A$.
In consequence, $\starA\cap{}^\sigma B={}^\sigma(A\cap B)$
for all sets $A,B$.

Following the common use in nonstandard analysis, to
simplify matters we will assume that ${}^*r=r$ for all $r\in\R$. This implies by transfer that ${}^*(r_1,\ldots,r_k)=(r_1,\ldots,r_k)$ for all ordered tuples of real numbers, i.e.\ ${}^\sigma(\R^k)=\R^k$. It follows that
hyper-extensions of real sets and functions are actual 
extensions:

\begin{itemize}
\item
$A\subseteq\starA$ for every $A\subseteq\R^k$,
\item
If $f:A\to B$ where $A\subseteq\R^k$ and $B\subseteq\R^h$,
then $\starf$  is an extension of $f$,
that is, $\starf(a)=f(a)$ for every $a\in A$.
\end{itemize}

In nonstandard analysis it is always 
assumed that the star map satisfies the following

\begin{itemize}
\item
\emph{Properness condition:} $\starN\ne\N$.
\end{itemize}

\begin{proposition}
If the properness condition $\starN\ne\N$ holds then
${}^\sigma A\ne\starA$ for every infinite $A$.
\end{proposition}

\begin{proof}
Given an infinite set $A$, pick a surjective map $f:A\to\N$. Then also the hyper-extension 
$\starf:\starA\to\starN$ is surjective, and 
$$\starN\ =\ \{\starf(\alpha)\mid \alpha\in\starA\}\ =\ 
\{\starf({}^*a)\mid a\in A\}\ =\ 
\{{}^*(f(a))\mid a\in A\}\ =\ \{{}^*n\mid n\in\N\}\ =\ \N.$$
\end{proof}

As a first consequence of the properness condition, one gets
a nonstandard characterization of finite sets
as those sets that are not ``extended'' by hyper-extensions.

\begin{proposition}\label{finitedoesntenlarge}
For every set $A$ one has the equivalence:
``$A$ is finite if and only if $\starA={}^\sigma A$".
(When $A\subseteq\R^k$, this is the same as ``$A$ is finite if
and only if $\starA=A$''.)
\end{proposition}

\begin{proof}
If $A=\{a_1,\ldots,a_k\}$ is finite, we already saw
in Proposition \ref{starpropsets} (\ref{startuple})
that $\starA=\{{}^*a_1,\ldots,{}^*a_k\}=\{{}^*a\mid a\in A\}$.
Conversely, if $A$ is infinite, we can pick a 
surjective function $f:A\to\N$. Then also $\starf:\starA\to\starN$
is onto. Now notice that for every $a\in A$, one has
that $(\starf)({}^*a)={}^*(f(a))\in\N$ 
(recall that ${}^*n=n$ for every $n\in\N$).
Then if $\xi\in\starN\setminus\N$ there
exists $\alpha\in\starA\setminus\{{}^*a\mid a\in A\}$ with $\starf(\alpha)=\xi$.
\end{proof}

One can safely extend the simplifying assumption ${}^*r=r$ from real numbers $r$
to elements of any given mathematical object $X$ under study
(but of course not for all mathematical objects).

\begin{itemize}
\item
\emph{Unless explicitly mentioned otherwise,
when studying a specific mathematical object $X$ by nonstandard analysis,
we will assume that ${}^*x=x$ for all $x\in X$, so that 
$X={}^\sigma X\subseteq\starX$.}
\end{itemize}

It is worth mentioning at this point that 
star maps satisfying the \emph{transfer principle}\index{transfer principle} and the \emph{properness condition} do actually exist. Indeed, they can be easily constructed 
by means of ultrafilters, or, equivalently,\index{ultrafilter}
by means of maximal ideals of rings of functions (see Section \ref{sec-ultrapower}).

We end this section with an example of using properness to give a short proof of a classical fact. (This nonstandard proof was suggested by D.A. Ross.)

\begin{theorem}[Sierpinski]
Given $a_1,\ldots,a_n,b\in \mathbb{R}^{>0}$, the set $$E:=\left\{(x_1,\ldots,x_n)\in \N^n \ : \ \frac{a_1}{x_1}+\cdots +\frac{a_n}{x_n}=b\right\}$$  is finite. 
\end{theorem}

\begin{proof}
Suppose, towards a contradiction, that $E$ is infinite.  Then there is $x=(x_1,\ldots,x_n)\in {}^{\ast}E\setminus E$.  Without loss of generality, we may assume that there is $k\in \{1,\ldots,n\}$ such that $x_1,\ldots,x_{k}\in \starN \setminus \N$ and $x_{k+1},\ldots,x_n\in \N$.  We then have
$$\frac{a_1}{x_1}+\cdots +\frac{a_k}{x_k}=b-\left(\frac{a_{k+1}}{x_{k+1}}+\cdots+\frac{a_n}{x_n}\right).$$  We have now arrived at a contradiction for the left hand side of the equation is a positive infinitesimal element of ${}^{\ast}\mathbb{R}$ while the right hand side of the equation is a positive standard real number.
\end{proof}

\section{The transfer principle\index{transfer principle}, in practice}

As we already pointed out, a correct application of \emph{transfer}
needs a precise understanding of the notion of elementary property.
Basically, a property is elementary if it talks about the elements of 
some given structures and \emph{not} about their subsets
or the functions between them.\footnote{
~In logic, properties that talks about elements
of a given structure are called \emph{first-order properties};
properties about subsets of the given structure 
are called \emph{second-order};
properties about subsets or relations on the given structure 
are called \emph{third-order}; and so forth.}
Indeed, in order to correctly apply the \emph{transfer principle}\index{transfer principle}, 
one must always point out the range of quantifiers, and
formulate them in the forms ``$\,\forall\,x\in X\ldots$'' and 
``$\,\exists\,y\in Y\ldots$'' for suitable specified sets $X,Y$.
With respect to this, the following remark is particularly relevant.

\begin{remark}
Before applying \emph{transfer}, all quantifications on subsets
``$\forall\,x\subseteq X\ldots$'' or
``$\exists\,x\subseteq X\ldots$''
must be reformulated as ``$\,\forall\,x\in\Pow(X)\ldots$'' and
``$\,\exists\,x\in\Pow(X)\ldots$'', respectively, where
$\Pow(X)=\{A\mid A\subseteq X\}$ is the \emph{powerset} of $X$.
Similarly, all quantifications on functions $f:A\to B$
must be bounded by $\Fun(A,B)$, the \emph{set of functions} from $A$ to $B$.
We stress that these instructions are pivotal because in general 
${}^*\Pow(X)\ne\Pow(\starX)$ and ${}^*\Fun(A,B)\ne\Fun(\starA,\starB)$,
as we will show in Proposition \ref{properinclusions}.
\end{remark}

\begin{example}
Consider the property: ``$<$ is a linear ordering on the set $A$''.
Notice first that $<$ is a binary relation on $A$, and hence its hyper-extension
${}^*\!\!\!<$ is a binary relation on $\starA$.
By definition, $<$ is a linear ordering if and only if
the following are satisfied:
\begin{enumerate}
\item[(a)]
$\forall x\in A\ (x\not< x)$,
\item[(b)]
$\forall x,y,z\in A\ (x<y \text{ and } y<z)
\Rightarrow x<z$,
\item[(c)]
$\forall x,y\in A\ (x<y \text{ or } y<x \text{ or } x=y)$.
\end{enumerate}

Notice that the three formulas above are elementary.
Then we can apply \emph{transfer} and conclude
that: ``${}^*\!\!\!<$ is a linear ordering on $\starA$."
\end{example}

Whenever confusion is unlikely, some asterisks
will be omitted. So, for instance, we will write 
$+$ to denote both the sum operation on
$\N$, $\Z$, $\Q$ and $\R$, and the corresponding 
operations on $\starN$, $\starZ$, $\starQ$ and $\starR$, respectively,
as given by the hyper-extension ${}^*+$. 

Similarly as in the example above, it is readily verified that
the properties of a discretely ordered ring, as well as the 
properties of a real-closed ordered field, are elementary because they 
just talk about the \emph{elements} of the given structures. 
Thus, by a direct application of \emph{transfer},
one obtain the following results, which generalize the properties presented in Section 
\S \ref{sec-warmingup}.

\begin{theorem}
\
\begin{enumerate}
\item
$\starR$, endowed with the hyper-extensions of the sum,
product, and order on $\R$, is a real-closed ordered field.\footnote
{~Recall that an ordered field is \emph{real closed}
if every positive element is a square, and every polynomial of 
odd degree has a root.}
\item
$\starZ$ is an unbounded discretely ordered subring of $\starR$,
whose positive part is $\starN$.
\item
The ordered subfield $\starQ\subset\starR$ 
is the quotient field of $\starZ$.
\end{enumerate}
\end{theorem}

Again by direct applications of \emph{transfer},
one also obtains the following properties.

\begin{proposition}
\

\begin{enumerate}
\item
Every non-zero $\nu\in\starN$ has a successor
$\nu+1$ and a predecessor $\nu-1$.\footnote{
~An element $\eta$ is the \emph{successor} of $\xi$ 
(or $\xi$ is the \emph{predecessor} of $\eta$) if $\xi<\eta$
and there are no elements $\zeta$ with $\xi<\zeta<\eta$.}
\item
For every positive $\xi\in\starR$ there exists a unique
$\nu\in\starN$ with $\nu\leq\xi<\nu+1$. As a result,
$\starN$ is unbounded in $\starR$.
\item
The hyperrational numbers $\starQ$, as well as the
hyperirrational numbers ${}^*(\R\setminus\Q)=\starR\setminus\starQ$,
are dense in $\starR$.
\end{enumerate}
\end{proposition}

\begin{proposition}\label{initialsegment}
$(\N,\le)$ is an initial segment of $(\starN,\le)$,
that is, if $\nu\in\starN\setminus\N$, then $\nu>n$ for all $n\in\N$,
\end{proposition}

\begin{proof}
For every $n\in\N$, by \emph{transfer} one obtains the validity
of the following elementary formula: ``$\forall x\in\starN\ (x\ne 1\ \text{and}\ \ldots$
$\text{and}\ x\ne n)\Rightarrow x>n$'', and hence the proposition holds.
\end{proof}

To get a clearer picture of the situation, examples
of \emph{non}-elementary properties that are not
preserved under hyper-extensions, are now in order.

\begin{example}\label{ex-wellordering}
The property of well-ordering (that is, every nonempty subset 
has a least element) and of completeness of an ordered set are \emph{not} elementary. Indeed, they concern the \emph{subsets} of the given ordered set. 
Notice that these properties are not preserved by hyper-extensions. In fact, $\N$ is well-ordered but $\starN$ is not 
(\emph{e.g.}, the set of infinite
hyper-natural numbers has no least element). The real line $\R$ is complete but
$\starR$ is not (\emph{e.g.}, the set of infinitesimal
numbers is bounded with no least upper bound).
\end{example}

\begin{remark}\label{rmk-wellordering}
\emph{Transfer} applies also to the 
the well-ordering property of $\N$, provided one formalizes it as:
``\emph{Every nonempty element of $\Pow(\N)$
has a least element}''.
(The property ``$X$ has a least element''
is elementary: ``there exists $x\in X$ such that
for every $y\in X$, $x\le y$.'')
In this way, one gets:
``\emph{Every nonempty element of ${}^*\Pow(\N)$ has a least
element}''. The crucial point here is that ${}^*\Pow(\N)$ is a 
only a proper subfamily of $\Pow(\starN)$
(see Proposition \ref{properinclusions} below).
So, the well-ordering property is \emph{not} an elementary
property of $\N$, but it is an elementary property of 
$\Pow(\N)$. Much the same observations can be made
about the completeness property.
Indeed, virtually all properties of mathematical objects
can be formalized by elementary formulas,
provided one uses the appropriate parameters.
\end{remark}

A much more slippery example of a \emph{non}-elementary property
is the following.

\begin{example}
The Archimedean property of an ordered field $\mathbb{F}$
is \emph{not} elementary. Notice that
to formulate it, one needs to also consider the 
substructure $\N\subset\mathbb{F}$:
\begin{center}
\emph{``For all positive $x\in\mathbb{F}$ there exists $n\in\N$ 
such that $nx>1$.''}
\end{center}
While the above is an elementary property
of the pair $(\mathbb{F},\N)$ since it talks about the
elements of $\mathbb{F}$ and $\N$ combined, 
it is \emph{not} an elementary property of the 
ordered field $\mathbb{F}$ alone. 
In regard to this, we remark that the following expression:
\begin{center}
\emph{``For all positive $x\in\mathbb{F}$ it is $x>1$ or 
$2x>1$ or $3x>1$ or \ $\ldots$\  or $nx>1$ or \ $\ldots$.''}
\end{center}
is \emph{not} a formula, because it would consist
in an infinitely long string of symbols if written in full.
Notice that the Archimedean property is not preserved by hyper-extensions.
For instance, $\R$ is Archimedean, but
the hyperreal line \index{hyperreal number}$\starR$ is not, being an ordered
field that properly extends $\R$ (see  Proposition \ref{nonArchimedean}).

Similarly, the properties of being infinitesimal, finite,
or infinite are \emph{not} elementary properties of elements
in a given ordered field
$\mathbb{F}$, because to formulate them one needs
to also consider the substructure $\N\subset\mathbb{F}$ as a parameter.
\end{example}

\section{The ultrapower model}\label{sec-ultrapower}

It is now time to justify what we have seen in the previous sections
and show that star maps that satisfy the \emph{transfer} principle
do actually exist.
Many researchers using nonstandard methods, 
especially those who do not have
a strong background in logic,  
feel more comfortable in directly working with a model.
However we remark that this is not necessary.
Rather, it is worth stressing that all one needs in practice is a good understanding of the \emph{transfer}
principle and its use, whereas the underlying construction
of a specific star map does not play a crucial role.\footnote
{~There are a few exceptions to this statement, but we will never
see them in the combinatorial applications presented in this book.}
The situation is similar to what happens when 
working in real analysis: what really matters is that $\mathbb{R}$ is a complete Arichmedean ordered field, along with the fact that such a field actually exist; whereas
the specific construction of the real line (\emph{e.g.}, by means
of Dedekind cuts or by a suitable quotient of the set of Cauchy sequences)
is irrelevant when developing the theory.

\subsection{The ultrapower construction}\label{ultrapower}
The ultrapower\index{ultrapower} construction relies on ultrafilters\index{ultrafilter}
and so, to begin with, let us fix an ultrafilter $\U$
on a set of indices $I$. For simplicity, in the following we will focus on ultrapowers of $\mathbb{R}$. However, the same
construction can be carried out by starting with any 
mathematical structure.

\begin{definition}
The \emph{ultrapower} of $\R$ modulo the ultrafilter $\U$,\index{ultrafilter}
denoted $\R^I/\U$, is the quotient of the family of real $I$-sequences 
$\R^I=\Fun(I,\R)=\{\sigma\mid \sigma:I\to\R\}$ 
modulo the equivalence relation $\equiv_\U$ defined by setting:
$$\sigma\equiv_\U\tau\ \Leftrightarrow\
\{i\in I\mid\sigma(i)=\tau(i)\}\in\U.$$
\end{definition}

Notice that the properties of being a filter on $\U$ guarantee that
$\equiv_\U$ is actually an equivalence relation.  Equivalence classes are denoted by using square brackets:
$[\sigma]=\{\tau\in\Fun(I,\R)\mid\tau\equiv_\U\sigma\}$.
The pointwise sum and product operations on the ring $\Fun(I,\R)$
are inherited by the ultrapower\index{ultrapower}. Indeed,
it is easily verified that the following definitions are well-posed:
$$[\sigma]\boldsymbol{+}[\tau]=[\sigma+\tau]\quad\text{and}\quad
[\sigma]\boldsymbol{\cdot}[\tau]=[\sigma\cdot\tau].$$

The order relation $\boldsymbol{<}$ on the ultrapower\index{ultrapower} is defined by putting:
$$[\sigma]\boldsymbol{<}[\tau]\ \Leftrightarrow\ 
\{i\in I\mid \sigma(i)<\tau(i)\}\in\U.$$

\begin{proposition}
The ultrapower\index{ultrapower} $(\R^I/\U, \boldsymbol{+},\boldsymbol{\cdot},
\boldsymbol{<},\mathbf{0},\mathbf{1})$
is an ordered field.
\end{proposition}

\begin{proof}
All properties of an ordered field are directly proved by using the properties of an ultrafilter.
For example, to prove that $\boldsymbol{<}$ is a total
ordering, one considers the partition
$I=I_1\cup I_2\cup I_3$ where $I_1=\{i\in I\mid \sigma(i)<\tau(i)\}$,
$I_2=\{i\in I\mid\sigma(i)=\tau(i)\}$ and $I_3=\{i\in I\mid \sigma(i)>\tau(i)\}$:
exactly one out of the three sets belongs to $\U$,
and hence exactly one out of 
$[\sigma]<[\tau]$, $[\sigma]=[\tau]$, or $[\sigma]>[\tau]$ holds.
As another example, let us show that every
$[\sigma]\neq\boldsymbol{0}$ has a multiplicative inverse.
By assumption, $A=\{i\in I\mid\sigma(i)=0\}\notin\mathcal{U}$, 
and so the complement $A^c=\{i\in I\mid\sigma(i)\neq 0\}\in \mathcal{U}$. 
Now pick any $I$-sequence $\tau$ such that $\tau(i)=1/{\sigma(i)}$
whenever $i\in A^c$. Then
$A^c\subseteq \{i\in I\mid \sigma(i)\cdot\tau(i)=1\}\in\U$,
and hence $[\sigma]\cdot[\tau]=\boldsymbol{1}$.
\end{proof}

There is a canonical way of embedding $\R$ into its ultrapower\index{ultrapower}.

\begin{definition}
The \emph{diagonal embedding} $d:\R\to\R^I/\U$
is the function that associates to every real number $r$ the equivalence
class of the corresponding constant $I$-sequence $[c_r]$.
\end{definition}

It is readily seen that $d$ is a 1-1 map that preserves
sums, products and the order relation.
As a result, without loss of generality,
we can identify every $r\in\R$ with its diagonal
image $d(r)=[c_r]$, and assume that $\R\subseteq\R^I/\U$ is an
ordered subfield.

Notice that if $\U=\U_j$ is principal then the corresponding
ultrapower\index{ultrapower} $\R^I/\U_j=\R$ is trivial. Indeed, in this
case one has $\sigma\equiv_{\U_j}\tau\Leftrightarrow\sigma(j)=\tau(j)$. Thus, every sequence is equivalent to the constant
$I$-sequence with value $\sigma(j)$, and the diagonal
embedding $d:\R\to\R^I/\U_j$ is onto. 

\begin{remark}\label{rem-isohyperreals}
Under the \emph{Continuum Hypothesis},
one can show that for every pair $\U,\V$ of non-principal ultrafilters on $\N$,
the hyperreal numbers given by the corresponding
ultrapower models $\R^\N/\U\cong\R^\N/\V$ are
isomorphic as ordered fields.\footnote{
~This is because, under the Continuum Hypothesis, they are
$\aleph_1$-saturated models of cardinality $\aleph_1$ in a finite language.}
\end{remark}

\subsection{Hyper-extensions in the ultrapower model}

In this section we will see how the ultrapower\index{ultrapower} $\R^I/\U$ can be
made a model of the hyperreal numbers\index{hyperreal number} of nonstandard analysis.
Let us start by denoting 
$$\starR\ =\ \R^I/\U.$$

We now have to show that the ordered field $\starR$ 
has all the special features that make it a set
of hyperreal numbers\index{hyperreal number}. To this end, we will define a 
\emph{star map} on the family of all sets of ordered tuples 
of real numbers and of all real functions, in such a way 
that the \emph{transfer} principle holds.  

\begin{definition}
Let $A\subseteq\R$. 
Then its \emph{hyper-extension} $\starA\subseteq\starR$ is defined as
the family of all equivalence classes of $I$-sequences that take
values in $A$, that is:
$$\starA\ =\ A^I/\U\ =\ \left\{[\sigma]\mid\sigma:I\to A\right\}\ \subseteq\ \starR.$$

Similarly, if $A\subseteq\R^k$ is a set of real $k$-tuples, then
its \emph{hyper-extension} is defined as
$$\starA\ =\ \left\{([\sigma_1],\ldots,[\sigma_k])\mid
(\sigma_1,\ldots,\sigma_k):I\to A\right\}\ \subseteq\ \starR^k$$
where we denoted 
$(\sigma_1,\ldots,\sigma_k):i\mapsto(\sigma_1(i),\ldots,\sigma_k(i))$.
\end{definition}

Notice that, by the properties of ultrafilter, for every
$\sigma_1,\ldots,\sigma_k,\tau_1,\ldots,\tau_k:I\to\R$, one has
$$\{i\in I\mid (\sigma_1(i),\ldots,\sigma_k(i))=
(\tau_1(i),\ldots,\tau_k(i))\}\in\U\ \Longleftrightarrow\
\sigma_s\equiv_\U\tau_s\ \text{for every }s=1,\ldots,k.$$
In consequence, the above definition is well-posed,
and one has that $([\sigma_1],\ldots,[\sigma_n])\in\starA\Leftrightarrow
\{i\mid(\sigma_1,\ldots,\sigma_n)\in A\}\in\U$.

We also define the star map on real ordered tuples by setting
$${}^*(r_1,\ldots,r_k)=(r_1,\ldots,r_k).$$
Recall that we identified every $r\in\R$ with the 
equivalence class $[c_r]$ of the corresponding constant sequence and
so, by letting ${}^*r=r=[c_r]$,
we have that $A\subseteq\starA$ for every $A\subseteq\R^k$.

We have already seen that $\starR$ is
an ordered field that extends the real line.
As a result, every rational function $f:\R\to\R$ is 
naturally extended to a function $\starf:\starR\to\starR$.
However, here we are interested in extending \emph{all}
real functions $f:A\to B$ where $A$ and $B$ are set of real tuples,
to functions $\starf:\starA\to\starB$. With ultrapowers,
this can be done in a natural way.

\begin{definition}
Let $f:A\to B$ where $A,B\subseteq\R$. Then the
\emph{hyper-extension} of $f$ is the function $\starf:\starA\to\starB$
defined by setting $\starf([\sigma])=[f\circ\sigma]$ for every
$\sigma:I\to A$.

\medskip
$$\begin{picture}(80,40)(0,0)
      \put(0,40){\makebox(0,0){$A$}}
      \put(80,40){\makebox(0,0){$B$}}
      \put(40,0){\makebox(0,0){$I$}}
      \put(40,48){\makebox(0,0){$f$}}
      \put(18,16){\makebox(0,0){$\sigma$}}
      \put(72,16){\makebox(0,0){$f\circ\sigma$}}
      \put(14,40){\vector(1,0){54}}
      \put(35,9){\vector(-1,1){26}}
      \put(48,9){\vector(1,1){24}}
\end{picture}$$

\smallskip
If $f:A\to B$ is a function of several variables where
$A\subseteq\R^k$ and $B\subseteq\R$, 
then $\starf:\starA\to\starB$ is defined by setting
for every $\sigma_1,\ldots,\sigma_k:I\to\R$:
$$\starf([\sigma_1],\ldots,[\sigma_k])\ =\ 
[\langle f(\sigma_1(i),\ldots,\sigma_k(i))\mid i\in I\rangle].$$



\end{definition}

Similarly as for hyper-extensions of sets of tuples,
it is routine to check that the properties of a filter
guarantee that the above definition is well-posed.

Let us now see that the ultrapower\index{ultrapower} model has all the desired properties.

\begin{theorem}
The hyper-extensions of real ordered tuples, sets of ordered real tuples and
real functions, as defined above, 
satisfy all the properties itemized in Propositions 
\ref{starpropsets}, \ref{starproprelations},
and \ref{starpropfunctions}:\footnote
{~Since hyper-extensions of families of sets have 
not been defined, properties 
(\ref{starfamilymembership}) and (\ref{starfamilyunion})
of Proposition \ref{starpropsets} are not included in the list. 
Clearly, (4), (5), (6), (7), (10) only applies when
$A,B,A_i$ are sets.}

\begin{enumerate}
\item
$a=b\Leftrightarrow{}^*a={}^*b$.
\item
$A$ is a set if and only if $\starA$ is a set.
\item
${}^*\emptyset=\emptyset$.
\item
$A\subseteq B\Leftrightarrow \starA\subseteq\starB$.
\item
${}^*(A\cup B)=\starA\cup\starB$.
\item
${}^*(A\cap B)=\starA\cap\starB$.
\item
${}^*(A\setminus B)=\starA\setminus\starB$.
\item
${}^*\{a_1,\ldots,a_k\}=\{a_1,\ldots,a_k\}$.
\item
${}^*(a_1,\ldots,a_k)=(a_1,\ldots,a_k)$.
\item
${}^*(A_1\times\ldots\times A_k)=\starA_1\times\ldots\times\starA_k$.
\item
${}^*\{(a,a)\mid a\in A\}=\{(\xi,\xi)\mid \xi\in A\}$.
\item
$R$ is a $k$-ary relation if and only if ${}^*R$ is a $k$-ary relation.
\item
${}^*\{a\mid \exists b\ R(a,b)\}=\{\xi\mid \exists \zeta\ {}^*R(\xi,\zeta)\}$,
that is, ${}^*\text{domain}(R)=\text{domain}({}^*R)$.
\item
${}^*\{b\mid \exists a\ R(a,b)\}=\{\zeta\mid \exists \xi\ {}^*R(\xi,\zeta)\}$,
that is, ${}^*\text{range}(R)=\text{range}({}^*R)$.
\item
${}^*\{(a,b)\mid R(b,a)\}=\{(\xi,\zeta)\mid {}^*R(\zeta,\xi)\}$.
\item
${}^*\{(a,b,c)\mid S(c,a,b)\}=\{(\xi,\zeta,\eta)\mid {}^*S(\xi,\eta,\zeta)\}$.
\item
${}^*\{(a,b,c)\mid S(a,c,b)\}=\{(\xi,\zeta,\eta)\mid {}^*S(\xi,\eta,\zeta)\}$.
\item\label{starfunction}
$f$ is a function if and only if $\starf$ is a function.
\item
${}^*\text{domain}(f)=\text{domain}(\starf)$.
\item
${}^*\text{range}(f)=\text{range}(\starf)$.
\item
$f:A\to B$ if and only if $\starf:\starA\to\starB$.
\item
${}^*\text{graph}(f)=\text{graph}(\starf)$.
\item
$(\starf)(a)=f(a)$ for every $a\in\text{domain}(f)$.
\item 
If $f:A\to A$ is the identity, then $\starf:\starA\to\starA$ 
is the identity, that is ${}^*(1_A)=1_{{}^*\!A}$.
\item
${}^*\{f(a)\mid a\in A\}=\{\starf(\xi)\mid\xi\in\starA\}$,
that is ${}^*(f(A))=\starf(\starA)$.
\item
${}^*\{a\mid f(a)\in B\}=\{\xi\mid\starf(\xi)\in\starB\}$,
that is ${}^*(f^{-1}(B))=(\starf)^{-1}(\starB)$.
\item
${}^*(f\circ g)=\starf\circ\starg$.
\item
${}^*\{(a,b)\in A\times B\mid f(a)=g(b)\}=
\{(\xi,\zeta)\in\starA\times\starB\mid \starf(\xi)=\starg(\zeta)\}$.
\end{enumerate}
\end{theorem}

\begin{proof}
All proofs of the above properties are
straightforward applications of the definitions
and of the properties of ultrafilters.
As an example, let us see here property (13) in detail.
We leave the others to the reader as exercises. 

Let $\Lambda=\{a\mid \exists b\ R(a,b)\}$
and let $\Gamma=\{\xi\mid \exists \zeta\ {}^*R(\xi,\zeta)\}$.
We have to show that ${}^*\Lambda=\Gamma$.
If $\sigma:I\to\Lambda$ then for every $i$ 
there exists an element $\tau(i)$ such that $R(\sigma(i),\tau(i))$.
Then ${}^*R([\sigma],[\tau])$ and so $[\sigma]\in\Gamma$.
This shows the inclusion ${}^*\Lambda\subseteq\Gamma$.
Conversely, $[\sigma]\in\Gamma$ if and only if
${}^*R([\sigma],[\tau])$ for some $I$-sequence $\tau$.
Since $([\sigma],[\tau])\in{}^*R$, the set
$\Theta=\{i\mid(\sigma(i),\tau(i))\in R\}\in\U$,
so also the superset $\{i\mid \sigma(i)\in\Lambda\}\supseteq\Theta$
belongs to $\U$. We conclude that
$[\sigma]\in{}^*\Lambda$, as desired.
\end{proof}

We disclose that the previous theorem essentially states that our
defined star map satisfies the \emph{transfer} principle.
Indeed, once the notion of elementary property
will be made fully rigorous, one can show that
\emph{transfer} is actually equivalent to the validity
of the properties listed above.

\begin{remark}
A ``strong isomorphism" between two sets of hyperreals
$\starR$ and ${}^\star\R$ is defined as a bijection $\psi:\starR\to{}^\star\R$
that it coherent with hyper-extensions, that is,
$(\xi_1,\ldots,\xi_k)\in\starA\Leftrightarrow
(\Psi(\xi_1),\ldots,\Psi(\xi_k))\in {}^\star A$
for every $A\subseteq\R^k$ and for every $\xi_1,\ldots,\xi_k\in\starR$,
and $\starf(\xi_1,\ldots,\xi_k)=\eta\Leftrightarrow
{}^\star f(\Psi(\xi_1),\ldots,\Psi(\xi_k))=\Psi(\eta)$
for every $f:\R^k\to\R$ and for every
$\xi_1,\ldots,\xi_k,\eta\in\starR$.
Then one can show that two ultrapower models
$\R^\N/\U$ and $\R^\N/\V$ are ``strongly isomorphic"
if and only if the ultrafilters $\U\cong\V$ are isomorphic,
that is, there exists a permutation $\sigma:\N\to\N$ such that
$A\in\U\Leftrightarrow\sigma(A)\in\V$ for every $A\subseteq\N$.
We remark that there exist plently of non-isomorphic ultrafilters
(indeed, one can show that there are $2^\mathfrak{c}$-many
distinct classes of isomorphic ultrafilters on $\N$).
This is to be contrasted with the previous Remark \ref{rem-isohyperreals},
where the notion of isomorphism between sets of hyperreals
was limited to
the structure of ordered field.
\end{remark}

\subsection{The properness condition in the ultrapower model}

In the previous section, we observed that principal ultrafilters
generate trivial ultrapowers. Below, we precisely isolate
the class of those ultrafilters that produce models
where the properness condition $\N\ne\starN$ 
(as well as $\R\ne\starR$) holds.

Recall that an ultrafilter $\U$ is called
\emph{countably incomplete}\index{ultrafilter!countably incomplete}
if it is not closed under countable intersections, that is,
if there exists a countable family $\{I_n\}_{n\in\N}\subseteq\U$
such that $\bigcap_{n\in\N} I_n\notin\U$.
We remark that all non-principal ultrafilters on $\N$ or on $\R$
are countably incomplete.\footnote
{~The existence of non-principal ultrafilters that are countably 
complete is equivalent to the existence of the so-called 
\emph{measurable cardinals}, a kind of inaccessible cardinals 
studied in the hierarchy of large cardinals, and whose
existence cannot be proved by \textsf{ZFC}. 
In consequence, if one sticks to the usual principles of mathematics, 
it is safe to assume that every non-principal ultrafilter is countably incomplete.}

\begin{exercise}
An ultrafilter $\U$ on $I$ is countably incomplete
if and only if there exists a countable partition $I=\bigcup_{n\in\N} J_n$
where $J_n\notin\U$ for every $n$.
\end{exercise}

\begin{proposition}
In the ultrapower model modulo the ultrafilter $\U$ on $I$,
the following properties are equivalent:
\begin{enumerate}
\item
Properness condition: $\starN\ne\N$;
\item
$\U$ is countably incomplete.
\end{enumerate}
\end{proposition}

\begin{proof}
Assume first that $\starN\ne\N$. Pick a sequence $\sigma:I\to\N$
such that $[\sigma]\notin\N$. Then 
$I_n=\{i\in I\mid \sigma(i)\ne n\}\in\U$ for every $n\in\N$,
but $\bigcap_n I_n=\emptyset\notin\U$.
Conversely, if $\U$ is countably incomplete, pick
a countable partition $I=\bigcup_n J_n$ where $J_n\notin\U$ for every $n$,
and pick the sequence $\sigma:I\to\N$ where
$\sigma(i)=n$ for $i\in J_n$.
Then $[\sigma]\in\starN$ but $[\sigma]\ne d(n)$ for every $n$.
\end{proof}

%

In the sequel we will always assume that ultrapower models
are constructed by using ultrafilters $\U$ that are
countably incomplete.

\subsection{An algebraic presentation}\label{algebraic}

The ultrapower\index{ultrapower} model can be presented in an alternative, 
but equivalent, purely algebraic fashion where
only the notion of quotient field of a ring
modulo a maximal ideal is assumed.(~See \cite{benci_purely_2005} for details.)
Here are the steps of the construction.

\begin{itemize}
\item
Consider $\text{Fun}(I,\R)$, the
ring of real valued sequences where the sum and 
product operations are defined pointwise.
\item
Let $\igoth$ be the ideal of those sequences that have \emph{finite support}:
$$\igoth\ =\ \{\sigma\in\text{Fun}(I,\R)\mid
\sigma(i)=0\ \text{for all but at most finitely many}\ i\}.$$
\item
Extend $\igoth$ to a maximal ideal $\mgoth$, 
and define the hyperreal numbers as the quotient field:
$$\starR\ =\ {\text{Fun}(I,\R)}/\mgoth.$$
\item
For every subset $A\subseteq\R$, its hyper-extension is defined by:
$$\starA\ =\ \{\sigma+\mgoth\mid \sigma:I\to A\}\ \subseteq\ \starR.$$
So, \emph{e.g.}, the \emph{hyper-natural numbers} $\starN$
are the cosets $\sigma+\mgoth$ of $I$-sequences 
$\sigma:I\to\N$ of natural numbers.
\item
For every function $f:A\to B$ where $A,B\subseteq\R$, its
hyper-extension $\starf:\starA\to\starB$ is defined by setting
for every $\sigma:I\to A$:
$$\starf(\sigma+\mgoth)\ =\ (f\circ\sigma)+\mgoth.$$
\end{itemize}

\medskip
It can be directly verified that $\starR$ is an ordered field whose
positive elements are $\starR^+=\text{Fun}(\N,\R^+)/\mgoth$,
where $\R^+$ is the set of positive reals.
By identifying each $r\in\R$ with the coset $c_r+\mgoth$
of the corresponding constant sequence, one obtains that $\R$ is a
proper subfield of $\starR$.

Notice that, as in the case of the ultrapower model, the above definitions are naturally
extended to hyper-extensions of sets of real tuples and of functions
between sets of real tuples.
%

\begin{remark}
The algebraic approach presented here
is basically equivalent to the ultrapower\index{ultrapower} model. 
Indeed, for every function $f:I\to\R$, let us
denote by $Z(f)=\{i\in I\mid f(i)=0\}$ its zero-set.
If $\mathfrak{m}$ is a maximal ideal of the ring $\text{Fun}(I,\R)$, 
then it is easily shown that the family
$\U_{\mathfrak{m}}=\{Z(f)\mid f\in\mathfrak{m}\}$
is an ultrafilter on $\N$.
Conversely, if $\U$ is an ultrafilter on $\N$,
then $\mathfrak{m}_\U=\{f\mid Z(f)\in\U\}$
is a maximal ideal of the ring $\Fun(I,\R)$.
The correspondance between $\U$-equivalence classes $[\sigma]$
and cosets $\sigma+\mathfrak{m}_\U$ yields an
isomorphism between the ultrapower\index{ultrapower} $\R^I/\U$ 
and the quotient $\Fun(I,\R)/{\mgoth_\U}$.
\end{remark}

\section{Internal and external objects}

We are now ready to introduce a fundamental class of objects
in nonstandard analysis, namely the internal objects.
In a way, they are similar to the open sets in topology,
or to the measurable sets in measure theory, because they
are those objects that behave ``nicely" in our theory.
Indeed, elementary properties of subsets or of functions transfer
to the corresponding internal objects (see below).

Recall that the \emph{star map} does not
preserve the properties of \emph{powersets} and \emph{function sets}.
For instance, we have noticed in the previous sections that
that there are (nonempty) sets
in $\mathcal{P}(\starN)$ with no least element, and there
are (nonempty) sets in $\mathcal{P}(\starR)$ that are bounded
but have no least upper bound (see Example \ref{ex-wellordering} and
Remark \ref{rmk-wellordering}).
However, by the \emph{transfer principle}\index{transfer principle}, the family 
$\mathcal{P}(A)$ of all
subsets of a set $A$ and ${}^*\mathcal{P}(A)$
satisfy the same properties. Similarly, the family $\Fun(A,B)$ of all
functions $f:A\to B$ and ${}^*\Fun(A,B)$ satisfy the same properties. 
Let us now elaborate on this, and start with two
easy observations.

\begin{proposition}
\

\begin{enumerate}
\item
Every element of the hyper-extension 
${}^*\mathcal{P}(A)$ is a subset of $\starA$, that is, 
${}^*\mathcal{P}(A)\subseteq\mathcal{P}(\starA)$;
\item
Every element of the hyper-extension 
${}^*\Fun(A,B)$ is a function $f:\starA\to\starB$,
that is,  ${}^*\Fun(A,B)\subseteq\Fun(\starA,\starB)$.
\end{enumerate}
\end{proposition}

\begin{proof}
(1). Apply \emph{transfer} to the elementary property:
$\forall x\in\mathcal{P}(A)\ \forall y\in x\ \,y\in A$.
(2) Apply \emph{transfer} to the elementary property:
$\forall x\in\Fun(A,B)\ ``x\ \text{is a function}"\ \text{and}\ 
\text{dom}(x)=A\ \text{and}\ \text{range}(x)\subseteq B$.
\end{proof}

Consequently, it is natural to consider the
elements in ${}^*\mathcal{P}(A)$ as the ``nice'' subsets of
$\starA$, and the elements in ${}^*\Fun(A,B)$ as the ``nice'' functions 
from $\starA$ to $\starB$.

\begin{definition}\label{def-internal}
Let $A,B$ be sets. 
The elements of ${}^*\mathcal{P}(A)$ are called
the \emph{internal subsets}\index{internal!set} of $\starA$ and the
elements of ${}^*\Fun(A,B)$ are called the
\emph{internal functions}\index{internal!function} from 
$\starA$ to $\starB$.
More generally, an \emph{internal} object\index{internal!object} 
is any element $B\in\starY$
that belongs to some hyper-extension.
\end{definition}

The following facts about functions are easily verified, and
the proofs are left as exercises.

\begin{proposition}\label{internalfunctions}
\
\begin{enumerate}
\item
A function $F$ is internal 
if and only if it belongs to the hyper-extension ${}^*\F$ 
of some set of functions $\F$;
\item
A function $F:A\to B$ is internal if and only if
there exist sets $X,Y$ such that
$A\in{}^*\Pow(X)$, $B\in{}^*\Pow(Y)$, and 
$F\in{}^*\{f\ \text{function}\mid \text{domain}(f)\subseteq X\ \text{and}\ 
\text{range}(f)\subseteq Y\}$.
\end{enumerate}
\end{proposition}

In consequence, domain and range of an internal function are internal sets.

First examples of internal objects are given by the
hyperreal numbers\index{hyperreal number} $\xi\in\starR$, and also by all ordered tuples of
hyperreal numbers $(\xi_1,\ldots,\xi_k)\in\starR^k$.
Notice that hyper-extensions $\starX$ themselves 
are internal objects, since trivially $\starX\in{}^*\{X\}=\{\starX\}$.

\begin{itemize}
\item
\textbf{Rule of thumb.}\
Properties about \emph{subsets} of a set $A$
transfer to the \emph{internal subsets}\index{internal!set} of $\starA$,
and properties about functions $f:A\to B$ trasfer to
the \emph{internal functions} from $\starA$ to $\starB$.
\end{itemize}

For instance, the \emph{well-ordering} property of $\N$
is transferred to:\
``Every nonempty \emph{internal} subset of $\starN$
has a least element'', and the \emph{completeness} property of $\R$
transfers to: ``Every nonempty \emph{internal} subset of $\starR$
that is bounded above has a least upper bound''.

The following is a useful closure property
of the class of internal objects\index{internal!object}\index{internal!definition}.

\begin{theorem}[Internal Definition Principle]
Let $\varphi(x,y_1,\ldots,y_k)$ be an elementary formula.
If $A$ is an internal set and $B_1,\ldots,B_n$ are internal objects,
then the set
$\{x\in A\mid \varphi(x,B_1,\ldots,B_n)\}$ is also internal.
\end{theorem}

\begin{proof}
By assumption, there exists a family of sets $\F$ and
sets $Y_i$ such that $A\in{}^*\F$ and $B_i\in\starY_i$ for $i=1,\ldots,n$. 
Pick any family $\G\supseteq\F$ that is closed under subsets, 
that is, $C'\subseteq C\in\G\Rightarrow C'\in\G$.
(For example, one can take $\G=\bigcup\{\mathcal{P}(C)\mid C\in\F\}$.)
Then the following is a true elementary property of the objects $\G,Y_1,\ldots,Y_n$:
$$P(\G,Y_1,\ldots,Y_n):\quad
\forall x\in\G\,\forall y_1\in Y_1\,\ldots\,\forall y_n\in Y_n\ 
\exists z\in\G\ \text{such that}\ 
``z=\{t\in x\mid \varphi(t,y_1,\ldots,y_n)\}."\footnote
{~The subformula ``$z=\{t\in x\mid \varphi(t,y_1,\ldots,y_n)\}$"
is elementary because it denotes the conjuction of the two formulas:
\begin{center}
``$\forall t\in z\ (t\in x\ \text{and}\ \varphi(t,y_1,\ldots,y_n))$''\ \
and\  \ ``$\forall t\in x\ (\varphi(t,y_1,\ldots,y_n)\Rightarrow t\in z)$''.
\end{center}}$$
By \emph{transfer}, the property $P({}^*\G,\starY_1,\ldots,\starY_n)$
is also true, and since $A\in{}^*\G, B_i\in\starY_i$, we obtain the
existence of an internal set $C\in{}^*\G$ such that
$C=\{t\in A\mid \varphi(x,B_1,\ldots,B_n)\}$, as desired.
\end{proof}

As direct applications of the above principle, one obtains
the following properties for the class of internal objects\index{internal!object}.

\begin{proposition}\label{internal}
\
\begin{enumerate}
\item
The class $\mathcal{I}$ of internal sets\index{internal!set} is closed under
unions, intersections, set-differences, finite sets and tuples,
Cartesian products, and under images and preimages of
internal functions.
\item
If $A\in\mathcal{I}$ is an internal set, then 
the set of its internal subsets 
$\mathcal{P}(A)\cap\mathcal{I}\in\mathcal{I}$ is itself
internal.
\item
If $A,B$ are internal sets, then 
the set $\Fun(A,B)\cap\mathcal{I}\in\mathcal{I}$ of internal functions\index{internal!function} between them is itself internal.
\end{enumerate}
\end{proposition}

\begin{proof}
(1). If $A$ and $B$ are internal sets, say $A\in{}^*\Pow(X)$
and $B\in{}^*\Pow(Y)$, then
$A\cup B=\{t\in \starX\cup\starY\mid t\in A\ \text{or}\ x\in B\}$
is internal by the \emph{Internal Definition Principle}.
The other properties are easily proved in the same fashion.

(2). Let $X$ be such that $A\in{}^*\Pow(X)$. 
It is easily verified that 
$\mathcal{P}(A)\cap\mathcal{I}=\{B\in{}^*\Pow(X)\mid B\subseteq A\}$,
and so the \emph{Internal Definition Principle} applies.

(3). Pick $X,Y$ such that $A\in{}^*\Pow(X)$ and 
$B\in{}^*\Pow(Y)$. By Proposition \ref{internalfunctions},
we know that 
$$\Fun(A,B)\cap\mathcal{I}=
\{F\in{}^*\F\mid \text{domain}(F)=A\ \text{and}\ \text{range}(F)\subseteq B\}$$
where $\F=\{f\ \text{function}\mid \text{domain}(f)\subseteq X\ \text{and}\ 
\text{range}(f)\subseteq Y\}$, and so $\Fun(A,B)\cap\mathcal{I}$ is internal
by the \emph{Internal Definition Principle}.
\end{proof}

\begin{definition}
An object that is not internal is called \emph{external}.
\end{definition}

Although ``bad'' with respect to \emph{transfer}\index{transfer principle},
there are relevant examples of external sets that are useful 
in the application of nonstandard methods.

\begin{example}
\

\begin{enumerate}
\item
The set of infinitesimal\index{hyperreal number!infinitesimal} hyperreal\index{hyperreal number} numbers is \emph{external}.
Indeed, it is a bounded subset of $\starR$ without least upper bound.
\item
The set of infinite hypernatural \index{hypernatural number!infinite} numbers is \emph{external}.
Indeed, it is a nonempty subset of $\starN$ without a least element.
\item
The set $\N$ of finite hypernatural numbers \index{hypernatural number} is \emph{external},
otherwise the set-difference $\starN\setminus\N$ 
of infinite numbers would be internal.
\end{enumerate}
\end{example}

The above examples shows that ${}^*\Pow(\N)\ne\Pow(\starN)$
and ${}^*\Pow(\R)\ne\Pow(\starR)$.
More generally, we have

\begin{proposition}\label{properinclusions}
\
\begin{enumerate}
\item
For every infinite set $A$, the set ${}^\sigma A=\{{}^*a\mid a\in A\}$
is external.
\item
Every infinite hyperextension $\starA$ has external subsets, that is,
the inclusion ${}^*\Pow(A)\subset\Pow(\starA)$ is proper.
\item
If the set $A$ is infinite and
$B$ contains at least two elements, then the inclusion
${}^*\Fun(A,B)\subset \Fun(\starA,\starB)$ is proper.
\end{enumerate}
\end{proposition}

\begin{proof}
(1). Pick a surjective map $\psi:A\to\N$. Then also the hyper-extension
${}^*\psi:\starA\to\starN$ is surjective. If by contradiction
${}^\sigma A$ was internal, also its image under ${}^*\psi$ would be,
and this is not possible, since
$${}^*\psi\left({}^\sigma A\right)\ =\ 
\left\{{}^*\psi({}^*a)\mid a\in A\right\}\ =\ 
\left\{{}^*(\psi(a))\mid a\in A\right\}\ =\ 
\left\{\psi(a)\mid a\in A\right\}\ =\ \N.$$

(2). Notice first that $A$ is infinite, because if
$A=\{a_1,\ldots,a_n\}$ was finite 
then also $\starA=\{{}^*a_1,\ldots,{}^*a_n\}$  would be finite.
Recall that ${}^*\Pow(A)$ is the set of all internal subsets of $\starA$.
Since ${}^\sigma A\subset\starA$ is external by (1),
${}^\sigma A\in\Pow(\starA)\setminus{}^*\Pow(A)$.

(3). Recall that ${}^*\Fun(A,B)$ is the set of all internal 
functions $f:\starA\to\starB$. 
Pick an external subset $X\subset A$, pick $b_1\ne b_2$ in $B$,
and let $f:\starA\to\starB$ be the function where
$f(a)={}^*b_1$ if $a\in X$ and $f(a)={}^*b_2$ if $a\notin X$.
Then $f$ is external, otherwise the preimage $f^{-1}({}^*\{b_1\})=X$
would be internal.
\end{proof}

We warn the reader that becoming familiar with the distinction
between internal and external objects is probably the hardest part
of learning nonstandard analysis.

\subsection{Internal objects in the ultrapower model}

The ultrapower\index{ultrapower} model $\starR=\R^I/\U$ that we introduced in 
Section \ref{sec-ultrapower} can be naturally extended so as
to include also hyper-extensions of families 
of sets of real tuples, and of families of functions.

Let us start by observing that every $I$-sequence 
$T=\langle T_i\mid i\in I\rangle$ of
sets of real numbers $T_i\subseteq\R$ 
determines a set $\widehat{T}\subseteq\starR$ 
of hyperreal numbers in a natural way, by letting
$$\widehat{T}\ =\ 
\left\{[\sigma]\in\starR\,\bigm|\{i\in I\mid \sigma(i)\in T_i\}\in\U\right\}.$$

\begin{definition}\label{ultraextensionfamilies}
If $\F\subseteq\mathcal{P}(\R)$,
then its \emph{hyper-extension} ${}^*\F\subseteq{}^*\mathcal{P}(\R)$
is defined as
$${}^*\F\ =\ \left\{\widehat{T}\bigm| T:I\to\F\right\}.$$
\end{definition}

We remark that the same definition above also applies to 
families $\F\subseteq\mathcal{P}(\R^k)$ of sets of $k$-tuples,
where for $I$-sequences $T:I\to\mathcal{P}(\R^k)$ one lets
$\widehat{T}=\left\{([\sigma_1],\ldots,[\sigma_k])\in\starR^k\bigm|
\{i\in I\mid (\sigma_1(i),\ldots,\sigma_k(i))\in T_i\}\in\U\right\}$.

According to Definition \ref{def-internal},
$A\subseteq\starR$ is internal if and only if $A\in{}^*\Pow(\R)$.
So, in the ultrapower model, $A\subseteq\starR$ is internal
if and only if $A=\widehat{T}$ for
some $I$-sequence $T:I\to\mathcal{P}(\R)$.

Analogously as above, every $I$-sequence 
$F=\langle F_i\mid i\in I\rangle$ of real functions 
$F_i:\R\to\R$ determines a 
function $\widehat{F}:\starR\to\starR$ on the hyperreal numbers 
by letting for every $\sigma:I\to\R$:
$$\widehat{F}([\sigma])\ =\ [\langle F_i(\sigma(i))\mid i\in I\rangle].$$

The internal functions \index{internal!function}from $\starR$ to $\starR$
in the ultrapower\index{ultrapower} model are precisely those that are determined by 
some $I$-sequence $F:I\to\Fun(\R,\R)$.

\begin{definition}
If $\G\subseteq\Fun(\R,\R)$, then its \emph{hyper-extension} 
${}^*\G\subseteq{}^*\Fun(\R,\R)$ is defined as
$${}^*\G\ =\ \left\{\widehat{F}\bigm| F:I\to\G\right\}.$$
\end{definition}

If $F=\langle F_i\mid i\in I\rangle$ is an $I$-sequence
of functions $F_i:\R^k\to\R$ of several variables,
one extends the above definition by letting
$\widehat{F}:\starR^k\to\starR$ be the function
where for every $\sigma_1,\ldots,\sigma_k:I\to\R$:
$$\widehat{F}([\sigma_1],\ldots,[\sigma_k])\ =\ 
[\langle F_i(\sigma(i))\mid i\in I\rangle].$$
Indeed, also in this case, if $\G\subseteq\Fun(\R^k,\R)$ 
then one puts ${}^*\G=\left\{\widehat{F}\bigm| F:I\to\G\right\}$.

\section{Hyperfinite sets}

In this section we introduce a fundamental tool
in nonstandard analysis, namely the class of \emph{hyperfinite sets}\index{hyperfinite!set}. 
Although they may contain infinitely many elements,
hyperfinite sets satisfy the same ``elementary properties" as finite sets.
For this reason they are instrumental in applications as
a convenient bridge between finitary statements and infinitary notions.

\begin{definition}\label{hyperfiniteset}\index{hyperfinite!set}
A \emph{hyperfinite set} $A$ is an element of the hyper-extension ${}^*\F$
of a family $\F$ of finite sets.
\end{definition}

In particular, hyperfinite\index{hyperfinite!set} sets are internal objects.

\begin{remark}
In the ultrapower\index{ultrapower} model, the hyperfinite subsets of $\starR$
are defined according to Definition \ref{ultraextensionfamilies}. 
Precisely, $A\subseteq\starR$ is hyperfinite\index{hyperfinite!set} if and only if
there exists a sequence $\langle T_i\mid i\in I\rangle$
of finite sets $T_i\subset\R$ such that $A=\widehat{T}$, that is, 
for every $\sigma:I\to\R$,
$[\sigma]\in A\Leftrightarrow \{i\in I\mid\sigma(i)\in T_i\}\in\U$.
\end{remark}

Let us start with the simplest properties of hyperfinite\index{hyperfinite!set} sets.

\begin{proposition}\label{hyperfinitebasics}
\

\begin{enumerate}
\item
A subset $A\subseteq\starX$ is hyperfinite\index{hyperfinite!set} if and only if $A\in{}^*\Fin(X)$,
where $\Fin(X)=\{A\subseteq X\mid A\ \text{is finite}\}$.
\item
Every finite set of internal objects\index{internal!object} is hyperfinite.
\item
A set of the form $\starX$ for some standard set $X$ is hyperfinite\index{hyperfinite!set} if and only if $X$ is finite.
\item
If $f:A\to B$ is an internal function\index{internal!function}, and $\Omega\subseteq A$ is
a hyperfinite\index{hyperfinite!set} set, then its image $f(\Omega)=\{\starf(\xi)\mid\xi\in\Omega\}$
is hyperfinite as well.  In particular, internal subsets of hyperfinite sets are hyperfinite.
\end{enumerate}
\end{proposition}

\begin{proof}
(1). If $A$ is a hyperfinite\index{hyperfinite!set} subset of $\starX$, then $A$ is internal,
and hence $A\in{}^*\mathcal{P}(X)$. So, if $\F$ is a family of finite sets with 
$A\in{}^*\F$, then $A\in{}^*\mathcal{P}(X)\cap{}^*\F=
{}^*(\mathcal{P}(X)\cap\F)\subseteq{}^*\Fin(X)$.
The converse implication is trivial.

(2). Let $A=\{a_1,\ldots,a_k\}$, and pick $X_i$ such that $a_i\in\starX_i$.
If $X=\bigcup_{i=1}^n X_i$, then $A\in{}^*\Fin(X)$, as it is easily
shown by applying \emph{transfer} to the elementary property:
``$\forall x_1,\ldots,x_k\in X\ \{x_1,\ldots,x_k\}\in\Fin(X)$''.

(3). This is a direct consequence of \emph{transfer} and the definition of hyperfinite set.


(4). Pick $X$ and $Y$ with $A\in{}^*\mathcal{P}(X)$ and 
$B\in{}^*\mathcal{P}(Y)$.
Then apply \emph{transfer} to the property: 
``For every $C\in\mathcal{P}(X)$, for every $D\in\mathcal{P}(Y)$,
for every $f\in\Fun(C,D)$ and for every $F\in\Fin(X)$
with $F\subseteq C$,
the image $f(F)\in\Fin(Y)$''.
\end{proof}

\begin{example}
For every pair $N<M$ of (possibly infinite) 
hypernatural numbers\index{hypernatural number}, the interval
$$[N,M]_{\starN}\ =\ \{\alpha\in\starN\mid N\le\alpha\le M\}$$
is hyperfinite. Indeed, applying \emph{transfer} to the property:
``For every $x,y\in\N$ with $x<y$, the set $[x,y]_\N=\{a\in\N\mid x\le a\le y\}\in\Fin(\N)$",
one obtains that $[N,M]_{\starN}\in{}^*\Fin(\N)$.\footnote
{~More formally, one transfers the formula:
``$\forall x,y\in\N\ [(x<y\Rightarrow(\exists A\in\Fin(\N)\ 
\forall z\,(z\in A\leftrightarrow x\le z\le y))]$''.}
More generally, it follows from transfer that every bounded internal set of hyperintegers
is hyperfinite.
\end{example}

Whenever confusion is unlikely, we will omit the subscript, and 
write directly $[N,M]$ to denote the interval of hypernatural numbers
determined by $N,N\in\starN$.

\begin{definition}\label{hyperfinitesequence}\index{hyperfinite!sequence}
A \emph{hyperfinite sequence} is an internal function whose domain
is a hyperfinite set $A$.  
\end{definition}

Typical examples of hyperfinite sequences are defined on initial
segments $[1,N]\subset\starN$ of the hypernatural numbers. In this case
we use notation $\langle \xi_\nu\mid\nu=1,\ldots,N\rangle$

By \emph{transfer} from the property:
``For every nonempty finite set $A$ there exists a unique $n\in\N$ such that
$A$ is in bijection with the segment $\{1,\ldots,n\}$,''
one obtains that there is a well-posed definition of cardinality 
for hyperfinite sets.

\begin{definition}
The \emph{internal cardinality}\index{internal!cardinality}
$|A|_h$ of a nonempty hyperfinite\index{hyperfinite!set} set $A$
is the unique hypernatural number \index{hypernatural number}$\alpha$ such that
there exists an internal bijection $f:[1,\alpha]\to A$.
\end{definition}

\begin{proposition} The internal cardinality satisfies the following properties:

\begin{enumerate}
\item
If the hyperfinite set $A$ is finite, then $|A|_h=|A|$.

\smallskip
\item
For any $\nu\in \starN$, we have
$|[1,\nu]|_h=\nu$. More generally, we have $|[\alpha,\beta]|_h=\beta-\alpha+1$.
\end{enumerate}
\end{proposition}

\begin{proof}
(1). If $A$ is a finite internal
set of cardinality $n$, then every bijection
$f:[1,n]\to A$ is internal, in that a finite set
of internal ordered pairs.

(2). The map $f:[1,\beta-\alpha+1]\to[\alpha,\beta]$
where $f(i)=\alpha+i-1$ is an internal bijection. 
\end{proof}

When confusion is unlikely, we will drop the subscript and
directly write $|A|$ to also denote the internal cardinality of 
a hyperfinite set $A$.

The following is a typical example of a property
that hyperfinite\index{hyperfinite!set} sets inherit from finite sets. It is
obtained by a straightforward application of \emph{transfer},
and its proof is left as an exercise.

\begin{proposition}
Every nonempty hyperfinite subset of $\starR$
has a least and a greatest element.
\end{proposition}

A relevant example of a hyperfinite\index{hyperfinite!set} set which is useful in applications is the following.

\begin{definition}\label{def-hyperfinitegrid}\index{hyperfinite!grid}
Fix an infinite $N\in\starN$. The corresponding
\emph{hyperfinite grid} $\mathbb{H}_N\subset\starQ$ is the hyperfinite set
that determines a partition of the interval $[1,N]\subset\starR$ 
of hyperreals\index{hyperreal number} into $N$-many intervals of equal infinitesimal 
length $1/N$. Precisely:
$$\mathbb{H}_N\ =\
\left\{\pm\, \frac{\alpha}{N}\,\Bigm|\,\alpha=0, 1, \ldots, N\right\}.$$
\end{definition}

We close this section with a couple of result about the
(infinite) cardinalities of hyperfinite\index{hyperfinite!set} sets. 

\begin{proposition}\label{hyperfinitecardinality}
If $\alpha\in\starN$ is infinite, then the corresponding
interval $[1,\alpha]\subset\starN$ has cardinality at least the cardinality of the 
\emph{continuum}.
\end{proposition}

\begin{proof}
For every real number $r\in(0,1)$, let
$$\psi(r)\ = \min\{\beta\in[1,\alpha]\mid r<\beta/\alpha\}.$$
Notice that the above definition is well-posed, because
$\{\beta\in\starN\mid r<\beta/\alpha\}$ is an internal bounded set
of hypernatural numbers\index{hypernatural number}, and hence a hyperfinite\index{hyperfinite!set} set.
The map $\psi:(0,1)_\R\to[1,\alpha]_{\starN}$ is 1-1. Indeed, 
$\psi(r)=\psi(s)\Rightarrow |r-s|<1/\alpha\Rightarrow r\sim s\Rightarrow r=s$
(recall that two real numbers that are infinitely close are necessarily equal).
Thus, we obtain the desired inequality
$\mathfrak{c}=|(0,1)_\R|\le|[1,\alpha]_{\starN}|$.
\end{proof}

\begin{corollary}
If $A$ is internal, then either $A$ is finite or
$A$ has at least the cardinality of the \emph{continuum}.
In consequence, every countably infinite set is external.
\end{corollary}

\begin{proof}
It is easily seen by \emph{transfer} that an internal\index{internal!set} set $A$ is 
either hyperfinite\index{hyperfinite!set}, and hence it is in bijection with an interval 
$[1,\alpha]\subset\starN$, or there exists an internal
1-1 function\index{internal!function} $f:\starN\to A$.
In the first case, if $\alpha\in\N$ is finite, then trivially $A$ is finite.
Otherwise $|A|=[1,\alpha]\ge\mathfrak{c}$ by the previous proposition.
In the second case, if $\alpha$ is any infinite hypernatural 
number\index{hypernatural number}, then
$|A|\ge|\starN|\ge|[1,\alpha]|\ge\mathfrak{c}$.
\end{proof}

\subsection{Hyperfinite sums}\label{sec-hyperfinitesums}

Similarly to finite sums of real numbers,
one can consider \emph{hyperfinite} sums\index{hyperfinite!sum} of
hyperfinite sets of hyperreal numbers\index{hyperreal number}. 

\begin{definition}
If $f:A\to\R$ then for every nonempty hyperfinite subset 
$\Omega\subset\starA$, one defines the
corresponding \emph{hyperfinite sum}\index{hyperfinite!sum} by setting:
$$\sum_{\xi\in \Omega}\,\starf(\xi)\ :=\ {}^*S_f(\Omega),$$
where $S_f:\Fin(A)\setminus\{\emptyset\}\to\R$ is the function
$\{r_1<\ldots<r_k\}\mapsto f(r_1)+\ldots+f(r_k)$.

As a particular case, if $a=\langle a_n\mid n\in\N\rangle$ 
is a sequence of real numbers and $\alpha\in\starN$ is a 
hypernatural number\index{hypernatural number}, 
then the corresponding \emph{hyperfinitely long sum}\index{hyperfinite!sum} is defined as
$$\sum_{i=1}^\alpha a_i\ =\ {}^*S_a(\alpha)$$
where $S_a:\N\to\R$ is the function $n\mapsto a_1+\ldots+a_n$.
\end{definition}

\begin{remark}
More generally, the above definition can be extended to hyperfinite sums
$\sum_{\xi\in\Omega}F(\xi)$\index{hyperfinite!sum} where 
$F:\starA\to\starR$ is an internal function, 
and $\Omega\subseteq\starA$ is a nonempty hyperfinite subset.
Precisely, in this case one sets
$\sum_{\xi\in\Omega}F(\xi)={}^*\mathcal{S}(F,\Omega)$, where
$\mathcal{S}:\Fun(A,\R)\times(\Fin(A)\setminus\{\emptyset\})\to\R$ 
is the function $(f,G)\mapsto\sum_{x\in G}f(x)$.
\end{remark}


Let us mention in passing that hyperfinite sums can be used
to directly define integrals. Indeed, if $N\in\starN$ is any infinite 
hypernatural number\index{hypernatural number} and $\mathbb{H}$ is the corresponding 
hyperfinite grid (see Definition \ref{def-hyperfinitegrid}), 
then for every $f:\R\to\R$ and for every $A\subseteq\R$,
one defines the \emph{grid integral} by putting:
$$\int_A f(x)d_{\mathbb{H}}(x)\ =\ 
\st\left(\sum_{\xi\in\mathbb{H}\cap\starA}{}^*f(\xi)\right).$$ 
Notice that the above definition applies to \emph{every}
real function $f$ and to \emph{every} subset $A$.
Moreover, it can be shown that if $f:[a,b]\to\R$ is a Riemann integrable function
defined on an interval, then the grid integral 
coincides with the usual Riemann integral. 


\section{Overflow and underflow principles}
\index{overflow principle}\index{underflow principle}

\begin{proposition}[Overflow principles]\label{overflow}
\

\begin{enumerate}
\item
$A\subseteq\N$ is infinite if and only if its hyper-extension $\starA$
contains an infinite number. 
\item
If $B\subseteq\starN$ is internal\index{internal!set} and $B\cap\N$ is infinite 
then $B$ contains an infinite number.
\item
If $B\subseteq\starN$ is internal and $\N\subseteq B$ then 
$[1,\alpha]\subseteq B$ for some infinite $\alpha\in\starN$.
\end{enumerate}
\end{proposition}

\begin{proof}
Item 1 follows from Propositions \ref{finitedoesntenlarge} and \ref{initialsegment}.  For item 2, suppose that $B$ does not contain an infinite number.  Then $B$ is bounded above in $\starN$.  By \emph{transfer}, $B$ has a maximum, which is necessarily an element of $\N$, contradicting that $B\cap \N$ is infinite.  For item 3, let $C:=\{\alpha\in \starN \ : \ [1,\alpha]\subseteq B\}$.  Then $C$ is internal and $\N\subseteq C$ by assumption.  By item 2 applied to $C$, there is $\alpha\in C$ that is infinite. This $\alpha$ is as desired.
\end{proof}

\begin{proposition}[Underflow principles]
\

\begin{enumerate}
\item
If $B\subseteq\starN$ is internal and $B$ contains arbitrarily
small infinite numbers, then $B$ contains a finite number.
\item
If $B\subseteq\starN$ is internal and $[\alpha,+\infty)\subseteq B$ 
for every infinite $\alpha\in\starN$ then
then $[n,+\infty)\subseteq B$ for some finite $n\in\N$.
\end{enumerate}
\end{proposition}

\begin{proof}
For item 1, suppose that $B$ does not contain a finite number.  Then the minimum of $B$ is necessarily infinite, contradicting the assumption that $B$ contains arbitrarily small infinite numbers.  Item 2 follows by applying item 1 to the internal set $C:=\{\alpha\in \starN \ : \ [\alpha,+\infty)\subseteq B\}$.
\end{proof}

In practice, one often says they are using \emph{overflow} when they are using any of the items in Proposition \ref{overflow} and likewise for \emph{underflow}.  
Below we will present a use of \emph{overflow} in graph theory.

\subsection{An application to graph theory}

Recall that a \emph{graph} is a set $V$ (the set of 
\emph{vertices}) endowed with an anti-reflexive and symmetric binary
relation $E$ (the set of \emph{edges}). 
Notice that if $G=(V,E)$ is a graph\index{graph} then also 
its hyper-extension ${}^*G=({}^*V,{}^*E)$ is a graph. 
By assuming as usual that ${}^*v=v$ for all $v\in V$,
one has that $G$ is a sub-graph of ${}^*G$.
A graph $G=(V,E)$ is \emph{locally finite}\index{graph!locally finite} if
for every vertex $v\in V$, its \emph{set of neighbors} 
$N_G(v)=\{u\in V\mid \{u,v\}\in E\}$ is finite. 
One has the following simple nonstandard characterization.

\begin{proposition}
A graph $G=(V,E)$ is locally finite if and only if
${}^*(N_G(v))\subseteq V$ for every $v\in V$.
\end{proposition}

\begin{proof}
If $G$ is locally finite then for every $v\in V$ the set
of its neighbors $N_G(v)=\{u_1,\ldots,u_n\}$ is finite,
and so ${}^*N_G(v)=\{{}^*u_1,\ldots,{}^*u_n\}=\{u_1,\ldots,u_n\}\subseteq V$.
Conversely, if $G$ is not locally finite, then there exists a vertex $v\in V$
such that $N_G(v)$ is infinite, and we can pick an element
$\tau\in{}^*(N_G(v))\setminus N_G(v)$. 
Now, $\tau\notin V$, as otherwise
$\tau\in{}^*(N_G(v))\cap V=N_G(v)$, a contradiction.
\end{proof}

Recall that a \emph{finite path} in a graph $G=(V,E)$
is a finite sequence $\langle v_i\mid i=1,\ldots,n\rangle$
of pairwise distinct vertexes such that $\{v_i,v_{i+1}\}\in E$ for every $i<n$.
A graph is \emph{connected} if for every pair of
distinct vertices $u,u'$ there exists a finite path $\langle v_i\mid i=1,\ldots,n\rangle$
where $v_1=u$ and $v_n=u'$.
An \emph{infinite path} is a sequence $\langle v_i\mid i\in\N\rangle$
of pairwise distinct vertexes such that $\{v_i,v_{i+1}\}\in E$ for every $i\in\N$.


\begin{theorem}[K\"onig's Lemma - I]
Every infinite connected graph that is locally finite
contains an infinite path.
\end{theorem}

\begin{proof}
Given a locally finite connected graph $G=(V,E)$ where $V$ is infinite,
pick $u\in V$ and $\tau\in {}^*V\setminus V$.
Since $G$ is connected, by \emph{transfer} there exists a hyperfinite
sequence
$\langle v_i\mid i=1,\ldots,\mu\rangle$ for some $\mu\in\starN$
where $v_1=u$ and $\{v_i,v_{i+1}\}\in{}^*E$ for every $i<\mu$.
By local finiteness, ${}^*(N_G(v_1))\subseteq V$
and so $v_2\in V$ and $\{v_1,v_2\}\in E$. Then, by induction, it is easily
verified that the restriction $\langle v_i\mid i\in\N\rangle$
of the above sequence to the finite indexes is an infinite path in $G$.
\end{proof}

A simple but relevant application of \emph{overflow}
proves the following equivalent formulation in terms of trees.

\begin{theorem}[K\"onig's Lemma - II]
Every infinite, finitely branching tree has an infinite path.
\end{theorem}

\begin{proof}
Let $T_n$ denote the nodes of the tree of height $n$.  Since $T$ is finitely branching, 
each $T_n$ is finite.  Since $T$ is infinite, each $T_n\not=\emptyset$.  
By \emph{overflow}\index{overflow principle}, there is $N>\N$ such that $T_N\not=\emptyset$.  
Fix $x\in T_N$.  Then $\{y\in T \mid y\text{ is connected }x\}$ is an infinite branch in $T$.
\end{proof}

\section{The saturation principle}

The \emph{transfer principle}\index{transfer principle} is all that one needs to
to develop the machinery of nonstandard analysis, 
but for advanced applications another property is also necessary, namely:

\begin{definition}
\emph{Countable Saturation Principle:}\index{saturation!countable}\
Suppose $\{B_n\}_{n\in\N}\subseteq\starA$ is a
countable family of internal sets
with the finite intersection property.
Then
$\bigcap_{n\in\N}B_n\neq\emptyset$.
\end{definition}

\begin{exercise}
Assume countable saturation. Then for every sequence $\langle B_n\mid n\in\N\rangle$
of internal elements can be extended to an internal sequence $\langle B_n\mid n\in\starN\rangle$,
that is, there exists an internal function $\sigma$ with domain $\starN$
and such that $\sigma(n)=B_n$ for every $n\in\N$.
\end{exercise}

Countable saturation\index{saturation!countable} will be instrumental in the
definition of \emph{Loeb measures}.
In several contexts, stronger saturation principles are assumed 
where also families of larger size are allowed.
Precisely, if $\kappa$ is a given uncountable cardinal,
then one considers the following.

\begin{definition}
$\kappa$-\emph{saturation property}\index{saturation!$\kappa $-}:\
If $\mathcal{B}\subseteq\starA$ is a family of
internal subsets of cardinality where $|\mathcal{B}|<\kappa$, and if
$\mathcal{B}$ has the finite intersection property,
then $\,\bigcap_{B\in\mathcal{B}}{}B\neq\emptyset$.
\end{definition}

Notice that, in this terminology, countable saturation\index{saturation!countable} is
$\aleph_1$-saturation.

In addition to countable saturation, in the applications
presented in this book, we will only use
the following weakened version of $\kappa$-saturation\index{enlarging property!$\kappa $-},
where only families of hyper-extensions are considered.

\begin{definition}
$\kappa$-\textit{enlarging property}:\
Suppose $\F\subseteq\mathcal{P}(A)$ has cardinality
$|\F|<\kappa$. If $\F$ has the finite intersection property,
then $\bigcap_{F\in\F}{}^*F\neq\emptyset$.\footnote{
~We remark that the enlarging property is strictly
weaker than saturation, in the sense that for every infinite $\kappa$
there are models of nonstandard analysis where the 
$\kappa$-enlarging property holds but $\kappa$-saturation fails.}
\end{definition}

As a first important application of the enlarging property,
one obtains that sets are included in a hyperfinite subset
of their hyper-extension. 

\begin{proposition}\label{hypapprox}
If the $\kappa$-enlarging property holds, then for
every set $X$ of cardinality $|X|<\kappa$
there exists a hyperfinite subset $H\subseteq\starX$
such that $X\subseteq H$.
\end{proposition}

\begin{proof}
For each $a\in X$, let $X_a:=\{Y\subseteq X \ : \ Y \text{ is finite and }a\in X\}$.  One then applies the $\kappa$-enlarging property to the family $\mathcal{F}:=\{X_a \ : \ a\in X\}$ to obtain $H\in \bigcap_{a\in X}{}^{\ast}X_a$. Such $H$ is as desired.
\end{proof}

%
%
%

Although it will not play a role in this book, we would be remiss if we did not mention the following example:

\begin{example}\label{topology}
Let $(X,\tau)$ be a topological space
with \textit{character} $<\kappa$, that is, such that
each point $x\in X$ has a base of neighborhoods $\mathcal{N}_x$ 
of cardinality less than $\kappa$. If we assume
the $\kappa$-enlarging property, the intersection
$\mu(x)=\bigcap_{U\in\mathcal{N}_x}{}^*U$ is nonempty.
In the literature, $\mu(x)$ is called the \textit{monad}
of $x$. Monads are the basic ingredient in applying nonstandard
analysis to topology, starting with the following characterizations
(see, \emph{e.g.}, \cite{lindstrom_invitation_1988} Ch.III):

\begin{itemize}
\item
$X$ is \textit{Hausdorff} if and only if $\mu(x)\cap\mu(y)=\emptyset$ 
whenever $x\ne y$;
\item
$A\subseteq X$ is \textit{open} if and only if
for every $x\in A$, $\mu(a)\subseteq{}^*A$;
\item
$C\subseteq X$ is \textit{closed} if and only if
for every $x\notin C$, $\mu(x)\cap{}^*C=\emptyset$;
\item
$K\subseteq X$ is \textit{compact} if and only if
${}^*K\subseteq\bigcup_{x\in K}\mu(x)$.
\end{itemize}
\end{example}

\begin{exercise}
Use the above characterization of compactness to show that $\beta S$ is the Stone-\v{C}ech compactification of $S$.  (Hint:  To prove that $\beta S$ is compact, given $\u\in {}^{\ast}\beta S$, define $\mathcal{V}:=\{A\subseteq S \ : \ {}^{\ast}A\in \u\}$.  Show that $\u\in \beta S$ and $\u\in \mu(\mathcal{V})$.  To establish the universal property of $\beta S$, given a function $f:S\to K$ where $K$ is a compact Hausdorff space, define $\overline{f}:\beta S\to K$ by defining $\overline{f}(\u_\alpha):=\st(f(\alpha))$, where $\st(f(\alpha))$ is the unique $x\in K$ such that $f(\alpha)\in \mu(x)$.  Notice that this map is well-defined.)
\end{exercise}
\subsection{Saturation in the ultrapower model}

We now show that the ultrapower\index{ultrapower} model $\starR=\R^I/\U$ 
introduced in Section \ref{sec-ultrapower} provides an example of nonstandard map that satisfies
saturation. Let us start with a direct combinatorial proof
in the case of ultrapowers modulo ultrafilters on $\N$.

\begin{theorem}
For every non-principal ultrafilter $\U$ on $\N$,
the corresponding ultrapower model satisfies countable saturation.
\end{theorem}

\begin{proof}
Let $\{B_n\}$ be a countable family of internal subsets of $\starR$
with the finite intersection property. For every $n$, pick
a function $T_n:\N\to\Pow(\R)$ such that 
$$B_n\ =\ \widehat{T}_n\ =\ 
\left\{[\sigma]\in\starR\mid \{i\in \N\mid \sigma(i)\in T_n\}\in\U\right\}.$$
For any fixed $n$, pick an element
$\tau(n)\in T_1(n)\cap\cdots\cap T_n(n)$
if that intersection is nonempty. Otherwise, pick an element
$\tau(n)\in T_1(n)\cap\cdots\cap T_{n-1}(n)$
if that intersection is nonempty, and so forth 
until $\tau(n)$ is defined.
We agree that $\tau(n)=0$ in case $T_1(n)=\emptyset$.
By the definition of $\tau$, one has the following property:

\begin{itemize}
\item
If $T_1(n)\cap\cdots\cap T_k(n)\neq\emptyset$
and $n\ge k$ then $\tau(n)\in T_1(n)\cap\ldots\cap T_k(n)$.
\end{itemize}

\noindent
Now let $k$ be fixed. By the finite intersection property,
$\widehat{T}_1\cap\ldots\cap\widehat{T}_k\ne\emptyset$, so
there exists $\sigma:\N\to\R$ such that
$\Lambda_j=\{i\in\N\mid \sigma(i)\in T_j(i)\}\in\U$ for every $j=1,\ldots,k$.
In particular, the set of indexes
$\Gamma(k)=\{i\in\N\mid T_1(i)\cap\ldots\cap T_k(i)\ne\emptyset\}\in\U$
because it is a superset of $\Lambda_1\cap\ldots\cap\Lambda_k\in\U$.
But then the set $\{i\in\N\mid\tau(i)\in T_1(i)\cap\ldots\cap T_k(i)\}\in\U$
because it is a superset of $\{i\in\Gamma(k)\mid i\ge k\}\in\U$.
We conclude that $[\tau]\in\widehat{T}_1\cap\ldots\cap\widehat{T}_k$.
As this holds for every $k$, the proof is completed.
\end{proof}

The above result can be extended to all ultrapower models
where the ultrafilter $\U$ on $I$ is 
\emph{countably incomplete}\index{ultrafilter!countably incomplete}
(recall that every non-principal ultrafilter on $\N$
is countably incomplete).

\begin{theorem}
For every infinite cardinal $\kappa$ there exist ultrafilters
$\U$ on the set $I=\text{Fin}(\kappa)$ of finite parts of $\kappa$
such that the corresponding ultrapower model 
satisfies the $\kappa^+$-enlarging property.
\end{theorem}

\begin{proof}
For every $x\in\kappa$, let $\widehat{x}=\{a\in I\mid x\in a\}$.
Then trivially the family $\mathcal{X}=\{\widehat{x}\mid x\in\kappa\}$ has
the finite intersection property. We claim that every ultrafilter
$\U$ that extends $\mathcal{X}$ has the desired property.

Suppose that the family $\F=\{B_x\mid x\in\kappa\}\subseteq\Pow(A)$ 
satisfies the finite intersection property. Then we can pick a sequence
$\sigma:I\to A$ such that $\sigma(a)\in \bigcap_{x\in a}A_x$ 
for every $a\in I$. The proof is completed by noticing that
$[\sigma]\in \starA_x$ for every $x\in\kappa$, since
$\{a\in I\mid \sigma(a)\in A_x\}\supseteq \widehat{x}\in\U$.
\end{proof}

A stronger result holds, but we will not prove it here because
it takes a rather technical proof, and we do not need that
result in the applications presented in this book.

\begin{theorem}
For every infinite cardinal $\kappa$ there exist ultrafilters
$\U$ on $\kappa$ (named $\kappa^+$-\emph{good ultrafilters})
such that the corresponding ultrapower models
satisfy the $\kappa^+$-saturation property.
\end{theorem}

\begin{proof}
See \cite[\S 6.1]{chang_model_1977}.
\end{proof}

\section{Hyperfinite approximation}

As established in Proposition \ref{hypapprox}, in sufficiently saturated structures, hyperfinite sets can be conveniently used as ``approximations''\index{hyperfinite!approximation}
of infinite structure. The fact that they behave as finite sets makes them 
particularly useful objects in applications of nonstandard analysis.
In this section we will see a few examples to illustrate this.  We assume that the nonstandard extension satisfies the $\kappa$-enlarging property, where $\kappa$ is larger than the cardinality of the objects under consideration.

\begin{theorem}
Every infinite set can be linearly ordered.
\end{theorem}

\begin{proof}
Let $X$ be an infinite set and take hyperfinite $H\subseteq {}^{\ast}X$ such that 
$\{{}^*x\mid x\in X\}\subseteq H$.  By \emph{transfer} applied to the corresponding property of finite sets, $H$ can be linearly ordered, 
whence so can $\{{}^*x\mid x\in X\}$, and hence $X$.
\end{proof}

The next theorem is a generalization of the previous one:

\begin{theorem}
Every partial order on a set can be extended to a linear order.
\end{theorem}

\begin{proof}
We leave it as an easy exercise by induction to show that every partial order on a finite set can be extended to a linear order.  Thus, we may precede as in the previous theorem.  This time, $H$ is endowed with the partial order it inherits from ${}^{\ast}X$, whence, by \emph{transfer}, this partial order can be extended to a linear order.  This linear order restricted to $X$ extends the original partial order on $X$.
\end{proof}

\begin{theorem}
A graph is $k$-colorable if and only if every finite subgraph is $k$-colorable.
\end{theorem}

\begin{proof}
Suppose that $G$ is a graph such that every finite subgraph is $k$-colorable.  Embed $G$ into a hyperfinite subgraph $H$ of ${}^{\ast}G$.  By \emph{transfer}, $H$ can be $k$-colored.  The restriction of this $k$-coloring to $G$ is a $k$-coloring of $G$.
\end{proof}


The next result plays an important role in the application of ultrafilter and nonstandard methods.  
Say that $f:\N\to \N$ is \emph{fixed-point free} if $f(n)\not=n$ for all $n\in \N$.
\begin{theorem}\label{fixedpointfree}
Suppose that $f:\N\to \N$ is fixed-point free.  Then there is a function $c:\N\to \{1,2,3\}$ (that is, a \emph{$3$-coloring of $\N$}) such that $c(f(n))\not=c(n)$ for all $n\in \N$.
\end{theorem}

\begin{proof}
In order to use hyperfinite approximation\index{hyperfinite!approximation}, we first need a finitary version of the theorem:

\

\noindent \textbf{Claim:}  For every finite subset $F\subseteq \N$, there is a 3-coloring $c_F$ of $F$ such that $c(f(n))\not=c(n)$ whenever $n,f(n)\in F$.

\noindent \textbf{Proof of Claim:}  We prove the claim by induction on the cardinality of $F$, the case $|F|=1$ being trivial since $F$ never contains both $n$ and $f(n)$.  Now suppose that $|F|>1$.  Fix $m\in F$ such that $|f^{-1}(m)\cap F|\leq 1$. Such an $m$ clearly exists by the Pigeonhole principle.  Let $G:=F\setminus \{m\}$.  By the induction assumption, there is a 3-coloring $c_G$ of $G$ such that $c(f(n))\not=c(n)$ whenever $n,f(n)\in G$.  One extends $c_G$ to a 3-coloring $c_F$ of $F$ by choosing $c_F(m)$ different from $c_G(f(m))$ (if $f(m)\in G$) and different from $c_G(k)$ if $k\in G$ is such that $f(k)=m$ (if there is such $k$). Since we have three colors to choose from, this is clearly possible.  The coloring $c_F$ is as desired.  

\

Now that the claim has been proven, let $H\subseteq \starN$ be hyperfinite such that $\N\subseteq H$.  By \emph{transfer}, there is an internal 3-coloring $c_H$ of $H$ such that $c(f(n))\not=c(n)$ whenever $n,f(n)\in H$.  Since $n\in \N$ implies $n,f(n)\in H$, we see that $c_H|\N$ is a 3-coloring of $H$ as desired. 
\end{proof}

\section*{Notes and references}
Nonstandard analysis was introduced by A. Robinson in the 1960s \cite{robinson_non-standard_1966}. Robinson's original approach was based on model theory. Shortly after, Luxemburg proposed an alternative appoach based on the ultrapower construction \cite{luxemburg_non-standard_1973}, which helped to further popularize nonstandard methods. Indeed, the ultrapower construction is still one of the most common ways to present nonstandard methods. This is the approach followed in \cite{goldblatt_lectures_1998}, which is an accessible introduction to nonstandard analysis, including a rigorous formulation and a detailed proof of the
\emph{transfer principle}\index{transfer principle}.
The foundations of nonstandard analysis are also presented in detail in \S 4.4 of \cite{chang_model_1977}. A survey of several different possible introductions to nonstandard
methods is given in \cite{benci_eightfold_2006}.
A nice introduction to nonstandard methods
for number theorists, including many examples,
is presented in \cite{jin_introduction_2008} (see also \cite{jin_applications_2000}).
Finally, a full development of nonstandard analysis can be
found in several monographs in the existing literature;
see \emph{e.g.} H.J.\ Keisler's classical book \cite{keisler_infinitesimal_1984},
or the comprehensive collection of surveys in \cite{arkeryd_nonstandard_1997}.

\chapter{Hyperfinite generators of ultrafilters}\label{hypgenerator}

Throughout this chapter, we fix an infinite set $S$ and 
we assume that ${}^*s=s$ for every $s\in S$,
so that $S\subseteq{}^*S$.  

\section{Hyperfinite generators}

An important observation is that elements of $\starS$ generate ultrafilters\index{ultrafilter} on $S$\index{hyperfinite!generator}: 

\begin{exercise}
Suppose that $\alpha\in \starS$.  
Set $\mathcal{U}_\alpha:=\{A\subseteq S \ : \ \alpha \in {}^{\ast}A\}$.  
\begin{enumerate}
\item $\mathcal{U}_\alpha$ is an ultrafilter\index{ultrafilter} on $S$.
\item $\mathcal{U}_\alpha$ is principal if and only if $\alpha\in S$.
\end{enumerate}
\end{exercise}

We call $\u_\alpha$ the \emph{ultrafilter on $S$ generated by $\alpha$}\index{ultrafilter}.  
Note that in the case that $\alpha\in S$, there is no conflict between the notation 
$\u_\alpha$ in this chapter and the notation $\u_\alpha$ from Chapter 1. 

\begin{exercise}
For $k\in \N$ and $\alpha\in \starN$, show that $k\u_\alpha=\u_{k\alpha}$.\footnote{
~Recall from Definition \ref{def-ku} 
that $A\in k\u\Leftrightarrow A/k=\{n\in\N\mid nk\in A\}\in\u$.}
\end{exercise}

Recall from Exercise \ref{imageultrafilter} that, for for every function $f:S\to T$ and for
every ultrafilter $\U$\index{ultrafilter} on $S$, the \emph{image ultrafilter} $f(\U)$\index{ultrafilter}
is the ultrafilter on $T$ defined by setting
$$f(\U)\ =\ \{B\subseteq T\mid f^{-1}(B)\in\U\}.$$

\begin{exercise}
Show that \index{ultrafilter}
$f(\U_\alpha)=\U_{f(\alpha)}$.
\end{exercise}

Since there are at most $2^{2^{|S|}}$ ultrafilters\index{ultrafilter} on $S$, 
if the nonstandard extension is $\kappa$-saturated for $\kappa>2^{2^{|S|}}$, 
then $|{}^*S|>2^{2^{|S|}}$ and we see that there must exist distinct 
$\alpha,\beta\in \starS\setminus S$ such that $\u_\alpha=\u_\beta$
(see Proposition \ref{prop-enlarging} and Exercise \ref{ex-enlarging} below).
This leads to the following notion, which is of central importance in Part II of this book.

\begin{definition}
Given $\alpha,\beta\in \starS$, we say that $\alpha$ and $\beta$ are \emph{$u$-equivalent}, written $\alpha\sim\beta$, if $\u_\alpha=\u_\beta$.
\end{definition}  

Here are some useful properties of this relation on $\starS$:

\begin{proposition}\label{simproperties}

\

\begin{enumerate}
\item If $\alpha,\beta\in S$, then $\alpha\sim\beta$ if and only if $\alpha=\beta$.
\item Suppose that $f:\N\to \N$ and $\alpha\sim \beta$.  Then $f(\alpha)\sim f(\beta)$.
\item Suppose that $f:\N\to \N$ and $\alpha$ is such that $f(\alpha)\sim \alpha$.  Then $f(\alpha)=\alpha$.
\end{enumerate}
\end{proposition}

\begin{proof}
Items (1) and (2) are easy and left to the reader.  We now prove (3).  Suppose that $f(\alpha)\not=\alpha$.  Let $A:=\{n\in \N \ : \ f(n)\not=n\}$.  Let $g:\N\to \N$ be fixed-point free such that $f|A=g|A$.  By Theorem \ref{fixedpointfree}, there is a 3-coloring $c$ of $\N$ such that $c(g(n))\not=c(n)$ for all $n\in \N$.  By transfer, $c(g(\alpha))\not=c(\alpha)$.  Since $\alpha\in \starA$, we have $f(\alpha)=g(\alpha)$, so $c(f(\alpha))\not=c(\alpha)$.  Setting, $i:=c(\alpha)$ and $X:=\{n\in \N \ : \ c(n)=i\}$, we have that $\alpha\in {}^{\ast}X$ but $f(\alpha)\notin {}^{\ast}X$, whence $f(\alpha)\not\sim \alpha$.
\end{proof}

We have seen that elements of $\starS$ generate ultrafilters on $S$\index{ultrafilter}.  
Under sufficient saturation, the converse holds:

\begin{proposition}\label{prop-enlarging}
Assume that the nonstandard universe has the
$(2^{|S|})^+$-enlarging property.  
Then for every $\mathcal{U}\in \beta S$, there is $\alpha\in \starS$ such that $\mathcal{U}=\mathcal{U}_\alpha$. 
\end{proposition}

\begin{proof}
Fix $\mathcal{U}\in \beta S$.  
It is clear that $\u$ is a family of subsets of $S$ of cardinality $|\u|\leq 2^{|S|}$ with the finite intersection property, whence, by the $(2^{|S|})^+$-enlarging property,
there is $\alpha\in \bigcap_{A\in \u}\starA$.  Observe now that $\u=\u_\alpha$.  
\end{proof}

\begin{exercise}\label{ex-enlarging}
Assume the $(2^{|S|})^+$-enlarging property. Show that for every
non-principal $\U\in\beta S\setminus S$ there exist
$|\starN|$-many $\alpha\in\starN$ such that $\U=\U_\alpha$.
\end{exercise}

By the previous proposition, the map $\alpha\mapsto \u_\alpha:\starS\to \beta S$ is surjective. This suggests that we define a topology on $\starS$, called the \emph{$u$-topology on $\starS$}, by declaring the sets $\starA$, for $A\subseteq S$, to be the basic open sets.\footnote
{~This topology is usually named ``$S$-topology" in the literature
of nonstandard analysis, where the ``S'' stands for ``standard".}
This topology, while (quasi)compact by the enlarging property, is not Hausdorff.  
In fact, $\alpha,\beta\in \starS$ are \emph{not} separated in the $u$-topology precisely 
when $\alpha\sim \beta$.  Passing to the separation, we get a compact 
Hausdorff space $\starS/\!\sim$ and the surjection $\starS\to \beta S$ 
defined above descends to a homeomorphism between the
quotient space $\starS/\!\sim$ and $\beta S$. So, while $\beta S$
is the ``largest" Hausdorff compactification of the discrete space $S$, a (sufficiently
saturated) hyper-extension of $S$ is an even larger space, which is still
compact (but non-Hausdorff) and has $\beta S$ as a quotient.

\section{The case of a semigroup again}\label{hg}
Let us now suppose, once again, that $S$ is the underlying set of a semigroup $(S,\cdot)$.  One might guess that, for $\alpha,\beta\in \starS$, we have that the equation $\u_{\alpha\cdot \beta}=\u_\alpha\odot\u_\beta$ holds.  Unfortunately, this is not the case:

\begin{example}
Fix any $\alpha \in \starN\setminus \N$.  We show that there is $\beta \in \starN$ such that 
$\u_\alpha\oplus \u_\beta\not=\u_\beta\oplus \u_\alpha$.  For this $\beta$, we must have that either $\u_\alpha\oplus \u_\beta\not=\u_{\alpha+\beta}$ or $\u_\beta\oplus \u_\alpha\not=\u_{\beta+\alpha}$.

Let $A=\bigcup_{n \text{ even}}[n^2,(n+1)^2)$. 
Take $\nu\in \starN$ such that $\nu^2\leq \alpha<(\nu+1)^2$.
Without loss of generality, we may assume that $\nu$ is even.  
(The argument when $\nu$ is odd is exactly the same.)      
First suppose that $(\nu+1)^2-\alpha$ is finite.  In this case, we let $\beta:=\nu^2$.  Note that $\{n\in \N \ : \ (A-n)\in \u_\alpha\}=\{n\in \N \ : \ n+\alpha\in \starA\}$ is finite by assumption, whence not in $\u_\beta$.  Consequently, $A\notin \u_\beta\oplus \u_\alpha$.  However, since $\alpha-\beta$ is necessarily infinite, we have $\{n\in \N \ : \ (A-n)\in \u_\beta\}=\{n\in \N \ : \ n+\beta \in\starA\}=\N$, whence a member of $\u_\alpha$ and thus $A\in \u_\alpha\oplus \u_\beta$.
If $(\nu+1)^2-\alpha$ is infinite, then set $\beta:=(\nu+1)^2$.  An argument analogous to the argument in the previous paragraph shows that $A\notin \u_\alpha\oplus \u_\beta$ but $A\in \u_\beta\oplus \u_\alpha$.  
\end{example}

\begin{remark}
The previous argument also gives a nonstandard proof of the fact that the center of $(\beta \N,\oplus)$ is precisely the set of principal ultrafilters.
\end{remark}

The previous example notwithstanding, there is a connection between $(\beta S,\cdot)$ and the nonstandard extension\index{nonstandard extension} of the semigroup $(S,\cdot)$.  To see this, for notational cleanliness, let us switch over to writing the semigroup operation of $S$ by $+$ (even though the semigroup need not be commutative).  Fix $\alpha,\beta \in \starS$. Define $A\cdot \mathcal{U}^{-1}_{\beta }$ to be the set $\{ a\in S: \{b\in S:a\cdot b\in A \}\in \mathcal{U}_\beta \}$. For $a\in S$, we have that $a \in A \cdot \mathcal{U}^{-1}_{\beta }$ if and only if $a\cdot {}^{\ast}\beta \in {}^{\ast }A$. By transfer, we have that ${}^{\ast } (A \cdot \mathcal{U}^{-1}_{\beta })=\{\gamma \in {}^{\ast}S : \gamma \cdot {}^{\ast} \beta \in {}^{\ast \ast} A \}$ Hence, we have that 

$$A\in \mathcal{U}_\alpha\odot \mathcal{U}_\beta \Leftrightarrow \alpha \in {}^{\ast } (A\cdot \mathcal{U}^{-1}_{\beta }) \Leftrightarrow \alpha \cdot  {}^{\ast} \beta \in {}^{\ast \ast} A$$.


Wait!  What is ${}^{\ast\ast}A$?  And what is ${}^{\ast}\beta$?  Well, our intentional carelessness was intended to motivate the need to be able to take nonstandard extensions of nonstandard extensions, that is, to be able to consider \emph{iterated nonstandard extensions}.  Once we give this precise meaning in the next chapter, the above informal calculation will become completely rigorous and we have a precise connection between the operation $\oplus$ on $\beta S$ and the operation $+$ on ${}^{\ast\ast}S$.  

We should also mention that it is possible for the equality $\u_{\alpha}\odot \u_{\beta}=\u_{\alpha \cdot \beta}$ to be valid.  Indeed, this happens when $\alpha$ and $\beta$ are \emph{independent} in a certain sense; see \cite{di_nasso_hypernatural_2015}.

\section*{Notes and references}
 The notion of nonstandard generator of an ultrafilter was initially isolated by Luxemburg in 
 \cite{luxemburg_general_1969}. It was later used by Puritz \cite{puritz_skies_1972,puritz_ultrafilters_1971} and by Cherlin and Hirshfeld \cite{cherlin_ultrafilters_1972} to study the Rudin-Keisler order among ultrafilters.  Model theorists will recognize hyperfinite generators of ultrafilters simply as realizations of the types corresponding to the ultrafilters.  

\chapter{Many stars:  iterated nonstandard extensions}

\section{The foundational perspective}

As we saw in the previous chapter, it is useful in applications
to consider iterated hyper-extensions
of the natural numbers, namely $\starN$, ${}^*\starN$, 
${}^{**}\starN$, and so forth. A convenient
foundational framework where such iterations make sense
can be obtained by considering models of nonstandard analysis
where the \emph{standard universe} and the \emph{nonstandard universe} 
coincide.\footnote
{~A construction of such star maps is given in Section \ref{allowiterated} of the foundational appendix.}
In other words, one works with a \emph{star map}
$$*:\VV\to\VV$$
from a universe into itself. Clearly,
in this case every hyper-extension $\starX$ belongs
to the universe $\VV$, so one can apply
the \emph{star map} to it, and obtain the ``second level" hyper-extension
${}^{**}X$, and so forth.

Let us stress that the \emph{transfer principle}\index{transfer principle} in this context
must be handled with much care. The crucial point to
keep in mind is that in the equivalence
$$P(A_{1},\ldots,A_{n})\ \Longleftrightarrow\ 
P(\starA_{1},\ldots,\starA_{n}),$$
the considered objects $A_1,\ldots,A_n$ could be themselves
iterated hyper-extensions. In this case, one simply has to add
one more ``star". Let us elaborate on this with a few examples.

\begin{example}
Recall that $\N$ is an initial segment of $\starN$, that is,
$$\N\subset\starN\ \text{and}\ 
\forall x\in\N\ \forall y\in\starN\setminus\N\ \ x<y.$$
Thus, by \emph{transfer}, we obtain that:
$$\starN\subset{}^{**}\N\ \text{and}\ 
\forall x\in\starN\ \forall y\in{}^{**}\N\setminus\starN\ \ x<y.$$
This means that $\starN$ is a proper initial segment of
the double hyper-image ${}^{**}\N$, that is,
every element of ${}^{**}\N\setminus\starN$ is larger
than all element in $\starN$.
\end{example}

\begin{example}
If $\eta\in\starN\setminus\N$, then by \emph{transfer}
${}^*\eta\in{}^{**}\N\setminus\starN$, and hence $\eta<{}^*\eta$.
Then, again by \emph{transfer}, one obtains that
the elements ${}^*\eta,{}^{**}\eta\in{}^{***}\N$
are such that ${}^{*}\eta<{}^{**}\eta$, and so forth.
\end{example}

The above example clarifies that the simplifying assumption
${}^*r=r$ that was adopted for every $r\in\R$ cannot be
extended to hold for all hypernatural numbers \index{hypernatural number}. Indeed, 
we just proved that $\eta\ne{}^*\eta$ for every $\eta\in\starN\setminus\N$.

\begin{example}
Since $\R\subset\starR$, by \emph{transfer} it follows that
$\starR\subset{}^{**}\R$. If $\varepsilon\in\starR$ is a positive
infinitesimal, that is, if $0<\varepsilon<r$ for every positive $r\in\R$,
then by \emph{transfer} we obtain that
$0<{}^*\varepsilon<\xi$ for every positive $\xi\in\starR$.
In particular, ${}^*\varepsilon<\varepsilon$.
\end{example}

Recall that, by Proposition \ref{applyingstar},
for every elementary formula $\varphi(x,y_1,\ldots,y_n)$
and for all objects $B,A_1,\ldots,A_n$, one has that
$${}^*\{y\in B\mid P(y,A_1,\ldots,A_n)\}\ =\ \{y\in\starB\mid P(y_1,\starA_1,\ldots,\starA_n)\}. \quad (\dagger)$$

Of course one can apply the above property also when
(some of) the parameters are hyper-extensions.


\begin{remark}
In nonstandard analysis, a hyper-extension $\starA$
is often called a ``standard" set. This terminology comes from the fact
that -- in the usual approaches -- one considers
a star map $*:\mathbb{S}\to\mathbb{V}$
between the ``standard universe" $\mathbb{S}$
and a ``nonstandard universe" $\mathbb{V}$. 
Objects $A\in\mathbb{S}$ are named ``standard" and, with some
ambiguity, also their hyper-extensions $\starA$ are named ``standard".\footnote
{~To avoid ambiguity, some authors call the
hyper-extensions $\starA\in\mathbb{V}$ 
``internal-standard".} 
Let us stress that the name ``standard" would be misleading
in our framework, where there is
just one single universe, namely the universe
of \emph{all} mathematical objects. 
Those objects of our universe that happen to be in the
range of the star map, are called hyper-extensions.
\end{remark}

\section{Revisiting hyperfinite generators}

In this subsection, we let $(S,+)$ denote an infinite semigroup.  Now that we have the ability to take iterated nonstandard extensions, we can make our discussion from the end of Section \ref{hg} precise:  

\begin{proposition}
For $\alpha,\beta\in \starS$, we have $\u_\alpha\odot \u_\beta=\u_{\alpha \cdot {}^{\ast}\beta}$.
\end{proposition}

\begin{proof}

By equation $(\dagger)$ from the previous section, we have that ${}^{\ast } (A \cdot \mathcal{U}^{-1}_{\beta })=\{\gamma \in {}^{\ast}S : \gamma \cdot {}^{\ast} \beta \in {}^{\ast \ast} A \}$. Hence, for $A\subseteq S$, we have that

$$A\in \mathcal{U}_\alpha\odot \mathcal{U}_\beta \Leftrightarrow \alpha \in {}^{\ast } (A\cdot \mathcal{U}^{-1}_{\beta }) \Leftrightarrow \alpha \cdot  {}^{\ast} \beta \in {}^{\ast \ast} A.$$

\end{proof}

\begin{exercise}
The \emph{tensor product} $\U\otimes\V$
of two ultrafilters\index{ultrafilter} on $S$ is the ultrafilter on $S\times S$ defined by:
$$\U\otimes\V\ =\ \{C\subseteq S\times S\mid\{s\in S\mid C_s\in\V\}\in\U\},$$
where $C_s=\{t\in S\mid (s,t)\in C\}$ is vertical $s$-fiber of $C$.  If $\alpha,\beta \in \starS$, prove that $\u_\alpha\otimes \u_\beta=\u_{(\alpha,{}^{\ast} \beta)}$.
\end{exercise}


We can extend this discussion to elements of higher nonstandard iterates of the universe.  Indeed, given $\alpha \in {}^{k\ast}S$, we can define $\mathcal{U}_\alpha:=\{A\subseteq S \ : \ \alpha\in {}^{k\ast}A\}$.  \index{hyperfinite!generator}

\begin{exercise}
For $\alpha \in {}^{k\ast}S$, prove that $\u_\alpha=\u_{{}^{\ast}\alpha}$.
\end{exercise}

For $\alpha,\beta \in\bigcup_k {}^{k\ast}S$, we define $\alpha\sim \beta$ 
if and only if $\mathcal{U}_\alpha=\mathcal{U}_\beta$.  
Note that $\alpha$ and $\beta$ may live in different levels of the 
iterated nonstandard extensions\index{nonstandard extension!iterated}.


%
%



\begin{exercise}
Prove that, for $\alpha_0,\ldots,\alpha_k\in \starN$ and $a_0,\ldots,a_k\in \N$, one has
$$a_0\mathcal{U}_{\alpha_0}\oplus \cdots \oplus a_k\mathcal{U}_{\alpha_k}=\mathcal{U}_{a_0\alpha_0+a_1{}^{\ast}\alpha_1+\cdots+a_k{}^{k\ast}\alpha_k}.$$
\end{exercise}

\begin{exercise}

\

\begin{enumerate}
\item Suppose that $\alpha,\alpha',\beta,\beta'\in \starN$ are such that $\alpha\sim \alpha'$ and $\beta\sim \beta'$.  Prove that $\alpha+{}^{\ast}\beta\sim \alpha'+{}^{\ast}\beta'$.
\item Find $\alpha,\alpha',\beta,\beta'$ as above with $\alpha+\beta\not\sim \alpha'+\beta'$.
\end{enumerate}
\end{exercise}



\section{The iterated ultrapower perspective}

The ultrapower\index{ultrapower} model does not naturally accommodate 
iterations of hyper-extensions, and in fact, one can be
easily puzzled when thinking of iterated hyper-extensions 
in terms of ``iterated ultrapowers". Let us try to clarify this point.

Let us fix an ultrafilter $\U$ on $\N$.\index{ultrafilter}
Since one can take the ultrapower\index{ultrapower} $\N^\N/\u$ of $\N$ to get a nonstandard 
extension of $\N$, it is natural to take an ultrapower\index{ultrapower} $(\N^\N/\u)^\N/\u$ of 
$\N^\N/\u$ to get a further nonstandard extension.  
The diagonal embedding $d:\N^\N/\u\to (\N^\N/\u)^\N/\u$ is the
map where $d(\alpha)$ is the equivalence class in $(\N^\N/\u)^\N/\u$ 
of the sequence that is constantly $\alpha$.  
We define ${}^*\alpha$ as $d(\alpha)$, but, unlike the first time when 
we took an ultrapower\index{ultrapower} and identified $n\in\N$ with $d(n)$, 
let us refrain from identifying $\alpha$ with ${}^*\alpha$.
Indeed, recall that, according to the 
theory developed in the first section of this chapter, ${}^*\alpha$ 
is supposed to be infinitely larger than $\alpha$.  
How do we reconcile this fact with the current construction?  
Well, unlike the first time we took an ultrapower\index{ultrapower}, a new phenomenon has occurred.  
Indeed, we now have a second embedding $d_0^\u:\N^\N/\u\to(\N^\N/\u)^\N/\u$ 
given by taking the ultrapower\index{ultrapower} of the diagonal embedding
$d_0:\N\to\N^\N/\u$.\footnote
{~Every map $f:A\to B$ yields
a natural map $f^\u:A^\N/\U\to B^\N/\U$ between their ultrapowers,
by setting $f^\u([\sigma])=[f\circ\sigma]$ for every $\sigma:\N\to A$.}
Precisely, if $\alpha=[\sigma]\in\N^\N/\u$ where $\sigma:\N\to\N$,
then $d_0^\U(\alpha)=[([c_{\sigma(1)}], [c_{\sigma(2)}], [c_{\sigma(3)}],\ldots)]$.
It is thus through this embedding that we identify $\alpha\in\N^\N/\u$ 
with its image $d_0^\u(\alpha)\in(\N^\N/\u)^\N/\u$.

It is now straightforward to see that $\alpha<d(\alpha)$ for all 
$\alpha\in\N^\N/\u\setminus \N$.  For example, if
$\alpha=[(1,2,3,\ldots)]\in\N^\N/\u$, then we identify 
$\alpha$ with $[([c_1],[c_2],[c_3],\ldots)]\in(\N^\N/\u)^\N/\U$.
Since $[c_n]<\alpha$ for all $n$, we have that 
$\alpha<[(\alpha,\alpha,\alpha,\ldots)]=d(\alpha)={}^*\alpha$.

Also, it is also straightforward to see that defining 
${}^{**}f$ as $(f^\u)^\u$ extends ${}^*f=f^\u$ for any function $f:\N\to\N$.  
Indeed, if $\alpha=[\sigma]\in\N^\N/\u$,
then we have that
$$(f^\u)^\u(\alpha)=(f^\u)^\u(d_0^\u(\alpha))=
[(f^\U([c_{\sigma(1)}]),f^\U([c_{\sigma(2)}]),\ldots)]=
[([c_{f(\sigma(1))}],[c_{f(\sigma(2))}],\ldots)]=d_0^\u([f\circ \sigma])
=[f\circ\sigma]=f^\u(\alpha).$$

\section*{Notes and references}

Iterated hyperextensions were introduced in \cite{di_nasso_iterated_2015}, where they are used to give a new approach to the proof of Rado's theorem in the theory of partition regularity of equations (see Chapter \ref{partreg} below).  Further applications to the study of partition regularity of equations are obtained in \cite{luperi_baglini_nonstandard_2015,luperi_baglini_partition_2014,luperi_baglini_finite_2016,di_nasso_fermat-like_2016,di_nasso_ramsey_2016}.  A survey on the main properties of iterated hyperextensions, also in relation with hyperfinite generators of ultrafilters, is presented in \cite{di_nasso_hypernatural_2015}.

 


\chapter{Idempotents}

\section{The existence of idempotents in semitopological semigroups}

\begin{definition}
Suppose that $(S,\cdot)$ is a semigroup.  We say that $e\in S$ is \emph{idempotent} if $e\cdot e=e$.
\end{definition}

The following classical theorem of Ellis\index{Ellis' Theorem} is the key to much of what we do.

\begin{theorem}
Suppose that $(S,\cdot)$ is a compact semitopological semigroup\index{semitopological semigroup}.  Then $S$ has an idempotent element\index{idempotent element}.
\end{theorem}

\begin{proof}
Let $\s$ denote the set of nonempty closed subsemigroups of $S$.  It is clear that the intersection of any descending chain of elements of $\s$ is also an element of $\s$, whence by Zorn's lemma, we may find $T\in \s$ that is minimal.  

Fix $s\in T$. We show that $s$ is idempotent.  Set $T_1:=Ts$.  Note that $T_1\not=\emptyset$ as $T\not=\emptyset$.  Since $S$ is a semitopological semigroup\index{semitopological semigroup} and $T$ is compact, we have that $T_1$ is also compact.  Finally, note that $T_1$ is also a subsemigroup of $S$:  
$$T_1\cdot T_1=(Ts)(Ts)\subseteq T\cdot T\cdot T\cdot s\subseteq T\cdot s=T_1.$$
We thus have that $T_1\in \s$.  Since $s\in T$, we have that $T_1\subseteq T$, whence by minimality of $T$, we have that $T_1=T$.  In particular, the set $T_2:=\{t\in T \ : \ t\cdot s=s\}$ is not empty.  Note that $T_2$ is also a closed subset of $T$, whence compact.  Once again, we note that $T_2$ is a subsemigroup of $S$.  Indeed, if $t,t'\in T_2$, then $tt'\in T$ and $(tt')\cdot s=t\cdot (t'\cdot s)=t\cdot s=s$.  We thus have that $T_2\in \s$.  By minimality of $T$, we have that $T_2=T$.  It follows that $s\in T_2$, that is, $s\cdot s=s$.
\end{proof}

The previous theorem and Theorem \ref{betaSsemitop} immediately give the following:

\begin{corollary}
Let $(S,\cdot)$ be a semigroup and let $T$ be any nonempty closed subsemigroup of $(\beta S,\odot)$.  Then $T$ contains an idempotent element\index{idempotent element}.
\end{corollary}

We refer to idempotent elements of $\beta S$ as \emph{idempotent ultrafilters}.\index{ultrafilter!idempotent}  Thus, the previous corollary says that any nonempty closed subsemigroup of $\beta S$ contains an idempotent ultrafilter.

%
%


Given the correspondence between ultrafilters\index{ultrafilter} on $S$ and elements of $\starS$, it is natural to translate the notion of idempotent ultrafilter\index{ultrafilter!idempotent} to the setting of $\starS$.  Suppose that $\alpha\in \starS$ is such that $\u_\alpha$ is an idempotent ultrafilter on $S$.  We thus have that $\u_\alpha=\u_\alpha\odot \u_\alpha=\u_{\alpha\cdot {}^\ast \alpha}$.  This motivates the following:

\begin{definition}
$\alpha\in \starS$ is \emph{$u$-idempotent} if $\alpha\cdot{}^{\ast}\alpha\sim\alpha$.
\end{definition}

We thus see that $\alpha\in \starS$ is $u$-idempotent\index{$u$-idempotent} if and only if $\u_\alpha$ is an idempotent ultrafilter\index{ultrafilter} on $S$.  The following exercise gives a nonstandard proof of \cite[Theorem 2.10]{bergelson_nonmetrizable_1990}.  

\begin{exercise}

\

\begin{enumerate}
\item Suppose that $\alpha\in \starN$ is idempotent\index{ultrafilter!idempotent}.  Prove that $2\alpha+{}^{\ast\ast}\alpha$, $2\alpha+{}^{\ast}\alpha+{}^{\ast\ast}\alpha$, and $2\alpha+2{}^{\ast}\alpha+{}^{\ast\ast}\alpha$ all generate the same ultrafilter, namely $2\u_\alpha\oplus\u_\alpha$.
\item Suppose that $\u\in \beta\N$ is idempotent and $A\in 2\u\oplus \u$.  Prove that $A$ contains a $3$-termed arithmetic progression\index{arithmetic progression}.
\end{enumerate}
\end{exercise}

We now seek an analog of the above fact that nonempty closed subsemigroups of $\beta S$ contain idempotents.  Suppose that $T\subseteq \beta S$ is a subsemigroup and that $\alpha,\beta\in \starS$ are such that $\u_\alpha,\u_\beta\in T$.  Since $\u_{\alpha\cdot{}^{\ast}\beta}=\u_\alpha\odot\u_\beta\in T$, we are led to the following definition:

\begin{definition}
$T\subseteq \starS$ is a \emph{$u$-subsemigroup} if, for any $\alpha,\beta\in T$, there is $\gamma\in T$ such that $\alpha\cdot{}^{\ast}\beta\sim \gamma$. 
\end{definition}

We thus have the following:

\begin{corollary}\label{existenceofuidempotents}
Suppose that $T\subseteq \starS$ is a nonempty closed $u$-subsemigroup.  Then $T$ contains a $u$-idempotent\index{$u$-idempotent} element.
\end{corollary}

\section{Partial semigroups}
We will encounter the need to apply the above ideas to the broader context of partial semigroups.

\begin{definition}
A \emph{partial semigroup} is
a set $S$ endowed with a partially defined binary operation $\left(
s,t\right) \mapsto s\cdot t$ that satisfies the following form of the associative law:  given $s_1,s_2,s_3\in S$, if either of the products $(s_1\cdot s_2)\cdot s_3$ or $s_1\cdot (s_2\cdot s_3)$ are defined, then so is the other and the products are equal.  The partial semigroup $(S,\cdot)$ is \emph{%
directed }if, for any finite subset $F$ of $S$, there exists $t\in S$ such
that the product $s\cdot t$ is defined for every $s\in F$. 
\end{definition}

For the rest of this chapter, we assume that $(S,\cdot)$ is a directed partial semigroup.

\begin{definition}
We call $\u\in \beta S$ \emph{cofinite} if, for all $s\in S$, we have $\{t\in S \ : \ s\cdot t \text{ is defined }\}\in \u$.  We let $\gamma S$ denote the set of all cofinite elements of $\beta S$.
\end{definition}

\begin{exercise}
$\gamma S$ is a \emph{nonempty} closed subset of $\beta S$.
\end{exercise}

We can define an operation $\odot$ on $\gamma S$ by declaring, for $\u,\mathcal{V}\in \gamma S$ and $A\subseteq S$, that $A\subseteq \u\odot \mathcal{V}$ if and only if $$\{s\in S \ : \ \{t\in S \ : \ s\cdot t \text{ is defined and }s\cdot t\in A\}\in \mathcal{V}\}\in \u.$$  Note that the operation $\odot$ is a totally defined operation on $\gamma S$ even though the original operation $\cdot$ was only a partially defined operation.

The next fact is very important but is somewhat routine given everything that has been proven thus far.  We thus leave the proof as a (lengthy) exercise; see also \cite[pages 31 and 32]{todorcevic_introduction_2010}.

\begin{theorem}\label{partial}
$(\gamma S,\odot)$ is a compact semitopological semigroup\index{semitopological semigroup}.  Consequently, every nonempty closed subsemigroup of $\gamma S$ contains an idempotent element\index{idempotent element}.
\end{theorem}

We once again give the nonstandard perspective on the preceding discussion.  Note that $\starS$ is naturally a partial semigroup with the nonstandard extension\index{nonstandard extension} of the partial seimgroup operation.  We say that $\alpha \in \starS$ is \emph{cofinite} if $s\cdot \alpha $ is defined for every $%
s\in S$.  We leave it to the reader to check that $\alpha$ is cofinite if and only if $\u_\alpha$ is a cofinite element of $\beta S$.  Consequently, Theorem \ref{partial} implies that any nonempty closed $u$-subsemigroup of the set of cofinite elements of $\starS$ contains an idempotent element\index{idempotent element}.

\begin{exercise}
Without using Theorem \ref{partial}, prove that, for any cofinite $\alpha,\beta \in \starS$, there is cofinite $\gamma\in \starS$ such that $\alpha\cdot {}^{\ast}\beta\sim \gamma$.  Compare your proof to the proof that $\u\odot\mathcal{V}\in \gamma S$ whenever $\u,\mathcal{V}\in \gamma S$.
\end{exercise}

%

\section*{Notes and references}

The existence of idempotent elements in any compact right topological semigroup is a classical result of Ellis \cite{ellis_lectures_1969}. The observation that this implies the existence of idempotent ultrafilters is due to Galvin. Idempotent ultrafilters play a fundamental role in the application of ultrafilter methods to combinatorics, starting from the Galvin--Glazer proof of Hindman's Theorem on sumsets (see Chapter \ref{chapter_hindman_gowers} below).

\chapter{Loeb measure}\label{Loeb}

\section{Premeasures and measures}

Fix a set $X$. A nonempty set $\mathcal{A}\subseteq \mathcal{P}(X)$ is an 
\emph{algebra} if it is closed under unions, intersections, and complements,
that is, if $A,B\in \mathcal{A}$, then $A\cup B$, $A\cap B$, and $X\setminus
A$ all belong to $\mathcal{A}$. If $\mathcal{A}$ is an algebra of subsets of 
$X$, then $\emptyset,X\in \mathcal{A}$. An algebra $\mathcal{A}$ on $X$ is
said to be a \emph{$\sigma$-algebra} if it is also closed under countable
unions, that is, if $A_1,A_2,\ldots$ all belong to $\mathcal{A}$, then so
does $\bigcup_{n=1}^\infty A_n$. A $\sigma$-algebra is then automatically
closed under countable intersections.

\begin{exercise}
Suppose that $X$ is a set and $\mathcal{O}\subseteq \mathcal{P}(X)$ is an
arbitrary collection of subsets of $X$. Prove that there is a smallest $%
\sigma $-algebra $\Omega $ containing $\mathcal{O}$. We call this $\sigma $%
-algebra the \emph{$\sigma $-algebra generated by $\mathcal{O}$} and denote
it by $\sigma (\mathcal{O})$.
\end{exercise}

\begin{remark}
When trying to prove that every element of $\sigma(\mathcal{O})$ has a
certain property, one just needs to show that the set of elements having
that property contains $\mathcal{O}$ and is a $\sigma$-algebra.
\end{remark}

Suppose that $\mathcal{A}$ is an algebra on $X$. A \emph{pre-measure} on $%
\mathcal{A}$ is a function $\mu:\mathcal{A}\to [0,+\infty]$ satisfying the
following two axioms:

\begin{itemize}
\item $\mu(\emptyset)=0$;

\item (Countable Additivity) If $A_1,A_2,\ldots,$ all belong to $\mathcal{A}$%
, are pairwise disjoint, \emph{and} $\bigcup_{n=1}^\infty A_n$ belongs to $%
\Omega$, then $\mu(\bigcup_{n=1}^\infty A_n)=\sum_{n=1}^\infty \mu(A_n)$.
\end{itemize}

If $\mathcal{A}$ is a $\sigma$-algebra, then a pre-measure is called a \emph{%
measure}. If $\mu$ is a measure on $X$ and $\mu(X)=1$, then we call $\mu$ a 
\emph{probability measure on $X$}.

\begin{exercise}
Fix $n\in \mathbb{N}$ and suppose that $X=\{1,2,\ldots,n\}$. Let $\mathcal{A}%
:=\mathcal{P}(X)$. Then $\mathcal{A}$ is an algebra of subsets of $X$ that
is actually a $\sigma$-algebra for trivial reasons. Define the function $\mu:%
\mathcal{A}\to [0,1]$ by $\mu(A)=\frac{|A|}{n}$. Then $\mu$ is a probability
measure on $\mathcal{A}$, called the \emph{normalized counting measure}.
\end{exercise}

\begin{exercise}
Suppose that $\mu:\mathcal{A}\to [0,+\infty]$ is a pre-measure. Prove that $%
\mu(A)\leq \mu(B)$ for all $A,B\in \mathcal{A}$ with $A\subseteq B$.
\end{exercise}

For subsets $A,B$ of $X$, we define the \emph{symmetric difference of $A$
and $B$} to be $A\triangle B:=(A\setminus B)\cup (B\setminus A)$.

\begin{exercise}
\label{productapprox} Suppose that $\mathcal{A}$ is an algebra and $\mu
:\sigma (\mathcal{A})\rightarrow \lbrack 0,\infty ]$ is a measure. Prove
that, for every $A\in \sigma (\mathcal{A})$ with $\mu (A)<\infty $ and every 
$\epsilon \in \mathbb{R}^{>0}$, there is $B\in \mathcal{A}$ such that $\mu
(A\triangle B)<\epsilon $.
\end{exercise}

For our purposes, it will be of vital importance to know that a pre-measure $%
\mu $ on an algebra $\mathcal{A}$ can be extended to a measure on a $\sigma $%
-algebra $\sigma (\mathcal{A})^{\prime }$ extending $\mathcal{A}$, a process
which is known as \emph{Carath\'{e}odory extension}. We briefly outline how
this is done. The interested reader can consult any good book on measure
theory for all the glorious details; see for instance \cite[Section 1.7]%
{tao_introduction_2011}.

Fix an algebra $\mathcal{A}$ of subsets of $X$ and a pre-measure $\mu $ on $%
\mathcal{A}$. For arbitrary $A\subseteq X$, we define the \emph{outer
measure of }$A$ to be 
\begin{equation*}
\mu ^{+}(A):=\inf \{\sum_{n\in \mathbb{N}}\mu (B_{n})\ |\ A\subseteq
\bigcup_{n\in \mathbb{N}}B_{n},\text{ each }B_{n}\in \mathcal{A}\}.
\end{equation*}%
Note that $\mu ^{+}(A)=\mu (A)$ for all $A\in \mathcal{A}$. Now although $%
\mu ^{+}$ is defined on all of $\mathcal{P}(X)$ (which is certainly a $%
\sigma $-algebra), it need not be a measure. However, there is a canonical $%
\sigma $-sub-algebra $\mathcal{A}_{m}$ of $\mathcal{P}(X)$, the so-called 
\emph{Carath\'{e}odory measurable }or \emph{$\mu ^{+}$-measurable subsets of 
$X$}, on which $\mu ^{+}$ is a measure. These are the sets $A\subseteq X$
such that%
\begin{equation*}
\mu ^{+}\left( E\right) =\mu ^{+}\left( A\cap E\right) +\mu ^{+}\left(
E\setminus A\right)
\end{equation*}%
for every other set $E\subset X$. Let us collect the relevant facts here:

\begin{fact}
\label{measurefacts} Let $X$ be a set, $\mathcal{A}$ an algebra of subsets
of $X$, and $\mu:\mathcal{A}\to [0,\infty]$ a pre-measure on $\mathcal{A}$
with associated outer measure $\mu^+$ and $\sigma$-algebra of $\mu^+$%
-measurable sets $\mathcal{A}_m$. Further suppose that $\mu$ is \emph{$%
\sigma $-finite}, meaning that we can write $X=\bigcup_{n\in \mathbb{N}}X_n$
with each $X_n\in \mathcal{A}$ and $\mu(X_n)<\infty$.

\begin{enumerate}
\item $\sigma(\mathcal{A})\subseteq \mathcal{A}_m$ and $\mu^+|\mathcal{A}%
=\mu $.

\item (Uniqueness) If $\mathcal{A}^{\prime }$ is another $\sigma$-algebra on 
$X$ extending $\mathcal{A}$ and $\mu^{\prime }:\mathcal{A}^{\prime }\to
[0,\infty]$ is a measure on $\mathcal{A}^{\prime }$ extending $\mu$, then $%
\mu^+$ and $\mu^{\prime }$ agree on $\mathcal{A}_m\cap \mathcal{A}^{\prime }$
(and, in particular, on $\sigma(\mathcal{A})$).

\item (Completeness) If $A\subseteq B\subseteq X$ are such that $B\in 
\mathcal{A}_m$ and $\mu^+(B)=0$, then $A\in \mathcal{A}_m$ and $\mu^+(A)=0$.

\item (Approximation Results)

\begin{enumerate}
\item If $A\in \mathcal{A}_m$, then there is $B\in \sigma(\mathcal{A})$
containing $A$ such that $\mu^+(B\setminus A)=0$. (So $\mathcal{A}_m$ is the 
\emph{completion} of $\sigma(\mathcal{A})$.)

\item If $A\in \mathcal{A}_m$ is such that $\mu^+(A)<\infty$, then for every 
$\epsilon \in \mathbb{R}^{>0}$, there is $B\in \mathcal{A}$ such that $%
\mu(A\triangle B)<\epsilon$.

\item Suppose that $A\subseteq X$ is such that, for every $\epsilon \in 
\mathbb{R}^{>0}$, there is $B\in \mathcal{A}$ such that $\mu (A\triangle
B)<\epsilon $. Then $A\in \mathcal{A}_{m}$.
\end{enumerate}
\end{enumerate}
\end{fact}

\begin{example}[Lebesgue measure]
Suppose that $X=\mathbb{R}$ and $\mathcal{A}$ is the collection of \emph{%
elementary sets}, namely the finite unions of intervals. Define $\mu :%
\mathcal{A}\rightarrow \lbrack 0,\infty ]$ by declaring $\mu (I)=\operatorname{length%
}(I)$ and $\mu (I_{1}\cup \cdots \cup I_{n})=\sum_{i=1}^{n}\mu (I_{j})$
whenever $I_{1},\ldots ,I_{n}$ are pairwise disjoint. The above
outer-measure procedure yields the $\sigma $-algebra $\mathcal{A}_{m}$,
which is known as the $\sigma $-algebra of \emph{Lebesgue measurable subsets
of $\mathbb{R}$} and usually denoted by $\mathfrak{M}$. The measure $\mu
^{+} $ is often denoted by $\lambda $ and is referred to as \emph{Lebesgue
measure}. The $\sigma $-algebra $\sigma (\mathcal{A})$ in this case is known
as the $\sigma $-algebra of \emph{Borel subsets of $\mathbb{R}$}, usually
denoted by $\mathcal{B}$. It can also be seen to be the $\sigma $-algebra
generated by the open intervals.
\end{example}

\section{The definition of Loeb measure}

How do we obtain pre-measures in the nonstandard context? Well, we obtain
them by looking at normalized counting measures on hyperfinite sets. Suppose
that $X$ is a hyperfinite set%
\index{hyperfinite!set}. We set $\mathcal{A}$ to be the set of \emph{internal%
} subsets of $X$. Then $\mathcal{A}$ is an algebra of subsets of $X$ that is
not (in general) a $\sigma $-algebra. For example, if $X=\left[ 1,N\right]
\subseteq {}^{\ast }\mathbb{N}$ for some $N\in {}^{\ast }\mathbb{N}\setminus 
\mathbb{N}$, then for each $n\in \mathbb{N}$, $A_{n}:=\{n\}$ belongs to $%
\mathcal{A}$, but $\bigcup_{n}A_{n}=\mathbb{N}$ does not belong to $\mathcal{%
A}$ as $\mathbb{N}$ is not internal.

If $A\in \mathcal{A}$, then $A$ is also hyperfinite. We thus define a
function $\mu :\mathcal{A}\rightarrow \lbrack 0,1]$ by $\mu (A):=%
\operatorname{st}\left( \frac{|A|}{|X|}\right) $. We claim that $\mu$ is a
pre-measure. It is easily seen to be \emph{finitely additive}, that is, $\mu
(A_{1}\cup \cdots \cup A_{n})=\sum_{i=1}^{n}\mu (A_{i})$ whenever $%
A_{1},\ldots ,A_{n}\in \mathcal{A}$ are disjoint. But how do we verify
countable additivity?%
\index{Loeb measure}

\begin{exercise}
If $A_1,A_2,\ldots$ all belong to $\mathcal{A}$ and $\bigcup_{n=1}^\infty
A_n $ also belongs to $\mathcal{A}$, then there is $k\in\mathbb{N}$ such
that $\bigcup_{n=1}^\infty A_n=\bigcup_{n=1}^k A_n$.
\end{exercise}

Thus, by the exercise, countable additivity is a trivial consequence of
finite additivity in this context. We may thus apply the Carath\'{e}odory
extension theorem from the previous section to obtain a probability measure $%
\mu ^{+}:\mathcal{A}\rightarrow \lbrack 0,1]$ extending $\mu $. The measure $%
\mu ^{+}$ is called the \emph{Loeb measure on $X$} and will be denoted $\mu
_{X}$. The elements of $\mathcal{A}$ are referred to as the \emph{Loeb
measurable subsets of $X$} and will be denoted by $\mathcal{L}_{X}$.%
\index{Loeb measure}

\begin{lemma}
\label{Lemma:approx-above}If $B\in \mathcal{L}_{X}$, then 
\begin{equation*}
\mu _{X}(B)=\inf \{\mu _{X}(A)\ |\ A%
\text{ is internal and }B\subseteq A\}\text{.}
\end{equation*}
\end{lemma}

\begin{proof}
The inequality $\leq $ is clear. Towards the other inequality, fix $\epsilon
\in \mathbb{R}^{>0}$. We need to find internal $A$ such that $B\subseteq A$
and $\mu _{X}(A)\leq \mu _{X}(B)+\epsilon $. Fix an increasing sequence of
internal sets $(A_{n}\ |\ n\in \mathbb{N})$ such that $B\subseteq
\bigcup_{n\in \mathbb{N}}A_{n}$ and $\mu _{X}\left( A_{n}\right) <\mu
_{X}(B)+\epsilon $ for every $n\in \mathbb{N}$. By countable saturation%
\index{saturation!countable}, we extend this sequence to an internal
sequence $(A_{n}\ |\ n\in {}^{\ast }\mathbb{N})$. By transfer, for each $%
k\in \mathbb{N}$, we have 
\begin{equation*}
(\forall n\in {}^{\ast }\mathbb{N})(n\leq k\rightarrow (A_{n}\subseteq A_{k}%
\text{ and }\mu _{X}(A_{n})<\mu _{X}(B)+\epsilon )).
\end{equation*}%
By 
\index{overflow principle}, there is $K>\mathbb{N}$ such that $\mu
_{X}(A_{K})\leq \mu _{X}(B)+\epsilon $. This concludes the proof.
\end{proof}

\begin{lemma}
\label{approx} If $B\in \mathcal{L}_{X}$ , then, for every $\epsilon \in 
\mathbb{R}^{>0}$, there are internal subsets $C,A$ of $X$ such that $%
C\subseteq B\subseteq A$ and $\mu _{X}(A\setminus C)<\epsilon $.
\end{lemma}

\begin{proof}
Fix $\epsilon >0$. By Lemma \ref{Lemma:approx-above} applied to $B$, there
is an internal set $A$ containing $B$ such that $\mu _{X}(A)< \mu _{X}(B)+%
\frac{\epsilon}{2} $. By Lemma \ref{Lemma:approx-above} applied to\textrm{\ }%
$A\setminus B$, there is an internal set $R$ containing $A\setminus B$ such
that $\mu _{X}\left( R\right) < \mu _{X}\left( A\setminus B\right) +\frac{%
\epsilon}{2} < \epsilon $. Set now $C:=A\setminus R$ and observe that $C$ is
an internal set contained in $B$. Furthermore we have that $\mu _{X}\left(
A\setminus C\right) \leq \mu _{X}\left( R\right) < \epsilon $. This
concludes the proof.
\end{proof}

There are many interesting things to say about Loeb measure. It is crucial
for applications of nonstandard analysis to many different areas of
mathematics. More information on the Loeb measure can be found in \cite%
{albeverio_nonstandard_1986,arkeryd_nonstandard_1997}. We will see later in
this book that Loeb measure allows us to treat densities on the natural
numbers as measures, allowing us to bring in tools from measure theory and
ergodic theory into combinatorial number theory.%
\index{Loeb measure}


\section{Lebesgue measure via Loeb measure}

\label{LebesgueLoeb}

The purpose of this section is to see that Lebesgue measure can be
constructed using a suitable Loeb measure. The connection between these
measures serves as a useful motivation for the results of Chapter \ref%
{sumset} on sumsets of sets of positive density.%
\index{Loeb measure}

\begin{theorem}
Suppose that $N>\mathbb{N}$ and consider the hyperfinite set%
\index{hyperfinite!set} $X:=\{0,%
\frac{1}{N},\frac{2}{N},\ldots ,\frac{N}{N}=1\}$ and the function $\operatorname{st}%
:X\rightarrow \lbrack 0,1]$. Define a $\sigma $-algebra $\mathcal{A}$ on $%
[0,1]$ by $A\in \mathcal{A}$ if and only if $\operatorname{st}^{-1}(A)\in \mathcal{L}%
_{X}$. For $A\in \mathcal{A}$, define $\nu (A):=\mu _{X}(\operatorname{st}^{-1}(A))$%
. Then $\mathcal{A}$ is the algebra of Lebesgue measurable subsets of $\left[
0,1\right] $ and $\nu $ is Lebesgue measure.
\end{theorem}

We outline the proof of this theorem in a series of steps. We denote by $%
\mathcal{B}$ the $\sigma $-algebra of Borel subsets of $\left[ 0,1\right] $,
by $\mathcal{M}$ the $\sigma $-algebra of measurable subsets of $\left[ 0,1%
\right] $, and by $\lambda $ the Lebesgue measure on $\mathcal{M}$.

\begin{exercise}
Prove that $\mathcal{A}$ is a $\sigma$-algebra and $\nu$ is a measure on $%
\mathcal{A}$.
\end{exercise}

\begin{exercise}
Fix $a,b\in [0,1]$ with $a<b$.

\begin{enumerate}
\item Prove that $X\cap (a,b)^*\in \mathcal{L} _X$ and $\mu_X(X\cap
(a,b)^*)=b-a$.

\item Prove that $\operatorname{st}^{-1}((a,b))=\bigcup_{n\in \mathbb{N}} (X\cap (a+%
\frac{1}{n},b-\frac{1}{n})^*)$.

\item Prove that $(a,b)\in \mathcal{A}$ and $\nu((a,b))=b-a$.
\end{enumerate}
\end{exercise}

We now use the fact that $\lambda $ is the only probability measure on $%
\mathcal{B}$ satisfying $\lambda (a,b)=b-a$ and that is invariant under
translations modulo $1$ to conclude that $\mathcal{B}\subseteq \mathcal{A}$
and $\nu |_{\mathcal{B}}=\lambda |_{\mathcal{B}}$.

\begin{exercise}
Conclude that $\mathcal{M}\subseteq \mathcal{A}$ and $\nu |_{\mathfrak{M}%
}=\lambda |_{\mathfrak{M}}$. (Hint: Use Fact \ref{measurefacts}.)
\end{exercise}

\begin{exercise}
Show that $\mathcal{A}\subseteq \mathcal{M}$. (Hint: if $B\in \mathcal{A}$,
then by Lemma \ref{approx}, there are internal $C,D\subseteq X$ such that $%
C\subseteq \operatorname{st}^{-1}(B)\subseteq D$ and $\mu _{X}(D\setminus
C)<\epsilon $. Set $C^{\prime }:=\operatorname{st}(C)$ and $D^{\prime
}:=[0,1]\setminus \operatorname{st}(X\setminus D)$. Notice that $C^{\prime }$ is
closed and $D^{\prime }$ is open, whence $C^{\prime },D^{\prime }\in 
\mathcal{B}\subseteq \mathcal{A}$. Prove that $C\subseteq \operatorname{st}%
^{-1}(C^{\prime })$ and $\operatorname{st}^{-1}(D^{\prime })\subseteq D$. Conclude
that $B\in \mathcal{M}$.)
\end{exercise}

%

%

\section{Integration}

There is a lot to say about the nonstandard theory of integration. We will
focus on the Loeb%
\index{Loeb measure} measure $\mu _{X}$ obtained from a hyperfinite set $X$%
\index{hyperfinite!set}. In this section, $X$ always denotes a hyperfinite
set.

First, if $F:X\to {}^{\ast}\mathbb{R}$ is an internal function such that $%
F(x)$ is finite for $\mu_X$-almost every $x\in X$, we define $\st(F):X\to 
\mathbb{R}$ by $\st(F)(x):=\st(F(x))$ whenever $F(x)$ is finite.
(Technically speaking, $\st(F)$ is only defined on a set of measure $1$, but
we will ignore this minor point.) If $f:X\to \mathbb{R}$ is a function and $%
F:X\to {}^{\ast}\mathbb{R}$ is an internal function such that $f(x)=\st(F)(x)
$ for $\mu_X$-almost every $x\in X$, we call $F$ a \emph{lift} of $f$. We
first characterize which functions have lifts.

\begin{proposition}
$f:X\to \mathbb{R}$ has a lift if and only if $f$ is $\mu_X$-measurable.
\end{proposition}

\begin{proof}
If $F$ is a lift of $f$, then for any $r\in \mathbb{R}$, we have 
\begin{equation*}
\mu_X\left(\{x\in X \ : \ f(x)<r\}\triangle \bigcup_{n\in \mathbb{N}%
}\left\{x\in X \ : \ F(x)<r-%
\frac{1}{n}\right\}\right)=0.
\end{equation*}
Since the latter set is clearly measurable and $\mu_X$ is a complete
measure, it follows that $\{x\in X \ : \ f(x)<r\}$ is measurable, whence $f$
is $\mu_X$-measurable.

For the converse, suppose that $f$ is $\mu_X$-measurable and fix a countable
open basis $\left\{ V_{n}\right\} $ for $\mathbb{R}$. For $n\in \mathbb{N}$,
set $U_{n}:=f^{-1}\left( V_{n}\right) \in \mathcal{L}_{X}$. By Lemma \ref%
{approx}, one can find, for every $n\in \mathbb{N}$, an increasing sequence $%
\left( A_{n,m}\right) $ of internal subsets of $U_{n}$ such that $\mu
_{X}\left( A_{n,m}\right) \geq \mu _{X}\left( U_{n}\right) -2^{-m}$ for
every $m\in \mathbb{N}$. It follows that the subset 
\begin{equation*}
X_{0}:=X\setminus \bigcup_{n\in \mathbb{N}}\left( U_{n}\setminus
\bigcup_{m\in \mathbb{N}}A_{n,m}\right)
\end{equation*}%
of $X$ has $\mu _{X}$-measure $1$. Observe now that, for every $n,m\in 
\mathbb{N}$, there exists an internal function $F:X\rightarrow {}^{\ast }%
\mathbb{R}$ such that $F\left( A_{\ell ,k}\right) \subset {}^{\ast }V_{\ell
} $ for $k\leq m$ and $\ell \leq n$. Therefore, by saturation, there exists
an internal function $F:X\rightarrow {}^{\ast }\mathbb{R}$ such that $%
F\left( A_{n,m}\right) \subset {}^{\ast }V_{n}$ for every $n,m\in \mathbb{N}$%
. It is clear that $f(x)=\operatorname{st}\left( F\left( x\right) \right) $ for
every $x\in X_{0}$, whence $F$ is a lift of $f$.
\end{proof}

The rest of this section is devoted towards understanding $\int fd\mu_X$ (in
the case that $f$ is $\mu_X$-integrable) and the ``internal integral'' $%
\frac{1}{|X|}\sum_{x\in X}F(x)$ of a lift $F$ of $f$. We first treat a
special, but important, case.

\begin{lemma}
\label{boundedintegration} Suppose that $F:X\to {}^{\ast}\mathbb{R}$ is an
internal function such that $F(x)$ is finite for all $x\in X$. Then $\st(F)$
is $\mu_X$-integrable and 
\begin{equation*}
\int \st(F)d\mu_X=\st\left(\frac{1}{|X|}\sum_{x\in X}F(x)\right).
\end{equation*}
\end{lemma}

\begin{proof}
Note first that the assumptions imply that there is $m\in \N$ such that $%
|F(x)|\leq m$ for all $x\in X$. It follows that $\st(F)$ is $\mu _{X}$%
-integrable. Towards establishing the displayed equality, note that, by
considering positive and negative parts, that we may assume that $F$ is
nonnegative. Fix $n\in \mathbb{N}$. For $k\in \{0,1,\ldots ,mn-1\}$, set $%
A_{k}:=\{x\in X\ :\ \frac{k}{n}\leq F(x)<\frac{k+1}{n}\}$, an internal set.
Since $\sum_{k}\frac{k}{n}\chi _{A_{k}}$ is a simple function below $\st(F)$%
, we have that $\sum_{k}\frac{k}{n}\mu _{X}(A_{k})\leq \int \st(F)d\mu _{X}$%
. However, we also have 
\begin{equation*}
\sum_{k}\frac{k}{n}\mu _{X}(A_{k})=\operatorname{st}\left( \frac{1}{|X|}%
\sum_{k}\sum_{x\in A_{k}}\frac{k}{n}\right) \geq \operatorname{st}\left( \frac{1}{|X|%
}\sum_{k}\sum_{x\in A_{k}}(F(x)-\frac{1}{n})\right) =\operatorname{st}\left( \frac{1%
}{|X|}\sum_{x\in X}F(x)\right) -\frac{1}{n}.
\end{equation*}%
It follows that $\operatorname{st}(\frac{1}{|X|}\sum_{x\in X}F(x))\leq \int \st%
(F)d\mu _{X}+\frac{1}{n}$. Since $n$ was arbitrary, we have that $\operatorname{st}(%
\frac{1}{|X|}\sum_{x\in X}F(x))\leq \int \st(F)d\mu _{X}$.

We leave the proof of the inequality $\int \st(F)d\mu_X\leq \operatorname{st}(\frac{1%
}{|X|}\sum_{x\in X}F(x))$ to the reader.
\end{proof}

We now seek to extend the previous lemma to cover situations when $F$ is not
necessarily bounded by a standard number. Towards this end, we need to
introduce the appropriate nonstandard integrability assumption. A $\mu_X$%
-measurable internal function $F:X\rightarrow {}^{\ast }\mathbb{R}$ is
called \emph{$S$-integrable} if:

\begin{enumerate}
\item The quantity 
\begin{equation*}
\frac{1}{|X|}\sum_{x\in X}|F(x)|
\end{equation*}%
is finite, and

\item for every internal subset $A$ of $X$ with $\mu_X(A)=0$, we have%
\begin{equation*}
\frac{1}{\left\vert X\right\vert }\sum_{x\in A}\left\vert F\left( x\right)
\right\vert\approx 0.
\end{equation*}
\end{enumerate}

Here is the main result of this section:

\begin{theorem}
\label{integratebycounting} Suppose that $f:X\to \mathbb{R}$ is a $\mu_X$%
-measurable function. Then $f$ is $\mu_X$-integrable if and only if $f$ has
an $S$-integrable lifting. In this case, for any $S$-integrable lift $F$ of $%
f$ and any internal subset $B$ of $X$, we have 
\begin{equation*}
\int_B fd\mu_X=\st\left(\frac{1}{|X|}\sum_{x\in B}F(x)\right).
\end{equation*}
\end{theorem}

\begin{proof}
We first note that, by taking positive and negative parts, we may assume
that $f$ is nonnegative. Moreover, by replacing $f$ with $f\cdot \chi _{B}$,
we may assume that $B=X$.

We first suppose that $F:X\rightarrow {}^{\ast }\mathbb{R}$ is a nonegative $%
S$-integrable function such that $F(x)$ is finite for $\mu _{X}$-almost
every $x$. For $n\in \starN$, set $B_{n}:=\{x\in X\ :\ F(x)\geq n\}$.

\ 

\noindent \textbf{Claim 1:} For every infinite $N\in {}^{\ast }\mathbb{N}$,
we have%
\begin{equation*}
\frac{1}{\left\vert X\right\vert }\sum_{x\in B_N} F(x) \approx 0\text{.}
\end{equation*}

\ 

\noindent \textbf{Proof of Claim 1:} Observe that%
\begin{equation*}
\frac{N\left\vert B_{N}\right\vert }{\left\vert X\right\vert }\leq \frac{1}{%
\left\vert X\right\vert }\sum_{x\in B_{N}}F\left( x\right) \leq \frac{1}{%
\left\vert X\right\vert }\sum_{x\in X}F\left( x\right) 
\end{equation*}%
Therefore%
\begin{equation*}
\frac{\left\vert B_{N}\right\vert }{\left\vert X\right\vert }\leq \frac{1}{N}%
\frac{1}{\left\vert X\right\vert }\sum_{x\in X}F\left( x\right) \thickapprox
0
\end{equation*}%
since, by assumption, $\frac{1}{\left\vert X\right\vert }\sum_{x\in
X}F\left( x\right) $ is finite. It follows from the assumption that $F$ is $S
$-integrable that%
\begin{equation*}
\frac{1}{\left\vert X\right\vert }\sum_{x\in B_{N}}F\left( x\right)
\thickapprox 0\text{.}
\end{equation*}

In the rest of the proof, we will use the following notation: given a
nonegative internal function $F:X\rightarrow {}^{\ast }\mathbb{R}$ and $m\in %
\starN$, we define the internal function $F_{m}:X\rightarrow {}^{\ast }%
\mathbb{R}$ by $F_{m}\left( x\right) =\min \left\{ F\left( x\right)
,m\right\} $. Observe that $F_{m}(x)\leq F_{m+1}(x)\leq F(x)$ for every $%
m\in \starN$ and every $x\in X$. It follows from the Monotone Convergence
Theorem and the fact that, for $\mu _{X}$-almost every $x\in X$, the
sequence $(\st(F_{m}(x))\ :\ m\in \mathbb{N})$ converges to $\st(F(x))$,
that $\int \st(F_{m})d\mu _{X}\rightarrow \int \st(F)d\mu _{X}$.

\ 

\noindent \textbf{Claim 2:} We have 
\begin{equation*}
\operatorname{st}\left( \frac{1}{\left\vert X\right\vert }\sum_{x\in X} F\left(
x\right) \right) =\lim_{m\rightarrow +\infty }\operatorname{st}\left( \frac{1}{%
\left\vert X\right\vert }\sum_{x\in X}F_{m}\left( x\right) \right) \text{.}
\end{equation*}

\ 

\noindent \textbf{Proof of Claim 2:} It is clear that%
\begin{equation*}
\lim_{m\rightarrow \infty }\operatorname{st}\left( \frac{1}{\left\vert X\right\vert }%
\sum_{x\in X}F_{m}\left( x\right) \right) \leq \operatorname{st}\left( \frac{1}{%
\left\vert X\right\vert }\sum_{x\in X}F\left( x\right) \ \right) \text{.}
\end{equation*}%
For the other inequality, fix $M\in \starN$ infinite and observe that 
\begin{eqnarray*}
\frac{1}{\left\vert X\right\vert }\sum_{x\in X}F\left( x\right)  &=&\frac{1}{%
\left\vert X\right\vert }\sum_{x\in B_{M}}F\left( x\right) +\frac{1}{%
\left\vert X\right\vert }\sum_{x\in X\setminus B_{M}}F\left( x\right)  \\
&\approx &\frac{1}{\left\vert X\right\vert }\sum_{x\in X\setminus
B_{M}}F\left( x\right)  \\
&=&\frac{1}{\left\vert X\right\vert }\sum_{x\in X\setminus B_{M}}F_{M}\left(
x\right)  \\
&\leq &\frac{1}{\left\vert X\right\vert }\sum_{x\in X}F_{M}\left( x\right) 
\text{.}
\end{eqnarray*}%
Thus, given any $\epsilon >0$, we have that $\frac{1}{\left\vert
X\right\vert }\sum_{x\in X}F\left( x\right) \leq \frac{1}{\left\vert
X\right\vert }\sum_{x\in X}F_{M}\left( x\right) +\epsilon $ for all infinite 
$M$, whence, by underflow, we have that $\frac{1}{\left\vert X\right\vert }%
\sum_{x\in X}F\left( x\right) \leq \frac{1}{\left\vert X\right\vert }%
\sum_{x\in X}F_{m}\left( x\right) +\epsilon $ for all but finitely many $%
m\in \mathbb{N}$. It follows that $\operatorname{st}\left( \frac{1}{\left\vert
X\right\vert }\sum_{x\in X}F\left( x\right) \right) \leq \lim_{m\rightarrow
+\infty }\operatorname{st}\left( \frac{1}{\left\vert X\right\vert }\sum_{x\in
X}F_{m}\left( x\right) \right) $, as desired.

By Lemma \ref{boundedintegration}, Claim 2, and the discussion preceding
Claim 2, we have that $\st(F)$ is $\mu_X$-integrable and $\int \st(F)d\mu=\st%
\left(\frac{1}{|X|}\sum_{x\in X}F(x)\right)$, as desired.

We now suppose that $f$ is a nonnegative $\mu _{X}$-integrable function. We must show that $f$ has
an $S$-integrable lifting. Let $F$ be any nonnegative lifting of $f$. Note that, for
every infinite $M\in \starN$, that $F_{M}$ is also a lifting of $f$. We will
find an infinite $M\in {}^{\ast }\mathbb{N}$ such that $F_{M}$ is also $S$%
-integrable.

By the Monotone Convergence Theorem, for every $\epsilon >0$, we have that 
\begin{equation*}
\left\vert \int \st(F)d\mu _{X}-\int \st(F_{m})d\mu _{X}\right\vert
<\epsilon 
\end{equation*}%
holds for all but finitely many $m\in \mathbb{N}$. %
Therefore, by Lemma \ref{boundedintegration}, we have that%
\begin{equation*}
\left\vert \int \st(F)d\mu _{X}-\frac{1}{\left\vert X\right\vert }\sum_{x\in
X}F_{m}\left( x\right) \right\vert <\epsilon 
\end{equation*}%
holds for all but finitely many $m\in \mathbb{N}$. %
By transfer, there exists infinite $M\in {}^{\ast }\mathbb{N}$ such that%
\begin{equation*}
\int \st(F)d\mu _{X}=\operatorname{st}\left( \frac{1}{\left\vert X\right\vert }%
\sum_{x\in X}F_{M}\left( x\right) \right) 
\end{equation*}%
and%
\begin{equation*}
\int fd\mu _{X}=\operatorname{st}\left( \frac{1}{\left\vert X\right\vert }\sum_{x\in
X}F_{M}\left( x\right) \right) \text{.}
\end{equation*}%
We show that the function $F_{M}$ is $S$-integrable. Suppose that $B$ is an
internal subset of $X$ such that $\mu _{X}(B)=0$. Set 
\begin{equation*}
r:=\operatorname{st}\left( \frac{1}{\left\vert X\right\vert }\sum_{x\in B}\left\vert
F_{M}\left( x\right) \right\vert \right) \text{.}
\end{equation*}%
We wish to show that $r=0$. Towards this end, fix $m\in \mathbb{N}$. Then we
have that%
\begin{eqnarray*}
r+\int \st(F_{m})d\mu _{X} &=&r+\int_{X\setminus B}\st(F_{m})d\mu
_{X}\approx r+\frac{1}{\left\vert X\right\vert }\sum_{x\in X\setminus
B}F_{m}\left( x\right)  \\
&\leq &r+\frac{1}{\left\vert X\right\vert }\sum_{x\in X\setminus
B}F_{M}\left( x\right) \approx \frac{1}{\left\vert X\right\vert }\sum_{x\in
X}F_{M}\left( x\right) \approx \int \st(F)d\mu _{X}\text{.}
\end{eqnarray*}%
Letting $m\rightarrow +\infty $, we obtain that $r=0$, as desired.
\end{proof}

\begin{corollary}
\label{approxL1} Suppose $f\in L^{1}(X,\mathcal{L}_{X},\mu _{X})$ and $%
\epsilon >0$. Then there exists internal functions $F,G:X\rightarrow
{}^{\ast }\mathbb{R}$ such that $F\leq f\leq G$ $\mu _{X}$-almost everywhere
and 
\begin{equation*}
\max \left\{ \left\vert \int_{B}fd\mu _{X}-\frac{1}{|X|}\sum_{x\in
B}F(x)\right\vert ,\left\vert \int_{B}fd\mu _{X}-\frac{1}{|X|}\sum_{x\in
B}G(x)\right\vert \right\} \leq \epsilon
\end{equation*}%
for every internal subset $B$ of $X$.
\end{corollary}

\begin{proof}
Let $H:X\rightarrow {}^{\ast }\mathbb{R}$ be a lifting of $f$. Set $%
F:=H-\epsilon /2$ and $G:=H+\epsilon /2$. Since $\operatorname{st}\left( H\left(
x\right) \right) =f\left( x\right) $ for $\mu _{X}$-almost every $x\in X$,
we conclude that $F\left( x\right) \leq f\left( x\right) \leq G\left(
x\right) $ for $\mu _{X}$-almost every $x\in X$. Furthermore, if $B$ is an
internal subset of $X$, then by Lemma \ref{integratebycounting}, we have that%
\begin{equation*}
\left\vert \int_{B}fd\mu _{X}-\frac{1}{\left\vert X\right\vert }\sum_{x\in
B}F\left( x\right) \right\vert \leq \epsilon /2+\left\vert \int_{B}fd\mu
_{X}-\frac{1}{\left\vert X\right\vert }\sum_{x\in B}H\left( x\right)
\right\vert \leq \epsilon 
\end{equation*}%
and 
\begin{equation*}
\left\vert \int_{B}fd\mu _{X}-\frac{1}{|X|}\sum_{x\in B}G(x)\right\vert \leq
\epsilon /2+\left\vert \int_{B}fd\mu _{X}-\frac{1}{\left\vert X\right\vert }%
\sum_{x\in B}H\left( x\right) \right\vert \leq \epsilon \text{.}
\end{equation*}%
This concludes the proof.
\end{proof}

\section{Product measure\label{product}}

Suppose that $(X,\mathcal{A}_{X},\nu _{X})$ and $(Y,\mathcal{A}_{Y},\nu
_{Y}) $ are two probability measure spaces. We can then form their \emph{%
product} as follows: first, set $\mathcal{A}$ to be the set of finite unions
of rectangles of the form $A\times B$, where $A\in \mathcal{A}_{X}$ and $%
B\in \mathcal{A}_{Y}$. The elements of $\mathcal{A}$ are called \emph{elementary
sets}. It is an exercise to show that $\mathcal{A}$ is an algebra of subsets
of $X\times Y$ and that every element of $\mathcal{A}$ can be written as a
finite union of \emph{disjoint} such rectangles. We can then define a
pre-measure $\nu $ on $\mathcal{A}$ by $\mu (\bigcup_{i=1}^{n}(A_{i}\times
B_{i})):=\sum_{i=1}^{n}(\nu _{X}(A_{i})\cdot \nu _{Y}(B_{i}))$. Applying the
outer measure procedure, we get a measure $\nu _{X}\otimes \nu _{Y}:\mathcal{%
A}_{m}\rightarrow \lbrack 0,1]$ extending $\nu $. We denote $\mathcal{A}_{m}$
by $\mathcal{A}_{X}\otimes \mathcal{A}_{Y}$.

The following situation will come up in Chapter \ref{trianglechapter}:
suppose that $X$ and $Y$ are hyperfinite sets and we construct the Loeb%
\index{Loeb measure} measure spaces $(X,\mathcal{L} _X,\mu_X)$ and $(Y,%
\mathcal{L} _Y,\mu_Y)$. We are thus entitled to consider the product measure
space $(X\times Y,\mathcal{L} _X\otimes \mathcal{L} _Y,\mu_X\otimes \mu_Y)$.
However, $X\times Y$ is itself a hyperfinite set, whence we can consider its
Loeb%
\index{Loeb measure} measure space $(X\times Y, \mathcal{L} _{X\times
Y,L},\mu_{X\times Y})$. There is a connection:

\begin{exercise}
Show that $\mathcal{L}_{X}\otimes \mathcal{L}_{Y}$ is a sub-$\sigma $%
-algebra of $\mathcal{L}_{X\times Y}$ and that $\mu _{X\times Y}|_{(\mathcal{%
L}_{X}\otimes \mathcal{L}_{Y})}=\mu _{X}\otimes \mu _{Y}$.
\end{exercise}

In the proof of the triangle removal lemma in Chapter \ref{trianglechapter},
we will need to use the following Fubini-type theorem for Loeb%
\index{Loeb measure} measure on a hyperfinite set.

\begin{theorem}
\label{fubiniloeb} Suppose that $X$ and $Y$ are hyperfinite sets and $%
f:X\times Y\to \mathbb{R}$ is a bounded $\mathcal{L} _{X\times Y}$%
-measurable function. For $x\in X$, let $f_x:Y\to \mathbb{R}$ be defined by $%
f_x(y):=f(x,y)$. Similarly, for $y\in Y$, let $f^y:X\to \mathbb{R}$ be
defined by $f^y(x):=f(x,y)$. Then:

\begin{enumerate}
\item $f_{x}$ is $\mathcal{L}_{Y}$-measurable for $\mu _{X}$-almost every $%
x\in X$;

\item $f^{y}$ is $\mathcal{L}_{X}$-measurable for $\mu _{Y}$-almost every $%
y\in Y$;

\item The double integral can be computed as an iterated integral: 
\begin{equation*}
\int_{X\times Y}f(x,y)d\mu_{X\times Y}(x,y)=\int_X\left(\int_Y
f_x(y)d\mu_Y(y)\right)d\mu_X(x)=\int_Y\left(\int_X
f^y(x)d\mu_X(x)\right)d\mu_Y(y).
\end{equation*}
\end{enumerate}
\end{theorem}

\begin{proof}
After taking positive and negative parts, it suffices to consider the case
that $f$ is positive. Furthermore, by the Monotone Convergence Theorem, it
suffices to consider the case that $f$ is a step function. Then, by
linearity, one can restrict to the case that $f=\chi _{E}$ is the
characteristic function of a Loeb%
\index{Loeb measure} measurable set $E\subseteq X\times Y$. Now Lemma \ref%
{approx} and a further application of the Monotone Convergence Theorem
allows one to restrict to the case that $E$ is internal. In this case, for $%
x\in X$ we have that $\int_{Y}\chi _{E}(x,y)d\mu _{Y}(y)=%
\operatorname{st}\left( \frac{|E_{x}|}{|Y|}\right) $, where $E_{x}:=\{y\in Y\ :\
(x,y)\in E\}$. By Theorem \ref{integratebycounting}, we thus have 
\begin{equation*}
\int_{X}\left( \int_{Y}\chi _{E}(x,y)d\mu _{Y}(y)\right) d\mu _{X}(x)\approx 
\frac{1}{|X|}\sum_{x\in X}\frac{|E_{x}|}{|Y|}=\frac{|E|}{|X||Y|}\approx
\int_{X\times Y}\chi _{E}(x,y)d\mu _{X\times Y}(x,y).
\end{equation*}%
The other equality is proved in the exact same way.
\end{proof}

\section{Ergodic theory of hypercycle systems\label{hypercycle}}

\begin{definition}
If $(X,\mathcal{B},\mu )$ is a probability space, we say that a bijection $%
T:X\rightarrow X$ is a \emph{measure-preserving transformation} if, for all $%
A\in \mathcal{B}$, $T^{-1}(A)\in \mathcal{B}$ and $\mu (T^{-1}(A))=\mu (A)$.
The tuple $(X,\mathcal{B},\mu ,T)$ is called a \emph{measure-preserving
dynamical system}. A measure-preserving dynamical system $(Y,\mathcal{C},\nu
,S)$ is a \emph{factor }of $\left( X,\mathcal{B},\mu ,T\right) $ if there is
a function $\pi :X\rightarrow Y$ such that, for $A\subseteq Y$, $A\in 
\mathcal{C}$ if and only if $\pi ^{-1}\left( A\right) \in \mathcal{B}$, $\nu
=\pi _{\ast }\mu $---which means $\nu (A)=\mu \left( \pi ^{-1}(A)\right) $
for every $A\in \mathcal{C}$---and $\left( S\circ \pi \right) \left(
x\right) =\left( \pi \circ T\right) \left( x\right) $ for $\mu $-almost
every $x\in X$.
\end{definition}

\begin{example}
Suppose that $X=[0,N-1]$ is an infinite hyperfinite interval. Define $%
S:X\rightarrow X$ by $S(x)=x+1$ if $x<N$ and $S(N-1)=0$. Then $S$ is a
measure-preserving transformation and the dynamical system $(X,\mathcal{L}%
_{X},\mu _{X},S)$ will be referred to as a \emph{hypercycle system}.
\end{example}

The hypercycle system%
\index{hypercycle system} will play an important role later in the book. In
particular, we will need to use the \emph{pointwise ergodic theorem} for the
hypercycle system. While the proof of the general ergodic theorem is fairly
nontrivial, the proof for the hypercycle system, due to Kamae \cite%
{kamae_simple_1982}, is much simpler. In the rest of this section, we fix a
hypercycle system $(X,\Omega _{X},\mu _{X},S)$.%
\index{pointwise ergodic theorem}

\begin{theorem}[The ergodic theorem for the hypercycle system]
\label{ergodichypercycle} Suppose that $f\in L^1(X,\Omega,\mu)$. Define 
\begin{equation*}
\hat{f}(x):=\lim_{n\to \infty}\frac{1}{n}\sum_{i=0}^{n-1} f(S^ix)
\end{equation*}
whenever this limit exists. Then:

\begin{enumerate}
\item $\hat{f}(x)$ exists for almost all $x\in X$;

\item $\hat{f}\in L^1(X,\Omega,\mu)$;

\item $\int_X fd\mu=\int_X \hat{f}d\mu$.
\end{enumerate}
\end{theorem}

\begin{proof}
Without loss of generality, we may assume that $X=[0,N-1]$ for some $N>%
\mathbb{N}$ and $f\left( x\right) \geq 0$ for $\mu _{X}$-almost every $x\in
X $. We set 
\begin{equation*}
\overline{f}(x):=\limsup_{n\rightarrow \infty }\frac{1}{n}%
\sum_{i=0}^{n-1}f(S^{i}x)
\end{equation*}%
and 
\begin{equation*}
\underline{f}(x):=\liminf_{n\rightarrow \infty }\frac{1}{n}%
\sum_{i=0}^{n-1}f(S^{i}x).
\end{equation*}%
Note that $\overline{f},\underline{f}$ are $\mu _{X}$-measurable and $S$%
-invariant. It suffices to show that $\overline{f},\underline{f}\in
L^{1}(X,\Omega ,\mu )$ and that 
\begin{equation*}
\int_{X}\overline{f}d\mu \leq \int_{X}fd\mu \leq \int_{X}\underline{f}d\mu .
\end{equation*}%
Towards this end, fix $\epsilon >0$ and $m\in \mathbb{N}$. By Lemma \ref%
{approxL1}, we may find internal functions $F,G:\left[ 0,N-1\right]
\rightarrow {}^{\ast }\mathbb{R}$ such that:

\begin{itemize}
\item for all $x\in X$, we have $f(x)\leq F(x)$ and $G(x)\leq \min \{\bar{f}%
(x),m\}$;

\item for every internal subset $B$ of $X$%
\begin{equation*}
\max \left\{ \left\vert \int_{B}fd\mu -\frac{1}{N}\sum_{x\in
B}F(x)\right\vert ,\left\vert \int_{B}\min \left\{ \bar{f},m\right\} d\mu -%
\frac{1}{N}\sum_{x\in B}G(x)\right\vert \right\} <\epsilon .
\end{equation*}
\end{itemize}

By definition of $\bar{f}$, for each $x\in X$, there is $n\in \mathbb{N}$
such that $\min \{\bar{f}(x),m\}\leq \frac{1}{n}\sum_{i=0}^{n-1}f(S^{i}x)+%
\epsilon $. For such an $n$ and $k=0,1,\ldots ,n-1$, we then have that 
\begin{equation*}
G(S^{k}x)\leq \min \{\bar{f}(S^{k}x),m\}=\min \{\bar{f}(x),m\}\leq \frac{1}{n%
}\sum_{i=0}^{n-1}f(S^{i}x)+\epsilon \leq \frac{1}{n}%
\sum_{i=0}^{n-1}F(S^{i}x)+\epsilon ,
\end{equation*}%
whence it follows that 
\begin{equation}
\sum_{i=0}^{n-1}G(S^{i}x)\leq \sum_{i=0}^{n-1}F(S^{i}x)+n\epsilon \text{.%
\label{Equation:bound}}
\end{equation}%
Since the condition in \eqref{Equation:bound} is internal, the function $%
\rho :X\rightarrow {}^{\ast }\mathbb{N}$ that sends $x$ to the least $n$
making \eqref{Equation:bound} hold for $x$ is internal. Note that $\rho
(x)\in \mathbb{N}$ for all $x\in K$, whence $\sigma :=\max_{x\in X}\rho
(x)\in \mathbb{N}$.

Now one can start computing the sum $\sum_{x=0}^{N}G(x)$ by first computing 
\begin{equation*}
\sum_{x=0}^{\rho (0)-1}G(x)=\sum_{x=0}^{\rho (0)-1}G(S^{x}0),
\end{equation*}%
which is the kind of sum appearing in \eqref{Equation:bound}. Now in order
to continue the computation using sums in which \eqref{Equation:bound}
applies, we next note that 
\begin{equation*}
\sum_{x=\rho (0)}^{\rho (0)+\rho (\rho (0))-1}G(x)=\sum_{x=0}^{\rho (\rho
(0))-1}G(S^{x}\rho (0)).
\end{equation*}%
This leads us to define, by internal recursion, the internal sequence $(\ell
_{j})$ by declaring $\ell _{0}:=0$ and $\ell _{j+1}:=\ell _{j}+\rho (\ell
_{j})$. It follows that we have 
\begin{equation*}
\sum_{x=0}^{\ell _{J}-1}G(x)=\sum_{j=0}^{J-1}\sum_{i=0}^{\rho (\ell
_{j})-1}G(S^{i}\rho \left( \ell _{j}\right) )\leq
\sum_{j=0}^{J-1}\sum_{j=0}^{\rho (\ell _{j})-1}F(S^{i}x)+\rho (\ell
_{j})\epsilon =\sum_{x=0}^{\ell _{J}-1}F(x)+\ell _{J}\epsilon .
\end{equation*}%
As a result, we have that, whenever $\ell _{J}<N$,%
\begin{equation*}
\frac{1}{N}\sum_{x=0}^{\ell _{J}-1}G(x)\leq \frac{1}{N}\sum_{x=0}^{\ell
_{J}-1}F(x)+\epsilon .
\end{equation*}%
Now take $J$ such that $N-\sigma \leq \ell _{J}<N$. Since $\sigma \in 
\mathbb{N}$ and $G(x)\leq m$ for every $x\in X$, we have that 
\begin{eqnarray*}
\int_{X}\min \{\bar{f},m\}d\mu &\leq &\frac{1}{N}\sum_{x=0}^{N-1}G(x)+%
\epsilon \approx \frac{1}{N}\sum_{x=0}^{\ell _{J}-1}G(x)+\epsilon \\
&\leq &\frac{1}{N}\sum_{x=0}^{\ell _{J}-1}F(x)+2\epsilon \approx \frac{1}{N}%
\sum_{x=0}^{N-1}F(x)+2\epsilon \leq \int_{X}fd\mu +3\epsilon .
\end{eqnarray*}%
Letting $m\rightarrow \infty $ and then $\epsilon \rightarrow 0$, we get
that $\overline{f}\in L^{1}(X,\Omega ,\mu )$ and $\int_{X}\overline{f}d\mu
\leq \int_{X}fd\mu $. The inequality $\int_{X}fd\mu \leq \int_{X}\underline{f%
}d\mu $ is proven similarly.
\end{proof}

In \cite{kamae_simple_1982}, Kamae uses the previous theorem to prove the
ergodic theorem for an arbitrary measure-preserving dynamical system. In
order to accomplish this, he proves the following result, which is
interesting in its own right.%
\index{hypercycle system}

\begin{theorem}[Universality of the hypercycle system]
\label{universality} Suppose that $(Y,\mathcal{B},\nu )$ is a standard
probability space\footnote{%
Unfortunately, \emph{standard} is used in a different sense than in the rest
of this book. Indeed, here, a standard probability space is simply a
probability space which is isomorphic to a quotient of $\left[ 0,1\right] $
endowed with the Borel $\sigma $-algebra and Lebesgue measure.} and $%
T:Y\rightarrow Y$ is an measure-preserving transformation. Then $(Y,\mathcal{%
B},\nu ,T)$ is a factor of the hypercycle system $(X,\Omega _{X},\mu _{X},S)$%
.
\end{theorem}

\begin{proof}
As before, we may assume that $X=[0,N-1]$ for some $N>\mathbb{N}$. Without
loss of generality, we can assume that $(Y,\mathcal{B},\nu )$ is atomless,
and hence isomorphic to $\left[ 0,1\right] $ endowed with the Borel $\sigma $%
-algebra and the Lebesgue measure. Consider the Borel map $r:\left[ 0,1%
\right] \rightarrow \left[ 0,1\right] ^{\mathbb{N}}$ given by $r(y)\left(
n\right) =h\left( T^{n}y\right) $ and the measure $r_{\ast }\nu $ on the
Borel $\sigma $-algebra of $\left[ 0,1\right] ^{\mathbb{N}}$. Then $r$
defines an isomorphism between $\left( Y,\mathcal{B},\nu ,T\right) $ and a
factor of the unilateral Bernoulli shift on $\left[ 0,1\right] ^{\mathbb{N}}$%
. Therefore, it is enough to consider the case when $\left( Y,\mathcal{B}%
,\nu ,T\right) $ is the unilateral Bernoulli shift on $\left[ 0,1\right] ^{%
\mathbb{N}}$ endowed with the Borel $\sigma $-algebra $\mathcal{B}$ and some
shift-invariant Borel probability measure $\nu $.

We now define the factor map $\pi :X\rightarrow \lbrack 0,1]^{\mathbb{N}}$.
In order to do this, we fix $\alpha \in \lbrack 0,1]^{\mathbb{N}}$ such that 
$\lim_{n\rightarrow \infty }%
\frac{1}{n}\sum_{i=0}^{n-1}f(T^{i}\alpha )=\int_{[0,1]^{\mathbb{N}}}f(y)d\nu 
$ for all $f\in C([0,1]^{\mathbb{N}})$. Such an $\alpha $ is called \emph{%
typical} in \cite{kamae_simple_1982} and is well-known to exist.\footnote{%
Of course, one can use the ergodic theorem to prove the existence of typical
elements. However, we need a proof that typical elements exist that does not
use the ergodic theorem. One can see, for example, \cite[Lemma 2]%
{kamae_simple_1982} for such a proof.}

By transfer, one can identify ${}^{\ast }(\left[ 0,1\right] ^{\mathbb{N}})$
with the set of internal functions from ${}^{\ast }\mathbb{N}$ to ${}^{\ast }%
\left[ 0,1\right] $. By compactness of $\left[ 0,1\right] ^{\mathbb{N}}$,
one can deduce that, given $\xi \in {}^{\ast }(\left[ 0,1\right] ^{\mathbb{N}%
})$, there exists a unique element $\operatorname{st}\left( \xi \right) \in \left[
0,1\right] ^{\mathbb{N}}$ such that $\xi \thickapprox \operatorname{st}\left( \xi
\right) $, in the sense that, for every open subset $U$ of $\left[ 0,1\right]
^{\mathbb{N}}$, one has that $\xi \in {}^{\ast }U$ if and only if $\operatorname{st}%
\left( \xi \right) \in {}^{\ast }U$. Concretely, one can identify $\operatorname{st}%
\left( \xi \right) $ with the element of $\left[ 0,1\right] ^{\mathbb{N}}$
such that $\operatorname{st}\left( \xi \right) \left( n\right) =\operatorname{st}\left( \xi
\left( n\right) \right) $ for $n\in \mathbb{N}$.

The function $\mathbb{N}\rightarrow \left[ 0,1\right] ^{\mathbb{N}}$, $%
n\mapsto T^{n}\alpha $ has a nonstandard extension%
\index{nonstandard extension} ${}^{\ast }\mathbb{N}\rightarrow {}^{\ast }(%
\left[ 0,1\right] ^{\mathbb{N}})$. Given $i\in \lbrack 0,N-1]$, define $\pi
(i):=%
\operatorname{st}(T^{i}\alpha )$. We must show that $\pi _{\ast }\mu _{X}=\nu $ and
that $\left( T\circ \pi \right) (i)=\left( \pi \circ S\right) (i)$ for $\mu
_{X}$-almost every $i\in \left[ 0,N-1\right] $. For $f\in C(\left[ 0,1\right]
^{\mathbb{N}})$, we have that 
\begin{equation*}
\int_{\lbrack 0,1]^{\mathbb{N}}}f(y)d\nu =\lim_{n\rightarrow \infty }\frac{1%
}{n}\sum_{i=0}^{n-1}f(T^{i}\alpha )\thickapprox \frac{1}{N}%
\sum_{i=0}^{N-1}f\left( T^{i}\alpha \right) \approx \int_{X}\left( f\circ
\pi \right) d\mu _{X}.
\end{equation*}%
Note that the first step uses the fact that $\alpha $ is typical and the
last step uses the fact that $f$ is continuous and Theorem \ref%
{integratebycounting}. This shows that%
\begin{equation*}
\int_{\lbrack 0,1]^{\mathbb{N}}}fd\nu =\int_{X}\left( f\circ \pi \right)
d\mu _{X}=\int_{\left[ 0,1\right] ^{\mathbb{N}}}fd\pi _{\ast }\mu _{X}
\end{equation*}%
and hence $\nu =\pi _{\ast }\mu _{X}$.

To finish, we show that $\left( T\circ \pi \right) (i)=\left( \pi \circ
S\right) (i)$ for $\mu _{X}$-almost every $i\in X$. Fix $i\in \lbrack 0,N-2]$%
. Then we have 
\begin{equation*}
T(\pi (i))=T(\operatorname{st}(T^{i}\alpha ))=\operatorname{st}(T^{i+1}\alpha ))=\pi (S(i)),
\end{equation*}%
where the second equality uses the fact that $T$ is continuous.
\end{proof}

From Theorems \ref{ergodichypercycle} and \ref{universality}, we now have a
proof of the ergodic theorem for measure-preserving systems based on
standard probability spaces. It only requires one more step to obtain the
ergodic theorem in general.

\begin{corollary}[The ergodic theorem]
\index{pointwise ergodic theorem} Suppose that $(Y,\mathcal{B},\nu,T)$ is a
measure-preserving dynamical system and $f\in L^1(X,\Omega,\mu)$. Define $%
\hat{f}(x):=\lim_{n\to \infty}\frac{1}{n}\sum_{i=0}^{n-1} f(T^ix)$ whenever
this limit exists. Then:

\begin{enumerate}
\item $\hat{f}(x)$ exists for almost all $x\in Y$;

\item $\hat{f}\in L^1(Y,\mathcal{B},\nu)$;

\item $\int_Y fd\nu=\int_Y \hat{f}d\nu$.
\end{enumerate}
\end{corollary}

\begin{proof}
Let $\tau :Y\rightarrow \mathbb{R}^{\mathbb{N}}$ be given by $\tau
(y)(n):=f(T^{n}y)$. Let $\mathcal{C}$ denote the Borel $\sigma $-algebra of $%
\mathbb{R}^{\mathbb{N}}$. Let $\sigma $ be the shift operator on $\mathbb{R}%
^{\mathbb{N}}$. Let $g:\mathbb{R}^{\mathbb{N}}\rightarrow \mathbb{R}$ be
given by $g(\alpha )=\alpha (0)$. It is then readily verified that the
conclusion of the ergodic theorem for $(Y,\mathcal{B},\nu ,T,f)$ is
equivalent to the truth of the ergodic theorem for $(\mathbb{R}^{\mathbb{N}},%
\mathcal{C},\tau _{\ast }\nu ,\sigma ,g)$, which, as we mentioned above,
follows from Theorems \ref{ergodichypercycle} and \ref{universality}.
\end{proof}

\section*{Notes and references}
The Loeb measure construction was introduced by Loeb in 1973 \cite{loeb_conversion_1975}. The  Loeb measure plays a crucial role in several applications of nonstandard methods to a wide variety of areas of mathematics, including measure theory, probability theory, and analysis. A survey of such applications can be found in \cite{cutland_loeb_2000}. The nonstandard proof of the ergodic theorem due to Kamae \cite{kamae_simple_1982,kamae_simple_1997} is just a single but insightful example of the usefulness of Loeb measure. The Loeb measure also underpins the nonstandard perspective on the Furstenberg correspondence theorem, which in turns opens the gates to application of nonstandard methods to additive number theory (see Chapter \ref{chapter_densitites} below).

%
%
%

\part{Ramsey theory}

\chapter{Ramsey's Theorem}

\section{Infinite Ramsey's Theorem}

Recall that a \emph{graph} is a pair $(V,E)$ where
$V$ is the set of \emph{vertices}, and the set of \emph{edges}
$E\subseteq V\times V$ is an anti-reflexive and symmetric binary
relation on $V$. If $X\subseteq V$ is such that $%
(x,x^{\prime })\in E$ (resp. $(x,x^{\prime })\notin E$) for all distinct $%
x,x^{\prime }\in X$, we say that $X$ is a \emph{clique} (resp. \emph{%
anticlique}) in $(V,E)$.

\begin{theorem}[Ramsey's theorem for pairs]
\label{Theorem:Ramsey-pairs}If $\left( V,E\right) $ is an infinite graph,
then $\left( V,E\right) $ either contains an infinite clique or an infinite
anticlique.
\end{theorem}

\begin{proof}
Let $\xi $ be an element of ${}^{\ast }V$ that does not belong to $V$.
Consider the element $\left( \xi ,{}^{\ast }\xi \right) \in {}^{\ast \ast }V$%
. There are now two possibilities: either $\left( \xi ,{}^{\ast }\xi \right)
\in {}^{\ast \ast }E$ or $\left( \xi ,{}^{\ast }\xi \right) \notin {}^{\ast
\ast }E$. We only treat the first case, the second case being entirely
similar. We recursively define a one-to-one sequence $\left( x_{n}\right) $
in $V$ such that the set $\left\{ x_{n}:n\in \mathbb{N}\right\} $ forms a
clique in $\left(V,E\right) $. Towards this end, suppose that $d\in \mathbb{N%
}$ and $x_0,\ldots,x_{d-1}$ are distinct elements of $V$ such that, for all $%
1\leq i<j<d$, we have

\begin{itemize}
\item $\left( x_{i},x_{j}\right) \in E$, and

\item $\left( x_{i},{}\xi \right) \in {}^{\ast }E$.
\end{itemize}

Consider now the statement \textquotedblleft there exists $y\in {}^{\ast }V$
such that, for $i<d$, $y$ is different from $x_{i}$, and $\left(
x_{i},y\right) \in {}{}^{\ast}E$, and $\left( y,{}^{\ast}\xi \right) \in {}^{\ast
\ast}E$\textquotedblright, whose truth is witnessed by $\xi$. It follows by transfer that there
exists $x_{d}\in V$ different from $x_{i}$ for $i<d$, such that $\left(
x_{i},x_{d}\right) \in E$ for $i<d$, and $\left( x_{d},\xi \right) \in
{}{}^{\ast }E$. This concludes the recursive construction. %
\end{proof}

In order to prove the full Ramsey theorem, we need the notion of a
hypergraph. Given $m\in \N$, an \emph{$m$-regular hypergraph} is a set $V$
of vertices together with a subset $E$ of $V^{m}$ that is
permutation-invariant and has the property that $\left( x_{1},\ldots
,x_{m}\right) \in E$ implies that $x_{1},\ldots ,x_{m}$ are pairwise
distinct. A \emph{clique} (resp.\ \emph{anticlique}) for $\left( V,E\right) $
is a subset $Y$ of $V$ with the property that $\left( y_{1},\ldots
,y_{m}\right) \in E$ (resp. $\left( y_{1},\ldots ,y_{m}\right) \notin E$)
for any choice of pairwise distinct elements $y_{1},\ldots ,y_{m}$ of $Y$.

\begin{theorem}[Ramsey's theorem]
\label{Theorem:Ramsey-hypergraph}If $\left( V,E\right) $ is an infinite $m$%
-regular hypergraph, then $\left( V,E\right) $ contains an infinite clique
or an infinite anticlique.
\end{theorem}

\begin{proof}
For simplicity, we consider the case when $m=3$. Let $\xi $ be an element of 
${}^{\ast }V$ that does not belong to $V$. As before, there are now two
cases, depending on whether $\left( \xi ,{}^{\ast }\xi ,{}^{\ast \ast }\xi
\right) $ belongs to ${}^{\ast \ast \ast }E$ or not. Once again, we only
treat the first case.

We recursively define a one-to-one sequence $\left( x_{n}\right) $ of
elements of $V$ such that $\{x_n \ : \ n\in \N\}$ forms a clique for $V$.
Towards this end, suppose that $d\in \mathbb{N}$ and $x_0,\ldots,x_{d-1}$
are distinct elements of $V$ such that, for all $1\leq i<j<k<d$, we have:

\begin{itemize}
\item $\left( x_{i},x_{j},x_{k}\right) \in E$,

\item $\left( x_{i},x_{j},\xi \right) \in {}^{\ast }E$, and

\item $\left( x_{i},\xi ,{}^{\ast }\xi \right) \in {}^{\ast \ast }E$.
\end{itemize}

Consider now the statement \textquotedblleft there exists $y\in {}^{\ast }V$
such that $y$ is different from $x_{i}$ for $1\leq i<d$, $\left(
x_{i},x_{j},y\right) \in {}^{\ast }E$ for every $1\leq i<j<d$, $\left(
x_{i},y,^{\ast }\xi \right) \in {}^{\ast \ast }E$ for every $1\leq i<d$, and 
$\left( y,{}^{\ast }\xi ,{}^{\ast \ast }\xi \right) \in {}^{\ast \ast \ast
}E $.\textquotedblright\ Note that $\xi $ witnesses the truth of the
statement in the nonstandard extension\index{nonstandard extension}. Therefore, by transfer, there is an
element $x_{d}$ of $V$ distinct from $x_{i}$ for $1\leq i<d$ for which the
above items remain true for all $1\leq i<j<k\leq d$. This completes the
recursive construction.
\end{proof}

Ramsey's theorem\index{Ramsey's theorem} is often stated in the language of colorings\index{coloring}. Given a set $%
X $ and $m\in \N$, we let $X^{[m]}$ denote the set of $m$-element subsets of 
$X $. If $X\subseteq \N$, we often identify $X^{[m]}$ with the set of pairs $\{(x_{1},\ldots
,x_{m})\in X^{m}\ :\ x_{1}<\cdots <x_{m}\}$. Given $k\in \N$, a $k$-coloring \index{coloring}
of $X^{[m]}$ is a function $c:X^{[m]}\rightarrow \{1,\ldots ,k\}$. In this
vein, we often refer to the elements of $\{1,\ldots ,k\}$ as colors.
Finally, a subset $Y\subseteq X$ is \emph{monochromatic for the coloring $c$}
if the restriction of $c$ to $Y^{\left[ m\right] }$ is constant. Here is the
statement of Ramsey's theorem for colorings.

\begin{corollary}
\label{Corollary:Ramsey-infinite}For any $k,m\in \N$, any infinite set $V$,
and any $k$-coloring $c$ of $V^{[m]}$, there is an infinite subset of $V$
that is monochromatic for the coloring $c$.
\end{corollary}

\begin{proof}
By induction, it suffices to consider the case $k=2$. We identify a coloring 
$c:V^{[m]}\to \{1,2\}$ with the $m$-regular hypergraph $(V,E)$ satisfying $%
(x_1,\ldots,x_m)\in E$ if and only if $c(\{x_1,\ldots,x_m\})=1$ for distinct 
$x_1,\ldots,x_m\in V$. An infinite clique (resp. anticlique) in $(V,E)$
corresponds to an infinite set with color $1$ (resp. $2$), whence the
corollary is merely a restatement of our earlier version of Ramsey's theorem\index{Ramsey's theorem}.
\end{proof}

\begin{remark}
Ramsey's Theorem cannot be extended to finite colorings of 
the infinite parts $V^{[\infty]}=\{A\subseteq V\mid A\ \text{is infinite}\}$.
Indeed, pick a copy of the natural numbers $\N\subseteq V$,
pick an infinite $\alpha \in\starN\setminus\N$, and for $A\in V^{[\infty]}$
set $c(A)=1$ if the internal cardinality $|\starA\cap[1,\alpha]|$ is odd,
and $c(A)=2$ otherwise. Then 
$c:V^{[\infty]}\to\{1,2\}$ is a $2$-coloring with the property that
$X^{[\infty]}$ is \emph{not} monochromatic for any
infinite $X\subseteq V$ since, \emph{e.g.}, $c(X)\ne c(X\setminus\{x\})$
for every $x\in X$.
\end{remark}

\section{Finite Ramsey Theorem}

Corollary \ref{Corollary:Ramsey-infinite} is often referred to as the
infinite Ramsey theorem. We now deduce from it the finite Ramsey theorem. We
first need a bit of notation.

\begin{definition}
Given $k,l,m,n\in \N$, we write $l\to (n)^m_k$ if every coloring\index{coloring} of $[l]^{[m]}$ with $k$ colors has a homogeneous set of size $n$.
\end{definition}

\begin{corollary}[Finite Ramsey Theorem]
For every $k,m,n\in \N$, there is $l\in \N$ such that $l\to (k)^n_m$.
\end{corollary}

\begin{proof}
Suppose the theorem is false for a particular choice of $k,m,n$. Then for
every $l\in \N$, there is a \textquotedblleft bad\textquotedblright\
coloring $c:[l]^{[m]}\rightarrow \{1,\ldots ,k\}$ with no monochromatic
subset of size $n$. We can form a finitely branching tree of bad colorings
with the partial order being inclusion. Since there is a bad coloring for
every such $l$, we have that the tree is infinite. By K\"{o}nig's Lemma,
there is an infinite branch. This branch corresponds to a coloring of $[\N%
]^{[m]}\rightarrow \{1,\ldots ,k\}$ with no monochromatic subset of size $n$, contradicting the Infinite Ramsey Theorem.
\end{proof}

\section{Rado's Path Decomposition Theorem}\index{Rado's Path Decomposition Theorem}

In this section, by a \emph{path in $\N$} we mean a (finite or infinite) injective sequence of natural numbers.  For a finite path $(a_0,\ldots,a_n)$ from $\N$, we refer to $a_n$ as the \emph{end of the path}.

Suppose that $c:\N^{[2]}\to \{1,\ldots,r\}$ is an $r$-coloring of $\N^{[2]}$.  For $i\in \{1,\ldots,r\}$, we say that a path $P=(a_n)$ has color $i$ if $c(\{a_n,a_{n+1}\})=i$ for all $n$.

\begin{theorem}[Rado's Path Decomposition Theorem]
Suppose that $c:\N^{[2]}\to \{1,\ldots,r\}$ is an $r$-coloring of $\N^{[2]}$.  Then there is a partition of $\N$ into paths $P_1,\ldots,P_r$ such that each $P_i$ has color $i$.
\end{theorem}

\begin{proof}
First, fix $\alpha\in \starN$.  For $m\in \N$ and $i\in \{1,\ldots,r\}$, we say that $m$ has color $i$ if $c(\{m,\alpha\})=i$.  We now recursively define disjoint finite paths $P_{1,k},\ldots,P_{r,k}$ such that, whenever $P_{i,k}\not=\emptyset$, then the end of $P_{i,k}$ has color $i$ (in the sense of the previous sentence).  

To start, we define $P_{i,0}=\emptyset$ for each $i=1,\ldots,r$.  Now assume that $P_{i,k-1}$ has been constructed for $i=1,\ldots,r$.  If $k$ belongs to some $P_{i,k-1}$, then set $P_{i,k}:=P_{i,k-1}$ for all $i=1,\ldots,r$.  Otherwise, let $i$ be the color of $k$ and let $e$ be the end of $P_{i,k-1}$.  Since $c(\{k,\alpha\})=c(\{e,\alpha\})=i$, by transfer, we can find $f\in \N$ larger than all numbers appearing in $\bigcup_{i=1}^r P_{i,k-1}$ such that $c(\{k,f\})=c(\{e,f\})=i$.  We then set $P_{j,k}:=P_{j,k-1}$ for $j\not=i$ and $P_{i,k}:=P_{i,k-1}\concat (f,k)$.  Note that the recursive assumptions remain true.

For $i=1,\ldots,r$, we now set $P_i$ to be the union of $P_{i,k}$ for $k\in \mathbb{N}$.  It is clear that $P_1,\ldots,P_r$ are as desired.
\end{proof}

\section{Ultrafilter trees}

Given a set $X$, we let $X^{[<\infty]}$ (resp. $X^{[\infty]})$ denote the set of finite (resp. infinite) subsets of $X$.  Given $s\in \N^{[<\infty]}$ and $X\subseteq \N$, we say that $s$ is an \emph{initial segment} of $X$, denoted $s\sqsubseteq X$, if there is $j\in \N$ such that $s=\{j\in X \ : \ j\leq i\}$.  

\begin{definition}
A subset $T$ of $\N^{[<\infty]}$ is called a \emph{tree on $\N$} if $T\not=\emptyset$ and for all $s,t\in \N^{[<\infty]}$, if $s\sqsubseteq t$ and $t\in T$, then $s\in T$.
\end{definition}

For a tree $T$ on $\N$, we set
$$[T]:=\{X\in \N^{[\infty]} \ : \ \forall s\in \N^{[<\infty]}(s\sqsubseteq X \Rightarrow s\in T)\}.$$  If there is an element of $T$ that is $\sqsubseteq$-maximal with respect to the property that it is $\sqsubseteq$-comparable to every element of $T$, we call this (necessarily unique) element of $T$ the \emph{stem} of $T$, denoted $\stem(T)$. Finally, given $s\in T$, we set $T/s:=\{t\in T \ : \ s\sqsubseteq t\}$.

\begin{definition}
Let $\vec \u=\langle \u_s \ : \ s\in \N^{[<\infty]}\rangle$ be a family of nonprincipal ultrafilters on $\N$ and let $T$ be a tree on $\N$.  We say that $T$ is a \emph{$\vec \u$-tree} if it has a stem, $T/\stem(T)$ is nonempty, and for all $s\in T/\{ \stem(T)$\} , we have that $\{n\in \N \ : \ s\cup\{n\}\in T\}\in \u_s$.
\end{definition}

Note that a $\vec \u$-tree $T$ contains no $\sqsubseteq$-maximal elements and that, for every $s\in T$, there is $X\in [T]$ such that $s\sqsubseteq X$.

The goal of this section is to prove the following Ramsey-theoretic statement about ultrafilter trees, recently proven by Trujillo in \cite{trujillo_abstract_2016}:

\begin{theorem}\label{ultraramsey}
Suppose that $\vec \u=\langle \u_s \ : \ s\in \N^{[<\infty]}\rangle$ is a sequence of non-principal ultrafilters on $\N$, $T$ is a $\vec \u$-tree on $\N$, and $\mathcal{X}\subseteq \N^{[\infty]}$.  Then there is a $\vec \u$-tree $S\subseteq T$ with $\stem(S)=\stem(T)$ such that one of the following holds:
\begin{enumerate}
\item $[S]\subseteq \mathcal{X}$;
\item $[S]\cap \mathcal{X}=\emptyset$;
\item for every $\vec \u$-tree $S'$ with $S'\subseteq S$, we have $[S']\not\subseteq \mathcal{X}$ and $[S']\cap \mathcal{X}\not=\emptyset$.
\end{enumerate}
\end{theorem}

Using hyperfinite generators of ultrafilters, we immediately obtain the following notion:

\begin{definition}
Let $\vec \alpha=\langle \alpha_s \ : \ s\in \N^{[<\infty]}\rangle$ be a family of infinite elements of $\starN$ and let $T$ be a tree on $\N$.  We say that $T$ is a \emph{$\vec \alpha$-tree} if, for all $s\in T/\stem(T)$, we have that $s\cup\{\alpha_s\}\in {}^{\ast}T$.
\end{definition}

Before proving Theorem \ref{ultraramsey}, we need one key lemma:

\begin{lemma}\label{trujillolemma}
Fix $\vec \alpha=\langle \alpha_s \ : \ s\in \N^{[<\infty]}\rangle$ with each $\alpha_s$ infinite.  Suppose that $C\subseteq \N^{[<\infty]}$ is such that, for all $s\in C$, we have that $s\cup\{\alpha_s\}\in {}^{\ast}C$.  Then for all $\vec \alpha$-trees $T$, if $\stem(T)\in C$, then there is a $\vec \alpha$-tree $S\subseteq T$ with $\stem(S)=\stem(T)$ such that $S/\stem(S)\subseteq C$.
\end{lemma}

\begin{proof}
Suppose that $T$ is a $\vec \alpha$-tree with $\stem(T)\in H$.  We first recursively define sets $L_n\subseteq C\cap T$ as follows.  Set $L_0:=\{\stem(T)\}$.  Supposing that $L_n$ has been defined, we set $$L_{n+1}:=\{s\cup \{m\} \ : \ s\in L_n,\  m>\max(s), \text{ and }s\cup \{m\}\in C\cap T\}.$$  We now set 
$$S:=\{s\in \N^{[<\infty]} \ : \ s\sqsubseteq \stem(T)\}\cup \bigcup_{n=0}^\infty L_n.$$  We claim that this $S$ is as desired.  It follows directly by induction that $S$ is a tree on $\N$ and that $S\subseteq T$.  Moreover, by the hypothesis on $C$ and the fact that $T$ is an $\vec \alpha$-tree, we have that $S$ is also an $\vec \alpha$-tree.  It is clear that $\stem(T)\sqsubseteq \stem(S)$.  However, since $\stem(T)\cup \{\alpha_{\stem(T)}\}\in {}^{\ast}L_1$, we have that $\{n\in \N \  \ \stem(T)\cup \{n\}\in L_1\}$ is infinite, whence it follows that $\stem(S)=\stem(T)$.  Finally, $S/\stem(S)=\bigcup_{n=0}^\infty L_n\subseteq C$.
\end{proof}

We can now prove Theorem \ref{ultraramsey} in its equivalent nonstandard formulation:

\begin{theorem}\label{nsultraramsey}
Suppose that $\vec \alpha=\langle \alpha_s \ : \ s\in \N^{[<\infty]}\rangle$ is a sequence of infinite elements of $\starN$, $T$ is a $\vec \alpha$-tree on $\N$, and $\mathcal{X}\subseteq \N^{[\infty]}$.  Then there is a $\vec \alpha$-tree $S\subseteq T$ with $\stem(S)=\stem(T)$ such that one of the following holds:
\begin{enumerate}
\item $[S]\subseteq \mathcal{X}$;
\item $[S]\cap \mathcal{X}=\emptyset$;
\item for every $\vec \alpha$-tree $S'$ with $S'\subseteq S$, we have $[S']\not\subseteq \mathcal{X}$ and $[S']\cap \mathcal{X}\not=\emptyset$.
\end{enumerate}
\end{theorem}

\begin{proof}
We introduce the following three sets:
$$A:=\{s\in \N^{[<\infty]} \ : \ \text{ there is a } \vec\alpha\text{-tree }S\subseteq T \text{ with }\stem(S)=s \text{ and }[S]\subseteq \mathcal{X}\},$$
$$B:=\{s\in \N^{[<\infty]} \ : \ \text{ there is a } \vec\alpha\text{-tree }S\subseteq T \text{ with }\stem(S)=s \text{ and }[S]\subseteq [\N]^{\infty}\setminus \mathcal{X}\},$$
$$C:=\N^{[<\infty]}\setminus (A\cup B).$$

\

\noindent \textbf{Claim:}  If $s\in C$, then $s\cup \{\alpha_s\}\in {}^{\ast}C$.

\

\noindent \textbf{Proof of claim:}  We argue by contrapositive, whence we assume that $s\cup\{\alpha_s\}\in \starA \cup {}^{\ast}B$.  We only treat the case that $s\cup \{\alpha_s\}\in \starA$, the other case being similar.  Let $D:=\{n\in \N \ : \ s\cup \{n\}\in A\}$. Note that $\alpha_s\in {}^{\ast}D$.  For each $n\in D$, let $T_n$ be a $\vec \alpha$-tree with $\stem(T_n)=s\cup\{n\}$ and $[T_n]\subseteq \mathcal{X}$.  Let $S:=\bigcup_{n\in D}T_n$.  Observe that:
\begin{enumerate}
\item[(i)] $S$ is a tree, 
\item[(ii)] $\stem(S)=s$, 
\item[(iii)] $\{s\cup \{n\} \ : \ n\in D\}\subseteq S$, and
\item[(iv)] $[S]=\bigcup_{n\in D}[T_n]\subseteq \mathcal{X}$.
\end{enumerate}

It remains to show that $S$ is a $\vec \alpha$-tree, for then $s\in A$, as desired.  Thus, given $t\in S$, we need $t\cup\{\alpha_t\}\in {}^{\ast}S$.  If $t=s$, then $s\cup\{\alpha_s\}\in {}^{\ast}S$ by item (iii) and the above observation that $\alpha_s\in {}^{\ast}D$.  Otherwise, there is $n\in D$ such that $t\in T_n/(s\cup \{n\})$.  Since $T_n$ is a $\vec\alpha$-tree, we have that $t\cup\{\alpha_t\}\in {}^{\ast}T_n\subseteq {}^{\ast}S$.  This finishes the proof of the claim.

\

It is clear that if $\stem(T)\in A$ (resp. $\stem(T)\in B$), then item (1) (resp. item (2)) of the conclusion of the theorem holds.  We may thus suppose that $\stem(T)\in C$.  By Lemma \ref{trujillolemma}, there is a $\vec \alpha$-tree $S\subseteq T$ with $\stem(S)=\stem(T)$ such that $S/\stem(S)\subseteq C$.  We claim that this $S$ is as desired.  Indeed, suppose that $S'$ is a $\vec \alpha$-tree with $S'\subseteq S$.  Then $\stem(S')\in S/\stem(S)\subseteq C$.  It follows from the definition of $C$ that $[S']\not\subseteq \mathcal{X}$ and $[S']\cap \mathcal{X}\not=\emptyset$, as desired.  
\end{proof}

We offer one application of Theorem \ref{ultraramsey}.  Given a tree $T$ on $\N$ and $n\in \N$, we set $T(n):=T\cap \N^{[n]}$.

\begin{corollary}[Ramsey's Theorem for $\vec \u$-trees]
Fix $n\in \N$ and $A\subseteq \N^{[n]}$.  Further fix a  sequence $\vec \u=\langle \u_s \ : \ s\in \N^{[<\infty]}\rangle$ of nonprincipal ultrafilters on $\N$ and a $\vec \u$-tree $T$.  Then there is a $\vec \u$-tree $S\subseteq T$ with $\stem(S)=\stem(T)$ such that either $S(n)\subseteq A$ or $S(n)\cap A=\emptyset$.
\end{corollary}

\begin{proof}
For each $Y\subseteq \N$ with $|Y|\geq n$, set $r_n(Y)\in \N^{[n]}$ to be the unique $s\in \N^{[n]}$ with $s\sqsubseteq Y$.  Set $\mathcal{X}:=\{Y\in \N^{[\infty]} \ : \ r_n(Y)\in A\}$.  We apply Theorem \ref{ultraramsey} to $\vec \u$, $T$, and $\mathcal{X}$, obtaining a $\vec\u$-tree $S$ with $S\subseteq T$ and $\stem(S)=\stem(T)$.  Note that $S$ cannot satisfy item (3) in the conclusion of Theorem \ref{ultraramsey}:  if $S'$ is a $\vec\u$-tree with $S'\subseteq S$ and $|\stem(S')|\geq n$, then either $r_n(\stem(S'))\in A$ (whence $[S']\subseteq \mathcal{X}$) or $r_n(\stem(S'))\notin A$ (whence $[S']\cap \mathcal{X}=\emptyset)$.  Consequently, either $[S]\subseteq \mathcal{X}$ (whence $S(n)\subseteq A$) or $[S]\cap \mathcal{X}=\emptyset$ (whence $S(n)\cap A=\emptyset$).
\end{proof}

\section*{Notes and references} The Ramsey theorem was proved in the foundational original paper of Ramsey \cite{ramsey_problem_1930}. In fact, in this paper the theorem is obtained as an intermediate step towards establishing a result in propositional logic, hence the title ``On a problem of formal logic''. Although van der Waerden's theorem (presented in the next chapter) predates Ramsey's theorem, the latter is universally recognized as the foundational result of the area which is now called Ramsey theorem.
While Ramsey's theorem did not initially receive  too much attention, it was later ``rediscovered'' in the 1950s by Erd\H{o}s and Rado who recognized its fundamental importance and provided several applications, such as Rado's decomposition theorem \cite{rado_monochromatic_1978}. For more on the metamathematics of Rado's Decomposition Theorem, 
see \cite{cholak_rado_2016},  whose ultrafilter proof of the theorem is essentially the proof given here.  Ultrafilter trees were first introduced by Blass in \cite{blass_selective_1988} and are part of a much larger theory of \emph{local Ramsey theory}, which is extensively developed in the book \cite{todorcevic_introduction_2010}.

\chapter{van der Waerden's and Hales-Jewett Theorems}

\section{van der Waerden's theorem\label{Section:vdW}}

The van der Waerden theorem\index{van der Waerden's Theorem} is one of the earliest achievements of what is
now called Ramsey theory. Indeed, it was established by van der Waerden in 1928 
\cite{van_der_waerden_beweis_1927}, thus predating Ramsey's theorem itself. The theorem is concerned with the notion of \emph{arithmetic progressions} \index{arithmetic progression}in the
set $\mathbb{N}$ of natural numbers. More precisely, for $k\in \N$, a $k$-term arithmetic
progression in $\N$ is a set of the form $a+d[0,k):=\left\{
a,a+d,a+2d,\ldots ,a+\left( k-1\right) d\right\} $ for some $%
a,d\in \mathbb{N}$. A $k$-term arithmetic progression\index{arithmetic progression} is also called an
arithmetic progression of length $k$. An arithmetic progression in ${}^{\ast
}\mathbb{N}$ is defined in a similar fashion, where one can actually consider $%
k$-term arithmetic progressions for $k\in {}^{\ast }\mathbb{N}$. 

Recall that, for $k\in \mathbb{N}$, a $k$-\emph{coloring} of a set $A$ is a function from $A$ to the set $[1,k]=\{ 1,\dots , k\}$. A \emph{finite coloring} of a $A$ is a $k$-coloring for some $k\in \mathbb{N}$. A subset $B$ of $A$ is \emph{monochromatic} with respect to a coloring $c$ if it is contained in the preimage of $i$ under $c$ for some $i\in [1,k]$. A collection $\mathcal{C}$ of subsets of $\mathbb{N}$ is \emph{partition regular} if it is closed under supersets and, for any $A\in \mathcal{C}$ and finite coloring $c$ of $A$, there is a monochromatic $B\subseteq A$ such that $B\in \mathcal{C}$. \index{partition regular}

\begin{theorem}
The following are equivalent:

\begin{enumerate}
\item Every finite coloring of $\mathbb{N}$ admits arbitrarily long monochromatic arithmetic progressions\index{arithmetic progression}.

\item For every $r,k\in \N$, there is $l\in \N$ such that every $r$-coloring of $[1,l]$  admits a monochromatic $k$-term arithmetic
progression.

\item The property of containing arbitrarily long arithmetic progressions\index{arithmetic progression} is
partition regular.
\end{enumerate}
\end{theorem}

\begin{proof}
(1)$\Rightarrow $(2) Suppose that (2) fails for some $k,r$. By overflow\index{overflow principle},
there is $L>\N$ and an internal $r$-coloring\index{coloring} of $[1,L]$ with no
monochromatic $k$-term arithmetic progression\index{arithmetic progression}. By considering the
restriction of $c$ to $\N$, we get an $r$-coloring of $\N$ with no
monochromatic $k$-term arithmetic progression, whence (1) fails.

(2)$\Rightarrow $(3) Suppose that (2) holds. Towards establishing (3), fix a
set $A$ containing arbitrarily long arithmetic progressions and a partition
of $A$ into two pieces $A=B_{1}\sqcup B_{2}$. Fix $k\in \N$. Let $l$ witness
the truth of (2) with 2 colors and $k$-term arithmetic progressions. Fix an
arithmetic progression $x+[0,l)d\subseteq A$. For $i=1,2$, let $%
C_{i}:=\{n\in \lbrack 0,l)\ :\ x+nd\in B_{i}\}$. Then there is $i\in \{1,2\}$
such that $C_{i}$ contains an arithmetic progression $y+[0,k)e$. It follows
that $(x+yd)+[0,k)de$ is a $k$-term arithmetic progression\index{arithmetic progression} contained in $%
B_{i}$. Since some $i$ must work for infinitely many $k$'s, we see that some 
$B_{i}$ contains arbitrarily long arithmetic progressions.

(3)$\Rightarrow $(1) This is obvious.
\end{proof}

%

The following is a nonstandard presentation of the proof of van der
Waerden's theorem from \cite{graham_short_1974}; see also \cite[Section 2.3]%
{promel_ramsey_2013}. First, some terminology. For $k,m\in \N$ and $g,h\in
[0,k]^m$, we say that $g$ and $h$ are equivalent, written $g\equiv h$, if $g$
and $h$ agree up to the last occurrence of $k$.

\begin{definition}
For $k,m\in \N$, let $S(m,k,r,n)$ be the statement: for any $r $-coloring\index{coloring} of 
$\left[ 1,n\right] $, there exist $a,d_{0},\ldots ,d_{m-1}\in \left[ 1,n%
\right] $ such that $a+k\sum_{j<m}d_{j}\in \left[ 1,n\right] $ and, for any $%
g,h\in \left[ 0,k\right] ^{m}$ such that $g\equiv h$, the elements $%
a+\sum_{j<m}g_{j}d_{j}$ and $a+\sum_{j<m}h_{j}d_{j}$ have the same color. We
then let $S(m,k)$ be the statement: for all $r\in \N$, there is $n\in \N$
such that $S(m,k,r,n)$ holds.
\end{definition}

We first observe that even though the statement $S(m,k,r,n)$ considers
colorings of $[1,n]$, it is readily verified that its truth implies the
corresponding statement for colorings of any interval of length $n$.

We next observe that the finitary van der Waerden\index{van der Waerden's Theorem} theorem is the statement
that $S(k,1)$ holds for all $k\in \N$. Indeed, suppose that $S(k,1)$ holds
and fix $r\in \N$. Fix $n\in \N$ such that $S(k,1,r,n)$ holds. Let $%
c:[1,n]\to [1,r]$ be an $r$-coloring of $[1,n]$. Then there is $a,d\in [1,n]$
such that $a+kd\in [1,n]$ and, since all elements of $[0,k]^1$ are
equivalent, we get that $c(a+gd)=c(a+hd)$ for all $g,h\in [0,k]$, whence we
get a monochromatic arithmetic progression\index{arithmetic progression} of length $k+1$.

If $\nu\in \starN$, then we also consider the internal statement $%
S(m,k,r,\nu)$ which is defined exactly as its standard counterpart except
that it only considers internal $r$-colorings\index{coloring} of $[1,\nu]$.

\begin{lemma}
$S(k,m)$ is equivalent to the statement: for all $r\in \N$ and all $\nu\in %
\starN\setminus \N$, we that have that $S(m,k,r,\nu)$ holds.
\end{lemma}

\begin{proof}
First suppose that $S(k,m)$ holds. Given $r\in \N$, take $n\in \N$ such that 
$S(k,m,r,n)$ holds. Fix $\nu \in \starN\setminus \N$ and consider an
internal $r$-coloring $c$ of $[1,\nu ]$. Then $c|_{\left[ 1,n\right] }$ is
an $r$-coloring of $[1,n]$, whence the validity of $S(k,m,r,n)$ yields the
desired conclusion. Conversely, if $S(k,m,r,\nu )$ holds for all $\nu \in %
\starN\setminus \N$, then by underflow there is $n\in \N$ such that $%
S(k,m,r,n)$ holds.
\end{proof}

\begin{theorem}
$S(k,m)$ holds for all $k,m\in \N$.
\end{theorem}

\begin{proof}
Suppose, towards a contradiction, that $S(k,m)$ fails for the pair $(k,m)$
and that $(k,m)$ is lexicographically least with this property.

\ 

\noindent \textbf{Claim:} $m=1$.

\noindent \textbf{Proof of Claim:} Suppose the claim is false. We obtain a
contradiction by showing that $S(k,m,r,\nu )$ holds for all $r\in \N$ and
all $\nu \in \starN\setminus \N$. Towards this end, fix $r\in \N$, $\nu \in %
\starN\setminus \N$, and an internal coloring $c:[1,\nu ]\rightarrow \lbrack
1,r]$. Since $S(k,m-1)$ is true, there is $M\in \N$ such that $S(k,m-1,r,M)$
is true. Write $\nu =NM+s$ with $0\leq s<M$. Note that $N\in \starN\setminus %
\N$. Consider the internal coloring $c_{N}:[1,N]\rightarrow \lbrack 1,r^{M}]$
given by 
\begin{equation*}
c_{N}(i):=(c((i-1)M+1),\ldots ,c((i-1)M+M)).
\end{equation*}%
Since $S(k,1,r,N)$ holds, there is an arithmetic progression\index{arithmetic progression} $%
b+d,b+2d,\ldots ,b+kd$ contained in $[1,N]$ that is monochromatic for the
coloring $c_{N}$. Next, since $S(k,m-1,r,M)$ holds, by considering $%
c|_{[(b-1)M,bM]}$, we see that there are $a,d_{0},\ldots ,d_{m-2}\in \lbrack
(b-1)M,bM]$ such that $a+k\sum_{j<m-1}d_{j}\in \lbrack (b-1)M,bM]$ and, for
any $g,h\in \left[ 0,k\right] ^{m-1}$ such that $g\equiv h$, the elements $%
a+\sum_{j<m}g_{j}d_{j}$ and $a+\sum_{j<m}h_{j}d_{j}$ have the same color
with respect to $c$.

Set $d_{m-1}:=dM$. We claim that $a,d_{0},\ldots ,d_{m-1}$ are as desired.
First note that $a+k\sum_{j<m}d_{j}\leq bM+kdM\leq NM\leq \nu $. Next
suppose that $g,h\in \lbrack 0,k]^{m}$ are such that $g\equiv h$. We wish to
show that $a+\sum_{j<m}g_{j}d_{j}$ and $a+\sum_{j<m}h_{j}d_{j}$ have the
same color. If the last occurrence of $k$ is $m-1$, then this is obvious.
Otherwise, we see that $g\mid m-1=h\mid m-1$, whence by assumption $%
a+\sum_{j<m-1}g_{j}d_{j}$ and $a+\sum_{j<m-1}h_{j}d_{j}$ have the same
color. Write $a+\sum_{j<m-1}g_{j}d_{j}=(b-1)M+p$ with $p\in \lbrack 1,M]$.
Then $a+\sum_{j<m}g_{j}d_{j}=(b-1)M+p+g_{M-1}dM=(b+g_{m-1}d-1)M+p$, which
has the same color as $(b-1)M+p$ by assumption. Likewise, $%
a+\sum_{j<m-1}h_{j}d_{j}=(b-1)M+q$ with $q\in \lbrack 1,M]$, whence $%
a+\sum_{j<m}h_{j}d_{j}=(b-1)M+q+h_{M-1}dM=(b+h_{m-1}d-1)M+q$, which has the
same color as $(b-1)M+q$ by assumption. Thus, $a+\sum_{j<m-1}g_{j}d_{j}$ and 
$a+\sum_{j<m-1}h_{j}d_{j}$ have the same color, proving the claim.

\bigskip

Since $S(k,1)$ fails, necessarily we have $k>1$. We will arrive at a
contradiction by showing that $S(k,1)$ in fact holds. Fix $r\in \mathbb{N}$, 
$\nu \in {}^{\ast }\mathbb{N}$ infinite, and an internal $r$-coloring $c$ of 
$\left[ 1,\nu \right] $. By minimality of $\left( k,1\right) $, we have that
there exist $a,d_{0},\ldots d_{r-1}\in \left[ 1,\nu \right] $ such that $%
a+r\sum_{j<r}d_{j}\in \left[ 1,\nu \right] $ and, for any $g,h\in \left[
1,k-1\right] ^{r}$ with $g\equiv h$, we have $a+\sum_{j<r}g_{j}d_{j}$ and $%
a+\sum_{j<r}h_{j}d_{j}$ have the same color. Observe that there are $r+1$ $r$%
-tuples that are obtained by concatenating a (possibly empty) $r$-tuple of $%
\left( k-1\right) $'s and a (possibly empty) $r$-tuple of $0$'s. Hence, by
the pigeonhole principle, there exist $1\leq s<t\leq r$ such that $a+\left(
k-1\right) \sum_{i<s}d_{i}$ and $a+\left( k-1\right) \sum_{i<t}d_{i}$ have
the same color. We also have that $a+\left( k-1\right) \sum_{i<s}d_{i}$ and $%
a+\left( k-1\right) \sum_{i<s}d_{i}+j\sum_{s\leq i<t}d_{i}$ have the same
color for every $j<k-1$. Therefore, setting $a^{\prime
}:=a+(k-1)\sum_{i<s}d_{i}$ and $d^{\prime }:=\sum_{s\leq i<t}d_{i}$, we have
that $a^{\prime }+jd^{\prime }$, for $j<k$, all have the same color. Since $%
\nu \in \starN\setminus \N$ and $c$ were arbitrary, this witnesses that $%
S(k,1)$ holds, yielding the desired contradiction.
\end{proof}

We will see in the next section that the Hales-Jewett theorem allows us to
immediately conclude a generalization of the van der Waerden theorem.

\section{The Hales-Jewett theorem}

Let $L$ be a \emph{finite} set (alphabet). We use the symbol $x$ to denote a 
\emph{variable }not in $L$. We let $W_{L}$ denote the set of finite strings
of elements of $L$ (called \emph{words} in $L$), and $W_{Lx}$ denote the set
of finite strings of elements of $L\cup \left\{ x\right\} $ with the
property that $x$ appears at least once (called \emph{variable words}). We
denote (variable) words by $v,w,z$ and letters by $a,b,c$. If $w$ is a
variable word\index{variable word} and $a$ is a letter, then we denote by $w\left[ a\right] $ the
word obtained from $w$ by replacing every occurrence of $x$ with $a$. For
convenience, we also set $w\left[ x\right] :=w$. The concatenation of two
(variable) words $v,w$ is denoted by $v%
\mathord{
    \mathchoice
    {\raisebox{1ex}{\scalebox{.7}{$\frown$}}}
    {\raisebox{1ex}{\scalebox{.7}{$\frown$}}}
    {\raisebox{.7ex}{\scalebox{.5}{$\frown$}}}
    {\raisebox{.7ex}{\scalebox{.5}{$\frown$}}}
  }w$.

\begin{definition}
Fix a sequence $\left( w_{n}\right) $ of variable words

\begin{enumerate}
\item The \emph{partial subsemigroup of $W_{L}$ generated by $\left(
w_{n}\right) $}, denoted $\left[ \left( w_{n}\right) \right] _{W_{L}}$, is
the set of all words $w_{n_{0}}\left[ a_{0}\right] 
\mathord{
    \mathchoice
    {\raisebox{1ex}{\scalebox{.7}{$\frown$}}}
    {\raisebox{1ex}{\scalebox{.7}{$\frown$}}}
    {\raisebox{.7ex}{\scalebox{.5}{$\frown$}}}
    {\raisebox{.7ex}{\scalebox{.5}{$\frown$}}}
  }\cdots 
\mathord{
    \mathchoice
    {\raisebox{1ex}{\scalebox{.7}{$\frown$}}}
    {\raisebox{1ex}{\scalebox{.7}{$\frown$}}}
    {\raisebox{.7ex}{\scalebox{.5}{$\frown$}}}
    {\raisebox{.7ex}{\scalebox{.5}{$\frown$}}}
  }w_{n_{k-1}}\left[ a_{k-1}\right]$, where $k\in \mathbb{N}$, $n_{0}<\cdots
<n_{k-1}$, and $\ a_{0},\ldots ,a_{k-1}\in L.$

\item The \emph{partial subsemigroup of $W_{Lx}$ generated by $\left(
w_{n}\right) $}, denoted $\left[ \left( w_{n}\right) \right] _{W_{Lx}}$, is
the set of all words $w_{n_{0}}\left[ \lambda_{0}\right] 
\mathord{
    \mathchoice
    {\raisebox{1ex}{\scalebox{.7}{$\frown$}}}
    {\raisebox{1ex}{\scalebox{.7}{$\frown$}}}
    {\raisebox{.7ex}{\scalebox{.5}{$\frown$}}}
    {\raisebox{.7ex}{\scalebox{.5}{$\frown$}}}
  }\cdots 
\mathord{
    \mathchoice
    {\raisebox{1ex}{\scalebox{.7}{$\frown$}}}
    {\raisebox{1ex}{\scalebox{.7}{$\frown$}}}
    {\raisebox{.7ex}{\scalebox{.5}{$\frown$}}}
    {\raisebox{.7ex}{\scalebox{.5}{$\frown$}}}
  }w_{n_{k-1}}\left[ \lambda_{k-1}\right]$, where $k\in \mathbb{N}$, $%
n_{0}<\cdots <n_{k-1}$, $\ \lambda_{0},\ldots ,\lambda_{k-1}\in L\cup\{x\}$,
and some $\lambda_i=x.$ %
\end{enumerate}
\end{definition}

\begin{theorem}[Infinite Hales-Jewett]
\label{Theorem:HJ}\index{Hales-Jewett's theorem}For every finite coloring\index{coloring} of $W_{L}\cup W_{Lx}$ there
exists an infinite sequence $\left( w_{n}\right) $ of variable words such
that $\left[ \left( w_{n}\right) \right] _{W_{L}}$ and $\left[ \left(
w_{n}\right) \right] _{W_{Lx}}$ are both monochromatic.
\end{theorem}

There is also a finitary version of the Hales-Jewett theorem. Suppose that $%
x_{1},\ldots ,x_{m}$ are variables. A variable word\index{variable word} $w$ in the variables $%
x_{1},\ldots ,x_{m}$ in the alphabet $L$ is a string of symbols in $L\cup
\left\{ x_{1},\ldots ,x_{m}\right\} $ such that, for every $1\leq i\leq m$, $%
x_{i}$ occurs in $w$, and for every $1\leq i<j\leq m$, the first occurrence of $x_{i}$ precedes the first occurrence of $x_{j}$. The word $w\left[ a_{1},\ldots
,a_{m}\right] $ obtained from $w$ by substituting the variable $x_{i}$ with
the letter $a_{i}$ for $i=1,2,\ldots ,m$ is defined in the obvious way.

\begin{corollary}[Finite Hales-Jewett]
\index{Hales-Jewett's theorem}For any finite alphabet $L$ and any $r,m\in \mathbb{N}$ there exists $n\in 
\mathbb{N}$ such that for any $r$-coloring\index{coloring} of the set $W_{L}\left( n\right) $
of $L$-words of length $n$ there exist a variable word\index{variable word} $w$ of length $n$ in
the alphabet $L$ and variables $x_{1},\ldots ,x_{m}$ such that the
\textquotedblleft combinatorial $m$-subspace\textquotedblright\ $\left\{ w%
\left[ a_{1},\ldots ,a_{m}\right] :a_{1},\ldots ,a_{n}\in L\right\} $ is
monochromatic.
\end{corollary}

A combinatorial $m$-subspace for $m=1$ is usually called a \emph{%
combinatorial line}.

\begin{proof}
We let $W_{Lx}(n)$ denote the
elements of $W_{Lx}$ of length $n$ and $W_{L}\left( n\right) $ denote the
elements of $W_{L}$ of length $n$. Suppose, towards a contradiction, that
there is $r\in \N$ such that, for each $n$, there is a \textquotedblleft
bad\textquotedblright\ $r$-coloring of $W_{L}(n)$ that admits no
monochromatic combinatorial line. By a compactness argument there is an $r$%
-coloring $c$ of $W_{L}$ such that the restriction of $c$ to $W_{L}\left(
n\right) $ is a bad $r$-coloring for every $n\in \mathbb{N}$. By the
Infinite Hales-Jewett Theorem\index{Hales-Jewett's theorem}, there is a sequence $(w_{i})$ for which $%
[(w_{i})]_{W_{L}}$ is monochromatic. For $i=1,2,\ldots ,m$, rename the
variable $x$ of $w_{i}$ by $x_{i}$, and consider the variable word\index{variable word} $w:=w_{1}%
\mathord{
    \mathchoice
    {\raisebox{1ex}{\scalebox{.7}{$\frown$}}}
    {\raisebox{1ex}{\scalebox{.7}{$\frown$}}}
    {\raisebox{.7ex}{\scalebox{.5}{$\frown$}}}
    {\raisebox{.7ex}{\scalebox{.5}{$\frown$}}}
  }w_{2}%
\mathord{
    \mathchoice
    {\raisebox{1ex}{\scalebox{.7}{$\frown$}}}
    {\raisebox{1ex}{\scalebox{.7}{$\frown$}}}
    {\raisebox{.7ex}{\scalebox{.5}{$\frown$}}}
    {\raisebox{.7ex}{\scalebox{.5}{$\frown$}}}
  }\cdots 
\mathord{
    \mathchoice
    {\raisebox{1ex}{\scalebox{.7}{$\frown$}}}
    {\raisebox{1ex}{\scalebox{.7}{$\frown$}}}
    {\raisebox{.7ex}{\scalebox{.5}{$\frown$}}}
    {\raisebox{.7ex}{\scalebox{.5}{$\frown$}}}
  }w_{m}$ in the variables $\left\{ x_{1},\ldots ,x_{m}\right\} $. If $n$ is
the length of $w$, then by the choice of $w_{1},\ldots ,w_{m}$ the
combinatorial subspace $\left\{ w\left[ a_{1},\ldots ,a_{m}\right]
:a_{1},\ldots ,a_{n}\in L\right\} $ is monochromatic. This contradicts the
fact that the restriction of $c$ to $W_{L}\left( n\right) $ is a bad $r$%
-coloring\index{coloring}.
\end{proof}

From the Hales-Jewett theorem\index{Hales-Jewett's theorem} one can deduce a multidimensional
generalization of van der Waerden's theorem\index{Gallai's theorem}, known as \emph{Gallai's theorem}%
.

\begin{theorem}[Gallai]
Fix $d\in \mathbb{N}$, a finite $F\subset \mathbb{N}^{d}$, and $r\in \mathbb{%
N}$. Then there exists $n\in \mathbb{N}$ such that, for any $r$-coloring\index{coloring} of $%
\left[ -n,n\right] ^{d}$, there exist $\boldsymbol{a}\in \mathbb{N}^{d}$ and 
$c\in \mathbb{N}$ such that the \emph{affine image} $\boldsymbol{a}%
+cF:=\left\{ a+c\boldsymbol{x}:\boldsymbol{x}\in F\right\} $ of $F$ is
monochromatic.
\end{theorem}

\begin{proof}
Consider the finite alphabet $L=F$. For $n\in \mathbb{N}$, consider the map $%
\Psi _{n}:W_{L}\left( n\right) \rightarrow \mathbb{N}^{d}$ defined by $%
\Psi_n(\left( \boldsymbol{a}_{1},\ldots ,\boldsymbol{a}_{n}\right))= 
\boldsymbol{a}_{1}+\cdots +\boldsymbol{a}_{n}$. Observe that $\Psi _{n}$
maps a combinatorial line to an affine image of $F$. Thus the conclusion
follows from the finitary Hales-Jewett theorem.
\end{proof}

In the rest of the section we present the proof of Theorem \ref{Theorem:HJ}.
Consider $W_{L}$ and $W_{L}\cup W_{Lx}$ as semigroups with respect to
concatenation. Thus their nonstandard extensions\index{nonstandard extension} ${}^{\ast }W_{L}$ and $%
{}^{\ast }W_{L}\cup {}^{\ast }W_{Lx}$ have canonical semigroup operations
with respect to the nonstandard extension\index{nonstandard extension} of the concatenation operation,
which we still denote by \textquotedblleft $%
\mathord{
    \mathchoice
    {\raisebox{1ex}{\scalebox{.7}{$\frown$}}}
    {\raisebox{1ex}{\scalebox{.7}{$\frown$}}}
    {\raisebox{.7ex}{\scalebox{.5}{$\frown$}}}
    {\raisebox{.7ex}{\scalebox{.5}{$\frown$}}}
  }$\textquotedblright . The elements of ${}^{\ast }W_{L}$ can be regarded
as hyperfinite strings of elements of ${}^{\ast }L$, and similarly for $%
{}^{\ast }W_{Lx}$. For every $a\in L\cup \left\{ x\right\} $ we also denote
by $\varpi \mapsto \varpi \left[ a\right] $ the nonstandard extension\index{nonstandard extension} of the
substitution operation $W_{Lx}\rightarrow W_{L}$, $w\mapsto w\left[ a\right] 
$.

\begin{lemma}
\label{Lemma:HJ}There exists a $u$-idempotent\index{$u$-idempotent} $\varpi {}$ in ${}^{\ast
}W_{Lx}$ and a $u$-idempotent\index{$u$-idempotent} $\upsilon \in {}^{\ast }W_{L}$ such that $%
\varpi 
\mathord{
    \mathchoice
    {\raisebox{1ex}{\scalebox{.7}{$\frown$}}}
    {\raisebox{1ex}{\scalebox{.7}{$\frown$}}}
    {\raisebox{.7ex}{\scalebox{.5}{$\frown$}}}
    {\raisebox{.7ex}{\scalebox{.5}{$\frown$}}}
  } {}^{\ast }\upsilon \sim \upsilon 
\mathord{
    \mathchoice
    {\raisebox{1ex}{\scalebox{.7}{$\frown$}}}
    {\raisebox{1ex}{\scalebox{.7}{$\frown$}}}
    {\raisebox{.7ex}{\scalebox{.5}{$\frown$}}}
    {\raisebox{.7ex}{\scalebox{.5}{$\frown$}}}
  } {}^{\ast }\varpi \sim \varpi $ and $\varpi \left[ a\right] \sim \upsilon 
$ for every $a\in L$.
\end{lemma}

\begin{proof}
Fix an enumeration $\left\{ a_{1},\ldots ,a_{m}\right\} $ of $L$. We define,
by recursion on $k=1,\ldots ,m$, $u$-idempotent\index{$u$-idempotent} elements $\varpi _{1},\ldots
,\varpi _{m}$ of ${}^{\ast }W_{Lx}$ and $\upsilon _{1},\ldots ,\upsilon _{m}$
of ${}^{\ast }W_{L}$ such that, for $1\leq i\leq j\leq m$,

\begin{enumerate}
\item $\varpi _{j}\left[ a_{i}\right] \sim \upsilon_j $, and

\item $\varpi _{j}\sim \varpi _{j} 
\mathord{
    \mathchoice
    {\raisebox{1ex}{\scalebox{.7}{$\frown$}}}
    {\raisebox{1ex}{\scalebox{.7}{$\frown$}}}
    {\raisebox{.7ex}{\scalebox{.5}{$\frown$}}}
    {\raisebox{.7ex}{\scalebox{.5}{$\frown$}}}
  } {}^{\ast }\upsilon _{i}\sim \upsilon _{i} 
\mathord{
    \mathchoice
    {\raisebox{1ex}{\scalebox{.7}{$\frown$}}}
    {\raisebox{1ex}{\scalebox{.7}{$\frown$}}}
    {\raisebox{.7ex}{\scalebox{.5}{$\frown$}}}
    {\raisebox{.7ex}{\scalebox{.5}{$\frown$}}}
  } {}^{\ast }\varpi _{j}$.
\end{enumerate}

\noindent Supposing this has been done, the conclusion of the lemma holds by
taking $\varpi:=\varpi_m$ and $\upsilon:=\upsilon_m$.

To begin, we let $\varpi _{0}$ be any nontrival $u$-idempotent\index{$u$-idempotent} element of $%
{}^{\ast }W_{Lx}$ and set $\upsilon _{1}:=\varpi _{0}\left[ a_{1}\right] $,
which we note is an idempotent element of ${}^\ast W_L$. Let $\rho _{1}$ be
an element of ${}^{\ast }W_{Lx}$ such that $\rho _{1} \sim \varpi _{0} 
\mathord{
    \mathchoice
    {\raisebox{1ex}{\scalebox{.7}{$\frown$}}}
    {\raisebox{1ex}{\scalebox{.7}{$\frown$}}}
    {\raisebox{.7ex}{\scalebox{.5}{$\frown$}}}
    {\raisebox{.7ex}{\scalebox{.5}{$\frown$}}}
  } {}^{\ast }\upsilon _{1}$. Observe that $\rho _{1}\left[ a_{1}\right]
\sim \upsilon _{1}$ and $\rho _{1} 
\mathord{
    \mathchoice
    {\raisebox{1ex}{\scalebox{.7}{$\frown$}}}
    {\raisebox{1ex}{\scalebox{.7}{$\frown$}}}
    {\raisebox{.7ex}{\scalebox{.5}{$\frown$}}}
    {\raisebox{.7ex}{\scalebox{.5}{$\frown$}}}
  } {}^{\ast }\upsilon _{1}\sim \rho _{1}$. Thus, the compact $u$-semigroup\index{$u$-semigroup} 
\begin{equation*}
\left\{ z\in {}^{\ast }W_{Lx}:z\left[ a_{1}\right] \sim \upsilon _{1} \text{
and }z \mathord{ \mathchoice {\raisebox{1ex}{\scalebox{.7}{$\frown$}}}
{\raisebox{1ex}{\scalebox{.7}{$\frown$}}}
{\raisebox{.7ex}{\scalebox{.5}{$\frown$}}}
{\raisebox{.7ex}{\scalebox{.5}{$\frown$}}} } {}^{\ast }\upsilon _{1}\sim
z\right\}
\end{equation*}%
is nonempty, whence it contains a $u$-idempotent\index{$u$-idempotent} $\beta _{1}$. We now fix $%
\varpi _{1}\in {}^{\ast}W_{Lx}$ such that $\varpi_1\sim \upsilon _{1} 
\mathord{
    \mathchoice
    {\raisebox{1ex}{\scalebox{.7}{$\frown$}}}
    {\raisebox{1ex}{\scalebox{.7}{$\frown$}}}
    {\raisebox{.7ex}{\scalebox{.5}{$\frown$}}}
    {\raisebox{.7ex}{\scalebox{.5}{$\frown$}}}
  } {}^{\ast }\beta _{1}$. It follows now that $\varpi_1$ is $u$-idempotent\index{$u$-idempotent}
and $\varpi_1$ and $\upsilon_1$ satisfy (1) and (2) above.

Suppose that $\varpi _{i},\upsilon _{i}$ have been defined for $1\leq i\leq
k<m$ satisfying (1) and (2) above. Set $\upsilon _{k+1}:=\varpi _{k}\left[
a_{k+1}\right] $. Observe that $\upsilon _{k+1}\sim \upsilon _{k+1} 
\mathord{
    \mathchoice
    {\raisebox{1ex}{\scalebox{.7}{$\frown$}}}
    {\raisebox{1ex}{\scalebox{.7}{$\frown$}}}
    {\raisebox{.7ex}{\scalebox{.5}{$\frown$}}}
    {\raisebox{.7ex}{\scalebox{.5}{$\frown$}}}
  } {}^{\ast }\upsilon _{i}\sim \upsilon _{i} 
\mathord{
    \mathchoice
    {\raisebox{1ex}{\scalebox{.7}{$\frown$}}}
    {\raisebox{1ex}{\scalebox{.7}{$\frown$}}}
    {\raisebox{.7ex}{\scalebox{.5}{$\frown$}}}
    {\raisebox{.7ex}{\scalebox{.5}{$\frown$}}}
  } {}^{\ast }\upsilon _{k+1}$ for $1\leq i\leq k+1$. Let $\rho_{k+1}$ be an
element of ${}^{\ast}W_{Lx}$ such that $\rho _{k+1}\sim \varpi _{k} 
\mathord{
    \mathchoice
    {\raisebox{1ex}{\scalebox{.7}{$\frown$}}}
    {\raisebox{1ex}{\scalebox{.7}{$\frown$}}}
    {\raisebox{.7ex}{\scalebox{.5}{$\frown$}}}
    {\raisebox{.7ex}{\scalebox{.5}{$\frown$}}}
  } {}^{\ast }\upsilon _{k+1}$. Observe that $\upsilon _{i} 
\mathord{
    \mathchoice
    {\raisebox{1ex}{\scalebox{.7}{$\frown$}}}
    {\raisebox{1ex}{\scalebox{.7}{$\frown$}}}
    {\raisebox{.7ex}{\scalebox{.5}{$\frown$}}}
    {\raisebox{.7ex}{\scalebox{.5}{$\frown$}}}
  } {}^{\ast }\rho _{k+1}\sim \rho _{k+1} 
\mathord{
    \mathchoice
    {\raisebox{1ex}{\scalebox{.7}{$\frown$}}}
    {\raisebox{1ex}{\scalebox{.7}{$\frown$}}}
    {\raisebox{.7ex}{\scalebox{.5}{$\frown$}}}
    {\raisebox{.7ex}{\scalebox{.5}{$\frown$}}}
  } {}^{\ast }\upsilon _{i}\sim \rho _{k+1}$ and $\rho _{k+1}\left[ a_{i}%
\right] \sim \upsilon _{k+1}$ for $1\leq i\leq k+1$. Thus, the compact $u$%
-semigroup%
\begin{equation*}
\left\{ z\in {}^{\ast }W_{Lx}:z\left[ a_{i}\right] \sim \upsilon _{k+1} 
\text{ and }z{} \mathord{ \mathchoice
{\raisebox{1ex}{\scalebox{.7}{$\frown$}}}
{\raisebox{1ex}{\scalebox{.7}{$\frown$}}}
{\raisebox{.7ex}{\scalebox{.5}{$\frown$}}}
{\raisebox{.7ex}{\scalebox{.5}{$\frown$}}} } {}^{\ast }\upsilon _{i}\sim z 
\text{ for }1\leq i\leq k+1\right\}
\end{equation*}%
is nonempty, whence it contains a $u$-idempotent\index{$u$-idempotent} element $\beta _{k+1} $.
Finally, fix $\varpi_{k+1}$ in ${}^{\ast}W_{Lx}$ such that $\varpi
_{k+1}\sim \upsilon _{k+1} 
\mathord{
    \mathchoice
    {\raisebox{1ex}{\scalebox{.7}{$\frown$}}}
    {\raisebox{1ex}{\scalebox{.7}{$\frown$}}}
    {\raisebox{.7ex}{\scalebox{.5}{$\frown$}}}
    {\raisebox{.7ex}{\scalebox{.5}{$\frown$}}}
  } {}^{\ast }\beta _{k+1}$. It follows that $\varpi_{k+1}$ is $u$%
-idempotent and (1) and (2) continue to hold for $\varpi_i$ and $\upsilon_i$
for $1\leq i\leq k+1$. This completes the recursive construction and the
proof of the lemma.
\end{proof}

In the statement of the following proposition, we assume that $\varpi$ and $%
\upsilon$ are as in the conclusion of Lemma \ref{Lemma:HJ}.

\begin{proposition}
\label{Proposition:HJ}Suppose that $A\subset W_{L}$ and $B\subset W_{Lx}$
are such that $\upsilon \in {}^{\ast }A$ and $\varpi \in {}^{\ast }B$. Then
there exists an infinite sequence $\left( w_{n}\right) $ in $W_{Lx}$ such
that $\left[ \left( w_{n}\right) \right] _{W_{L}}$ is contained in $A$ and $%
\left[ \left( w_{n}\right) \right] _{W_{Lx}}$ is contained in $B$.
\end{proposition}

\begin{proof}
Set $C:=A\cup B$. Observe that $\varpi $ satisfies, for every $a,b\in L\cup
\left\{ x\right\} $, 
\begin{eqnarray*}
\varpi \left[ a\right] &\in &{}^{\ast }C \\
\varpi \left[ a\right] \mathord{ \mathchoice
{\raisebox{1ex}{\scalebox{.7}{$\frown$}}}
{\raisebox{1ex}{\scalebox{.7}{$\frown$}}}
{\raisebox{.7ex}{\scalebox{.5}{$\frown$}}}
{\raisebox{.7ex}{\scalebox{.5}{$\frown$}}} } {}{}^{\ast }\varpi \left[ b%
\right] &\in &{}^{\ast \ast }C\text{.}
\end{eqnarray*}%
Therefore, by transfer, there exists $w_{0}\in W_{Lx}$ that satisfies, for
every $a_{0},a_{1}\in L\cup \left\{ x\right\} $, 
\begin{eqnarray*}
w_{0}\left[ a_{0}\right] &\in &C \\
w_{0}\left[ a_{0}\right] \mathord{ \mathchoice
{\raisebox{1ex}{\scalebox{.7}{$\frown$}}}
{\raisebox{1ex}{\scalebox{.7}{$\frown$}}}
{\raisebox{.7ex}{\scalebox{.5}{$\frown$}}}
{\raisebox{.7ex}{\scalebox{.5}{$\frown$}}} }{}\varpi \left[ a_{1}\right]
&\in &{}^{\ast }C\text{.}
\end{eqnarray*}%
From this we also have, for every $a_{0},a_{1},b\in L\cup \left\{ x\right\} $%
, that,%
\begin{equation*}
w_{0}\left[ a_{0}\right] \mathord{ \mathchoice
{\raisebox{1ex}{\scalebox{.7}{$\frown$}}}
{\raisebox{1ex}{\scalebox{.7}{$\frown$}}}
{\raisebox{.7ex}{\scalebox{.5}{$\frown$}}}
{\raisebox{.7ex}{\scalebox{.5}{$\frown$}}} }\varpi \left[ a_{1}\right] %
\mathord{ \mathchoice {\raisebox{1ex}{\scalebox{.7}{$\frown$}}}
{\raisebox{1ex}{\scalebox{.7}{$\frown$}}}
{\raisebox{.7ex}{\scalebox{.5}{$\frown$}}}
{\raisebox{.7ex}{\scalebox{.5}{$\frown$}}} } {}{}^{\ast }\varpi \left[ b%
\right] \in {}^{\ast \ast }C\text{.}
\end{equation*}%
Therefore, by transfer, there exists $w_{1}\in W_{Lx}$ that satisfies, for
every $a_{0},a_{1},a_{2}\in L\cup \left\{ x\right\} $:%
\begin{eqnarray*}
w_{0}\left[ a_{0}\right] &\in &C \\
w_{1}\left[ a_{1}\right] &\in &C \\
w_{0}\left[ a_{0}\right] \mathord{ \mathchoice
{\raisebox{1ex}{\scalebox{.7}{$\frown$}}}
{\raisebox{1ex}{\scalebox{.7}{$\frown$}}}
{\raisebox{.7ex}{\scalebox{.5}{$\frown$}}}
{\raisebox{.7ex}{\scalebox{.5}{$\frown$}}} }w_{1}\left[ a_{1}\right] &\in &C
\\
w_{0}\left[ a_{0}\right] \mathord{ \mathchoice
{\raisebox{1ex}{\scalebox{.7}{$\frown$}}}
{\raisebox{1ex}{\scalebox{.7}{$\frown$}}}
{\raisebox{.7ex}{\scalebox{.5}{$\frown$}}}
{\raisebox{.7ex}{\scalebox{.5}{$\frown$}}} }\varpi \left[ a_{2}\right] &\in
&{}^{\ast }C \\
w_{1}\left[ a_{1}\right] \mathord{ \mathchoice
{\raisebox{1ex}{\scalebox{.7}{$\frown$}}}
{\raisebox{1ex}{\scalebox{.7}{$\frown$}}}
{\raisebox{.7ex}{\scalebox{.5}{$\frown$}}}
{\raisebox{.7ex}{\scalebox{.5}{$\frown$}}} }\varpi \left[ a_{2}\right] &\in
&{}^{\ast }C \\
w_{0}\left[ a_{0}\right] \mathord{ \mathchoice
{\raisebox{1ex}{\scalebox{.7}{$\frown$}}}
{\raisebox{1ex}{\scalebox{.7}{$\frown$}}}
{\raisebox{.7ex}{\scalebox{.5}{$\frown$}}}
{\raisebox{.7ex}{\scalebox{.5}{$\frown$}}} }w_{1}\left[ a_{1}\right] %
\mathord{ \mathchoice {\raisebox{1ex}{\scalebox{.7}{$\frown$}}}
{\raisebox{1ex}{\scalebox{.7}{$\frown$}}}
{\raisebox{.7ex}{\scalebox{.5}{$\frown$}}}
{\raisebox{.7ex}{\scalebox{.5}{$\frown$}}} }\varpi \left[ a_{2}\right] &\in
&{}^{\ast }C\text{.}
\end{eqnarray*}%
Proceeding recursively, one can assume that at the $n$-th step elements $%
w_{0},\ldots ,w_{n-1}$ of $W_{Lx}$ have been defined such that, for every $%
n_{1}<\cdots <n_{k}<n$ and $a_{0},\ldots ,a_{n-1},a\in L\cup \left\{
x\right\} $, one has that%
\begin{eqnarray*}
w_{n_{1}}\left[ a_{n_{1}}\right] \mathord{ \mathchoice
{\raisebox{1ex}{\scalebox{.7}{$\frown$}}}
{\raisebox{1ex}{\scalebox{.7}{$\frown$}}}
{\raisebox{.7ex}{\scalebox{.5}{$\frown$}}}
{\raisebox{.7ex}{\scalebox{.5}{$\frown$}}} }\cdots \mathord{ \mathchoice
{\raisebox{1ex}{\scalebox{.7}{$\frown$}}}
{\raisebox{1ex}{\scalebox{.7}{$\frown$}}}
{\raisebox{.7ex}{\scalebox{.5}{$\frown$}}}
{\raisebox{.7ex}{\scalebox{.5}{$\frown$}}} }w_{n_{k}}\left[ a_{n_{k}}\right]
&\in &C \\
w_{n_{1}}\left[ a_{n_{1}}\right] \mathord{ \mathchoice
{\raisebox{1ex}{\scalebox{.7}{$\frown$}}}
{\raisebox{1ex}{\scalebox{.7}{$\frown$}}}
{\raisebox{.7ex}{\scalebox{.5}{$\frown$}}}
{\raisebox{.7ex}{\scalebox{.5}{$\frown$}}} }\cdots \mathord{ \mathchoice
{\raisebox{1ex}{\scalebox{.7}{$\frown$}}}
{\raisebox{1ex}{\scalebox{.7}{$\frown$}}}
{\raisebox{.7ex}{\scalebox{.5}{$\frown$}}}
{\raisebox{.7ex}{\scalebox{.5}{$\frown$}}} }w_{n_{k}}\left[ a_{n_{k}}\right] %
\mathord{ \mathchoice {\raisebox{1ex}{\scalebox{.7}{$\frown$}}}
{\raisebox{1ex}{\scalebox{.7}{$\frown$}}}
{\raisebox{.7ex}{\scalebox{.5}{$\frown$}}}
{\raisebox{.7ex}{\scalebox{.5}{$\frown$}}} }\varpi \left[ a\right] &\in
&{}^{\ast }C\text{.}
\end{eqnarray*}%
From this one deduces also that for every $a,b\in L\cup \left\{ x\right\} $
one has that%
\begin{equation*}
w_{n_{1}}\left[ a_{n_{1}}\right] \mathord{ \mathchoice
{\raisebox{1ex}{\scalebox{.7}{$\frown$}}}
{\raisebox{1ex}{\scalebox{.7}{$\frown$}}}
{\raisebox{.7ex}{\scalebox{.5}{$\frown$}}}
{\raisebox{.7ex}{\scalebox{.5}{$\frown$}}} }\cdots \mathord{ \mathchoice
{\raisebox{1ex}{\scalebox{.7}{$\frown$}}}
{\raisebox{1ex}{\scalebox{.7}{$\frown$}}}
{\raisebox{.7ex}{\scalebox{.5}{$\frown$}}}
{\raisebox{.7ex}{\scalebox{.5}{$\frown$}}} }w_{n_{k}}\left[ a_{n_{k}}\right] %
\mathord{ \mathchoice {\raisebox{1ex}{\scalebox{.7}{$\frown$}}}
{\raisebox{1ex}{\scalebox{.7}{$\frown$}}}
{\raisebox{.7ex}{\scalebox{.5}{$\frown$}}}
{\raisebox{.7ex}{\scalebox{.5}{$\frown$}}} }\varpi \left[ a\right] \mathord{
\mathchoice {\raisebox{1ex}{\scalebox{.7}{$\frown$}}}
{\raisebox{1ex}{\scalebox{.7}{$\frown$}}}
{\raisebox{.7ex}{\scalebox{.5}{$\frown$}}}
{\raisebox{.7ex}{\scalebox{.5}{$\frown$}}} } {}{}^{\ast }\varpi \left[ b%
\right] \in {}^{\ast \ast }C\text{.}
\end{equation*}%
Hence, by transfer one obtains $w_{n}\in W_{Lx}$ such that for every $%
n_{1}<\cdots <n_{k}\leq n$ and $a_{0},\ldots ,a_{n},a\in L\cup \left\{
x\right\} $, one has that%
\begin{eqnarray*}
w_{n_{1}}\left[ a_{n_{1}}\right] &\mathord{ \mathchoice
{\raisebox{1ex}{\scalebox{.7}{$\frown$}}}
{\raisebox{1ex}{\scalebox{.7}{$\frown$}}}
{\raisebox{.7ex}{\scalebox{.5}{$\frown$}}}
{\raisebox{.7ex}{\scalebox{.5}{$\frown$}}} }&\cdots \mathord{ \mathchoice
{\raisebox{1ex}{\scalebox{.7}{$\frown$}}}
{\raisebox{1ex}{\scalebox{.7}{$\frown$}}}
{\raisebox{.7ex}{\scalebox{.5}{$\frown$}}}
{\raisebox{.7ex}{\scalebox{.5}{$\frown$}}} }w_{n_{k}}\left[ a_{n_{k}}\right]
\in C \\
w_{n_{1}}\left[ a_{n_{1}}\right] &\mathord{ \mathchoice
{\raisebox{1ex}{\scalebox{.7}{$\frown$}}}
{\raisebox{1ex}{\scalebox{.7}{$\frown$}}}
{\raisebox{.7ex}{\scalebox{.5}{$\frown$}}}
{\raisebox{.7ex}{\scalebox{.5}{$\frown$}}} }&\cdots \mathord{ \mathchoice
{\raisebox{1ex}{\scalebox{.7}{$\frown$}}}
{\raisebox{1ex}{\scalebox{.7}{$\frown$}}}
{\raisebox{.7ex}{\scalebox{.5}{$\frown$}}}
{\raisebox{.7ex}{\scalebox{.5}{$\frown$}}} }w_{n_{k}}\left[ a_{n_{k}}\right] %
\mathord{ \mathchoice {\raisebox{1ex}{\scalebox{.7}{$\frown$}}}
{\raisebox{1ex}{\scalebox{.7}{$\frown$}}}
{\raisebox{.7ex}{\scalebox{.5}{$\frown$}}}
{\raisebox{.7ex}{\scalebox{.5}{$\frown$}}} }\varpi \left[ a\right] \in
{}^{\ast }C\text{.}
\end{eqnarray*}%
This concludes the recursive construction.
\end{proof}

Theorem \ref{Theorem:HJ} now follows immediately from Proposition \ref%
{Proposition:HJ}. Indeed, if $\left\{ A_{1},\ldots ,A_{r}\right\} $ is a
finite coloring\index{coloring} of $W_{L}\cup W_{Lx}$, then there exist $1\leq i,j\leq r$
such that $\upsilon \in {}^{\ast }A_{i}$ and $\varpi \in {}^{\ast }A_{j}$.

\section*{Notes and references} Van der Waerden's theorem \cite{van_der_waerden_beweis_1927} is chronologically one of the first results in Ramsey theory, only preceded by Schur's lemma on Schur triples \cite{schur_kongruenz_1916}. Both van der Waerden's theorem and Schur's lemma were motivated by problems in modular arithmetic; see also \cite[Chapter 2]{promel_ramsey_2013}. 

The Hales-Jewett theorem \cite{hales_regularity_1963} is an abstract Ramsey-theoretic result motivated by the mathematical study of positional games such as ``Tick-Tack-Toe'' or ``Go Moku''. The original proof of Hales and Jewett from \cite{hales_regularity_1963} was finitary and purely combinatorial. An infinitary proof was given by Bergelson, Blass, and Hindman in \cite{bergelson_partition_1994}; see also \cite[Chapter 2]{todorcevic_introduction_2010}. 
Combinatorial lines and combinatorial subspaces are also the object of the Graham--Rothschild theorem \cite{graham_ramseys_1971}. This was motivated by a conjecture of Rota on a geometric analogue of Ramsey's theorem. The conjecture was eventually established using similar methods by Graham, Leeb, and Rothschild \cite{graham_ramseys_1972}.

\chapter{From Hindman to Gowers}

\label{chapter_hindman_gowers}
\section{Hindman's theorem}

Hindman's theorem is another fundamental pigeonhole principle, which
considers the combinatorial configurations provided by sets of finite sums
of infinite sequences.

\begin{definition}

\begin{enumerate}
\item Given $F\subseteq \N$ finite and $(c_n)$ a sequence of distinct
elements from $A$, define $c_F:=\sum_{n\in F} c_n$, with the convention that 
$c_\emptyset=0$.

\item Given a (finite or infinite) sequence $(c_{n})$ of distinct elements
from $\N$, set $\FS((c_{n})):=\{c_{F}\ :\ F\subseteq \N\text{ finite,
nonempty}\}$.

\item We say that $A\subseteq \N$ is an \emph{FS-set} if there is an
infinite sequence $(c_{n})$ of distinct elements from $\N$ such that $\FS%
((c_{n}))\subseteq A$.
\end{enumerate}
\end{definition}

%

We first note, using the notation from Section \ref{Section:vdW}, that $%
S(m,2)$ implies the following theorem:

\begin{theorem}[Folkman's theorem]\index{Folkman's theorem}
For any $m,r\in \N$, there is $n\in \N$ such that, for any $r$-coloring of $%
[1,n]$, there are $d_0,\ldots,d_{m-1}\in [1,n]$ such that $\FS(d_n)$ is
monochromatic. 
\end{theorem}

In particular, for any finite coloring\index{coloring} of $\N$, there are arbitrarily large
finite sequences $\left( c_{1},\ldots ,c_{n}\right) $ in $\N$ such
that $\FS(c_{1},\ldots ,c_{n})$ is monochromatic. The main result of this
chapter, due to Hindman, allows us to find an \emph{infinite} sequence $%
\left( c_{n}\right) $ in $\N$ such that $\FS((c_{n}))$ is monochromatic.
Just as the infinite Ramsey theorem cannot just be deduced from its finite
form, Hindman's theorem\index{Hindman's theorem} cannot simply be deduced from Folkman's theorem\index{Folkman's theorem}.

\begin{theorem}
Suppose that $\alpha \in \starN$ is $u$-idempotent\index{$u$-idempotent}. Then for every $%
A\subseteq \N$, if $\alpha\in \starA$, then $A$ is an $\FS$-set. 
\end{theorem}

\begin{proof}
We define by recursion $x_{0}<x_{1}<\cdots <x_{n}$ such that $x_{F}\in A$
and $x_{F}+\alpha \in {}^{\ast }A$ for any $F\subseteq \left\{ 0,1,\ldots
,n\right\} $. Note that, since $\alpha $ is idempotent, we also have that $%
x_{F}+\alpha +{}^{\ast }\alpha \in {}^{\ast \ast }A$. Suppose that these
have been defined up to $n$. The statement \textquotedblleft there exists $%
w\in \starN$ such that $w>x_{n}$ and, for every subset $F$ of $\left\{
0,1,2,\ldots ,n\right\} $, $x_{F}+w\in {}^{\ast }A$ and $x_{F}+w+\staralpha%
\in {}^{\ast \ast }A$\textquotedblright\ holds. So, as witnessed by $w=\alpha$, 
by transfer there exists 
$x_{n+1}\in \N$ larger than $x_{n}$ such that $x_{F}+x_{n+1}\in A$ and $%
x_{F}+x_{n+1}+\alpha \in {}^{\ast }A$ for any $F\subseteq \left\{ 0,1,\ldots
,n\right\} $. This concludes the recursive construction.
\end{proof}

\begin{corollary}[Hindman]
For any finite coloring\index{coloring} of $\mathbb{N}$, there is a color that is an $\FS$%
-set. 
\end{corollary}

\begin{proof}
Let $\N:=C_1\sqcup\cdots \sqcup C_r$ be a finite coloring of $\N$. Let $%
\alpha$ be a $u$-idempotent\index{$u$-idempotent} element of $\starN$ and let $i$ be such that $%
\alpha\in {}^{\ast}C_i$. The result now follows from the previous theorem.
\end{proof}

\begin{lemma}
Suppose that $(c_n)$ is a sequence of distinct elements from $\N$. Then
there is an idempotent $\alpha\in \starN$ such that $\alpha\in{}^{\ast}\FS%
((c_n))$.
\end{lemma}

\begin{proof}
For each $m$, let $U_{m}$ be the closed subset ${}^{\ast }\FS((c_{n})_{n\geq
m})$ of ${}^{\ast }\mathbb{N}$. By compactness, we have that $%
S:=\bigcap_{m}U_{m}$ is a nonempty closed subset of $\starN$. We claim that $%
S$ is a $u$-subsemigroup of ${}^{\ast }\mathbb{N}$. Indeed, suppose that $%
\alpha ,\beta \in S$ and let $\gamma \in {}^{\ast }\mathbb{N}$ such that $%
\alpha +{}^{\ast }\beta \sim \gamma $. We claim that $\gamma \in S$. Fix $%
m\in \mathbb{N}$. We must show that ${}\gamma \in {}^{\ast }\FS%
((c_{n})_{n\geq m})$ or, equivalently, $\alpha +{}^{\ast }\beta \in {}^{\ast
\ast }\FS((c_{n})_{n\geq m})$. Write $\alpha =c_{F}$ for some hyperfinite $%
F\subseteq \{n\in \starN\ :\ n\geq m\}$. By transferring the fact that $%
\beta \in \bigcap_{m}S_{m}$, there is hyperfinite $G\subseteq \{n\in
{}^{\ast \ast }\N\ :\ n>\max (F)\}$ such that ${}^{\ast }\beta =c_{G}$, and so
$\alpha +{}^{\ast }\beta =c_{F}+c_{G}\in {}^{\ast \ast }\FS%
((c_{n})_{n\geq m})$. 

It follows that $S$ is a nonempty closed $u$-subsemigroup of $\starN$,
whence, by Corollary \ref{existenceofuidempotents}, there is an idempotent $%
\alpha \in S$, which, in particular, implies that $\alpha \in {}^{\ast }\FS%
((c_{n}))$.
\end{proof}

\begin{corollary}[Strong Hindman's Theorem]\label{stronghindman}
Suppose that $C$ is an FS-set and $C$ is partitioned into finitely many
pieces $C_1,\ldots,C_n$. Then some $C_i$ is an FS-set.
\end{corollary}

\begin{proof}
Take $(c_n)$ such that $\FS((c_n))\subseteq C$. Take $\alpha\in \starN$ $u$%
-idempotent such that $\alpha\in {}^{\ast}\FS((c_n))$. Then $\alpha\in
{}^{\ast}C$ as well, whence $\alpha\in{}^{\ast}C_i$ for a unique $%
i=1,\ldots,n$, and this $C_i$ is itself thus an FS-set.
\end{proof}

\begin{exercise}
Let $\operatorname{Idem}:=\{\alpha\in \starN \ : \ \alpha \text{ is $u$-idempotent}\}$.  Prove that $\alpha\in \overline{\operatorname{Idem}}$ if and only if:  for every $A\subseteq \N$, if $\alpha\in \starA$, then $A$ is an FS-set.  Here, $\overline{\operatorname{Idem}}$ denotes the closure of $\operatorname{Idem}$ in the $u$-topology.
\end{exercise}

%

\section{The Milliken-Taylor theorem}

We denote by $\mathbb{N}^{\left[ m\right] }$ the set of subsets of $\mathbb{N%
}$ of size $m$. We identify $\mathbb{N}^{\left[ m\right] }$ with the set of
ordered $m$-tuples of elements of $\mathbb{N}$ increasingly ordered. If $F,G$
are finite subsets of $\mathbb{N}$, we write $F<G$ if either one of them is
empty, or they are both nonempty and the maximum of $F$ is smaller than the
minimum of $G$. Recall that for $F\subseteq \N$ finite, we use the notation $%
x_{F}$ for $\sum_{i\in F}x_{i}$, where we declare $x_{F}=0$ when $F$ is
empty.

The goal of this section is to prove the following:

\begin{theorem}[Milliken-Taylor]
\label{MT} For any $m\in \mathbb{N}$ and finite coloring\index{coloring} of $\mathbb{N}^{%
\left[ m\right] }$, there exists an increasing sequence $\left( x_{n}\right) 
$ in $\mathbb{N}$ such that the set of elements of the form $\left\{
x_{F_{1}},\ldots ,x_{F_{m}}\right\}$ for finite nonempty subsets $%
F_{1}<\cdots <F_{m}$ of $\mathbb{N}$ is monochromatic.
\end{theorem}

We note that the Milliken-Taylor theorem\index{Milliken-Taylor's Theorem} is a simultaneous generalization of
Ramsey's theorem (by taking the finite sets $F_1,\ldots,F_m$ to have
cardinality one) and Hindman's theorem\index{Hindman's theorem} (by taking $m=1$).

The heart of the nonstandard approach is the following:

\begin{proposition}\label{prop-mt}
\label{MTNS} Suppose that $m\in \mathbb{N}$ and $\alpha \in {}^{\ast }%
\mathbb{N}$ is $u$-idempotent\index{$u$-idempotent}. If $A\subset \mathbb{N}^{\left[ m\right] }$
is such that ${}\left\{ \alpha ,{}^{\ast }\alpha ,\ldots ,{}^{\ast
(m-1)}\alpha \right\} \in {}^{\ast m}A$, then there exists an increasing
sequence $\left( x_{n}\right) $ in $\mathbb{N}$ such that $\left\{
x_{F_{1}},x_{F_{2}},\ldots ,x_{F_{m}}\right\} \in A$ for any finite nonempty
subsets $F_{1}<\cdots <F_{m}$ of $\mathbb{N}$.
\end{proposition}

\begin{proof}
We define by recursion an increasing sequence $(x_n)$ such that%
\begin{equation*}
\left\{x_{F_1},x_{F_2},\ldots,x_{F_j},\alpha,{}^{\ast}\alpha,\ldots,{}^{\ast
(m-j-1)}\alpha\right\}\in {}^{\ast (m-j)}A
\end{equation*}
and 
\begin{equation*}
\left\{ x_{F_{1}},x_{F_{2}},\ldots ,x_{F_{j-1}},x_{F_{j}}+{}\alpha ,{}^{\ast
}\alpha ,{}^{\ast \ast }\alpha ,\ldots ,{}^{\ast (m-j)}\alpha \right\} \in
{}^{\ast (m-j+1)}A
\end{equation*}
for every $1\leq j\leq m$ and finite $F_{1}<\cdots<F_{j}$ such that $%
F_{1},\ldots ,F_{j-1}$ are nonempty. It is clear that the sequence $(x_n)$
satisfies the conclusion of the proposition.

Suppose that we have constructed $x_1<\cdots<x_{n-1}$ satisfying the
recursive construction (where of course now $F_1,\ldots,F_j$ are subsets of $%
\{1,\ldots,n-1\}$). Since $\alpha $ is $u$-idempotent\index{$u$-idempotent}, we also have, for any 
$1\leq j\leq m$ and $F_{1},\ldots ,F_{j}$ as above, that%
\begin{equation*}
\left\{ x_{F_{1}},x_{F_{2}},\ldots ,x_{F_{j-1}},x_{F_{j}}+{}\alpha +{}^{\ast
}\alpha ,{}{}^{\ast \ast }\alpha ,\ldots ,{}^{\ast (m-j+1)}\alpha \right\}
\in {}^{\ast (m-j+2)}A\text{.}
\end{equation*}%
Therefore, by transfer there exists $x_{n}>x_{n-1}$ such that%
\begin{equation*}
\left\{ x_{F_{1}},x_{F_{2}},\ldots ,x_{F_{j-1}},x_{F_{j}}+{}x_{n},{}\alpha
,{}^{\ast }\alpha ,\ldots ,{}^{\ast (m-j-1)}\alpha \right\} \in {}^{\ast
(m-j)}A
\end{equation*}%
and%
\begin{equation*}
\left\{ x_{F_{1}},x_{F_{2}},\ldots ,x_{F_{j-1}},x_{F_{j}}+{}x_{n}+{}\alpha
,{}{}^{\ast }\alpha ,\ldots ,{}^{\ast (m-j)}\alpha \right\} \in {}^{\ast
(m-j+1)}A
\end{equation*}%
for any $1\leq j\leq m$ and $F_{1}<\cdots<F_{j}$ contained in $\left\{
1,2,\ldots ,n-1\right\} $ such that $F_{1},\ldots ,F_{j-1}$ are nonempty.
This concludes the recursive construction and the proof of the proposition.
\end{proof}

Theorem \ref{MT} follows immediately from Proposition \ref{MTNS}. Indeed,
suppose $\N^{[m]}=A_1\sqcup \cdots \sqcup A_r$ is a partition of $\N^{[m]}$.
Fix $\alpha\in \starN$ a $u$-idempotent\index{$u$-idempotent}. Let $i\in \{1,\ldots,r\}$ be such
that $\left\{ \alpha ,{}^{\ast }\alpha ,\ldots ,{}^{\ast m-1}\alpha \right\}
\in {}^{\ast m}A_i$. Then $A_i$ is the desired color.

Observe now that if $\lambda \in \mathbb{N}$ and $\alpha \sim \alpha +{}^{\ast
}\alpha $, then $\lambda \alpha \sim \lambda \alpha +\lambda {}^{\ast }\alpha $. Hence the same
proofs as above shows the following slight strengthening of Proposition \ref{prop-mt},
and hence of the Milliken-Taylor theorem\index{Milliken-Taylor's Theorem}.

\begin{proposition}
Suppose that $m\in \mathbb{N}$, $\lambda _{1},\ldots ,\lambda _{m}\in \mathbb{N}$, and $%
\alpha \in {}^{\ast }\mathbb{N}$ is $u$-idempotent\index{$u$-idempotent}. If $A\subset \mathbb{N}^{%
\left[ m\right] }$ is such that ${}\left\{ \alpha ,{}^{\ast }\alpha ,\ldots
,{}^{\ast (m-1)}\alpha \right\} \in {}^{\ast m}A$, then there exists an
increasing sequence $\left( x_{n}\right) $ in $\mathbb{N}$ such that $%
\left\{ \lambda _{1}x_{F_{1}},\ldots ,\lambda _{m}x_{F_{m}}\right\} \in A$ for any finite
nonempty subsets $F_{1}<\cdots <F_{m}$ of $\mathbb{N}$.
\end{proposition}

\begin{theorem}
For any $m\in \mathbb{N}$, $\lambda _{1},\ldots ,\lambda _{m}\in \mathbb{N}$, and finite
coloring of $\mathbb{N}^{\left[ m\right] }$, there exists an increasing
sequence $\left( x_{n}\right) $ in $\mathbb{N}$ such that the set of
elements of the form $\left\{ \lambda _{1}x_{F_{1}},\ldots ,\lambda _{m}x_{F_{m}}\right\}$ 
for finite nonempty subsets $F_{1}<\cdots <F_{m}$ of $\mathbb{N}$ is
monochromatic.
\end{theorem}

From the previous theorem, it is straightforward to deduce an ``additive''
version:

\begin{corollary}
For any $m\in \mathbb{N}$, $c_{1},\ldots ,c_{m}\in \mathbb{N}$, and finite
coloring\index{coloring} of $\mathbb{N}$, there exists an increasing sequence $\left(
x_{n}\right) $ in $\mathbb{N}$ such that the set of elements of the form $%
c_{1}x_{F_{1}}+\cdots +c_{m}x_{F_{m}}$ for finite nonempty subsets $%
F_{1}<\cdots <F_{m}$ of $\mathbb{N}$ is monochromatic.
\end{corollary}

\section{Gowers' theorem}

\begin{definition}
For $k\in \mathbb{N}$, we let $\mathrm{FIN}_{k}$\index{$FIN_k$} denote the set of
functions $b:\mathbb{N}\rightarrow \{0,1,\ldots ,k\}$ with $\mathrm{Supp}(b)$
finite and such that $k$ belongs to the range of $b$. Here, $\mathrm{Supp}%
\left( b\right) :=\{n\in \N\ :\ b(n)\not=0\}$ is the \emph{support of $b$}. We
extend the definition of $\mathrm{FIN}_{k}$ to $k=0$ by setting $\mathrm{FIN}%
_{0}$ to consist of the function on $\N$ that is identically $0$.
\end{definition}

Note that, after identifying a subset of $\mathbb{N}$ with its
characteristic function, $\mathrm{FIN}_{1}$ is simply the set of nonempty
finite subsets of $\mathbb{N}$. We endow $\mathrm{FIN}_{k}$ with a partial semigroup operation $\left(
b_{0},b_{1}\right) \mapsto b_{0}+b_{1}$ which is defined only when $\mathrm{%
Supp}\left( b_{0}\right) <\mathrm{Supp}\left( b_{1}\right) $. 

By transfer, ${}^{\ast }\mathrm{FIN}_{k}$ is the set of internal functions $%
b:{}^{\ast }\mathbb{N}\rightarrow \{0,1,\ldots ,k\}$ with hyperfinite
support that have $k$ in their range. The partial semigroup operation on $%
\mathrm{FIN}_k$ extends also to ${}^{\ast}\mathrm{FIN}_k$. We say that $%
\alpha\in {}^{\ast }\mathrm{FIN}_{k}$ is \emph{cofinite} if its support is
disjoint from $\mathbb{N}$. Thus, if $\alpha,\beta\in {}^{\ast}\mathrm{FIN}%
_k $ are cofinite and $i<j$, then the sum ${}^{\ast i}\alpha+{}^{\ast
j}\beta $ exists.

Gowers' original theorem\index{Gowers' Theorem} considers the \emph{tetris operation} $T:\mathrm{FIN%
}_{k}\rightarrow \mathrm{FIN}_{k-1}$\index{tetris operation} given by $T(b)(n):=\max \left\{
b(n)-1,0\right\} $. In this section, we prove a more general version of
Gowers' theorem by considering a wider variety of functions $\mathrm{FIN}%
_{k}\rightarrow \mathrm{FIN}_{j}$ for $j\leq k$. First, for $k\in \N$, by a 
\emph{regressive map on $k$} or \emph{generalized tetris operation}\index{tetris operation!generalized}\index{regressive map}, we mean a nondecreasing surjection $%
f:[0,k]\rightarrow \lbrack 0,f(k)]$. Given a regressive map $f$ on $k$,
one can define a corresponding operation $f:\mathrm{FIN}_{k}\rightarrow 
\mathrm{FIN}_{f(k)}$ by setting $f\left( b\right) :=f\circ b$. Note also
that if $l\leq k$, then $f|_{\left[ 0,l\right] }$ is a regressive map\index{regressive map} on $l
$, whence we can also consider $f:\mathrm{FIN}_{l}\rightarrow \mathrm{FIN}%
_{f(l)}$.

Given $n\in \N$, we set $\mathrm{FIN}_{[0,n]}:=\bigcup_{k=0}^{n}\mathrm{FIN}%
_{k}$ . Note that $\mathrm{FIN}_{\left[ 0,n\right] }$ is also a partial
semigroup given by pointwise addition and defined on pairs of functions with
disjoint supports. If $f$ is a regressive map\index{regressive map} on $n$, then as we already
recalled, $f|_{\left[ 0,k\right] }$ is a regressive map\index{regressive map} on $k$ for $1\leq
k\leq n$, whence $f$ yields a function $f:\mathrm{FIN}_{\left[ 0,n\right]
}\rightarrow \mathrm{FIN}_{\left[ 0,f(n)\right] }$.

Given a regressive map\index{regressive map} $f$ on $n$, we get the nonstandard extension\index{nonstandard extension} $%
f:{}^{\ast }\mathrm{FIN}_{n}\rightarrow {}^{\ast }\mathrm{FIN}_{f(n)}$ and $%
f:{}^{\ast }\mathrm{FIN}_{\left[ 0,n\right] }\rightarrow {}^{\ast }\mathrm{%
FIN}_{\left[ 0,f(n)\right] }$. In addition, if $\alpha ,\beta \in {}^{\ast }%
\mathrm{FIN}_{\left[ 0,n\right] }$ are cofinite and $i<j$, then ${}^{\ast
i}\alpha +{}^{\ast j}\beta $ exists and $f({}^{\ast i}\alpha +{}^{\ast
j}\beta )=f({}^{\ast i}\alpha )+f({}^{\ast j}\beta )$.

If $\alpha _{k}\in {}^{\ast }\mathrm{FIN}_{k}$ for $k=1,\ldots ,n$, we say
that a tuple $\langle\alpha _{1},\ldots ,\alpha _{n}\rangle$ is \emph{coherent} if $f(\alpha
_{k})\sim \alpha _{f(k)}$ for all $k=1,\ldots ,n$ and all regressive maps $%
f$ on $n$. 
It is easy to verify that the set $Z$ of all cofinite coherent tuples is a compact $u$-semigroup.
We note that $Z$ is nonempty. Indeed, let $%
\alpha _{1}\in {}^{\ast }\mathrm{FIN}_{1}$ be any cofinite element. For $%
k=2,\ldots ,n$, let $\alpha _{k}\in {}^{\ast }\mathrm{FIN}_{k}$ have the
same support as $\alpha _{1}$ and take only the values $0$ and $k$. It is
immediate that $(\alpha _{1},\ldots ,\alpha _{n})\in Z$.

Finally, we introduce some convenient notation. Given $\alpha _{0},\alpha
_{1},\ldots ,\alpha _{j}\in {}^{\ast }\mathrm{FIN}_{\left[ 0,n\right] }$ and 
$j\in \N$, we set 
\begin{equation*}
\bigoplus_{i=0}^{j}\alpha _{i}:=\alpha _{0}+{}^{\ast }\alpha _{1}+\cdots
+{}^{\ast j}\alpha _{j}\text{.}
\end{equation*}%
Thus, if each $\alpha _{i}$ is cofinite and $f$ is a regressive map on $n$,
we have the convenient equation 
\begin{equation*}
f(\bigoplus_{i=1}^{j}\alpha _{i})=\bigoplus_{i=1}^{j}f(\alpha _{i}).
\end{equation*}

\begin{lemma}
\label{Lemma:Gowers}Fix $n\in \mathbb{N}$. Then, for $k=1,\ldots,n$, there
exist cofinite $u$-idempotents $\alpha _{k}\in {}^{\ast }\mathrm{FIN}_{k}$
such that:

\begin{enumerate}
\item $\langle\alpha_1,\ldots,\alpha_n\rangle$ is a coherent tuple, and

\item $\alpha _{j}+{}^{\ast }\alpha _{k}\sim \alpha _{k}+{}^{\ast }\alpha
_{j}\sim \alpha _{k}$ for every $1\leq j\leq k\leq n$.
\end{enumerate}
\end{lemma}

\begin{proof}
We define, by recursion on $k=1,2,\ldots ,n$, a sequence of $u$-idempotents 
\begin{equation*}
\boldsymbol{\alpha }^{(k)}=(\alpha _{1}^{(k)},\ldots ,\alpha _{n}^{(k)})\in Z
\end{equation*}
such that, for $1\leq i\leq j\leq k\leq n$, one has that

\begin{enumerate}
\item[(a)] $\alpha _{i}^{(k)}\sim\alpha _{i}^{(j)}$,

\item[(b)] $\alpha _{j}^{(k)}+{}^{\ast }\alpha _{i}^{(k)}\sim \alpha
_{j}^{(k)}$.
\end{enumerate}
\index{Gowers' Theorem!for multiple tetris operations}
To begin the construction, let $\boldsymbol{\alpha }^{(1)}$ be any
idempotent element of $Z$. Now suppose now that $k<n$ and $\boldsymbol{%
\alpha }^{(1)},\ldots ,\boldsymbol{\alpha }^{(k)}$ have been constructed
satisfying (a) and (b). Consider the closed $u$-semigroup\index{$u$-semigroup} $Z_k$ consisting
of sequences $\boldsymbol{\beta }=\left( \beta _{1},\ldots ,\beta
_{k}\right) \in Z$ such that:

\begin{enumerate}
\item[(i)] $\beta _{j}\sim \alpha _{j}^{(k)}$ for $1\leq j\leq k$, and

\item[(ii)] $\beta _{j}+{}^{\ast }\beta _{i}\sim \beta _{j}$ for $1\leq
i<j\leq n$ and $1\leq i\leq k$.
\end{enumerate}

We claim that $Z_{k}$ is nonempty. Indeed, we claim it contains the sequence 
$\boldsymbol{\beta }=(\beta _{1},\ldots ,\beta _{k})$, where $\beta _{j}\in
{}^{\ast }\mathrm{FIN}_{j}$ is such that 
\begin{equation*}
\beta _{j}\sim \bigoplus_{i=0}^{j-1}\alpha _{j-i}^{(k)}.
\end{equation*}%
To see that $\boldsymbol{\beta }$ is coherent, fix a regressive map\index{regressive map} $f$ on 
$n$. For a given $j\in \lbrack 1,k]$, we have that%
\begin{equation*}
f\left( \beta _{j}\right) \sim \bigoplus_{i=0}^{j-1}f(\alpha
_{j-i}^{(k)})\sim \bigoplus_{i=0}^{j-1}\alpha _{f(j-i)}^{(k)}\sim
\bigoplus_{i=0}^{f(j)-1}\alpha _{f(j)-i}^{(k)}\sim \beta _{f(j)}.
\end{equation*}%
The second equivalence uses that $\boldsymbol{\alpha }^{(k)}$ is coherent,
while the third equivalence uses that $f$ is a regressive map\index{regressive map} and that $%
\boldsymbol{\alpha }^{(k)}$ is a $u$-idempotent\index{$u$-idempotent}. 
Next observe that, since $\boldsymbol{\alpha }^{(k)}$ satisfies (b), we have
that $\beta _{j}\sim \bigoplus_{i=0}^{j-k}\alpha _{j-i}^{(k)}$ for $%
j=1,\ldots ,n$, and, moreover, that $\beta _{j}\sim \alpha _{j}^{(k)}$ for $%
j=1,2,\ldots ,k$. Thus, if $1\leq i<j\leq n$ and $1\leq i\leq k$, it follows
that 
\begin{equation*}
\beta _{j}+{}^{\ast }\beta _{i}\sim \bigoplus_{i=0}^{j-k}\alpha
_{j-i}^{(k)}+{}^{\ast k}\alpha _{i}\sim \bigoplus_{i=0}^{j-k}\alpha
_{j-i}^{(k)},
\end{equation*}%
where the last equivalence follows from (b). This concludes the proof that $%
\boldsymbol{\beta }$ belongs to $Z_{k}$.

Since $Z_{k}$ is a nonempty closed $u$-semigroup\index{$u$-semigroup}, it contains an idempotent $%
\boldsymbol{\alpha }^{(k+1)}$. It is clear that $\boldsymbol{\alpha }^{(k+1)}
$ satisfies (a) and (b). This concludes the recursive construction.

For $k=1,\ldots ,n$, we fix $\alpha _{k}\in {}^{\ast }\mathrm{FIN}_{k}$ such
that%
\begin{equation*}
\alpha _{k}\sim \bigoplus_{i=1}^{k}\alpha _{i}^{(i)}\text{.}
\end{equation*}%
We claim that $\alpha _{1},\ldots ,\alpha _{n}$ are as in the conclusion of
the lemma. Towards this end, first fix a regressive map\index{regressive map} $f$ on $n$. We
then have that%
\begin{equation*}
f\left( \alpha _{j}\right) \sim \bigoplus_{i=1}^{k}f(\alpha _{i}^{(i)})\sim
\bigoplus_{i=1}^{k}\alpha _{f(i)}^{(i)}\sim \bigoplus_{i=1}^{f(k)}\alpha
_{i}^{(i)}\sim \alpha _{f(j)},
\end{equation*}%
where the second to last step uses the fact that $f$ is a regressive map\index{regressive map},
that that the $\alpha _{i}^{(k)}$'s are $u$-idempotent\index{$u$-idempotent}, and that (a) holds.
We thus have that $\alpha _{1},\ldots ,\alpha _{n}$ are coherent. We now
show that (2) holds. Fix $1\leq j\leq k\leq n$. We then have 
\begin{equation*}
\alpha _{k}+{}^{\ast }\alpha _{j}\sim \bigoplus_{i=1}^{k}\alpha
_{i}^{(i)}+\bigoplus_{i=1}^{j}\alpha _{i}^{(i)}\sim
\bigoplus_{i=1}^{k}\alpha _{i}^{(i)}\sim \alpha _{k},
\end{equation*}%
where the second to last equivalence repeatedly uses the fact that $\alpha
_{k}^{(k)}+{}^{\ast }\alpha _{i}^{(i)}\sim \alpha _{k}^{(k)}$ for $1\leq
i\leq k$. 
A similar computation shows that $\alpha _{j}+{}^{\ast }\alpha _{k}\sim
\alpha _{k}$, establishing (2) and finishing the proof of the lemma.
\end{proof}

We say that a sequence $\left( x_{i}\right) $ in $\mathrm{FIN}_{n}$ is a 
\emph{block sequence} if $\mathrm{Supp}\left( x_{i}\right) <\mathrm{Supp}%
\left( x_{j}\right) $ for $i<j$.

\begin{theorem}
Suppose that $\alpha _{k}\in {}^{\ast }\mathrm{FIN}_{k}$ for $k=1,2,\ldots ,n
$ are as in the previous lemma. For $k=1,\ldots ,n$, suppose that $%
A_{k}\subset \mathrm{FIN}_{k}$ is such that ${}\alpha _{k}\in {}^{\ast }A_{k}
$. Then there exists a block sequence $\left( x_{i}\right) $ in $\mathrm{FIN}%
_{n}$ such that, for every finite sequence $f_{1},\ldots ,f_{\ell }$ of
regressive maps on $n$, we have $f_{1}\left( x_{1}\right) +\cdots +f_{\ell
}\left( x_{\ell }\right) \in A_{\max (f_{1}(n),\ldots ,f_{\ell}(n))}$.
\end{theorem}

\begin{proof}
By recursion on $d$, we define a block sequence $(x_{d})$ in $\mathrm{FIN}%
_{n}$ such that, for every sequence $f_{1},\ldots ,f_{d+1}$ of regressive maps $n$, we have%
\begin{equation*}
f_{1}\left( x_{1}\right) +\cdots +f_{d}\left( x_{d}\right) \in A_{\max
(f_{1}(n),\ldots ,f_{d}(n))}
\end{equation*}%
and%
\begin{equation*}
f_{1}\left( x_{1}\right) +\cdots +f_{d}\left( x_{d}\right) +f_{d+1}\left(
\alpha _{n}\right) \in {}^{\ast }A_{\max (f_{1}(n),\ldots ,f_{d+1}(n))}\text{%
.}
\end{equation*}%
Suppose that $x_{1},\ldots ,x_{d}$ has been constructed satisfying the
displayed properties. Suppose that $f_{1},\ldots ,f_{d+2}$ are regressive maps on $n$. Then since 
\begin{equation*}
f_{d+1}(\alpha _{n})+f_{d+2}({}^{\ast }\alpha _{n})\sim \alpha
_{f_{d+1}(n)}+{}^{\ast }\alpha _{f_{d+2}(n)}\sim \alpha _{\max
(f_{d+1}(n),f_{d+2}(n))}\sim f_{d+p}(\alpha _{n}),
\end{equation*}%
where $p\in \{1,2\}$ is such that $\max (f_{d+1}(n),f_{d+2}(n))=f_{d+p}(n)$,
the inductive hypothesis allows us to conclude that%
\begin{equation*}
f_{1}\left( x_{1}\right) +\cdots +f_{d-1}\left( x_{d}\right) +f_{d+1}\left(
\alpha _{n}\right) +f_{d+2}\left( {}^{\ast }\alpha _{n}\right) \in {}^{\ast
\ast }A_{\max (f_{1}(n),\ldots ,f_{d+2}(n))}\text{.}
\end{equation*}%
Therefore, by transfer, we obtain $x_{d+1}\in \mathrm{FIN}_{n}$ such that $%
\mathrm{Supp}\left( x_{d+1}\right) >\mathrm{Supp}\left( x_{d}\right) $, and,
for any sequence $f_{1},\ldots ,f_{d+2}$ of regressive maps on $n$, we
have that%
\begin{equation*}
f_{1}\left( x_{1}\right) +\cdots +f_{d}\left( x_{d+1}\right) \in {}A_{\max
(f_{1}(n),\ldots ,f_{d+1}(n))}
\end{equation*}%
and 
\begin{equation*}
f_{1}\left( x_{1}\right) +\cdots +f_{d+1}\left( x_{d+1}\right)
+f_{d+2}\left( {}\alpha _{n}\right) \in {}^{\ast }A_{\max (f_{1}(n),\ldots
,f_{d+2}(n))}\text{.}
\end{equation*}%
This concludes the recursive construction.
\end{proof}

\begin{corollary}[Generalized Gowers]
For any finite coloring of $\mathrm{FIN}_{n}$, there exists a block sequence 
$\left( x_{i}\right) $ in $\mathrm{FIN}_{n}$ such that the set of elements
of the form $f_{1}\left( x_{1}\right) +\cdots +f_{\ell}\left( x_{\ell}\right) $
for $\ell\in \mathbb{N}$ where $f_{1},\ldots ,f_{\ell }$ 
and regressive maps on $n$
such that $n=\max (f_{1}(n),\ldots ,f_{\ell}(n))$, is monochromatic.
\end{corollary}

\begin{proof}
If $\mathrm{FIN}_n=B_1\sqcup \cdots \sqcup B_r$ is a partition of $\mathrm{%
FIN}_n$, apply the previous theorem with $A_n:=B_i$ where $\alpha_n\in
{}^{\ast}B_i$.
\end{proof}

Gowers' original theorem is a special case of the previous corollary by
taking each $f_{i}$ to be an iterate of the tetris operation. One can also
obtain a common generalization of Gowers' theorem and the Milliken-Taylor
theorem. We let $\mathrm{FIN}_{k}^{\left[ m\right] }$ be the set of $m$%
-tuples $\left( x_{1},\ldots ,x_{m}\right) $ in $\mathrm{FIN}_{k}$ such that 
$\mathrm{Supp}\left( x_{i}\right) <\mathrm{Supp}\left( x_{j}\right) $ for $%
1\leq i<j\leq m$. Suppose that $\left( x_{d}\right) $ is a sequence in $%
\mathrm{FIN}_{n}$. Suppose that $F=\left\{ a_{1},\ldots ,a_{r}\right\} $ is
a finite nonempty subset of $\mathbb{N}$. We let $\mathcal{S}\left(
F,k\right) $ be the set of tuples $\boldsymbol{f}=\left( f_{j}\right) _{j\in
F}$ such that $f_{j}:\left\{ 0,1,\ldots ,n\right\} \rightarrow \left\{
0,1,\ldots ,k_{j}\right\} $ is a nondecreasnig surjection and $\max \left\{
k_{j}:j\in F\right\} =k$. For such an element $\boldsymbol{f}$ we let $x_{%
\boldsymbol{f}}$ be the sum $f_{a_{1}}\left( x_{a_{1}}\right) +\cdots
+f_{a_{r}}\left( x_{a_{r}}\right) $. When $F$ is empty, by convention we let 
$\mathcal{S}\left( F,k\right) $ contain a single element $\boldsymbol{f}%
=\varnothing $, and in such case $x_{\boldsymbol{f}}=0$.

\begin{theorem}
Let $\langle\alpha _{1},\ldots ,\alpha _{n}\rangle$ be as in Lemma \ref{Lemma:Gowers}.
Suppose that $A_{k}\subset \mathrm{FIN}_{k}^{\left[ m\right] }$ for $%
k=1,2\ldots ,n$ is such that $\left( \alpha _{k},{}^{\ast }\alpha
_{k},\ldots ,{}^{\ast (m-1)}\alpha _{k}\right) \in {}^{\ast m}A_{k}$. Then
there exists a block sequence $\left( x_{d}\right) $ in $\mathrm{FIN}_{n}$
such that, given $k\in \{1,\ldots,n\}$, nonempty finite subsets $%
F_1<\ldots<F_m$ of $\N$, and $f_i\in \mathcal{S}(F_i,k)$ for $i=1,\ldots,m$,
we have that $\{x_{\boldsymbol{f}_1},\ldots,x_{\boldsymbol{f}_m}\}\in A_k$.%
%
%
%
%
%
%
\end{theorem}

\begin{proof}
We define by recursion a block sequence $\left( x_{d}\right) $ in $\mathrm{%
FIN}_{n}$ such that, for all $k\in \{1,\ldots ,n\}$, all $1\leq j\leq m$,
all finite $F_{1},\ldots ,F_{j}\subseteq \N$ with $F_{1}<\cdots <F_{j}$ and $%
F_{1},\ldots ,F_{j-1}$ nonempty, and all $\boldsymbol{f_{i}}\in \mathcal{S}%
(F_{i},k)$, we have 
\begin{equation*}
\left\{ x_{\boldsymbol{f}_{1}},x_{\boldsymbol{f}_{2}},\ldots ,x_{\boldsymbol{%
f}_{j}},\alpha _{k},{}^{\ast }\alpha _{k},\ldots ,{}^{\ast (m-j-1)}\alpha
_{k}\right\} \in {}^{\ast (m-j)}A_{k}
\end{equation*}%
and 
\begin{equation*}
\left\{ x_{\boldsymbol{f}_{1}},x_{\boldsymbol{f}_{2}},\ldots ,x_{\boldsymbol{%
f}_{j-1}},x_{\boldsymbol{f}_{j}}+{}\alpha _{k},{}^{\ast }\alpha_k ,{}^{\ast
\ast }\alpha _{k},\ldots ,{}^{\ast (m-j)}\alpha _{k}\right\} \in {}^{\ast
(m-j+1)}A_{k}\text{.}
\end{equation*}%
It is clear that the sequence $(x_{d})$ is as desired. 

Suppose that $x_{1},\ldots ,x_{d}$ have been constructed satisfying the
above assumption. From the properties of the sequence $\alpha _{1},\ldots
,\alpha _{n}$, we see that the second condition also implies, for all $1\leq
s\leq k$:%
\begin{equation*}
\{x_{\boldsymbol{f}_{1}},\ldots ,x_{\boldsymbol{f}_{j}}+\alpha _{k}+{}^{\ast
}\alpha _{s},{}^{\ast \ast }\alpha _{k},\ldots ,{}^{\ast \left( m-j+1\right)
}\alpha _{k}\}\in {}{}^{\ast (m-j+2)}A_{k}
\end{equation*}%
and%
\begin{equation*}
\{x_{\boldsymbol{f}_{1}},\ldots ,x_{\boldsymbol{f}_{j}}+\alpha _{s}+{}^{\ast
}\alpha _{k},{}^{\ast \ast }\alpha _{k},\ldots ,{}^{\ast \ast \left(
m-j+1\right) }\alpha _{k}\}\in {}{}^{\ast (m-j+2)}A_{k}\text{.}
\end{equation*}%
It follows from transfer that we can find $x_{d+1}$ with $\mathrm{Supp}%
(x_{d+1})>\mathrm{Supp}(x_{d})$ as desired.
\end{proof}

\section*{Notes and references} The finitary form of Hindman's theorem, initially conjectured by Graham and Rothschild in \cite{graham_ramseys_1971}, was first proved by Hindman by purely combinatorial methods \cite{hindman_finite_1974}. It had been previously observed by Galvin---see also \cite{hindman_existence_1972}---that the existence of an idempotent ultrafilters (which was unknown at the time) implies the infinitary version of Hindman's theorem. The existence of idempotent ultrafilters was later established by Glazer; see \cite{comfort_ultrafilters:_1977}. Remarkably, Hindman's original combinatorial proof was significantly more technical and required a substantial amount of bookkeeping. Another short proof of Hindman's theorem was also obtained shortly later by Baumgartner \cite{baumgartner_short_1974}.

Gowers' theorem \cite{gowers_lipschitz_1992} was motivated by a problem on the geometry of the Banach space $c_0$. While Gowers' original proof was infinitary and used ultrafilter methods, explicit purely combinatorial proofs of the corresponding finitary statement were later obtained by Ojeda-Aristizabal \cite{ojeda-aristizabal_finite_2015} and Tyros \cite{tyros_primitive_2015}.  The more general version of Gowers' theorem presented in this chapter was established in \cite{lupini_gowers_2017}. This answered a question of Barto\v{s}ov\'{a} and Kwiatkowska from \cite{bartosova_lelek_2013}, where the corresponding finitary version is proved with different methods.

\chapter{Partition regularity of equations}\label{partreg}

\section{Characterizations of partition regularity}

Let $F(X_1,\ldots,X_n)$ be a polynomial over $\Z$.  We begin with a proposition giving a nonstandard characterization of ultrafilters all of whose sets contain zeroes of $F$.

\begin{proposition}\label{charultrasolutions}
Suppose that $\u\in \beta\N$.  The following are equivalent:
\begin{enumerate}
\item For every $A\in \u$, there are [distinct] $x_1,\ldots,x_n\in A$ such that $F(x_1,\ldots,x_n)=0$.
\item There exists $k\in \N$ and [distinct] $\alpha_1,\ldots,\alpha_n\in {}^{k\ast}\N$ such that $\u=\u_{\alpha_i}$ for all $i=1,\ldots,n$ and $F(\alpha_1,\ldots,\alpha_n)=0$.
\end{enumerate}
\end{proposition}

\begin{proof}
First assume that (1) holds.  For $A\in \u$, set 
\begin{equation*}
X_{A}:=\left\{(\alpha _{1},\ldots ,\alpha _{n})\in \starN^{n}\ :\
[\bigwedge_{i\not=j}\alpha _{i}\not=\alpha _{j}]\wedge
\bigwedge_{i}\alpha _{i}\in {}^{\ast }A\wedge {}^{\ast }F(\alpha _{1},\ldots ,\alpha _{n})=0\right\}.
\end{equation*}%
It is clear that the family $(X_A)$ has the finite intersection property, so by saturation there is $(\alpha_1,\ldots,\alpha_n)\in \bigcap_A X_A$; this tuple witnesses the truth of (2).

Conversely, suppose that (2) holds.  Suppose that $\alpha _{1},\ldots ,\alpha _{n}\in {}^{k\ast}\N$ are as
in (2).  Suppose that $%
A\in \mathcal{U}$. Then the statement \textquotedblleft there exist $i\in
\left\{ 1,2,\ldots ,r\right\} $ and [distinct] $x_{1},\ldots ,x_{n}\in
{}^{\ast }{}^{k\ast}A$ such that $F(x_{1},\ldots ,x_{n})=0$%
\textquotedblright\ holds in the $k$-fold iterated nonstandard universe, as witnessed by $%
\alpha _{1},\ldots ,\alpha _{n}$; the desired conclusion follows from
$k$ applications of transfer.

\end{proof}

\begin{definition}
An ultrafilter $\u$ is a \emph{witness} of the [injective] partition
regularity of the equation $F(X_1,\ldots,X_n)=0$ when $\u$
satisfies the equivalent conditions of Proposition 
\ref{charultrasolutions}. In this case, we also simply say that
$\u$ an [injective] \emph{$F$-witness}.
\end{definition}

We now connect this notion with the standard Ramsey-theoretic notion of partition regular equation.

\begin{definition}
A polynomial $F(X_{1},\ldots ,X_{n})$ is said to be \emph{[injectively] partition regular%
} (on the natural numbers $\N$)
if, for every finite partition $\mathbb{N}=C_{1}\sqcup \cdots \sqcup C_{r}$, 
there exists $i\in \{1,\ldots ,r\}$ and there exist [distinct] $%
x_{1},\ldots ,x_{n}\in C_{i}$ such that $F(x_{1},\ldots ,x_{n})=0$.\index{partition regular}\index{partition regular!injectively}
\end{definition}

\begin{proposition}
\label{ultraregularity} 
$F(X_1,\ldots,X_n)$ is [injectively] partition regular if and only if there is an [injective] 
$F$-witness if and only there exist [distinct] $u$-equivalent $\alpha_1\sim\ldots\sim\alpha_n$
such that $F(\alpha_1,\ldots,\alpha_n)=0$.
\end{proposition}

\begin{proof}
First suppose that $F(X_{1},\ldots ,X_{n})=0$ is
[injectively] partition regular. Given $A\subseteq \mathbb{N}$, consider the
set 
\begin{equation*}
Y_{A}:=\{(\alpha _{1},\ldots ,\alpha _{n})\in \starN^{n}\ :\
[\bigwedge_{i\not=j}\alpha _{i}\not=\alpha _{j}]\wedge
\bigwedge_{i,j}(\alpha _{i}\in {}^{\ast }A\leftrightarrow \alpha _{j}\in
{}^{\ast }A)\wedge {}^{\ast }F(\alpha _{1},\ldots ,\alpha _{n})=0\}.
\end{equation*}%
Observe that the family $(Y_{A})_{A\subseteq \mathbb{N}}$ has the finite
intersection property. Indeed, given $A_{1},\ldots ,A_{m}\subseteq \mathbb{N}
$, let $C_{1},\ldots ,C_{k}$ be the atoms of the boolean algebra generated
by $A_{1},\ldots ,A_{m}$. Since the equation $F(X_{1},\ldots ,X_{n})=0$ is
[injectively] partition regular, there is $i\in \{1,\ldots ,k\}$ and
[distinct] $x_{1},\ldots ,x_{n}\in C_{i}$ such that $F(x_{1},\ldots ,x_{n})=0
$; it follows that $(x_{1},\ldots ,x_{n})\in \bigcap_{i=1}^{m}Y_{A_{i}}$.
Thus, by saturation, there is $(\alpha _{1},\ldots ,\alpha _{n})\in
\bigcap_{A\subseteq \mathbb{N}}Y_{A}$. Then
these $\alpha _{1},\ldots ,\alpha _{n}\in\starN
$ are $u$-equivalent and $F(\alpha_1,\ldots,\alpha_n)=0$.
Clearly, $\u=\u_{\alpha_i}$ is the desired $F$-witness.



The converse direction is trivial.
\end{proof}


As an example, let us
give a nonstandard proof of the following result, first shown by
T.C. Brown and V. R\H{o}dl \cite{brown_monochromatic_1991}.

\begin{theorem}
A homogeneous equation $P(X_1,\ldots,X_n)=0$ is [injectively]
partition regular if and only if the corresponding
equation with reciprocals $P(1/X_1,\ldots,1/X_n)=0$ is
[injectively] partition regular.
\end{theorem}

\begin{proof}
Assume first that $P(X_1,\ldots,X_n)=0$ is partition regular,
and let $\alpha_1,\ldots,\alpha_n\in\starN$ be such that $P(\alpha_1,\ldots,\alpha_n)=0$.
Pick any infinite $\xi\in\starN$. Note that ${}^\ast\xi>\alpha_i$, and so
the factorial $({}^\ast\xi)!={}^\ast(\xi!)$ is a multiple of $\alpha_i$ for all $i$.
Let $\zeta_i:={}^\ast(\xi!)/\alpha_i\in{}^{**}\N$. Then
$\alpha_1\sim\ldots\sim\alpha_n\Rightarrow 
\zeta_1\sim\ldots\sim\zeta_n$ and
$P(1/\zeta_1,\ldots,1/\zeta_n)=(1/{}^\ast(\xi!))^d P(\alpha_1,\ldots,\alpha_n)=0$,
where $d$ is the degree of $P$. This shows that $P(1/X_1,\ldots,1/X_n)$ is
partition regular. Clearly, by the same argument also the converse implication follows.
Finally, note that the $\alpha_i$'s are distinct if and only if
the $\zeta_i$'s are distinct, and so the equivalence holds also in the injective case.

\end{proof}

\section{Rado's theorem}

In this section we use the characterization of partition regularity shown above
to prove the following version of the classical theorem of Rado for a single equation:
 

\begin{theorem}
Suppose that $k>2$ and $c_1,\ldots,c_k\in \Z$ are such that $c_1+\cdots+c_k=0
$. Then the equation $c_1X_1+\cdots+c_kX_k=0$ is injectively partition
regular.
\end{theorem}


Indeed, we will prove a strengthening of Rado's theorem below. First, given
a polynomial $P(X):=\sum_{j=0}^{n}b_{j}X^{j}\in \mathbb{Z}\left[ X\right]
$ and $\xi \in \starZ$, set $\tilde{P}(\xi ):=\sum_{j=0}^{n}b_{j}{}^{j\ast
{}}\xi \in {}^{(j+1)\ast {}}\mathbb{Z}$. We note the following corollary of
Proposition \ref{ultraregularity}.

\begin{corollary}
\label{Corollary:part} Suppose that $c_{1},\ldots ,c_{k}\in \mathbb{Z}$ are such that
there exist [distinct] polynomials $P_{1}(X),\ldots ,P_{k}(X)\in 
\mathbb{Z}[X]$ and $\xi ,\eta \in \starN$ for which

\begin{enumerate}
\item $c_1P_1(X)+\cdots+c_kP_k(X)=0$, and

\item $\tilde{P}_i(\xi)\sim \eta$ for each $i=1,\ldots,k$.
\end{enumerate}

\noindent Then $\mathcal{U}_\eta$ witnesses that $c_1X_1+\cdots+c_kX_k=0$ is
[injectively] partition regular.
\end{corollary}

\begin{proof}
For each $i=1,\ldots,k$, let $\alpha_i:=\tilde{P}_i(\xi)$; by assumption,
for each $i$ we have $\mathcal{U}_{\alpha_i}=\mathcal{U}_\eta$. It is also
clear that $c_1\alpha_1+\cdots+c_k\alpha_k=0$. By the previous proposition,
we have that $\mathcal{U}_\eta$ witnesses the partition regularity of $%
c_1X_1+\cdots+c_kX_k=0$.

Suppose in addition that the $P_i$'s are distinct; to conclude injective
partition regularity, we must show that the $\alpha_i$'s are distinct.
Suppose that $\alpha_i=\alpha_j$, that is, $\tilde{P}_i(\xi)=\tilde{P}_j(\xi)
$. Write $P_i(X):=\sum_{l=0}^m r_l X^l$ and $P_j(X)=\sum_{l=0}^m s_l X^l$, where at least one between $r_m$ and $s_m$ is nonzero. We
then have that $(r_m-s_m){}^{m\ast}\xi=-\sum_{l=0}^{m-1}(r_l-s_l){}^{l
\ast}\xi$. The only way that this is possible is that $r_m=s_m=0$;
continuing inductively in this manner, we see that $P_i=P_j$, yielding the
desired contradiction.
\end{proof}

In light of the previous corollary, it will be useful to find a standard
condition on a family of polynomials $P_1,\ldots,P_k\in \Z[X]$ such that, for every idempotent $\xi\in \starN$, we have that all $\tilde{P}_i(\xi)$'s
are $u$-equivalent. The next definition captures such a condition.

\begin{definition}
Following \cite{di_nasso_iterated_2015},
we define the equivalence relation $\approx_u$ on finite strings of integers
to be the smallest equivalence relation satisfying the following three
properties:

\begin{itemize}
\item $\emptyset \approx_u \langle 0 \rangle$;

\item If $a\in \mathbb{Z}$, then $\langle a\rangle\approx_u \langle
a,a\rangle$;

\item If $\sigma\approx_u \sigma^{\prime }$ and $\tau\approx_u\tau^{\prime }$%
, then concatenations $\sigma\tau \approx_u \sigma^{\prime }\tau^{\prime }$.
\end{itemize}

If $P,Q\in \mathbb{Z}[X]$ are polynomials, then we write $P\approx_u Q$ to
mean that their strings of coefficients are $u$-equivalent.
\end{definition}

\begin{lemma}
\label{uequiv} Let $P,Q\in \mathbb{Z}[X]$ have positive leading coefficient.
If $P\approx _{u}Q$, then for
every idempotent $\xi \in \starN$, we have $\tilde{P}(\xi )\sim \tilde{Q}%
(\xi )$.
\end{lemma}


\begin{proof}
Fix an idempotent $\xi\in \starN$. The lemma follows from the following
facts:

\begin{itemize}
\item $\sum_{j=0}^{m}a_{j}{}^{j\ast }\xi \sim \sum_{j=0}^{i}a_{j}{}^{j\ast
}\xi +a_{i}{}^{(i+1)\ast }\xi +\sum_{j=i+1}^{m}a_{j}{}^{(j+1)\ast }\xi $;

\item If $\sum_{j=0}^{m}a_{j}{}^{j\ast }\xi \sim \sum_{j=0}^{m^{\prime
}}a_{j}^{\prime j\ast }\xi $ and $\sum_{j=0}^{n}b_{j}{}^{j\ast }\xi \sim
\sum_{j=0}^{n^{\prime }}b_{j}^{\prime j\ast }\xi $, then 
\begin{equation*}
\sum_{j=0}^{m}a_{j}{}^{j\ast }\xi +\sum_{j=0}^{n}b_{j}{}^{(j+m)\ast }\xi
\sim \sum_{j=0}^{m^{\prime }}a_{j}^{\prime j\ast }\xi +\sum_{j=0}^{n^{\prime
}}b_{j}^{\prime (j+m^{\prime })\ast }\xi .
\end{equation*}
\end{itemize}
\end{proof}

We should mention that the converse of the previous lemma is true in an even
stronger form, namely that if $\tilde{P}(\xi )\sim \tilde{Q}(\xi )$ for 
\emph{some} idempotent $\xi \in \starN$, then $P\approx _{u}Q$. This follows
from \cite[Theorem T]{Maleki_Solving_2000}.

We can now give the nonstandard proof of the above mentioned version
of Rado's theorem. In fact, we prove
the more precise statement:

\begin{theorem}
Suppose that $k>2$ and $c_{1},\ldots ,c_{k}\in \mathbb{Z}$ are such that $%
c_{1}+\cdots +c_{k}=0$. Then there exists $a_{0}\,\ldots ,a_{k-2}\in \mathbb{%
N}$ such that, for every idempotent ultrafilter $\mathcal{U}$, we have that $%
a_{0}\mathcal{U}\oplus \cdots \oplus a_{k-2}\mathcal{U}$ witnesses the
injective partition regularity of the equation $c_{1}X_{1}+\cdots
+c_{k}X_{k}=0$.
\end{theorem}

\begin{proof}
Without loss of generality, we will assume that $c_{1}\geq c_{2}\geq \cdots
\geq c_{k}$. By Corollary \ref{Corollary:part} and Lemma \ref{uequiv}, we need to find 
$a_{0},\ldots ,a_{k-2}\in \mathbb{N}$ and distinct $P_{1}(X),\ldots
,P_{k}(X)\in \mathbb{Z}[X]$ such that $c_{1}P_{1}(X)+\cdots
+c_{k}P_{k}(X)=0$ and such that $P_{i}(X)\approx
_{u}\sum_{j=0}^{k-2}a_{j}X^{j}$ for each $i=1,\ldots ,k$. For appropriate $%
a_{0},\ldots ,a_{k-2}$, the following polynomials will be as needed:

\begin{itemize}
\item $P_1(X):=\sum_{j=0}^{k-2}a_jX^j+a_{k-2}X^{k-1}$;

\item $P_i(X):=\sum_{j=0}^{k-i-1}a_jX^j+\sum_{j=k-i+1}^{k-1}a_{j-1}X^j$ for $%
2\leq i\leq k-1$,

\item $P_k(X):=a_0+\sum_{j=1}^{k-1}a_{j-1}X^j$.
\end{itemize}

It is straightforward to check that $P_{i}(X)\approx
_{u}\sum_{j=0}^{k-2}a_{j}X^{j}$ for each $i=1,\ldots ,k$. Furthermore, since 
$a_{0},\ldots ,a_{k-2}$ are nonzero, the polynomials $P_{1}\left( X\right)
,\ldots ,P_{k}\left( X\right) $ are mutually distinct. It remains to show
that there are $a_{0},\ldots ,a_{k-2}\in \mathbb{N}$ for which 
$c_{1}P_{1}(X)+\cdots +c_{k}P_{k}(X)=0$. Since 
$c_{1}+\cdots +c_{k}=0$, the constant and leading terms of $%
c_{1}P_{1}(X)+\cdots +c_{k}P_{k}(X)$ are zero. So the equation $%
c_{1}P_{1}(X)+\cdots +c_{k}P_{k}(X)=0$ is equivalent to the system of
equations $\left( c_{1}+\cdots +c_{k-i}\right) \cdot a_{i-1}+\left(
c_{k-i+2}+\cdots +c_{k}\right) \cdot a_{i-2}$ for $i=1,2,\ldots ,k-1$. One
can then easily define recursively elements $a_{0},a_{1},\ldots ,a_{k-2}$
satisfying all these equations.
\end{proof}

We conclude this section by using the nonstandard characterization of partition regularity to prove the following theorem of Hindman, giving an example of a class of \emph{nonlinear} partition regular equations:

\begin{theorem}\label{hindmanPR}
For any $m,n\in \N$, the equation $x_1+\cdots+x_m-y_1\cdots y_n=0$ is injectively partition regular.
\end{theorem}

The idea of the nonstandard proof is due to L. Luperi Baglini \cite{luperi_baglini_partition_2014}.  The following proposition is the key idea.

\begin{proposition}\label{luperi}
Suppose that $P(X_1,\ldots,X_n)$ is a homogeneous linear equation that is injectively partition regular.  Then there is a \emph{multiplicatively idempotent} $\u\in \beta \N$ (that is, $\u\odot \u=\u$) that witnesses the injective partition regularity of $P$.
\end{proposition}

\begin{proof}
Let $I_P$ be the set of $P$-witnesses.  It suffices to show that $I_P$ is a nonempty, closed subsemigroup of $(\beta \N,\odot)$.  $I_P$ is nonempty by definition.  $I_P$ is closed since
it consists of those ultrafilters whose members $A$ all satisfy the property
$\Phi(A)$ of containing a tuple that is solution of $P$
(see Exercise \ref{ex-Phiclosed}).
Finally, we show that $I_P$ is closed under multiplication.  In fact, we show that $I_P$ is a two-sided ideal.  Suppose that $\u\in I_P$ and $\mathcal{V}\in \beta \N$.  Take distinct $\alpha_1,\ldots,\alpha_n$ such that $\u=\u_{\alpha_i}$ for $i=1,\ldots,n$
and $P(\alpha_1,\ldots,\alpha_n)=0$.  Also let $\beta$ be such that $\mathcal{V}=\u_\beta$.  We then have that $\alpha_1{}^{\ast}\beta, \ldots,\alpha_n{}^{\ast}\beta$ are distinct generators of $\u\odot \mathcal{V}$ and, setting $d$ to be the degree of $P$, we have
$$P(\alpha_1{}^{\ast}\beta,\ldots,\alpha_n{}^{\ast}\beta)={}^{\ast}\beta^dP(\alpha_1,\ldots,\alpha_n)=0.$$  It follows that $\u\odot \mathcal{V}$ belongs to $I_P$.  The proof that $I_P$ is a right-ideal is similar and left to the reader.
\end{proof}

We now prove Theorem \ref{hindmanPR} in the simple case $m=2$ and $n=3$.  Since $x_1+x_2-y=0$ is homogeneous and injectively partition regular by Rado's theorem, Proposition \ref{luperi} implies that we may find a multiplicative idempotent ultrafilter $\u$ that witnesses the injective partition regularity of $\u$.  Take distinct $\alpha_1,\alpha_2,\beta\in \starN$ all of which generate $\u$ and for which $\alpha_1+\alpha_2=\beta$.  For $i=1,2$, set $\gamma_i:=\alpha_i{}^{\ast}\beta^{\ast\ast}\beta$.  Note that $\gamma_1$ and $\gamma_2$ are also distinct generators of $\u$ and $\gamma_1+\gamma_2=\beta{}^{\ast}\beta^{\ast\ast}\beta$, whence $\gamma_1,\gamma_2,\beta,{}^{\ast}\beta,{}^{\ast\ast}\beta$ witness the injective partition regularity of $x_1+x_2-y_1\cdot y_2\cdot y_3=0$.

\section{Non-partition regularity of some equations}

Nonstandard methods have also played a role in establishing the
non-partition regularity of equations. We present here the simplest examples
of this type of result.\index{partition regular}

\begin{theorem}[\protect\cite{di_nasso_fermat-like_2016}]
Let $P(x_{1},\ldots ,x_{h}):=a_{1}x_{1}^{n_{1}}+\cdots +a_{h}x_{h}^{n_{h}}$,
with $n_{1}<\cdots <n_{h}$, where each $a_{i}\in \mathbb{Z}$ is odd and $h$
is odd. Then $P(x_{1},\ldots ,x_{h})=0$ is not partition regular.
\end{theorem}

\begin{proof}
Suppose, towards a contradiction, that there are $u$-equivalent $\xi
_{1},\ldots ,\xi _{h}\in \starN$ such that $P(\xi _{1},\ldots ,\xi _{h})=0$.
Let $f,g:\mathbb{N}\rightarrow \mathbb{N}$ be such that, for all $x\in 
\mathbb{N}$, we have $x=2^{f(x)}g(x)$ with $g(x)$ odd. Then, for each $%
i,j=1,\ldots ,h$, we have $f(\xi _{i})\sim f(\xi _{j})$. Set $\nu
_{i}:=f(\xi _{i})$ and $\zeta _{i}:=g(\xi _{i})$.

We next claim that, for distinct $i,j\in \{1,\ldots,h\}$, we have $%
n_i\nu_i\not=n_j\nu_j$. Indeed, if $n_i\nu_i=n_j\nu_j$, then $%
n_i\nu_i=n_j\nu_j\sim n_j\nu_i$, whence $n_i\nu_i=n_j\nu_i$ by Proposition \ref{simproperties} and
hence $\nu_i=0$. Since the $\nu_k$'s are all $u$-equivalent, it follows that 
$\nu_k=0$ for each $k$, whence each $\xi_i$ is odd. But then since $h$ is
odd, we have that $P(\xi_1,\ldots,\xi_h)$ is odd, contradicting that $%
P(\xi_1,\ldots,\xi_h)=0$.

By the previous paragraph, we can let $i\in \{1,\ldots,k\}$ be the unique
index for which $n_i\nu_i< n_j\nu_j$ for all $j=1,\ldots,k$. By factoring
out $2^{n_i\nu_i}$ from the equation $P(\xi_1,\ldots,\xi_h)=0$, we obtain
the contradiction 
\begin{equation*}
0=a_i\zeta_i^{n_i}+\sum_{j\not=i}a_j2^{n_j\nu_j-n_i\nu_i}\zeta_j^{n_j}\equiv
1 \mod 2.
\end{equation*}
\end{proof}

From the previous theorem, we see that many ``Fermat-like'' equations are
not partition regular:

\begin{corollary}
Suppose that $k,m,n$ are distinct positive natural numbers. Then the
equation $x^m+y^n=z^k$ is not partition regular.
\end{corollary}

In \cite{di_nasso_fermat-like_2016}, the previous corollary is extended to allow $m
$ and $n$ to be equal, in which case the equations are shown to be not
partition regular (as long as, in the case when $m=n=k-1$, one excludes the
trivial solution $x=y=z=2$). The methods are similar to the previous proof.
To further illustrate the methods, we conclude by treating two
simple cases.

The following result was first proven by Csikivari, Gyarmati, and Sarkozy
in \cite{csikvari_density_2012}. The nonstandard proof given below
uses the same argument as in \cite{green_monochromatic_2016}.

\begin{theorem}
If one excludes the trivial solution $x=y=z=2$, then the equation $x+y=z^{2}$
is not partition regular.
\end{theorem}

\begin{proof}
Suppose, towards a contradiction, that $\alpha ,\beta ,\gamma \in \starN%
\setminus \mathbb{N}$ are such that $\alpha \sim \beta \sim \gamma $ and $%
\alpha +\beta =\gamma ^{2}$. 
Without loss of generality, assume $\alpha\ge\beta$.
Let $f:\N\to\N$ be the function defined by
$2^{f(n)}\le n<2^{f(n)+1}$ for every $n$, and set $a:=f(\alpha)$. Notice that $a$ is infinite, as otherwise $\alpha$ would be finite and
$\alpha\sim\beta\sim\gamma$ would imply $\alpha=\beta=\gamma=2$, contrary to our hypothesis.  Now observe that 
$$2^a\le\alpha<\alpha+\beta=\gamma^2\le 2\alpha<2\cdot 2^{a+1}\ \Rightarrow\
2^{\frac{a}{2}}<\gamma<2^{\frac{a}{2}+1}.$$
This shows that either $f(\gamma)=\lfloor\frac{a}{2}\rfloor$ or
$f(\gamma)=\lfloor\frac{a}{2}\rfloor+1$. Since $a=f(\alpha)\sim f(\gamma)$,
we have either $a\sim\lfloor\frac{a}{2}\rfloor$ or $a\sim \lfloor\frac{a}{2}\rfloor+1$.
In both cases we reach a contradiction as we would either have $a=\lfloor\frac{a}{2}\rfloor$ (which cannot occur since $a$ is positive),
or $a=\lfloor\frac{a}{2}\rfloor+1$ (and hence $a=1$ or $a=2$; the former is impossible and the latter has been excluded by hypothesis).
\end{proof}

\begin{theorem}
The equation $x^2+y^2=z$ is not partition regular.
\end{theorem}

\begin{proof}
Notice first that the given equation does not have constant solutions.
Then suppose, towards a contradiction, that $\alpha ,\beta ,\gamma$
are infinite hypernatural numbers \index{hypernatural number} such that
$\alpha \sim \beta \sim \gamma $ and $%
\alpha^2 +\beta^2 =\gamma $. 
Notice that $\alpha,\beta,\gamma$ are even numbers, since
they cannot all be odd. Then we can write
\begin{equation*}
\alpha =2^{a}\alpha _{1},\quad \beta =2^{b}\beta _{1},\quad \gamma
=2^{c}\gamma _{1},
\end{equation*}%
with positive $a\sim b\sim c$ and 
with $\alpha _{1}\sim\beta _{1}\sim\gamma _{1}$ odd.

\

\noindent \textbf{Case 1:} $a<b$.
We then have that 
$2^{2a}(\alpha_1^2+2^{2b-2a}\beta_1^2)=2^{c}\gamma_1$. 
Since $\alpha_1^2+2^{2b-2a}\beta_1^2$ and $\gamma_1$ are odd, it follows
that $2a=c\sim a$, whence $2a=a$ by Proposition \ref{simproperties} and
hence $a=0$, a contradiction. 
If $b>a$ the proof is entirely similar.

\

\noindent \textbf{Case 2:} $a=b$.
In this case we have the equality
$2^{2a}(\alpha_1^2+\beta_1^2)=2^{c}\gamma_1$. 
Since $\alpha_1,\beta_1$ are odd, $\alpha_1^2+\beta_1^2\equiv 2\mod 4$,
and so $2^c\gamma_1=2^{2a+1}\alpha_2$ for a suitable odd number $\alpha_2$.
But then $2a+1=c\sim a$, whence $2a+1=a$, and we again obtain a contradiction.
\end{proof}

Our current knowledge concerning partition regularity of nonlinear equations is still very fragmentary and, unlike the case of linear equations, no general necessary and sufficient conditions are known. An outstanding open problem in the area concerns the partition regularity of the Pythagorean equation $x^2 +y^2 =z^2$. 

\section*{Notes and references}

Rado's theorem is one of the first general results in Ramsey theory \cite{rado_studien_1933}, building on previous work of Hilbert and Rado's advisor Schur. In particular, Rado's Theorem (in its extended
version about systems of equations) subsumes van der Waerden's Theorem on arithmetic progressions. Since then, only fragmented and isolated progress has been obtained in the study of partition regularity of more general (nonlinear) equations. Recently, a breakthrough was obtained in \cite{di_nasso_iterated_2015,di_nasso_hypernatural_2015,luperi_baglini_nonstandard_2015,di_nasso_ramsey_2016}, where very general necessary criteria for partition regularity of Diophantine equations are obtained using nonstandard methods and iterated hyperextensions, as well as sufficient criteria using algebra in the Stone-\v{C}ech compactification. The study of partition regularity of Diophantine equations can be seen as a particular instance of the more general problem of establishing the partition regularity of arbitrary configurations. One outstanding such problem, recently settled positively by Moreira using topological dynamics \cite{moreira_monochromatic_2017}, was the problem of partition regularity of the configuration $\{ x,x+y,xy\}$ in $\mathbb{N}$. It is still unknown at the time of writing whether the configuration $\{ x, y, x+y,xy\}$ in $\mathbb{N}$ is partition regular.

%
%
%

\part{Combinatorial Number Theory}

\chapter{Densities and structural properties}
\label{chapter_densitites}
\section{Densities}

In this section, $A$ and $B$ denote subsets of $\N$.
Recall that $\delta(A,n)=\frac{|A\cap [1,n]|}{n}$.

\begin{definition}

\

\begin{enumerate}
\item The \emph{upper density }\index{density!upper} of $A$ is defined to be
$$\overline{d}(A):=\limsup_{n\to \infty}\delta(A,n).$$ 
\item The \emph{lower density } of $A$ is defined to be
$$\underline{d}(A):=\liminf_{n\to \infty}\delta(A,n).$$
\item If $\overline{d}(A)=\underline{d}(A)$, then we call this common value the \emph{density of $A$} and denote it by $d(A)$.
\end{enumerate}
\end{definition}

The following exercise concerns the nonstandard characterizations of the aforementioned densities.
\begin{exercise}
Prove that 
$$\overline{d}\left( A\right)=\max\{\st(\delta(A,N)) \ : \ N\in \starN\setminus \N\} =\max\{\mu_N(\starA) \ : \ N\in \starN\setminus \N\},$$
where $\mu_N$ is the Loeb measure on $[1,N]$.
State and prove the corresponding statement for lower density.
\end{exercise}

The previous exercise illustrates why the nonstandard approach to densities is so powerful.  Indeed, while densities often ``feel'' like measures, they lack some of the key properties that measures possess.  However, the nonstandard approach allows us to treat densities as measures, thus making it possible to use techniques from measure theory and ergodic theory.

There is something artificial in the definitions of upper and lower density in that one is always required to take samples from initial segments of the natural numbers.  We would like to consider a more uniform notion of density which allows one to consider sets that are somewhat dense even though they do not appear to be so when considering only initial segments.  This leads us to the concept of (upper) Banach density.  In order to defined Banach density, we first need to establish a basic lemma from real analysis, whose nonstandard proof is quite elegant.

\begin{lemma}[Fekete]
Suppose that $\left( a_{n}\right) $ is a \emph{subadditive} sequence of positive real numbers, that is, $a_{m+n}\leq a_{m}+a_{n}$ for all $m,n$. Then the sequence $%
\left( \frac{1}{n}a_{n}\right) $ converges to $\inf \left\{ \frac{1}{n}%
a_{n}:n\in \mathbb{N}\right\} $.
\end{lemma}

\begin{proof}
After normalizing, we may suppose that $a_{1}=1$. This implies that $\frac{1%
}{n}a_{n}\leq 1$ for every $n\in \mathbb{N}$. Set $\ell:=\inf\{%
\frac{1}{n}a_{n} \ : \ n\in \N\}$. By transfer, there exists $\nu
_{0}\in {}^{\ast }\mathbb{N}$ infinite such that $\frac{1}{\nu _{0}}a_{\nu
_{0}}\thickapprox \ell $. Furtermore $\st\left( \frac{1}{\nu }a_{\nu }\right)\geq\ell$ for every $\nu \in {}^{\ast }\mathbb{N}$. Fix an infinite $%
\mu \in {}^{\ast }\mathbb{N}$ and observe that for $\nu \geq \mu \nu _{0}$
one can write $\nu =r\nu _{0}+s$ where $r\geq \mu $ and $s<\nu _{0}$.
Therefore%
\begin{equation*}
\frac{1}{\nu }a_{\nu }\leq \frac{ra_{\nu _{0}}+a_{s}}{r\nu _{0}+s}\leq \frac{%
a_{\nu _{0}}}{\nu _{0}}+\frac{a_{s}}{\mu s}\leq \frac{a_{\nu _{0}}}{\nu _{0}}%
+\frac{1}{\mu }\thickapprox \frac{a_{\nu _{0}}}{\nu _{0}}\thickapprox \ell 
\text{.}
\end{equation*}%
It follows that $\frac{1}{\nu }a_{\nu }\thickapprox \ell $ for every $\nu
\geq \mu \nu _{0}$, whence by transfer we have that, for every $\varepsilon >0$,
there exists $n_{0}\in \mathbb{N}$ such that $\left\vert \frac{1}{n}%
a_{n}-\ell \right\vert <\varepsilon $ for every $n\geq n_{0}$. Therefore the
sequence $\left( \frac{1}{n}a_{n}\right) $ converges to $\ell $.
\end{proof}

For each $n$, set $$\Delta_n(A):=\max\{\delta(A,I) \ : \ I\subseteq \N \text{ is an interval of length }n\}.$$  It is straightforward to verify that $(\Delta_n(A))$ is subadditive, whence, by Fekete's Lemma, we have that the sequence $(\Delta_n(A))$ converges to $\inf_n \Delta_n(A)$.

\begin{definition}
We define the \emph{Banach density of $A$} to be
$$\BD(A)=\lim_{n\to \infty} \Delta_n(A)=\inf_n\Delta_n(A).$$
\end{definition}

\begin{remark}
Unlike upper and lower densities, the notion of Banach density actually makes sense in any amenable (semi)group, although we will not take up this direction in this book.
\end{remark}

If $(I_n)$ is a sequence of intervals in $\N$ such that $\lim_{n\to \infty}|I_n|=\infty$ and $\BD(A)=\lim_{n\to \infty}\delta(A,I_n)$, then we say that $(I_n)$ \emph{witnesses the Banach density of $A$}.

Here is the nonstandard characterization of Banach density:



\begin{exercise}
For any $N\in \starN\setminus \N$, we have
$$\BD(A)=\max\{\st(\delta(\starA,I)) \ : \ I\subseteq \starN 
\text{ is an interval of length }N\}.$$
\end{exercise}

As above, if $I$ is an infinite hyperfinite interval such that $\BD(A)=\st(\delta(A,I))$, we also say that $I$ \emph{witnesses the Banach density of $A$}.

\begin{exercise}
Give an example of a set $A\subseteq \N$ such that $\overline{d}(A)=0$ but $\BD(A)=1$.
\end{exercise}

\begin{exercise}
Prove that Banach density is translation-invariant:  $\BD(A+n)=\BD(A)$,
where $A+n=\{a+n : a\in A\}$.
\end{exercise}

Banach density is also subadditive:
\begin{proposition}\label{BDsubadd}
For any $A,B\subseteq \N$, we have $\BD(A\cup B)\leq \BD(A)+\BD(B)$.  
\end{proposition}

\begin{proof}
Let $I$ be an infinite hyperfinite interval witnessing the Banach density of $A\cup B$.  Then
$$\BD(A\cup B)=\st(\delta(A\cup B,I))\leq \st(\delta(A,I))+\st(\delta(B,I))\leq \BD(A)+\BD(B).$$  
\end{proof}
The following ``fattening'' result is often useful.

\begin{proposition}\label{fat}
If $\BD(A)>0$, then $\lim_{k\to \infty}\BD(A+[-k,k])=1$.
\end{proposition}

\begin{proof}
Set $r:=\BD(A)$.  For each $k$, set $a_k:=\max_{x\in \mathbb{N}}|A\cap [x+1,x+k]|$, so $r=\lim_{k\to \infty}a_k/k$.  By the Squeeze Theorem, it suffices to show that $\BD(A+[-k,k])\geq \frac{r\cdot k}{a_k}$ for all $k$.  Towards this end, fix $k\in \N$ and $N\in \starN\setminus \N$ and take $x\in \starN$ such that $s:=|{}^{\ast}A\cap [x+1,x+N\cdot k]|/N \cdot k\approx r$.  For $i=0,1,\ldots,N-1$, set $J_i:=[x+ik+1,x+(i+1)k]$.  Set $\Lambda:=\{i \ | \ {}^{\ast}A\cap J_i\not=\emptyset\}$; observe that $\Lambda$ is internal.  We then have
$$s=\frac{|{}^{\ast}A\cap [x+1,x+N\cdot k]|}{N \cdot k}=\frac{\sum_{i\in \Lambda}|{}^{\ast}A\cap J_i|}{N\cdot k}\leq \frac{|\Lambda|\cdot a_k}{N \cdot k},$$ whence we can conclude that $|\Lambda|\geq s\cdot N \cdot k/a_k$.  Now note that if $i\in \Lambda$, then $J_i\subseteq {}^{\ast}A+[-k,k]$, so
$$\frac{|({}^{\ast}A+[-k,k])\cap [x+1,x+N\cdot k]|}{N \cdot k}\geq \frac{|\Lambda|\cdot k}{N \cdot k}\geq s\cdot k/a_k.$$  It follows that $\BD(A+[-k,k])\geq r \cdot k/a_k$.
\end{proof}

%



\section{Structural properties}

We now move on to consider structural notions of largeness.  In this section, $A$ continues to denote a subset of $\N$.

\begin{definition}
$A$ is \emph{thick}\index{thick} if and only if $A$ contains arbitrarily long intervals.  
\end{definition}

\begin{proposition}
$A$ is thick\index{thick} if and only if there is an infinite hyperfinite interval $I$ contained in $\starA$.
\end{proposition}

\begin{proof}
The backwards direction follows directly from transfer.  The forwards direction follows from the overflow principle\index{overflow principle} applied to the 
internal set $\{\alpha\in \starN \ : \ \starA \text{ contains an interval of length }\alpha\}$.
\end{proof}

\begin{corollary}
$A$ is thick\index{thick} if and only if $\BD(A)=1$.
\end{corollary}

\begin{proof}
The forwards direction is obvious.  For the backwards direction, let $N\in \starN$ be divisible by all elements of $\N$ and let $I$ be a hyperfinite interval of length $N$ witnessing the Banach density of $A$.  If $A$ is not thick\index{thick}, then there is $m$ such that $m\mid N$ and $A$ does not contain any intervals of length $m$.  Divide $I$ into $N/m$ many intervals of length $m$.  By transfer, each such interval contains an element of $\starN\setminus \starA$.  Thus
$$\BD(A)=\st(\delta(A,I))\leq \st\left(\frac{N-N/m}{N}\right)=1-1/m.$$ 
\end{proof}

\begin{definition}
$A$ is \emph{syndetic}\index{syndetic} if $\N\setminus A$ is not thick\index{thick}.
\end{definition}

Equivalently, $A$ is syndetic\index{syndetic} if there is $m$ such that all gaps of $A$ are of size at most $m$.
\begin{proposition}
$A$ is syndetic\index{syndetic} if and only if all gaps of $\starA$ are finite. 
\end{proposition}

\begin{proof}
The forward direction is immediate by transfer.  For the backwards direction, consider the set $$X:=\{\alpha\in \starN \ : \ \text{all gaps of }\starA \text{ are of size at most }\alpha\}.$$  By assumption, $X$ contains all elements of $\starN\setminus \N$, so by underflow, there is $m\in X\cap \N$.  In particular, all gaps of $A$ are of size at most $m$.
\end{proof}

\begin{definition}
$A$ is \emph{piecewise syndetic}\index{syndetic!piecewise} if there is a finite set $F\subseteq \N$ such that $A+F$ is thick\index{thick}.
\end{definition}

\begin{proposition}
If $A$ is piecewise syndetic\index{syndetic!piecewise}, then $\BD(A)>0$.  More precisely, 
if $F$ is a finite set such that $A+F$ is thick, then $\BD(A)\geq 1/|F|$.

\end{proposition}

\begin{proof}
Take finite $F\subseteq \N$ such that $A+F$ is thick\index{thick}.  Since Banach density is translation invariant, by Proposition \ref{BDsubadd}, we have
$$1=\BD(\N)=\BD(\bigcup_{x\in F}(A+x))\leq |F|\cdot \BD(A).$$
\end{proof}

The notion of being piecewise syndetic\index{syndetic!piecewise} is very robust in that it has many interesting reformulations:

\begin{proposition}
For $A\subseteq \N$, the following are equivalent:
\begin{enumerate}
\item $A$ is piecewise syndetic\index{syndetic!piecewise};
\item there is $m\in \N$ such that $A+[0,m]$ is thick\index{thick};
\item there is $k\in \N$ such that for every $N>\N$, there is a hyperfinite interval $I$ of length $N$ such that $\starA$ has gaps of size at most $k$ on $I$;
\item for every $N>\N$, there is a hyperfinite interval $I$ of length $N$ such that all gaps of $\starA$ on $I$ are finite;
\item there is $k\in \N$ and there is an infinite hyperfinite interval $I$ such that $\starA$ has gaps of size at most $k$ on $I$;
\item there is an infinite hyperfinite interval $I$ such that all gaps of $\starA$ on $I$ are finite;
\item there is $k\in \N$ such that, for every $n\in \N$, there is an interval $I\subseteq \N$ of length $n$ such that the gaps of $A$ on $I$ are of size at most $k$;
\item there is a thick\index{thick} set $B$ and a syndetic\index{syndetic} set $C$ such that $A=B\cap C$.
\end{enumerate}
\end{proposition}

\begin{proof}
Clearly (1) and (2) are equivalent and (3) implies (4).  Now assume that (3) fails.  In particular, if $X$ is the set of $k\in \starN$ for which there is a hyperfinite interval $I$ of length greater than $k$ on which $\starA$ has gaps of size greater than $k$, then $X$ contains all standard natural numbers.  By \index{overflow principle}, there is an infinite element of $X$, whence (4) fails.  Thus, (3) and (4) are equivalent.  (5) clearly implies (6) and (6) implies (5) follows from a familiar underflow argument.  (5) and (7) are also equivalent by transfer-overflow.

We now show (2) impies (3).  Fix $N>\N$.  By (2) and transfer, there is an interval $[x,x+N)\subseteq \starA+[0,m]$.  Thus, on $[x,x+N)$, $\starA$ has gaps of size at most $m$.

Clearly $(3)\Rightarrow (5)$.  Now suppose that (5) holds.  Choose $k\in \N$ and $M,N\in \starN$ such that $M<N$ and $N-M>\N$ such that $\starA$ has gaps of size at most $k$ on $[M,N]$.  Then $[M+k,N]\subseteq \starA+[0,k]$.  It follows by transfer that $A+[0,k]$ is thick\index{thick}, whence (2) holds.

Thus far, we have proven that (1)-(7) are equivalent.  Now assume that (7) holds and take $k\in \N$ and intervals $I_n\subseteq \N$ of length $n$ such that $A$ has gaps of size at most $k$ on each $I_n$.  Without loss of generality, the $I_n$'s are of distance at least $k+1$ from each other.  Let $B:=A\cup \bigcup_n I_n$ and let $C:=A\cup (\N\setminus B)$.  Clearly $B$ is thick\index{thick}.  To see that $C$ is syndetic\index{syndetic}, suppose that $J$ is an interval of size $k+1$ disjoint from $C$.  Then $J$ is disjoint from $A$ and $J\subseteq B$, whence $J\subseteq \bigcup_n I_n$.  Since the $I_n$'s are of distance at least $k+1$ from each other, $J\subseteq I_n$ for some $n$.  Thus, $J$ represents a gap of $A$ on $I_n$ of size $k+1$, yielding a contradiction.  It is clear that $A=B\cap C$.

Finally, we prove that (8) implies (7).  Indeed, suppose that $A=B\cap C$ with $B$ thick\index{thick} and $C$ syndetic\index{syndetic}.  Suppose that $k\in \N$ is such that all gaps of $C$ are of size at most $k$.  Fix $n\in \N$ and let $I$ be an interval of length $n$ contained in $B$.  If $J$ is an interval contained in $I$ of size $k+1$, then $J\cap C\not=\emptyset$, whence $J\cap A\not=\emptyset$ and (7) holds.
\end{proof}

Item (7) in the previous proposition explains the name piecewise syndetic\index{syndetic!piecewise}.  The following is not obvious from the definition:

\begin{corollary}
The notion of being piecewise is partition regular, meaning that if $A$ is piecewise syndetic\index{syndetic!piecewise} and $A=A_1\sqcup A_2$, then $A_i$ is piecewise syndetic for some $i=1,2$.
\end{corollary}

\begin{proof}
Suppose that $I$ is an infinite hyperfinite interval such that all gaps of $\starA$ on $I$ are finite.  Suppose that $I$ does not witness that $A_1$ is piecewise syndetic.  Then there is an infinite hyperfinite interval $J\subseteq I$ such that $J\cap \starA_1=\emptyset$.  It then follows that any gap of ${}^{\ast}A_2$ on $J$ must be finite, whence $J$ witnesses that $A_2$ is piecewise syndetic\index{syndetic!piecewise}.
\end{proof}

\begin{remark}
We note that neither thickness nor syndeticity are partition regular notions.  Indeed, if $A$ is the set of even numbers and $B$ is the set of odd numbers, then neither $A$ nor $B$ is thick\index{thick} but their union certainly is.  For syndeticity, let $(x_n)$ be the sequence defined by $x_1=1$ and $x_{n+1}:=x_n+n$.  Set $C:=\bigcup_{n \text{ even}}[x_n,x_n+n)$ and $D:=\bigcup_{n \text{ odd}}[x_n,x_n+n)$.  Then neither $C$ nor $D$ are syndetic\index{syndetic} but their union is $\N$, a syndetic set.
\end{remark}

The following is a nice consequence of the partition regularity of the notion of piecewise syndetic\index{syndetic!piecewise}.

\begin{corollary}
van der Waerden's theorem is equivalent to the statement that piecewise syndetic sets contain arbitrarily long arithmetic progressions.
\end{corollary}

\begin{proof}
First suppose that van der Waerden's theorem holds and let $A$ be a piecewise syndetic set.  Fix $k\in \N$; we wish to show that $A$ contains an arithmetic progression\index{arithmetic progression} of length $k$.  Take $m$ such that $A+[0,m]$ is thick\index{thick}.  Let $l$ be sufficiently large such that when intervals of length $l$ are partitioned into $m+1$ pieces, then there is a monochromatic arithmetic progression of length $k$.  Let $I\subseteq A+[0,m]$ be an interval of length $l$.  Without loss of generality, we may suppose that the left endpoint of $I$ is greater than $m$.  Let $c$ be the coloring\index{coloring} of $I$ given by $c(x):=$ the least $i\in [0,m]$ such that $x\in A+i$.  Then there is $i\in [0,m]$ and $x,d$ such that $x,x+d,\ldots,x+(k-1)d\in A+i$.  It follows that $(x-i),(x-i)+d,\ldots,(x-i)+(k-1)d\in A$.    

Conversely, suppose that piecewise syndetic sets\index{syndetic!piecewise} contain arbitrarily long arithmetic progressions\index{arithmetic progression}.  Fix a finite coloring $c$ of the natural numbers.  Since being piecewise syndetic is partition regular, some color is piecewise syndetic, whence contains arbitrarily long arithmetic progressions by assumption.
\end{proof}

\section{Working in $\Z$}

We now describe what the above densities and structural properties mean in the group $\Z$ as opposed to the semigroup $\N$.  Thus, in this section, $A$ now denotes a subset of $\Z$.

It is rather straightforward to define the appropriate notions of density.  Indeed, given any sequence $(I_n)$ of intervals in $\mathbb{Z}$ with $\lim_{n\to \infty}|I_n|=\infty$, we define
$$\overline{d}_{(I_n)}:=\limsup_{n\to \infty}\delta(A,I_n)$$ and
$$\underline{d}_{(I_n)}:=\liminf_{n\to \infty}\delta(A,I_n).$$

When $I_n=[-n,n]$ for each $n$, we simply write $\overline{d}(A)$ (resp. $\underline{d}(A)$) and speak of the \emph{upper} (resp. \emph{lower}) density of $A$.  Finally, we define the \emph{upper Banach density of $A$} to be 
$$\BD(A)=\lim_{n\to \infty}\max_{x\in \N}\delta(A,[x-n,x+n]).$$
Of course, one must verify that this limit exists, but this is proven in the exact same way as in the case of subsets of $\N$.

\begin{exercise}
Prove that 
$$\BD(A):=\max\{\overline{d}_{(I_n)}(A) \ : \ (I_n) \text{ a sequence of intervals with }\lim_{n\to \infty}|I_n|=\infty\}.$$
\end{exercise}

The notions of thickness and syndeticity for subsets of $\Z$ remains unchanged:  $A$ is thick\index{thick} if $A$ contains arbitrarily long intervals and $A$ is syndetic if $\Z\setminus A$ is not thick\index{thick}.  Similarly, $A$ is piecewise syndetic if there is a finite set $F\subseteq \Z$ such that $A+F$ is thick\index{thick}.  The following lemma is almost immediate:

\begin{lemma}\label{pwsZ}
$A$ is piecewise syndetic\index{syndetic!piecewise} if and only if there is a finite set $F\subseteq \Z$ such that, for every finite $L\subseteq \Z$, we have $\bigcap_{x\in L}(A+F+x)\not=\emptyset$.
\end{lemma}

\begin{exercise}
Formulate and verify all of the nonstandard equivalents of the above density and structural notions developed in the previous two sections for subsets of $\Z$.
\end{exercise}

The following well-known fact about difference sets has a nice nonstandard proof.

\begin{proposition}\label{differencesyndetic}
Suppose that $A\subseteq \mathbb{Z}$ is such that $\BD(A)>0$.  Then $A-A$ is syndetic.  In fact, if $\BD(A)=r$, then there is a finite set $F\subseteq \mathbb{Z}$ with $|F|\leq \frac{1}{r}$ such that $(A-A)+F=\mathbb{Z}$.
\end{proposition}

First, we need a lemma.

\begin{lemma}\label{deltasetnonstandard}
Let $N\in\starN\setminus\N$.
Suppose that $E\subseteq [1,N]$ is an internal set such that $\delta(E,N)\approx r$.  Then there is a finite $F\subseteq \mathbb{Z}$ with $|F|\leq 1/r$ such that $\mathbb{Z}\subseteq (E-E)+F$.
\end{lemma}

\begin{proof}
Fix $x_1\in \N$.  If $\mathbb{Z}\subseteq (E-E)+x_1$, then take $F=\{x_1\}$.  Otherwise, take $x_2\notin (E-E)+\{x_1\}$.  If $\mathbb{Z}\subseteq (E-E)+\{x_1,x_2\}$, then take $F=\{x_1,x_2\}$.  Otherwise, take $x_3\notin (E-E)+\{x_1,x_2\}$.

Suppose that $x_1,\ldots,x_k$ have been constructed in this fashion.  Note that the sets $E+x_i$, for $i=1,\ldots,k$, are pairwise disjoint.  Since each $x_i\in \mathbb{Z}$ and $N$ is infinite,
we have that $\delta((E+x_i),N)\approx r$.  It follows that 
$$\delta\left(\bigcup_{i=1}^k (E+x_i),N\right)=\frac{\sum_{i=1}^k|(E+x_i)\cap [1,N]|}{N}\approx kr.$$  It follows that the process must stop after $k$-many steps,
with $k\leq \frac{1}{r}$.
\end{proof}

\begin{proof}[of Proposition \ref{differencesyndetic}]
Set $r:=\BD(A)$.  Fix and infinite $N$ and take $x\in \starN$ such that $\delta({}^{\ast}A,[x+1,x+N])\approx r$.  Set $E:=({}^{\ast}A-x)\cap [1,N]$.  Then $\delta(E,N)\approx r$, whence there is finite $F\subseteq \mathbb{Z}$ with $|F|\leq 1/r$ such that $\mathbb{Z}\subseteq (E-E)+F$.  It follows that $\mathbb{Z}\subseteq ({}^{\ast}A-{}^{\ast}A)+F$, whence it follows by transfer that $\mathbb{Z}=(A-A)+F$.
\end{proof}

The analog of Proposition \ref{fat} for $\Z$ is also true:

\begin{proposition}\label{fatZ}
If $\BD(A)>0$, then $\lim_{k\to \infty}\BD(A+[-k,k])=1$.
\end{proposition}

However, for our purposes in Section \ref{quant}, we will need a more precise result.  Note that, a priori, for every $\epsilon>0$, there is $k_\epsilon$ and infinite hyperfinite interval $I_\epsilon$ such that $\delta(\starA+[-k_\epsilon,k_\epsilon],I_\epsilon)>1-\epsilon$.  The next proposition tells us that we can take a single interval $I$ to work for each $\epsilon$.  The proof is heavily inspired by the proof of \cite[Lemma 3.2]{beiglbock_sumset_2010}.

\begin{proposition}\label{goodinterval}
Suppose that $\BD(A)>0$.  Then there is an infinite hyperfinite interval $I\subseteq \Z$ such that, for every $\epsilon>0$, there is $k$ for which $\delta(\starA+[-k,k],I)>1-\epsilon$.
\end{proposition}

\begin{proof}
Pick $(I_n)$ a sequence of intervals in $\Z$ winessing the Banach density of $A$ and such that, for every $k$, we have that $\lim_{n\to \infty}\delta(A+[-k,k],I_n)$ exists.  (This is possible by a simple diagonalization argument.)  Fix an infinite $N$ and, for each $\alpha\in \starN$, set $G_\alpha:=(\starA+[-\alpha,\alpha])\cap I_N$.  Set $r:=\sup_{k\in \N} \mu_{I_N}(G_k)$.

\

\noindent \textbf{Claim:}  There is $K>\N$ such that:
\begin{enumerate}
\item[(i)]  For every $l\in \Z$, $\frac{|(l+G_K)\triangle G_K|}{|G_K|}\approx 0$.
\item[(ii)]  $\frac{|G_K|}{|I_N|}\approx r$.
\end{enumerate}

\

\noindent \textbf{Proof of Claim:}  For each $l\in \Z$, set $X_l$ to be the set of $\alpha\in \starN$ such that:
\begin{enumerate}
\item[(a)]  $\alpha\geq l$;
\item[(b)]  For all $x\in \Z$ with $|x|\leq l$, we have $\frac{|(x+G_\alpha)\triangle G_\alpha|}{|G_\alpha|}<\frac{1}{l}$;
\item[(c)]  $\left|\frac{|G_\alpha|}{|I_N|}-r\right|<\frac{1}{l}$.
\end{enumerate}
Since each $X_l$ is internal and unbounded in $\N$, by saturation there is $K\in \bigcap_l X_l$.  This $K$ is as desired.

\

Fix $K$ as in the Claim and set $G:=G_K$ and $\mu:=\mu_G$.  For $k\in \N$, we then have that
$$\delta(\starA+[-k,k],G)=\frac{|(\starA+[-k,k])\cap I_N|}{|G|}\approx \delta(\starA+[-k,k],I_N)\cdot \frac{1}{r},$$ whence we see that $\delta(\starA+[-k,k],G)\to 1$ as $k\to \infty$.\footnote{At this point, we may note that $G$ satisfies the conclusion of the proposition except that it is not an interval but instead a \emph{Folner approximation} for $\Z$.  While this would suffice for our purposes in Section \ref{quant}, we wanted to avoid having to introduce the theory of Folner approximations and instead opted to work a bit harder to obtain the above cleaner statement.}   

Now take $J$ to be an infinite hyperfinite interval such that $\frac{|(l+G_K)\triangle G_K|}{|G_K|}\approx 0$ for all $l\in J$; this is possible as a consequence of the
overflow principle\index{overflow principle}.  We claim that there is $t\in G$ such that $I:=t+J$ is as desired.  

For each $k$, take $n_k$ such that $\delta(\starA+[-n_k,n_k],G)>1-\frac{1}{k}$; without loss of generality, we may assume that $(n_k)$ is an increasing sequence.  Set $B_k:=\starA+[-n_k,n_k]$ and set $g_k:G\to [0,1]$ to be the $\l_G$-measurable function given by $g_k(t):=\st(\delta(B_k,t+J))$.  For each $t\in G$, we have that $(g_k(t))$ is a bounded nondecreasing sequence, whence converges to a limit $g(t)$.  By the Dominated Convergence Theorem, we have that $\int_G g(t)d\mu=\lim_{k\to \infty}\int_G g_k(t)d\mu$.  Now note that
$$\int_G g_k(t)d\mu\approx \frac{1}{|G|}\sum_{t\in G}\delta(B_k,t+J)=\frac{1}{|I|}\sum_{x\in J}\delta(B_k,x+G)\approx \delta(B_k,G)>1-\frac{1}{k}.$$  It follows that $\int_G g(t)d\mu=1$, whence $g(t)=1$ for some $t\in G$.  It is then clear that $I:=t+J$ is as desired.
\end{proof}

We call $I$ as in the conclusion of Proposition \ref{goodinterval} \emph{good for $A$}.  One can also prove the previous proposition using a Lebesgue Density Theorem for cut spaces; see \cite{di_nasso_high_2015}.

\section{Furstenberg's Correspondence Principle}
We end this chapter by explaining the nonstandard take on Furstenberg's correspondence principle.


\begin{theorem}[Furstenberg's Correspondence Principle]
Suppose that $A\subseteq \mathbb{Z}$ is such that $\BD(A)>0$.  Then there is a measure-preserving dynamical system $(X,\mathcal{B},\nu,T)$ and a measurable set $A_0\in \mathcal{B}$ such that $\nu(A_0)=\BD(A)$ and such that, for any finite set $F\subseteq \mathbb{Z}$, we have:
$$\BD\left(\bigcap_{i\in F} (A-i)\right)\geq \nu\left(\bigcap_{i\in F}T^{-i}(A_0)\right).$$
\end{theorem}

\begin{proof}
Fix $I\subseteq {}^{\ast} \mathbb{Z}$ witnessing the Banach density of $A$.  It is easy to verify that the hypercycle system $(I,\Omega,\mu,S)$ introduced in Section \ref{hypercycle} of Chapter \ref{Loeb} and the set $A_0:=\starA \cap I$ are as desired.
\end{proof}

Let us mention the ergodic-theoretic fact that Furstenberg proved\index{Furstenberg's Multiple Recurrence Theorem}:

\begin{theorem}[Furstenberg Multiple Recurrence Theorem]
Suppose that $(X,\mathcal{B},\nu,T)$ is a measure-preserving dynamical system, $A\in \mathcal{B}$ is such that $\nu(A)>0$, and $k\in \N$ is given.  Then there exists $n\in \N$ such that $\nu(A\cap T^{-n}(A)\cap T^{-2n}(A)\cap \cdots \cap T^{-(k-1)n}(A))>0$.
\end{theorem}

Notice that the above theorem, coupled with the Furstenberg Correspondence Principle\index{Furstenberg's Correspondence Principle}, yields Furstenberg's proof of Szemer\'edi's Theorem \index{Szemer\'edi's Theorem}.

\begin{theorem}[Szemeredi's Theorem]
If $A\subseteq \Z$ is such that $\BD(A)>0$, then $A$ contains arbitrarily long arithmetic progressions\index{arithmetic progression}.
\end{theorem}

Szemeredi's Theorem is the density version of van der Waerden's theorem and was originally proven by Szemeredi in \cite{szemeredi_sets_1975}.

We end this chapter giving a simpler application of the correspondence principle used by Bergelson in \cite{bergelson_density_1986} to give a quantitative version of Schur's Theorem.

Suppose that $c:\N\to \{1,\ldots,m\}$ is an $m$-coloring of $\N$.  Then Schur's theorem states that there is $i\in \{1,\ldots,m\}$ and $a,b\in \N$ such that $c(a)=c(b)=c(a+b)=i$.  (Note that Schur's theorem is an immediate corollary of Rado's Theorem.)  It is natural to ask whether or not a quantitative Schur's theorem could hold in the sense that there should be some color $C_i$ such that there are many $a,b\in \N$ with $c(a)=c(b)=c(a+b)=i$.  In \cite{bergelson_density_1986}, Bergelson proved the following precise version of that result:

\begin{theorem}\label{bergelsonschur}
Suppose that $c:\N\to \{1,\ldots,m\}$ is an $m$-coloring of $\N$ and $C_i:=\{n\in \N \ : \ c(n)=i\}$.  For  $i\in \{1,\ldots,n\}$ and $\epsilon>0$, set
$$R_{i,\epsilon}:=\{n\in C_i \ : \ \overline{d}(C_i\cap (C_i-n))\geq \overline{d}(C_i)^2-\epsilon\}.$$  
Then there is $i\in \{1,\ldots,n\}$ such that, for every $\epsilon>0$, we have $\overline{d}(R_{i,\epsilon})>0$.
\end{theorem}

We more or less follow Bergelson's original proof except we use the nonstandard version of the Furstenberg correspondence principle.  

\begin{definition}
We call $R\subseteq \N$ a \emph{set of nice recurrence} if:  given any dynamical system $(X,\mathcal{B},\mu,T)$, any $\mu(B)>0$, and any $\epsilon>0$, there is $n\in R$ such that $\mu(A\cap T^{-n}A)\geq \mu(A)^2-\epsilon$.
\end{definition}

\begin{proposition}\label{differencenice}
Let $S\subseteq \N$ be an infinite set.  Then $S-S$ is a set of nice recurrence.
\end{proposition}

\begin{proof}
Let $(s_i)$ be an enumeration of $S$ in increasing order.  It is straightforward to check that there must exist $i<j$ such that $\mu(T^{-s_i}A\cap T^{-s_j}A)\geq \mu(A)^2-\epsilon$.  It follows that $\mu(A\cap T^{-(s_j-s_i)}A)\geq \mu(A)^2-\epsilon$, as desired.
\end{proof}

\begin{exercise}\label{thickdifference}
If $E\subseteq \N$ is thick, then there is an infinite set $S\subseteq \N$ such that $S-S\subseteq E$.
\end{exercise}

\begin{corollary}\label{thicknice}
Suppose that $E\subseteq \N$ is thick and $E=C_1\cup \cdots \cup C_k$ is a partition of $E$.  Then some $C_i$ is a set of nice recurrence. 
\end{corollary}

\begin{proof}
By Exercise \ref{thickdifference}, we may take $S\subseteq \N$ such that $S-S\subseteq E$.  Define a coloring $c:S\to \{1,\ldots,k\}$ by declaring, for $s,s'\in S$ with $s<s'$, that $c(\{s,s'\}):=i$ if $c(s-s')=i$.  By Ramsey's theorem, there is an infinite $S'\subseteq S$ and $i\in \{1,\ldots,k\}$ such that $c([S']^2)=\{i\}$.  It follows that $S'-S'\subseteq C_i$.  By Proposition \ref{differencenice}, $S'-S'$, and hence $C_i$, is a nice set of recurrence.
\end{proof}

We are now ready to give the proof of Theorem \ref{bergelsonschur}.  First, without loss of generality, we may assume that there is $k\in \{1,\ldots,m\}$ such that $\overline{d}(C_i)>0$ for $i=1,\ldots,k$ and $C_1\cup\cdots\cup C_k$ is thick.  For ease of notation, for $p\in \N$, let $R_{i,p}:=R_{i,1/p}$.  It suffices to show that, for each $p\in \N$, there is $i_p\in \{1,\ldots,k\}$ such that $\overline{d}(R_{i_p,p})>0$.  Indeed, if this is the case, then by the Pigeonhole Principle, there is some $i\in \{1,\ldots,m\}$ such that $i_p=i$ for infinitely many $p$; this $i$ is as desired.

Towards this end, fix $p\in \N$ and, again for ease of notation, set $R_i:=R_{i,p}$.  Suppose, towards a contradiction, that $\overline{d}(R_{i})=0$ for each $i=1,\ldots,k$.  Set $D_i:=C_i\setminus R_i$.  Then $\overline{d}(D_i)=\overline{d}(C_i)$ and $D_1\cup \cdots D_k$ is thick.  By Corollary \ref{thicknice}, there is $i\in \{1,\ldots,k\}$ such that $D_i$ is a nice set of recurrence.  Take $N>\N$ such that $\overline{d}(D_i)=\mu_N({}^{\ast}D_i)$.  By applying the fact that $D_i$ is a nice set of recurrence to the hypercycle system based on $[1,N]$ and the measurable set $A:={}^{\ast}D_i\cap [1,N]$, we get that there is $n\in D_i$ such that 
$$\overline{d}(C_i\cap (C_i-n))\geq \overline{d}(D_i\cap (D_i-n))\geq \mu_N(A\cap T^{-n}A)\geq \mu(A)^2-\epsilon=\overline{d}(C_i)^2-\epsilon,$$  contradicting the fact that $n\notin R_i$.

\section*{Notes and references}  The first appearance of nonstandard methods in connection with densities and structural properties seems to be Leth's dissertation and subsequent article \cite{leth_some_1988}.  Proposition \ref{fat} was first proven by Hindman in \cite{hindman_density_1982}.  Partition regularity of piecewise syndeticity was first proven by Brown in \cite{brown_interesting_1971}.  Proposition \ref{differencesyndetic} was first proven by F\o lner in \cite{folner_generalization_1954}; the nonstandard proof is due to Di Nasso \cite{di_nasso_embeddability_2014}.  Furstenberg's Correspondence Principle was first established in \cite{furstenberg_ergodic_1977} where he gave his ergodic-theoretic proof of Szemer\'edi's theorem.  The nonstandard approach to the Furstenberg Correspondence Principle seems to have a somewhat nebulous history.  Indeed, while it was surely known to many experts that one could use hypercycle systems to prove the Furstenberg Correspondence Principle, the first appearance of this idea in the literature seems to be generalizations of the Furstenberg Correspondence due to Townser appearing in the paper \cite{towsner_convergence_2009}.

\chapter{Working in the remote realm}


\section{Remote realms and finite embeddability}

A useful combinatorial notion is the following:

\begin{definition}\label{def-finiteembeddability}\index{finite embeddability}
Let $X,Y$ be sets of integers.
We say that $X$ is \emph{finitely embeddable} in $Y$, 
and write $X\lhd Y$, if every finite configuration $F\subseteq X$
has a shifted copy $t+F\subseteq Y$.
\end{definition}

Finite embeddability preserves most of the fundamental
combinatorial notions that are commonly considered in combinatorics
of integer numbers.

\begin{proposition}\label{feproperties}
\
\begin{enumerate}
\item
A set is $\lhd$-maximal if and only if it is thick.
\item
If $X$ contains an arithmetic progression of length $k$
and distance $d$ and $X\lhd Y$, then
also $Y$ also contains an arithmetic progression of length $k$
and distance $d$.
\item
If $X$ is piecewise syndetic and $X\lhd Y$, then
also $Y$ is piecewise syndetic.
\item
If $X\lhd Y$, then $\BD(X)\le\BD(Y)$.
\end{enumerate}
\end{proposition}

\begin{proof}
(1). Clearly $X$ is maximal if and only if $\N\lhd X$
if and only if every finite interval $[1,n]$
has a shifted copy $[x+1,x+n]\subseteq X$.
(2) is trivial.  We leave the proofs of (3) and (4) to the reader.
\end{proof}

We stress the fact that while piecewise syndeticity is preserved
under finite embeddability, 
the property of being syndetic is not. Similarly,
the upper Banach density is preserved or increased 
under finite embeddability, but the upper asymptotic 
density is not.
A list of basic properties is itemized below.

\begin{proposition}
\
\begin{enumerate}
\item
If $X\lhd Y$ and $Y\lhd Z$, then $X\lhd Z$.
\item
If $X\lhd Y$ and $X'\lhd Y'$, then $X-X'\lhd Y-Y'$.
\item
If $X\lhd Y$,
then $\bigcap_{t\in G}(X-t)\lhd \bigcap_{t\in G}(Y-t)$ for every finite $G$.
\end{enumerate}
\end{proposition}

\begin{proof}
$(1)$ is straightforward from the definition of $\lhd$.

$(2)$. Given a finite $F\subseteq X-X'$, let $G\subseteq X$ and $G'\subseteq X'$
be finite sets such that $F\subseteq G-G'$. By the hypotheses,
there exist $t,t'$ such that $t+G\subseteq Y$ and
$t'+G'\subseteq Y'$. Then,
$(t-t')+F\subseteq(t+G)-(t'+G')\subseteq Y-Y'$.

$(3)$. Let a finite set $F\subseteq\bigcap_{t\in G}(X-t)$ be given.
Notice that $F+G\subseteq X$, so we can pick an element $w$ such that
$w+(F+G)\subseteq Y$. Then, $w+F\subseteq\bigcap_{t\in G}Y-t$.
\end{proof}

In a nonstandard setting, the finite embeddability $X\lhd Y$
means that $X$ is contained in some ``remote realm"
of the hyper-extension $\starY$.
This notion can be also
characterized in terms of \emph{ultrafilter-shifts}, as defined
by M.~Beiglb\"ock \cite{Beiglbock_Ultrafilter_2011}.

\begin{proposition}\label{fe}
Let $X,Y\subseteq\N$. Then the following are equivalent:
\begin{enumerate}
\item
$X\lhd Y$.
\item
$a+X\subseteq\starY$ for some $a\in\starN$.
\item
There exists an ultrafilter $\U$ on $\N$ such that
$X\subseteq Y-\U$, where the ultrafilter shift
$Y-\U:=\{x : Y-x\in\U\}$.
\end{enumerate}
\end{proposition}

\begin{proof}
$(1)\Rightarrow(2)$. 
Enumerate $X=\{x_n\mid n\in\N\}$.
By the hypothesis, the finite intersection
$\bigcap_{i=1}^n(Y-x_i)\ne\emptyset$. Then, by \emph{overspill},
there exists an infinite $N\in\starN$ such that
$\bigcap_{i=1}^N(\starY-x_i)$ is non-empty.
If $a\in\starN$ is in that intersection,
then clearly $a+x_i\in\starY$ for all $i\in\N$.

$(2)\Rightarrow(3)$.
Let $\U=\U_a$ be the ultrafilter
generated by $a\in\starN$. For every $x\in X$,
by the hypothesis, $a+x\in\starY\Rightarrow a\in{}^*(Y-x)$,
and hence $Y-x\in\U$, \emph{i.e.}, $x\in Y-\U$, as desired.

$(3)\Rightarrow(1)$. Given a finite $F\subseteq X$,
the set $\bigcap_{x\in F}(Y-x)$ is nonempty, because
it is a finite intersection of elements of $\U$.
If $t\in\Z$ is any element in that intersection, then
$t+F\subseteq Y$.
\end{proof}

One can also considers a notion
of \emph{dense embeddability} $X\lhd_d Y$
when every finite configuration $F\subseteq X$
has ``densely-many'' shifted copies included in $Y$,
\emph{i.e.}, if the intersection
$\bigcap_{x\in F}(Y-x)=\{t\in\Z\mid t+F\subseteq Y\}$
has positive upper Banach density (see \cite{di_nasso_embeddability_2014}).
A natural notion of finite embeddability can also
be defined between ultrafilters on $\N$,
by putting $\U\lhd \V$ when for every $B\in\V$ there exists $A\in\U$
with $A\lhd B$ (see \cite{blass_finite_2015}).

\section{Banach density as Shnirelmann density in the remote realm}

The title in this chapter refers to looking at copies of $\N$ starting at some infinite element $a\in \starN$ and then connecting some density of the set of points of this copy of $\N$ that lie in the nonstandard extension\index{nonstandard extension} of a set $A$ and some other density of the original set $A$ itself.  In this regard, given $A\subseteq \N$ and $a\in \starN$, we set $\overline{d}(\starA-a):=\overline{d}((\starA-a)\cap \N)$ and likewise for other notions of density.  We warn the reader that, in general, we do not identify $\starA-a$ and $(\starA-a)\cap \N$ \emph{as sets}, but since we have not defined the density of a subset of $\starN$, our convention should not cause too much confusion. 




The key observation of Renling Jin is that there is a strong converse to item (3) of Proposition \ref{feproperties}.  

\begin{proposition}
Suppose that $A\subseteq \N$ is such that $\BD(A)=r$.  Let $I$ be an interval of infinite hyperfinite length witnessing the Banach density of $A$.  Then for $\mu_I$-almost all $x\in I$, we have $d({}^{\ast}A-x)=r$.  
\end{proposition}

\begin{proof}
Write $I=[H,K]$ and consider the hypercycle system $(I,\l_i,\mu_I,S)$.  Let $f$ denote the characteristic function of ${}^{\ast}A\cap I$.  It follows that, for $x\in I^\#:=\bigcap_{n\in \N}[H,K-n]$, we have that
$$\frac{1}{n}\sum_{m=0}^{n-1}f(S^m(x))=\delta(\starA,[x,x+n-1]).$$ 
By the ergodic theorem for hypercycles (Theorem \ref{ergodichypercycle}), there is a $\l_I$-measurable function $\hat{f}$ such that, for $\mu_I$-almost all $x\in I$, we have that $$\lim_{n\to \infty}\frac{1}{n}\sum_{m=0}^{n-1}f(S^m(x))=\hat{f}(x).$$  Since $I^\#$ is a $\mu_I$-conull set, we will thus be finished if we can show that $\bar{f}$ is $\mu_I$-almost everywhere equal to $r$ on $I^\#$. 

Towards this end, first note that $\hat{f}(x)\leq r$ for $\mu_I$-almost all $x\in I^\#$.  Indeed, if $\hat{f}(x)>r$ for a positive measure set of $x\in I^\#$, then there would be some $x\in I^\#$ with $d({}^{\ast}A-x)>r$, whence $\BD(A)>r$ by transfer, yielding a contradiction.

Next note that, by the Dominated Convergence Theorem, we have that
$$\int_I \hat{f}(x)d\mu_I=\lim_{n\to \infty}\int_I \frac{1}{n}\sum_{m=0}^{n-1}f(S^m(x))d\mu_I=r,$$ where the last equality follows from the fact that $S$ is measure-preserving and that $\int_I \hat{f}(x)d\mu_I=\mu_I({}^{\ast}A)=r$.  By a standard measure theory argument, we have that $\hat{f}(x)=r$ for almost all $x\in I^\#$.
\end{proof}

\begin{remark}
In the context of the previous proposition, since $\mu_I(\starA)>0$, we can conclude that there is $x\in \starA$ such that $d(\starA-x)=r$.
\end{remark}

Summarizing what we have seen thus far:

\begin{theorem}\label{BDequiv}
For $A\subseteq \N$, the following are equivalent:
\begin{enumerate}
\item $\BD(A)\geq r$.
\item There is $B\lhd A$ such that $\underline{d}(A)\geq r$.
\item For any infinite hyperfinite interval $I$ witnessing the Banach density of $A$, we have $d({}^{\ast}A-x)\geq r$ for $\mu_I$-almost all $x\in I$.
\end{enumerate}
\end{theorem}

We now introduce a new notion of density.

\begin{definition}
For $A\subseteq \N$, we define the \emph{Shnirelman density}\index{density!Shnirelman} of $A$ to be
$$\sigma(A):=\inf_{n\geq 1}\delta(A,n).$$
\end{definition}

It is clear from the definition that $\underline{d}(A)\geq \sigma(A)$.  Note that the Shnirelman density\index{density!Shnirelman} is very sensitive to what happens for ``small'' $n$.  For example, if $1\notin A$, then $\sigma(A)=0$.  On the other hand, knowing that $\sigma(A)\geq r$ is a fairly strong assumption and thus there are nice structural results for sets of positive Shnirelman density\index{density!Shnirelman}.  We will return to this topic in the next section.

A crucial idea of Jin was to add one more equivalence to the above theorem, namely that there is $B\lhd A$ such that $\sigma(B)\geq r$; in this way, one can prove Banach density parallels of theorems about Shnirelman density\index{density!Shnirelman}.  To add this equivalence, one first needs a standard lemma.

\begin{lemma}
Suppose that $A\subseteq \N$ is such that $\underline{d}(A)=r$.  Then for every $\epsilon>0$, there is $n_0\in \N$ such that $\sigma(A-n_0)\geq r-\epsilon$.
\end{lemma}

\begin{proof}
Suppose that the lemma is false for a given $\epsilon>0$.  In particular, $\sigma(A)<r-\epsilon$, so there is $n_0\in \N$ such that $\delta(A,n_0)<r-\epsilon$.  Since $n_0$ does not witness the truth of the lemma, there is $n_1\in \N$ such that $\delta((A-n_0),n_1)<r-\epsilon$.  Continuing in this way, we find a sequence $(n_i)$ of natural numbers
such that, for all $i$, we have $\delta(A,[n_0+\ldots+n_i+1,n_0+\ldots+n_i+n_{i+1}])<r-\epsilon$
for all $i$. In consequence, the increasing sequence $(\sum_{i\le k}n_i)$
witnesses that $\underline{d}(A)\leq r-\epsilon$, yielding a contradiction.
\end{proof}

\begin{proposition}\label{Proposition:conversion}
Suppose that $\BD(A)\geq r$.  Then there is $B\lhd A$ such that $\sigma(B)\geq r$.
\end{proposition}

\begin{proof}
We seek $x\in \starN$ such that $\sigma(\starA-x)\geq r$.  Take $y\in \starN$ such that $d({}^{\ast}A-y)\geq r$.  By the previous lemma, for each $n\in \N$, there is $z_n\in \starN$ with $z_n\geq y$ such that $\sigma({}^{\ast}A-z_n)\geq r-1/n$.  By overflow\index{overflow principle}, for each $n\in \N$, there is infinite $K_n\in \starN$ such that, for each $m\leq K_n$, we have
$$\delta(({}^{\ast}A-z_n),m)\geq r-1/n.$$
Take an infinite $K\in \starN$ such that $K\leq K_n$ for each $n$.  
(This is possible by countable saturation.)
Let $$D:=\{\alpha\in \starN \ : \ (\exists z\in \starN)(\forall m\leq K)\delta(({}^{\ast}A-z),m)\geq r-1/\alpha\}.$$  Then $D$ is internal and $\N\subseteq D$, whence by 
overflow\index{overflow principle} there is infinite $N\in D$.  Take $x\in \starN$ such that $\delta(({}^{\ast}A-x),m)\geq r-1/N$ for all $m\leq N$.  In particular, for all $m\in \N$, we have $\delta(({}^{\ast}A-x),m)\geq r$, whence this $x$ is as desired.
\end{proof}

Theorem \ref{BDequiv} and Proposition \ref{Proposition:conversion} immediately yield:

\begin{corollary}
$\BD(A)\geq r$ if and only if there is $B\lhd A$ such that $\sigma(B)\geq r$.
\end{corollary}

We end this section with a curious application of Proposition \ref{Proposition:conversion}.  We will make more serious use of this technique in the next section.

\begin{proposition}
Szemeredi's Theorem is equivalent to the following (apparently weaker statement):  
There exists $\epsilon>0$ such that every set $A\subseteq\N$
with $\sigma(A)\geq 1-\epsilon$ contains arbitrarily long arithmetic progressions.
\end{proposition}

\begin{proof}
Fix $A\subseteq \N$ with $\BD(A)>0$; we wish to show that $A$ contains arbitrarily long arithmetic progressions.  By Proposition \ref{fat}, there is $k\in \N$ such that $\BD(A+[0,k])\geq 1-\epsilon$.  If $A+[0,k]$ contains arbitrarily long arithmetic progressions, then by van der Waerden's theorem, there is $i\in [0,k]$ such that $A+i$ contains arbitrarily long arithmetic progressions, whence so does $A$.  It follows that we may assume that $\BD(A)\geq 1-\epsilon$.

By Proposition \ref{Proposition:conversion}, we have $B\lhd A$ such that $\sigma(B)\geq 1-\epsilon$, whence, by assumption, we have that $B$ contains arbitrarily long arithmetic progressions\index{arithmetic progression}, and hence so does $A$.
\end{proof}

%
%
%

\section{Applications}

We use the ideas from the preceding section to derive some Banach density versions of theorems about Shnirelman density\index{density!Shnirelman}.  We first recall the following result of Shnirleman (see, for example, \cite[page 8]{halberstam_sequences_1983}):

\begin{theorem}
Suppose that $A\subseteq \N_0$ is such that $0\in A$ and $\sigma(A)>0$.  Then $A$ is a \emph{basis}, that is, there is $h\in \N$ such that $\Sigma_h(A)=\mathbb{N}$.
\end{theorem}

Using nonstandard methods, Jin was able to prove a Banach density version of the aforementioned result:

\begin{theorem}\label{banachbasis}
Suppose that $A\subseteq \N$ is such that $\gcd(A-\min(A))=1$ and $\BD(A)>0$.  Then $A$ is a \emph{Banach basis}, that is, there is $h\in \N$ such that $\Sigma_h(A)$ is thick\index{thick}.
\end{theorem}

Note that we must assume that $\gcd(A-\min(A))=1$, for if $\gcd(A-\min(A))=c>1$, then $hA\subseteq \{h\min(A)+nc \ : \ n\in \N\}$, which does not contain arbitrarily long intervals.

\begin{proof}[of Theorem \ref{banachbasis}]
Suppose $\BD(A)=r$ and $\gcd(A-\min(A))=1$.  The latter property guarantees the existence of $m\in \N$ such that $\Sigma_m(A-\min(A))$ contains two consecutive numbers, whence $c,c+1\in \Sigma_m(A)$ for some $c\in \N$.  By Proposition \ref{Proposition:conversion}, there is $a\in \starN$ such that $\sigma({}^{\ast}A-a+1)\geq r$.  In particular, $a\in {}^{\ast}A$.  Consequently, we have
$$\sigma(\Sigma_{1+m}({}^{\ast}A)-a-c)\geq \sigma({}^{\ast}A+\{c,c+1\}-a-c)\geq \sigma({}^{\ast}A-a+1)\geq r.$$  Since $0\in \Sigma_{1+m}({}^{\ast}A)-a-c$, Shnirleman's theorem implies that there is $n$ such that $\N\subseteq \Sigma_n(\Sigma_{1+m}({}^{\ast}A)-a-c)$.  By \index{overflow principle}, there is $N$ such that $[0,N]\subseteq \Sigma_n(\Sigma_{1+m}({}^{\ast}A)-a-c)$.  Set $h:=n(1+m)$, so $[0,N]+n(a+c)\subseteq {}^{\ast}(\Sigma_h(A))$.  By transfer, $\Sigma_h(A)$ contains arbitrarily long intervals.
\end{proof}


With similar methods, one can prove the Banach density analogue of the following theorem of Mann (see, for example,  \cite[page 5]{halberstam_sequences_1983}):

\begin{theorem}
\label{Theorem:Mann}Given $A,B\subseteq \mathbb{N}_0$ such that $0\in A\cap B$, we have $\sigma
\left( A+B\right) \geq \min \left\{ \sigma \left( A\right) +\sigma \left(
B\right) ,1\right\} $.
\end{theorem}

Observe that the exact statement of Mann's theorem is false if one replaces Shnirelman density\index{density!Shnirelman} by Banach density.  Indeed, if $A$ and $B$ are both the set of even numbers, then $\BD\left( A+B\right) =\frac{1}{2}$ but $\BD\left( A\right)
+\BD\left( B\right) =1$.  However, if one replaces $A+B$ by $A+B+\{0,1\}$, the Banach density version of Mann's theorem is true.

\begin{theorem}
\label{Theorem:Banach-Mann}Given $A,B\subseteq \mathbb{N}$, we have $%
\BD\left( A+B+\left\{ 0,1\right\} \right) \geq \min \left\{ \BD\left( A\right)
+\BD\left( B\right) ,1\right\} $.
\end{theorem}

The idea behind the proof of Theorem \ref{Theorem:Banach-Mann} is, as
before, to reduce to the case of Shnirelman density\index{density!Shnirelman} by replacing the given sets with hyperfinite shifts. In the course of the proof of Theorem \ref%
{Theorem:Banach-Mann}, we will need to use the following fact from additive number
theory (see, for example,  \cite[page 6]{halberstam_sequences_1983}):

\begin{theorem}[Besicovitch's theorem]
Suppose $A,B\subseteq \mathbb{N}$ and $s\in [0,1]$ are such that $1\in A
$, $0\in B$, and $\left\vert B\cap \lbrack 1,n]\right\vert \geq s \left(
n+1\right) $ for every $n\in \mathbb{N}$.  Then $\sigma \left( A+B\right)
\geq \min \left\{ \sigma \left( A\right) +\sigma \left( B\right) ,1\right\} $. 
\end{theorem}


\begin{proof}[of Theorem \protect\ref{Theorem:Banach-Mann}]
Set $r :=\BD\left( A\right) $ and $s :=\BD\left( B\right) $. We can
assume, without loss of generality, that $r \leq s \leq 1/2$. By
Proposition \ref{Proposition:conversion}, one can find $a\in {}^{\ast }A$ and 
$b\in {}^{\ast }B$ such that $\sigma \left( {}^{\ast }A-a+1{}\right) \geq
r $ and $\sigma \left( {}^{\ast }B-b+1\right) \geq s $. 

\

\noindent \textbf{Claim:}  For every $n\in \mathbb{N}$, one has that $\left\vert \left( {}^{\ast
}B+\left\{ 0,1\right\} \right) \cap \lbrack b+1,b+n]\right\vert \geq s
\left( n+1\right) $. 

\noindent \textbf{Proof of Claim:}  Let $[1,k_{0}]$ be the largest initial segment
of $\mathbb{N}$ contained in $\left( {}^{\ast }B+\left\{ 0,1\right\}
-b\right) \cap \mathbb{N}$ (if no such $k_0$ exists, then the claim is clearly true) and let $[1,k_{1}]$ be the largest initial
segment of $\mathbb{N}$ disjoint from${}\left( (^{\ast }B+\{0,1\})-\left(
b+k_{0}\right) \right) \cap \mathbb{N}$. We note the following:

\begin{itemize}
\item For $1\leq n\leq k_{0}$, we have that
\begin{equation*}
\left\vert( {}^{\ast }B+\left\{ 0,1\right\}) \cap \lbrack b+1,b+n]\right\vert
=n\geq (n+1)/2\geq s \left( n+1\right). 
\end{equation*}

\item For $k_{0}+1\leq n<k_0+ k_{1}$, since $%
\sigma \left( {}^{\ast }B-b+1\right) \geq s $, we have that %
\begin{eqnarray*}
\left\vert ({}^{\ast }B+\left\{ 0,1\right\}) \cap \lbrack b+1,b+n]\right\vert 
&\geq &\left\vert {}^{\ast }\left( B+1\right) \cap \lbrack
b+1,b+n]\right\vert  \\
&=&\left\vert {}^{\ast }\left( B+1\right) \cap \lbrack b+1,b+n+1]\right\vert 
\\
&\geq &s \left( n+1\right) \text{.}
\end{eqnarray*}

\item For $n\geq k_0+ k_{1}$, since $k_0+k_{1}+1\in {}^{\ast }B$, $%
k_0+k_1+1\notin {}^{\ast }B+1$, and $\sigma
\left( {}^{\ast }B-b+1\right) \geq s $, we have
that%
\begin{eqnarray*}
\left\vert ({}^{\ast }B+\left\{ 0,1\right\}) \cap \lbrack b+1,b+n]\right\vert 
&\geq &\left\vert {}^{\ast }\left( B+1\right) \cap \lbrack
b+1,b+n]\right\vert +1 \\
&\geq &s n+1\geq s \left( n+1\right) \text{.}
\end{eqnarray*}
\end{itemize}
These observations conclude the proof of the claim. 

One can now apply Besicovitch's
theorem to $\starA-a+1$ and ${}^{\ast}B+\{0,1\}-b$ (intersected with $\N$) to conclude that%
\begin{equation*}
\sigma \left( \left( ^{\ast }A-a+1\right) +\left( ^{\ast }B+\left\{
0,1\right\} -b\right) \right) \geq \min \left\{ \sigma \left( ^{\ast
}A-a+1\right) +s ,1\right\} \geq r +s \text{.}
\end{equation*}%
Finally, observe that ${}$%
\begin{equation*}
\left( ^{\ast }A-a+1\right) +\left( ^{\ast }B+\left\{ 0,1\right\} -b\right)
={}^{\ast }\left( A+B+\left\{ 0,1\right\} \right) -\left( a+b\right) \text{.}
\end{equation*}%
Hence
\begin{equation*}
\BD\left( A+B+\left\{ 0,1\right\} \right) \geq \sigma \left( {}^{\ast }\left(
A+B+\left\{ 0,1\right\} \right) -\left( a+b\right) \right) \geq r
+s \text{.}
\end{equation*}
\end{proof}

\section*{Notes and references}  The notion of finite embeddability was isolated 
and studied in \cite{di_nasso_embeddability_2014}, although it was implicit in several
previous papers of additive number theory.  The material in Sections 12.2 and 12.3 is from Renling Jin's paper \cite{Jin_Nonstandard_2001}. 

\chapter{Jin's Sumset Theorem}\label{sumset}

\section{The statement of Jin's Sumset Theorem and some standard consequences}
\begin{definition}
An initial segment $\mathrm{U}$ of $\starN_0$ is a \emph{cut} if $\mathrm{U}+\mathrm{U}\subseteq \mathrm{U}$.
\end{definition}

\begin{exercise}
If $\mathrm{U}$ is a cut, then either $\mathrm{U}$ is external or else $\mathrm{U}=\starN$.
\end{exercise}

\begin{example}\label{cutex}

\

\begin{enumerate}
\item $\N$ is a cut.
\item If $N$ is an infinite element of $\starN$, then $\mathrm{U}_N:=\{x\in \starN \ : \ \frac{x}{N}\approx 0\}$ is a cut.
\end{enumerate}
\end{example}

Fix a cut $\mathrm{U}$ of $\starN$ and suppose that $\mathrm{U}\subseteq [0,N)$.  Given $x,y\in \starN$, we write $x\sim_\mathrm{U}y$ if $|x-y|\in \mathrm{U}$; note that $\sim_\mathrm{U}$ is an equivalence relation on $\starN$.  We let $[x]_{\mathrm{U},N}$, or simply $[x]_N$ if no confusion can arise, denote the equivalence class of $x$ under $\sim_\mathrm{U}$ and we let $[0,N)/\mathrm{U}$ denote the set of equivalence classes.  We let $\pi_\mathrm{U}:[0,N)\to [0,N)/\mathrm{U}$ denote the quotient map.  The linear order on $[0,N)$ descends to a linear order on $[0,N)/\mathrm{U}$.  Moreover, one can push forward the Loeb\index{Loeb measure} measure on $[0,N)$ to a measure on $[0,N)/\mathrm{U}$, which we also refer to as Loeb\index{Loeb measure} measure.

\begin{example}\label{LL}
Fix $N\in \starN$ infinite and consider the cut $\mathrm{U}_N$ from Example \ref{cutex}.  Note that the surjection $f:[0,N)\to [0,1]$ given by $f(\beta):=\st(\beta/N)$ descends to a bijection of ordered sets $f:[0,N)/\mathrm{U}_N\to [0,1]$.  The discussion in Section \ref{LebesgueLoeb} of Chapter \ref{Loeb} shows that the measure on $[0,1]$ induced by the Loeb\index{Loeb measure} measure on $[0,N)/\mathrm{U}_N$ via $f$ is precisely Lebesgue measure.
\end{example}

For any cut $\mathrm{U}$ contained in $[0,N)$, the set $[0,N)/\mathrm{U}$ has a natural topology induced from the linear order, whence it makes sense to talk about category notions in $[0,N)/\mathrm{U}$.  (This was first considered in \cite{keisler_meager_1991}.)  It will be convenient to translate the category notions from $[0,N)/\mathrm{U}$ back to $[0,N]$:

\begin{definition}
$A\subseteq [0,N)$ is \emph{$\mathrm{U}$-nowhere dense} if $\pi_\mathrm{U}(A)$ is nowhere dense in $[0,N)/U$.  More concretely:  $A$ is $\mathrm{U}$-nowhere dense if, given any $a<b$ in $[0,N)$ with $b-a>\mathrm{U}$, there is $[c,d]\subseteq [a,b]$ with $d-c>\mathrm{U}$ such that $[c,d]\subseteq [0,N)\setminus A$.  If $A$ is not $\mathrm{U}$-nowhere dense, we say that $A$ is \emph{$\mathrm{U}$-somewhere dense}.
\end{definition}

Recall the following famous theorem of Steinhaus:

\begin{theorem}\label{Lebesgueinterval}
If $C,D\subseteq [0,1]$ have positive Lebesgue measure, then $C+ D$ contains an interval. 
\end{theorem}

For $x,y\in [0,N)$, set $x\oplus_N y:=x+y \mod N$.  For $A,B\subseteq [0,N)$, set $$A\oplus_N B:=\{x\oplus_N y \ : \ x\in A, y\in B\}.$$  In light of Example \ref{LL}, Theorem \ref{Lebesgueinterval} says that whenever $A,B\subseteq [0,N)$ are internal sets of positive Loeb\index{Loeb measure} measure, then $A\oplus_N B$ is $\mathrm{U}_N$-somewhere dense.  Keisler and Leth asked whether or not this is the case for any cut.  Jin answered this positively in \cite{jin_sumset_2002}:

\begin{theorem}[Jin's Sumset Theorem]\label{jinsumset}
If $\mathrm{U}\subseteq [0,N)$ is a cut and $A,B\subseteq [0,N)$ are internal sets with positive Loeb\index{Loeb measure} measure, then $A\oplus_N B$ is $\mathrm{U}$-somewhere dense.
\end{theorem}

\begin{exercise}
Prove Theorem \ref{Lebesgueinterval} from Theorem \ref{jinsumset}.
\end{exercise}

We will prove Theorem \ref{jinsumset} in the next section.  We now prove the following standard corollary of Theorem \ref{jinsumset}, which is often also referred to as Jin's sumset theorem, although this consequence was known to Leth beforehand.

\begin{corollary}\label{jinleth}
Suppose that $A,B\subseteq \N$ have positive Banach density\index{density!Banach}.  Then $A+B$ is piecewise syndetic\index{syndetic!piecewise}.
\end{corollary}

\begin{proof}
Set $r:=\BD(A)$ and $s:=\BD(B)$.  Fix $N\in \starN$ infinite and take $x,y\in \starN$ such that
$$\delta({}^{\ast}A\cap [x,x+N)|)\approx r, \quad \delta({}^{\ast}B\cap [y,y+N))\approx s.$$  Let $C:={}^{\ast}A-x$ and $D:={}^{\ast}B-y$, so we may view $C$ and $D$ as internal subsets of $[0,2N)$ of positive Loeb\index{Loeb measure} measure.  By Jin's theorem applied to the cut $\N$, we have that $C\oplus_{2N}D=C+D$ is $\N$-somewhere dense, that is, there is a hyperfinite interval $I$ such that all gaps of $C+D$ on $I$ have finite length.  By \index{overflow principle}, there is $m\in \N$ such that all gaps of $C+D$ on $I$ have length at most $m$.  Therefore, $x+y+I\subseteq {}^{\ast}(A+B+[0,m])$.  By transfer, for any $k\in \N$, $A+B+[0,m]$ contains an interval of length $k$, whence $A+B$ is piecewise syndetic.  
\end{proof}

It is interesting to compare the previous corollary to Proposition \ref{differencesyndetic}.  It is also interesting to point out that Corollary \ref{jinleth} can also be used to give an alternative proof of Theorem \ref{banachbasis}.  Indeed, suppose $\BD(A)>0$ and $\gcd(A-\min(A))=1$.  Then there is $h\in \N$ such that $A+A+[0,h]$ is thick\index{thick}.  It follows that $A+A+[x,x+h]$ is thick\index{thick} for all $x\in \N$.  As in the proof of Theorem \ref{banachbasis}, take $m$ and consecutive $a,a+1\in \Sigma_m(A)$.  Note that, for all $i=0,1,\ldots,h$, we have that $ha+i=i(a+1)+(h-i)a\in \Sigma_{hm}(A)$.  It follows that $A+A+[ha,ha+h]\subseteq \Sigma_{hm+2}(A)$, whence $\Sigma_{hm+2}(A)$ is thick\index{thick}.

\section{Jin's proof of the sumset theorem}

We now turn to the proof of Theorem \ref{jinsumset} given in \cite{jin_sumset_2002}.  Suppose, towards a contradiction, that there is a cut $\mathrm{U}$ for which the theorem is false.  If $H>\mathrm{U}$ and $A,B\subseteq [0,H)$ are internal, we say that $(A,B)$ is \emph{$(H,\mathrm{U})$-bad} if $\mu_H(A),\mu_H(B)>0$ and $A\oplus_H B$ is $\mathrm{U}$-nowhere dense.  We set 
$$r:=\sup\{\mu_H(A) \ : \ (A,B) \text{ is $(H,\mathrm{U})$ bad for some $H>\mathrm{U}$ and some } B\subseteq [0,H)\}.$$  By assumption, $r>0$.  We fix $\epsilon>0$ sufficiently small.  We then set
$$s:=\sup\{\mu_H(B) \ : (A,B) \text{ is $(H,\mathrm{U})$-bad for some $H>\mathrm{U}$ and some $A\subseteq [0,H)$ with }\mu_H(A)>r-\epsilon\}.$$  

By the definition of $r$, we have that $s>0$.  Also, by the symmetry of the definition of $r$, we have that $r\geq s$.  The following is slightly less obvious:

\

\noindent \textbf{Claim 1:}  $s<\frac{1}{2}+\epsilon$.

\noindent \textbf{Proof of Claim 1:}  Suppose, towards a contradiction, that $s\geq \frac{1}{2}+\epsilon$.  We may thus find $H>\N$ and an $(H,\mathrm{U})$-bad pair $(A,B)$ with $\mu_H(A)>\frac{1}{2}$ and $\mu_H(B)>\frac{1}{2}$.  Since addition modulo $H$ is translation invariant, it follows that for any $x\in [0,H)$, we have that $A\cap (x\ominus_H B)\not=\emptyset$, whence $A\oplus_H B=[0,H-1)$, which is a serious contradiction to the fact that $A\oplus_H B$ is $\mathrm{U}$-nowhere dense.

\

We now fix $\delta>0$ sufficiently small, $H>\mathrm{U}$ and an $(H,\mathrm{U})$-bad $(A,B)$ such that $\mu_H(A)>r-\epsilon$ and $\mu_H(B)>s-\delta$.  We will obtain a contradiction by producing $K>\mathrm{U}$ and $(K,\mathrm{U})$-bad $(A',B')$ such that $\mu_K(A')>r-\epsilon$ and $\mu_K(B')>s+\delta$, contradicting the definition of $s$.  

We first show that it suffices to find $K>U$ such that $K/H\approx 0$ and such that there are hyperfinite intervals $I,J\subseteq [0,H)$ of length $K$ for which
$$\st\left(\frac{|A\cap I|}{K}\right)>r-\epsilon \text{ and }\st\left(\frac{|B\cap J|}{K}\right)>s+\delta.$$

\noindent Indeed, suppose that $I:=[a,a+K)$ and $J:=[b,b+K)$ are as above.  Let $A':=(A\cap I)-a$ and $B':=(B\cap J)-b$.  Then $\mu_K(A')>r-\epsilon$ and $\mu_K(B')>s+\delta$.  It remains to see that $(A',B')$ is $(K,U)$-bad.  Since $A\oplus_H B$ is $\mathrm{U}$-nowhere dense, it is clear that $(A\cap I)\oplus_H (B\cap J)$ is also $\mathrm{U}$-nowhere dense.  Since $A'\oplus_H B'=((A\cap I)\oplus_H(B\cap J))\ominus (a+b)$, we have that $A'\oplus_H B'$ is $\mathrm{U}$-nowhere dense.  Since $K/H$ is infinitesimal, we have that $A'\oplus_H B'=A'\oplus_{2K}B'$.  It follows that $A'\oplus_K B'$ is the union of two $\mathrm{U}$-nowhere dense subsets of $[0,K)$, whence is also $\mathrm{U}$-nowhere dense, and thus $(A',B')$ is $(K,\mathrm{U})$-bad, as desired.

We now work towards finding the appropriate $K$.  By the definition of $\mathrm{U}$-nowhere dense, we have, for every $k\in \mathrm{U}$, that $A\oplus_H(B\oplus_H [-k,k]))=(A\oplus_H B)\oplus_H[-k,k]$ is $\mathrm{U}$-nowhere dense.  By the definition of $s$, it follows that $\mu_H(B\oplus_H [-k,k])\leq s$ for each $k\in \mathrm{U}$.  Since $U$ is external and closed under addition, it follows that there is $K>\mathrm{U}$ with $K/H$ infinitesimal such that $$\frac{|B\oplus_H [-K,K]|}{H}\leq s+\frac{\delta}{2}.$$  We finish by showing that this $K$ is as desired.

Let $\mathcal{I}:=\{[iK,(i+1)K) \ : \ 0\leq i\leq H/K -1\}$ be a partition of $[0,H-1)$ into intervals of length $K$ (with a negligible tail omitted).  Let $X:=\{i\in [0,H/K-1] \ : \ [iK,(i+1)K-1)\cap B=\emptyset\}$.

\

\noindent \textbf{Claim 2:}  $\frac{|X|}{|\mathcal{I}|}> \frac{1}{3}$.

\noindent \textbf{Proof of Claim 2:}  Suppose, towards a contradiction, that $\frac{|X|}{|\mathcal{I}|}\leq  \frac{1}{3}$.  Fix $i\notin X$ and $x\in [iK,(i+1)K)$.  Write $x=iK+j$ with $j\in [0,K-1]$.  Since $i\notin X$, there is $l\in [0,K-1)$ such that $iK+l\in B$.  It follows that $x=(iK+l)+(j-l)\in B\oplus_H [-K,K]$.  Consequently,
$$|B\oplus_H [-K,K]|\geq \sum_{i\notin X}K\geq \frac{2}{3}(H/K-1)\cdot K=\frac{2}{3}H-\frac{2}{3}K,$$
whence $$\frac{|B\oplus_H [-K,K]|}{H}\geq \frac{2}{3}-\frac{2}{3}\frac{K}{H}\approx \frac{2}{3},$$ which, for sufficiently small $\epsilon$ and $\delta$, contradicts the fact that $\frac{|B\oplus_H [-K,K]|}{H}\leq s+\frac{\delta}{2}$.

\

Let $\mathcal{I}':=\{[iK,(i+1)K) \ : \ i\notin X\}$.  As explained above, the following claim completes the proof of the theorem.

\

\noindent \textbf{Claim 3:}  There are $I,J\in \mathcal{I}$ such that
$$\st\left(\frac{|A\cap I|}{K}\right)>r-\epsilon \text{ and }\st\left(\frac{|B\cap J|}{K}\right)>s+\delta.$$

\noindent \textbf{Proof of Claim 3:}  We only prove the existence of $J$; the proof of the existence of $I$ is similar (and easier).  Suppose, towards a contradiction, that $\st(\frac{|B\cap J|}{K})\leq s+\delta$ for all $J\in \mathcal{I}$.  We then have
$$s-\delta<\frac{|B\cap [0,H-1)|}{H}=\frac{1}{H}\sum_{J\in \mathcal{I}'}|B\cap [iK,(i+1)K)|\leq \frac{1}{H}\cdot \frac{2}{3}\cdot(H/K)\cdot(s+\delta)K=\frac{2}{3}(s+\delta).$$  If $\delta\leq \frac{s}{5}$, then this yields a contradiction.

\section{Beiglb\"ock's proof}

It is straightforward to verify that Corollary \ref{jinleth} is also true for subsets of $\Z$:

\begin{corollary}\label{jinlethZ}
If $A,B\subseteq \Z$ are such that $\BD(A),\BD(B)>0$, then $A+B$ is piecewise syndetic.  
\end{corollary}

In this section, we give Beiglb\"ock's ultrafilter\index{ultrafilter} proof of Corollary \ref{jinlethZ} appearing in \cite{Beiglbock_Ultrafilter_2011}.  We first start with some preliminary facts on invariant means on $\Z$.

\begin{definition}
An \emph{invariant mean} on $\Z$ is a linear functional $\ell:B(\Z)\to \R$ that satisfies the following properties:
\begin{enumerate}
\item $\ell$ is positive, that is, $\ell(f)\geq 0$ if $f\geq 0$;
\item $\ell(1)=1$; and
\item $\ell(k.f)=\ell(f)$ for all $k\in \Z$ and $f\in B(\Z)$, where $(k.f)(x):=f(x-k)$.
\end{enumerate}
\end{definition}

There are many invariant means on $\Z$:

\begin{exercise}\label{meanexercise}
Suppose that $(I_n)$ is a sequence of intervals in $\Z$ with $|I_n|\to \infty$ as $n\to \infty$.  Fix $\u\in \beta\N$.  Define, for $f\in B(\Z)$, $\ell(f)=\lim_\u (\frac{1}{|I_n|}\sum_{x\in I_n}f(x))$.  Show that $\ell$ is an invariant mean on $\Z$.  
\end{exercise}

In fact, we have:

\begin{lemma}
For every $A\subseteq \Z$, there is an invariant mean $\ell$ on $\Z$ such that $\ell(1_A)=\BD(A)$.
\end{lemma}

\begin{proof}
Let $(I_n)$ be a sequence of intervals witnessing the Banach density\index{density!Banach} of $A$.  Fix nonprincipal $\u\in \beta \Z$.  Define $\ell$ as in Exercise \ref{meanexercise} for these choices of $(I_n)$ and $\u$.  It is clear that $\ell(1_A)=\BD(A)$.
\end{proof}

\begin{lemma}
For every invariant mean $\ell$ on $\Z$, there is a regular Borel probability  measure $\nu$ on $\beta \Z$ such that $\ell(1_A)=\nu(\overline{A})$ for every $A\subseteq \Z$.
\end{lemma}

\begin{proof}
Fix a mean $\ell$ on $\Z$.  Since $f\mapsto \beta f$ yields an isomorphism $B(\Z)\cong C(\beta \Z)$, the Riesz Representation Theorem yields a regular Borel probability measure $\nu$ on $\beta \Z$ such that $\ell(f)=\int_{\beta\Z} (\beta f) d\nu$ for all $f\in B(\Z)$.  In particular, $$\ell(1_A)=\int_{\beta \Z} (\beta 1_{A})d\nu =\nu(\overline{A}).$$
\end{proof}

The following lemma is the key to Beiglb\"ock's proof of Corollary \ref{jinlethZ}.

\begin{lemma}\label{Blemma}
For any $A,B\subseteq \Z$, there is $\u\in \beta \Z$ such that $\BD(A\cap (B-\u))\geq \BD(A)\cdot \BD(B)$.
\end{lemma}

\begin{proof}
Fix an invariant mean $\ell$ on $\Z$ such that $\ell(1_B)=\BD(B)$ and let $\nu$ be the associated Borel probability measure on $\beta\Z$.  Let $(I_n)$ be a sequence of intervals witnessing the Banach density\index{density!Banach} of $A$.  Define $f_n:\beta \Z\to [0,1]$ by $$f_n(\u):=\delta((A\cap (B-\u),I_n)=\frac{1}{|I_n|}\sum_{k\in A\cap I_n}1_{\overline{{B-k}}}(\u).$$  Set $f(\u):=\limsup_n f_n(\u)$ and note that $f(\u)\leq \BD(A\cap (B-\u))$ for all $\u\in \beta\Z$.  Fatou's Lemma implies
$$\int_{\beta \Z} fd\nu \geq \limsup_n \int_{\beta \Z}\frac{1}{|I_n|}\sum_{k\in A\cap I_n}1_{U_{B-k}}d\nu=\limsup_n \frac{1}{|I_n|}\sum_{k\in I_n\cap A} \ell(1_{B-k}).$$  Since $\ell$ is invariant, the latter term is equal to $ \limsup_n \delta(A,I_n)\cdot \ell(1_B)=\BD(A)\cdot \BD(B)$.  Thus, we have shown $\int_{\beta\Z} fd\nu\geq \BD(A)\cdot \BD(B)$.  In particular, there is some $\u\in \Z$ such that $f(\u) \geq \BD(A)\cdot\BD(B)$, as desired.
\end{proof}

Notice that, in the notation of the above proof, $\mu(\Z)=0$, whence we can take $\u$ as in the conclusion of the lemma to be nonprincipal.

We can now give Beiglb\"ock's proof of Corollary \ref{jinlethZ}.  Assume that $\BD(A),\BD(B)>0$.  Apply the previous lemma with $A$ replaced by $-A$ (which has the same Banach density\index{density!Banach}), obtaining $\u\in \beta \Z$ such that $C:=(-A)\cap (B-\u)$ has positive Banach density\index{density!Banach}.  By Lemma \ref{differencesyndetic}, $C-C$ is syndetic; since $C-C\subseteq A+(B-\u)$, we have that $A+(B-\u)$ is also syndetic. 

Suppose $s\in A+(B-\u)$.  Then for some $a\in A$, $B-(s-a)\in \u$, whence $a+B-s\in \u$ and hence $A+B-s\in \u$.  Thus, for any finite set $s_1,\ldots,s_n\in A+(B-\u)$, we have $\bigcap_{i=1}^n (A+B-s_i)\in \u$, and, in particular, is nonempty, meaning there is $t\in \Z$ such that $t+\{s_1,\ldots,s_n\}\subseteq A+B$.  We claim that this implies that $A+B$ is piecewise syndetic\index{syndetic!piecewise}.  Indeed, take $F\subseteq \Z$ such that $F+A+(B-\u)=\Z$.  We claim that $F+A+B$ contains arbitrarily long intervals.  To see this, fix $n\in \N$ and, for $i=1,\ldots,n$ take $s_i\in A+(B-\u)$ such that $i\in F+s_i$.  Take $t\in \Z$ such that $t+\{s_1,\ldots,s_n\}\subseteq A+B$.  Then $t+[1,n]\subseteq t+F+\{s_1,\ldots,s_n\}\subseteq F+(A+B)$, completing the proof.

\section{A proof with an explicit bound}

A proof of Corollary \ref{jinlethZ} can be given by using a simple
counting argument of finite combinatorics in the nonstandard setting.
In this way, one also obtains an explicit bound on the
number of shifts of the sumset that are needed to produce a thick set\index{thick}.

\begin{lemma}
Let $C\subseteq[1,n]$ and $D\subseteq[1,m]$ be finite
sets of natural numbers. Then there exists $k\le n$ such that
$$\frac{|(C-k)\cap D|}{m}\ \ge\
\frac{|C|}{n}\cdot\frac{|D|}{m}\ -\ \frac{|D|}{n}.$$
\end{lemma}

\begin{proof}
If $\chi:[1,n]\to\{0,1\}$ is the characteristic function of $C$,
then for every $d\in D$, we have
$$\frac{1}{n}\cdot\sum_{k=1}^{n}\chi(k+d)\ =\
\frac{|C\cap[1+d,n+d]|}{n}\ =\
\frac{|C|}{n}+\frac{e(d)}{n}$$
where $|e(d)|\le d$. Then:
\begin{eqnarray*}
\frac{1}{n}\cdot\sum_{k=1}^n\left(\frac{1}{m}\cdot\sum_{d\in D}\chi(k+d)\right) &  = &
\frac{1}{m}\cdot\sum_{d\in D}\left(\frac{1}{n}\cdot\sum_{x=1}^n\chi(k+d)\right)
\\
{} & = &
\frac{1}{m}\cdot\sum_{d\in D}\frac{|C|}{n}\ +\
\frac{1}{n m}\cdot\sum_{d\in D}e(d)\ =\ \frac{|C|}{n}\cdot\frac{|D|}{m}\ +\ e
\end{eqnarray*}
where
$$|e|\ =\ \left|\frac{1}{n m}\sum_{d\in D}e(d)\right|\ \le\
\frac{1}{n m}\sum_{d\in D}|e(d)|\ \le\ \frac{1}{n m}\cdot\sum_{d\in D}d\ \le\
\frac{1}{n m}\sum_{d\in D}m\ =\ \frac{|D|}{n}.$$
By the \emph{pigeonhole principle}, there must exist
at least one number $k\le n$ such that
$$\frac{|(C-k)\cap D|}{m}\ =\ 
\frac{|(D+k)\cap C|}{m}\ =\ 
\frac{1}{m}\cdot\sum_{d\in D}\chi(k+d)\ \ge\
\frac{|C|}{n}\cdot\frac{|D|}{m}\ -\ \frac{|D|}{n}.$$
\end{proof}

\begin{theorem}
Let $A,B\subseteq\Z$ have positive Banach densities
$\BD(A)=\alpha>0$ and $\BD(B)=\beta>0$.
Then there exists a finite set $F$ with $|F|\le\frac{1}{\alpha\beta}$
such that $(A+B)+F$ is thick\index{thick}. In particular, $A+B$ is piecewise syndetic.
\end{theorem}

\begin{proof}
Pick infinite $\nu,N\in\starN$ such that $\nu/N\approx 0$, and pick intervals
$[\Omega+1,\Omega+N]$ and $[\Xi+1,\Xi+\nu]$ 
such that
$$\frac{|\starA\cap[\Omega+1,\Omega+N]|}{N}\,\approx\,\alpha\quad\text{and}\quad
\frac{|(-\starB)\cap[\Xi+1,\Xi+\nu]|}{\nu}\,\approx\,\beta.$$
By applying the nonstandard version of the previous lemma to the hyperfinite sets
$C=(\starA-\Omega)\cap[1,N]\subseteq[1,N]$
and $D=(-\starB-\Xi)\cap[1,\nu]$, one obtains the existence
of a number $\zeta$ such that
$$\frac{|(C-\zeta)\cap D|}{\nu}\ \ge\
\frac{|C|}{N}\cdot\frac{|D|}{\nu}\ -\ \frac{|D|}{N}\ \approx\ \alpha\beta.$$
Finally, apply Lemma \ref{deltasetnonstandard}
to the internal set $E=(C-\zeta)\cap D\subseteq[1,\nu]$.
Since $|E|/\nu\approx\alpha\beta$, there exists a finite $F\subset\Z$
with $|F|\le\frac{1}{\alpha\beta}$ and such that $\Z\subseteq (E-E)+F$,
and hence, by \emph{overflow}, $I\subseteq(E-E)+F$ for some
infinite interval $I$. Since $E\subseteq\starA-\Omega$ and $E\subset-\starB-\Xi$,
it follows that ${}^*(A+B+F)=\starA+\starB+F$ includes the infinite interval 
$I+\Omega+\Xi+\zeta$, and hence it is thick\index{thick}.
\end{proof}

\section{Quantitative strengthenings}\label{quant}

We end this chapter by proving some technical strengthenings of Corollary \ref{jinlethZ}.  Indeed, in light of Lemma \ref{pwsZ}, the following theorem can be viewed as a ``quantitative'' strengthening of Corollary \ref{jinlethZ}:

\begin{theorem}\label{quantjinleth}
Suppose that $(I_n)$ is a sequence of intervals with $|I_n|\to \infty$ as $n\to \infty$.  Suppose that $A,B\subseteq \Z$ and $\BD(B)>0$.  Then:
\begin{enumerate}
\item If $\overline{d}_{(I_n)}(A)\geq r$, then there is a finite set $F\subseteq \Z$ such that, for every finite set $L\subseteq \Z$, we have $$\overline{d}_{(I_n)}\left(\bigcap_{x\in L}(A+B+F+x)\right)\geq r.$$
\item If $\underline{d}_{(I_n)}(A)\geq r$, then for every $\epsilon>0$, there is a finite set $F\subseteq \Z$ such that, for every finite set $L\subseteq \Z$, we have $$\underline{d}_{(I_n)}\left(\bigcap_{x\in L}(A+B+F+x)\right)\geq r-\epsilon.$$
\end{enumerate}
\end{theorem} 

In connection with item (2) of the previous theorem, it will turn out that $F$ depends only on $B$ and $\epsilon$ (but not on $A$ or $(I_n)$).  Moreover, item (2) is false if $r-\epsilon$ is replaced by $r$; see \cite{di_nasso_high_2016}.

In order to prove Theorem \ref{quantjinleth}, we need a preparatory counting lemma.  

%
%

\begin{lemma}\label{technicallemma}
Suppose that $(I_n)$ is a sequence of intervals in $\Z$ such that $|I_n|\to \infty$ as $n\to \infty$.  Further suppose $I$ is an infinite hyperfinite interval in ${}^{\ast}\Z$ and $A\subseteq \Z$.
\begin{enumerate}
\item If $\overline{d}_{(I_n)}\left( A \right)\geq r$, then there is $N>\N$ such that 
\begin{equation*}
\delta(\starA,I_N)\gtrsim r \text{\quad and\quad }\frac{1}{\left \vert I_N\right \vert}\sum_{x\in I_N}\delta(x-({}^*A\cap I_N),I)\gtrsim r\text{.} \quad(\dagger)
\end{equation*}%
\item If $\underline{d}_{(I_n)}\left( A\right)> r$, then there is $N_0>\N$ such that $(\dagger)$ holds for all $N\geq N_0$.
\end{enumerate}
\end{lemma}

\begin{proof}
For (1), first apply transfer to the statement ``for every finite interval $J\subseteq \Z$ and every natural number $k$, there exists $n\geq k$ such that
\begin{equation*}
\delta(A,I_n)>r -2^{-k}\text{\quad and\quad }\frac{1}{\left\vert
I_{n}\right\vert }\sum_{x\in J}\left\vert (I_{n}-x)\bigtriangleup
I_{n}\right\vert <2^{-k}\text{.\textquotedblright }
\end{equation*}%
Fix $K>\N$ and let $N$ be the result of applying the transferred statement to $I$ and $K$.  Set $C={}^{\ast }A\cap
I_{N }$ and let $\chi _{C}$ denote the characteristic function of $C$. We
have%
\begin{eqnarray*}
\frac{1}{\left\vert I_{N }\right\vert }\sum_{x\in I_{N }}\delta((x-C),Y) &=&\frac{1}{%
\left\vert I_{N }\right\vert }\sum_{x\in I_{N }}\frac{1}{\left\vert
I\right\vert }\sum_{y\in I}\chi _{C}(x-y) \\
&=&\frac{1}{\left\vert I\right\vert }\sum_{y\in I}\frac{\left\vert C\cap
(I_{N }-y)\right\vert }{\left\vert I_{N }\right\vert } \\
&\geq &\frac{\left\vert C\right\vert }{\left\vert I_{N }\right\vert }%
-\sum_{y\in I}\frac{\left\vert (I_{N }-y)\bigtriangleup I_{N
}\right\vert }{\left\vert I_{N }\right\vert } \\
&\thickapprox &r \text{.}
\end{eqnarray*}%
For (2), apply transfer to the statement ``for every finite interval $J\subseteq \Z$ and every natural number $k$, there exists $n_0\geq k$ such that, for all $n\geq n_0$, 
\begin{equation*}
\delta(A,I_n)>r -2^{-n_0}\text{\quad and\quad }\frac{1}{\left\vert
I_{n}\right\vert }\sum_{x\in J}\left\vert (I_{n}-x)\bigtriangleup
I_{n}\right\vert <2^{-n_0}\text{.\textquotedblright }
\end{equation*}%
Once again, fix $K>\N$ and let $N_0$ be the result of applying the transferred statement to $I$ and $K$.  As above, this $N_0$ is as desired.
\end{proof}




\begin{proof}[of Theorem \ref{quantjinleth}]
Fix an infinite hyperfinite interval $I$ that is good for $B$.  (See Proposition \ref{goodinterval}.)

For (1), assume that $\overline{d}_{(I_n)}(A)\geq r$.
Let $N$ be as in part (1) of Lemma \ref{technicallemma} applied to $I$ and $A$.  Once again, set $C:={}^*A\cap I_N$.  
Consider the $\mu _{I_{N }}$-measurable function%
\begin{equation*}
f\left( x\right) =\mathrm{st}(\delta(x-C,I)).
\end{equation*}%
By Lemma \ref{integratebycounting}, we have that
\begin{equation*}
\int_{I_{N }}fd\mu _{I_{N }}=\mathrm{st}\left( \frac{1}{|I_{N }|}%
\sum_{x\in I_{N }}\delta(x-C,I)\right) \geq r ,
\end{equation*}%
whence there is some standard $s>0$ such that $\mu _{I_{N }}\left( \left\{ x\in I_{N }:f\left( x\right) \geq 2s\right\}
\right) \geq r \text{.}$
%
%
%
Setting $\Gamma =\left\{ x\in I_{N }:\delta(x-C,I)\geq s\right\}$, we have that $\mu_{I_N}(\Gamma)\geq r$.
Since $I$ is good for $B$, we may take a finite subset $F$ of $\Z$ such that%
\begin{equation*}
\delta({}^{\ast }(B+F),I)%
>1-\frac{s}{2}.
\end{equation*}%
Fix $x\in \Z$.  Since $I$ is infinite, we have that%
\begin{equation*}
\delta({}^{\ast }\left( B+F+x\right),I)=\delta({}^{\ast }\left( B+F\right),(I-x))\approx \delta({}^{\ast }\left( B+F\right),I),
\end{equation*}%
whence $\delta({}^{\ast }\left( B+F+x\right),I)>1-s.$
Thus, for any $y\in \Gamma$, we have that $(y-C)\cap {}^*(B+F+x)\not=\emptyset$.
In particular, if $L$ is a finite subset of $\Z$, then $\Gamma \subseteq \;^{\ast }\left( \bigcap_{x\in L}A+B+F+x\right) \text{.}$
Therefore%
\begin{equation*}
\overline{d}_{(I_n)}\left( \bigcap_{x\in L}A+B+F+x\right) \geq \mu_{I_N}\left({}^*( \bigcap_{x\in L}A+B+F+x)\right)\geq \mu_{I_N}(\Gamma)\geq r \text{.}
\end{equation*}
This establishes (1).

Towards (2), note that we may suppose that $\underline{d}_{(I_n)}(A)>r$.  Fix $N _{0}>\N$ as in part (2) of Lemma \ref{technicallemma} applied to $I$ and $A$.  Fix $N \geq N _{0}$ and standard $\varepsilon >0$ with $\varepsilon
<r $.
Set 
\begin{equation*}
\Lambda:=\left\{ x\in I_{N }:\delta((x-C),I)\geq \varepsilon \right\}
\end{equation*}%
and observe that $\frac{\left\vert \Lambda\right\vert }{\left\vert I_{N }\right\vert }%
>r -\varepsilon \text{.}$
Since $I$ is good for $B$, we may fix a finite subset $F$ of $\Z$ such that%
\begin{equation*}
\delta({}^{\ast }(B+F),I)%
>1-\frac{\varepsilon}{2} \text{.}
\end{equation*}%
Fix $x\in \Z$.  Since $I$ is infinite, arguing as in the proof of part (1), we conclude that%
\begin{equation*}
\delta{}^{\ast }\left( B+F+x\right),I)>1-\varepsilon \text{.}
\end{equation*}%
Fix $L\subseteq \Z$ finite.  As in the proof of part (1), it follows that $\Lambda\subseteq \;^{\ast }\left( \bigcap_{x\in L}A+B+F+x\right)$
whence
\begin{equation*}
\delta\left({}^{\ast }\left( \bigcap_{x\in L}A+B+F+x\right),I_{N
}\right)\geq \frac{\left\vert \Lambda\right\vert }{\left\vert I_{N }\right\vert }> r
-\varepsilon \text{.}
\end{equation*}%
Since the previous inequality held for every $N \geq N _{0}$, by transfer we can conclude that there is $n_0$ such that, for all $n\geq n_0$, we have 
$$\delta\left(\left( \bigcap_{x\in L}A+B+F+x\right),I_{n
}\right)\geq r-\epsilon,$$
whence it follows that
\begin{equation*}
\underline{d}_{(I_n)}\left( \bigcap_{x\in L}A+B+F+x\right) \geq r
-\varepsilon \text{.}
\end{equation*}
\end{proof}

\section*{Notes and references} 
The space of cuts was first studied in the paper \cite{jin_sumset_2002} and Jin's Sumset Theorem solved Problem 9.13 in that paper negatively.  Jin gives a purely standard, finitary version of his proof of the Sumset Theorem in \cite{jin_standardizing_2004}; a simplified elementary standard proof was then given in \cite{di_nasso_elementary_2014}.
The proof given in Section 13.4 is due to Di Nasso \cite{di_nasso_embeddability_2014}.  The original proof of Theorem \ref{quantjinleth} given in \cite{di_nasso_high_2015} used a Lebesgue Density Theorem for the cut spaces $[0,H]/\mathrm{U}$.  Indeed, one can give a nice proof of Theorem \ref{Lebesgueinterval} using the standard Lebesgue density theorem and Example \ref{LL} suggested that perhaps a general Lebesgue density theorem holds for cut spaces.  Once this was established, the fact that one has many density points was used to strengthen the sumset theorem in the above manner.  The proof given in this chapter follows \cite{di_nasso_high_2016}, which actually works for all countable amenable groups rather than just $\Z$; other than the fact that Proposition \ref{fat} is more difficult to prove for amenable groups than it is for $\Z$, there is not much added difficulty in generalizing to the amenable situation.  We should also mention that the amenable group version of Corollary \ref{jinlethZ} was first proven by Beiglb\"ock, Bergelson, and Fish in \cite{beiglbock_sumset_2010}.

\chapter{Sumset configurations in sets of positive density}

\section{Erd\H{o}s' conjecture}

Just as Szemeredi's theorem is a ``density'' version of van der Waerden's theorem, it is natural to wonder if the density version of Hindman's theorem is true, namely:  does every set of positive density contain an FS set?  It is clear that the answer to this question is:  no!  Indeed, the set of odd numbers has positive density, but does not even contain $\operatorname{PS}(B)$ for any infinite set $B$.  Here, $\operatorname{PS}(B):=\{b+b' \ : \ b,b'\in B, \ b\not=b'\}$.  This example is easily fixed if we allow ourselves to translate the original set, so Erd\H{o}s conjectured that this was the only obstruction to a weak density version of Hindman's theorem, namely:  if $A\subseteq \N$ has positive density, then there is $t\in \N$ and infinite $B\subseteq A$ such that $t+\operatorname{FS}(B)\subseteq A$.  Straus provided a counterexample to this conjecture\footnote{It still seems to be open whether or not a set of positive density (of any kind) must contain a translate of $\operatorname{PS}(B)$ for some infinite $B$.} (see \cite{erdos_survey_1980}), whence Erd\H{o}s changed his conjecture to the following, which we often refer to as Erd\H{o}s' sumset conjecture\index{Erd\H{o}s' Sumset Conjecture} (see \cite{Nathanson_Sumsets_1980} and \cite[page 85]{Erdos_Old_1980}):

\begin{conjecturenew}\label{Erdosconj}
Suppose that $A\subseteq \N$ is such that $\underline{d}(A)>0$.  Then there exist infinite sets $B$ and $C$ such that $B+C\subseteq A$.
\end{conjecturenew}

 
%

It will be convenient to give a name to sets satisfying the conclusion of Conjecture \ref{Erdosconj}\index{Erd\H{o}s' Sumset Conjecture}.

\begin{definition}
We say that $A\subseteq \N$ has the \emph{sumset property} if there are infinite sets $B,C\subseteq\N$ such that $B+C\subseteq A$.
\end{definition}

Many sets that are structurally large have the sumset property as indicated by the following proposition.  While this result follows from standard results in the literature, we prefer to give the following elegant argument of Leth.\footnote{Indeed, if $A$ is piecewise syndetic, then $A+[0,k]$ is thick for some $k\in \N$.  Thick sets are easily seen to contain FS-sets, whence, by the Strong version of Hindman's theorem (Corollary \ref{stronghindman}), $A+i$ contains an FS-set for some $i\in [0,k]$.  It follows immediately that $A$ has the sumset property.}

\begin{proposition}
If $A$ is piecewise syndetic\index{syndetic!piecewise}, then $A$ has the sumset property.  More precisely, there is an infinite set $B\subseteq \N$ and $k\in \N$ such that $\operatorname{PS}(B)-k\subseteq A$.
\end{proposition}

\begin{proof}
Since $A$ is piecewise syndetic, there exists $m$ and an interval $[a,b]$ in $^{\ast}\mathbb{N}$ with $a$ and $b-a$ infinite
such that $^{\ast}A$ has no gaps of size larger than $m$ on $[a,b]$. \ Set $L:=(^{\ast}A-a)\cap \N$, so that $a+L\subseteq$
$^{\ast}A.$  Let $l$ be the first element in $^{\ast}L$ greater than or equal to $a$. \ Set $k:=l-a$.  Since $L$
contains no gaps of size larger than $m$, we know that $0\leq k\leq m$.  We now have: \
\[
l-k+L\subseteq\text{ }^{\ast}A\text{ and }l\in\text{ }^{\ast}L.
\]

Take $b_0\in L$ arbitrary.  Assume now that $b_0<b_1<\cdots<b_n\in L$ have been chosen so that $b_i+b_j-k\in A$ for $1\leq i<j\leq n$.  Since the statement ``there is $l\in {}^{\ast}L$ such that $l>b_n$ and $l-k+b_i\in {}^{\ast}A$ for $i=1,\ldots,n$'' is true, by transfer there is $b_{n+1}\in L$ such that $b_{n+1}>b_n$ and $b_i+b_{n-1}-k\in A$ for $i=1,\ldots,n$.  The set $B:=\{b_0,b_1,b_2,\ldots\}$ defined this way is as desired.
\end{proof}

The first progress on Erd\H{o}s' conjecture was made by Nathanson in \cite{Nathanson_Sumsets_1980}, where he proved the following:

\begin{theorem}[Nathanson]\label{nathansontheorem}
If $\BD(A)>0$ and $n\in\N$, then there are $B,C$ with $\BD(B)>0$ and $|C|=n$ such that $B+C\subseteq A$.
\end{theorem}

This theorem follows immediately by induction using the following lemma.  We take the opportunity here to give a short nonstandard proof.

\begin{lemma}[Kazhdan]
Suppose that $\BD(A)>0$ and $t\in\N$.  Then there is $B\subseteq A$ with $\BD(B)>0$ and $c\geq t$ such that $B+c\subseteq A$.
\end{lemma}

\begin{proof}
Let $I$ be an infinite interval such that $\mu_I(\starA)=\BD(A)$.  It follows that $\mu_I(\starA+t),\mu_I(\starA+2t),\ldots$ cannot all be pairwise almost everywhere disjoint, whence there are $k\leq l$ such that $\mu_I((\starA+kt)\cap (\starA+lt))>0$, whence $\BD((A+kt)\cap (A+lt))>0$.  Let $c:=|k-l|t$, so $\BD(A\cap (A+c))>0$.  Let $B:=(A\cap (A+c))-c$.  Then this $B$ and $c$ are as desired.
\end{proof}


After Nathanson's result, there had been very little progress made on proving Conjecture \ref{Erdosconj}.  In 2015, Di Nasso, Goldbring, Leth, Lupini, Jin, and Mahlburg proved the following result \cite{di_nasso_sumset_2015}:

\begin{theorem}\label{highdensityerdos}

\

\begin{enumerate}
\item If $\BD(A)>\frac{1}{2}$, then $A$ has the sumset property.
\item If $\BD(A)>0$, then there is $k\in \mathbb{N}$ such that $A\cup (A+k)$ has the sumset property.
\end{enumerate}
\end{theorem}

In the same paper, the authors establish that \emph{pseudorandom} sets also satisfy the sumset property.  Very recently, significantly building upon the ideas from \cite{di_nasso_sumset_2015}, Moreira, Richter, and Robertson proved Conjecture \ref{Erdosconj} in a very strong form \cite{moreira_proof_2018}:

\begin{theorem}
If $\BD(A)>0$, then $A$ has the sumset property.
\end{theorem}

The proof of the previous theorem is significantly beyond the scope of this book.  However, we believe that there is value in giving the proof of Theorem \ref{highdensityerdos} as it is a perfect example of the utility of nonstandard techniques in combinatorial number theory; the proof will be given in the next section. 

We end this section by establishing a nonstandard reformulation of the sumset property.  We will actually need the following more general statement:

\begin{proposition}\label{nonstandardsumsetshift}
Given $A\subseteq \N$ and $k\in \mathbb{Z}$, the following are equivalent:
\begin{enumerate}
\item there exists $B=\{b_1<b_2<\cdots\}$ and $C=\{c_1<c_2<\cdots\}$ such that $b_i+c_j\in A$ for $i\leq j$ and $b_i+c_j\in A+k$ for $i>j$;
\item there exist nonprincipal ultrafilters $\mathcal{U}$ and $\mathcal{V}$ on $\N$ such that $A\in \mathcal{U}\oplus \mathcal{V}$ and $A+k\in \mathcal{V}\oplus \mathcal{U}$;
\item there exist infinite $\beta,\gamma \in \starN$ such that $\beta+{}^{\ast}\gamma\in {}^{\ast\ast}A$ and $\gamma+^{\ast}\beta\in {}^{\ast\ast}A+k$.
\end{enumerate}
\end{proposition}

\begin{proof}
First suppose that (1) holds as witnessed by $B$ and $C$.  By assumption, the collection of sets
$$\{B\}\cup\{A-c \ : \ c\in C\}$$ has the finite intersection property with the Frech\'et filter, whence there is a nonprincipal ultrafilter $\mathcal{U}$\index{ultrafilter} on $\N$ extending this family.  Likewise, there is a nonprincipal ultrafilter $\mathcal{V}$ on $\N$ extending the family $\{C-k\}\cup \{A-b \ : \ b\in B\}$.  These $\mathcal{U}$ and $\mathcal{V}$ are as desired.

Next, given (2), take $\beta,\gamma\in \starN$ such that $\mathcal{U}=\mathcal{U}_\beta$ and $\mathcal{V}=\mathcal{U}_\gamma$.  These $\beta$ and $\gamma$ are as desired.

Finally, suppose that $\beta,\gamma\in \starN$ are as in (3).  We define $B=\{b_1<b_2<b_3<\cdots\}$ and $C=\{c_1<c_2<c_3<\cdots\}$ recursively as follows.  Suppose that $b_i$ and $c_j$ for $i,j=1,\ldots,n$ have been constructed so that, for all $i,j$ we have: 
\begin{itemize}
\item $b_i+c_j\in A$ if $i\leq j$;
\item $b_i+c_j\in A+k$ if $i>j$;
\item $b_i+\gamma\in {}^{\ast}A$;
\item $c_j+\beta\in {}^{\ast}A+k$.
\end{itemize}
Applying transfer to the statement ``there is $x\in \starN$ such that $x+c_j\in {}^{\ast}A+k$ for $j=1,\ldots,n$ and $x>b_n$ and $x+{}^{\ast}\gamma\in {}^{\ast\ast}A$'' (which is witnessed by $\beta$), we get $b_{n+1}\in \N$ such that $b_{n+1}>b_n$, $b_{n+1}+c_j\in A+k$ for $j=1,\ldots,n$ and for which $b_{n+1}+\gamma\in {}^{\ast}A$.  Next, apply transfer to the statement ``there is $y\in \starN$ such that $b_i+y\in {}^{\ast}A$ for $i=1,\ldots,n+1$ and $y>c_n$ and $y+{}^{\ast}\beta \in {}^{\ast\ast}A+k$'' (which is witnessed by $\gamma$), we get $c_{n+1}\in \N$ such that $c_{n+1}>c_n$ and for which $b_i+c_{n+1}\in A$ for $i=1,\ldots,n+1$ and for which $c_{n+1}+\beta\in {}^{\ast}A$.  This completes the recursive construction.
\end{proof}

Taking $k=0$ in the previous proposition yields a nonstandard reformulation of the sumset property.

\begin{corollary}\label{nonstandardsumset}
Given $A\subseteq \N$, the following are equivalent:
\begin{enumerate}
\item $A$ has the sumset property;
\item there exist nonprincipal ultrafilters $\mathcal{U}$ and $\mathcal{V}$ on $\N$ such that $A\in (\mathcal{U}\oplus \mathcal{V}) \cap (\mathcal{V}\oplus \mathcal{U})$;
\item there exist infinite $\xi,\eta \in \starN$ such that $\xi+{}^{\ast}\eta,\eta+^{\ast}\xi\in {}^{\ast\ast}A$.
\end{enumerate}
\end{corollary}

\section{A 1-shift version of Erd\H{o}s' conjecture}

In this section, we prove Theorem \ref{highdensityerdos}.  We first show how the first part of that theorem, together with Proposition \ref{nonstandardsumsetshift}, yields the second item of the theorem, which we state in an even more precise form.


\begin{proposition}\label{oneshift}
Suppose that $\BD(A)>0$.  Then there exists $B=\{b_1<b_2<\cdots\}$, $C=\{c_1<c_2<\cdots\}$, and $k\in \N$ such that $b_i+c_j\in A$ for $i\leq j$ and $b_i+c_j\in A+k$ for $i>j$.
\end{proposition}

\begin{proof}
By Proposition \ref{fat}, we may fix $n\in \N$ such that $\BD(A+[-n,n])>\frac{1}{2}$.  By Theorem \ref{highdensityerdos}(1) and Corollary \ref{nonstandardsumset}, we may take infinite $\beta,\gamma \in \starN$ such that $\beta+{}^{\ast}\gamma,\gamma+{}^{\ast}\beta\in {}^{\ast\ast}A+[-n,n]$.  Take $i,j\in [-n,n]$ such that $\beta+{}^{\ast}\gamma\in {}^{\ast\ast}A+i$ and $\gamma+{}^{\ast}\beta\in {}^{\ast\ast}A+j$.  Without loss of generality, $i<j$.  Set $k:=j-i$.  Then $\beta+{}^{\ast}(\gamma-i)\in {}^{\ast\ast}A$ and $(\gamma-i)+{}^{\ast}\beta\in {}^{\ast\ast}A+k$, whence the conclusion holds by Proposition \ref{nonstandardsumsetshift}.
\end{proof}

In order to prove the first item in Theorem \ref{highdensityerdos}, we need one technical lemma:

\begin{lemma}\label{Lemma:ergodic}
Suppose that $\BD(A)=r>0$.  Suppose further that $(I_n)$ is a sequence of intervals with witnessing the Banach density\index{density!Banach} of $A$.  Then there is $L\subseteq \N$ satisfying:
\begin{enumerate}
\item $\limsup_{n\to \infty}\frac{|L\cap I_n|}{|I_n|}\geq r$;
\item for all finite $F\subseteq L$, $A\cap \bigcap_{x\in F}(A-x)$ is infinite
\end{enumerate}
\end{lemma}

\begin{proof}
First, we note that it suffices to find $L$ satisfying (1) and 
\begin{enumerate}
\item[(2')] there is $x_0\in {}^{\ast}A\setminus A$ such that $x_0+L\subseteq {}^{\ast}A$. 
\end{enumerate}
 Indeed, given finite $F\subseteq L$ and $K\subseteq \N$, $x_0$ witnesses the truth of ``there exists $x\in {}^\ast A$ such that $x+F\subseteq {}^{\ast}A$ and $x\notin K$'' whence, by transfer, such an $x$ can be found in $\N$, establishing (2).  
 
 In the rest of the proof, we fix infinite $H\in \starN$ and let $\mu$ denote Loeb\index{Loeb measure} measure on $I_H$.  In addition, for any $\alpha\in \starN$ and hyperfinite $X\subseteq \starN$, we set $d_\alpha(X):=\frac{|X|}{|I_\alpha|}$.  Finally, we fix $\epsilon\in (0,\frac{1}{2})$.
 
Next we remark that it suffices to find a sequence $X_1,X_2,\ldots$ of internal subsets of $I_H$ and an increasing sequence $n_1<n_2<\cdots$ of natural numbers such that, for each $i$, we have:
\begin{enumerate}
\item[(i)] $\mu(X_i)\geq 1-\epsilon^i$ and, 
\item[(ii)] for each $x\in X_i$, we have $d_{n_i}({}^{\ast}A\cap (x+I_{n_i}))\geq r-\frac{1}{i}$.
\end{enumerate}

Indeed, suppose that this has been accomplished and set $X:=\bigcap_i X_i$.  Then $X$ is Loeb\index{Loeb measure} measurable and $\mu(X)>0$.  Fix $y_0\in X\setminus \N$ arbitrary and set $x_0$ to be the minimum element of ${}^{\ast}A$ that is greater than or equal to $y_0$; note that $x_0-y_0\in \N$ since $y_0\in X$.  Set $L:=({}^{\ast}A-x_0)\cap \N$; note that (2') is trivially satisfied.  To see that (1) holds, note that
 $$\limsup_{i\to \infty}d_{n_i}(L\cap I_{n_i})=\limsup_{i\to \infty}d_{n_i}({}^{\ast}A\cap (x_0+I_{n_i}))=\limsup_{i\to \infty}d_{n_i}({}^{\ast}A\cap (y_0+I_{n_i}))\geq r,$$ where the last inequality follows from the fact that $y_0\in X$.
 
 Thus, to finish the lemma, it suffices to construct the sequences $(X_i)$ and $(n_i)$.  Suppose that $X_1,\ldots,X_{i-1}$ and $n_1<\cdots<n_{i-1}$ have been constructed satisfying the conditions above.  For $\alpha\in \starN$, set $$Y_\alpha:=\{x\in I_H \ : \ d_\alpha({}^{\ast}A\cap (x+I_m))\geq r-\frac{1}{i}\}.$$  Set $Z:=\{\alpha\in \starN \ : \ n_{i-1}<\alpha \text{ and } d_H(Y_\alpha)>1-\epsilon^i\}$.  Note that $Z$ is internal.  It will be enough to show that $Z$ contains all sufficiently small infinite elements of $\starN$, for then, by underflow, there is $n_i\in Z\cap \N$.  Setting $X_i:=Y_{n_i}$, these choices of $X_i$ and $n_i$ will be as desired.

We now work towards proving that $Z$ contains all sufficiently small infinite elements of $\starN$.  First, we remark that we may assume, without loss of generality, that the sequences $(|I_n|)$ and $(b_n)$ are increasing, where $b_n$ denotes the right endpoint of $I_n$.  Fix $K\in \starN\setminus \N$ such that $2b_K/|I_H|\approx 0$.  We finish the proof of the lemma by proving that $K\in Z$, which we claim follows from the following two facts:
\begin{enumerate}
\item[(a)] for all $x\in I_H$, $\st(d_K(\starA\cap (x+I_K)))\leq r$;
\item[(b)] $\frac{1}{|I_H|}\sum_{x\in I_H}d_K(\starA\cap (x+I_K))\approx r$.
\end{enumerate}

To see that these facts imply that $K\in Z$, for $x\in I_H$, set $f(x):=d_K(\starA\cap (x+I_K))$.  It is enough to show that $f(x)\approx r$ for $\mu$-almost all $x\in I_H$.  Given $n$, let $A_n:=\{x\in I_H \ : \ f(x)<r-\frac{1}{n}\}$.  Suppose, towards a contradiction, that $\mu(A_n)=s>0$.  By (a), we may fix a positive infinitesimal $\eta$ such that $f(x)\leq r+\eta$ for all $x\in I_H$.  We then have
$$\frac{1}{|I_H|}\sum_{x\in I_H}f(x)=\frac{1}{|I_H|}\left[\sum_{x\in A_n}f(x)+\sum_{x\notin A_n}f(x)\right]<s(r-\frac{1}{n})+(1-s)(r+\eta).$$  Since the right-hand side of the above display is appreciably less than $s$, we get a contradiction to (b).

It remains to establish (a) and (b).  (a) follows immediately from the fact that $\BD(A)=r$.  To see (b), we first observe that
$$\frac{1}{|I_H|}\sum_{x\in I_H}d_K(\starA\cap (x+I_K))=\frac{1}{|I_K|}\sum_{y\in I_K}\frac{1}{|I_H|}\sum_{x\in I_H} \chi_{\starA}(x+y).$$  Fix $y\in I_K$.  Since $|\sum_{x\in I_H} \chi_{\starA}(x+y)-|\starA\cap I_H||\leq 2y\leq 2b_K$, we have that 
$$\left|\frac{1}{|I_H|}\sum_{x\in I_H} \chi_{\starA}(x+y)-d_H(\starA)\right|\approx 0.$$  Since a hyperfinite average of infinitesimals is infinitesimal, we see that
$$\frac{1}{|I_H|}\sum_{x\in I_H}d_K(\starA\cap (x+I_K))\approx \frac{1}{|I_K|}\sum_{y\in I_K}d_H(\starA)\approx r,$$ establishing (b).    
\end{proof}

\begin{remark}
A significant strengthening of the previous lemma was one of the main ingredients in the full resolution of Conjecture \ref{Erdosconj} given in \cite{}.
\end{remark}


\begin{proof}[of Theorem \ref{highdensityerdos}]
Set $r:=\BD(A)$.  Let $(I_n)$ witness the Banach density\index{density!Banach} of $A$ and let $L:=(l_n)$ be as in the previous lemma.  We may then define an increasing sequence $D:=(d_n)$ contained in $A$ such that $l_i+d_n\in A$ for $i\leq n$.\footnote{Notice that at this point we have another proof of Nathanson's Theorem \ref{nathansontheorem}:  if we set $B:=\{d_n,d_{n+1},\ldots\}$ and $C:=\{l_1,\ldots,l_n\}$, then $B+C\subseteq A$.}  Now take $N$ such that $\mu_{I_N}({}^{\ast}L)\geq r$.  Note also that $\mu_{I_N}(\starA-d_n)\geq r$ for any $n$.  Since $r>1/2$, for any $n$ we have that $\mu_{I_N}({}^{\ast}L\cap ({}^{\ast}A-d_n))\geq 2r-1>0$.  By a standard measure theory fact, by passing to a subsequence of $D$ if necessary, we may assume that, for each $n$, we have that $\mu_{I_N}({}^{\ast}L\cap \bigcap_{i\leq n}({}^{\ast}A-d_i))>0$.  In particular, for every $n$, we have that $L\cap \bigcap_{i\leq n}(A-d_i)$ is infinite.

We may now conclude as follow.  Fix $b_1\in L$ arbitrary and take $c_1\in D$ such that $b_1+c_1\in A$.  Now assume that $b_1<\cdots<b_n$ and $c_1<\cdots<c_n$ are taken from $L$ and $D$ respectively such that $b_i+c_j\in A$ for all $i,j=1,\ldots,n$.  By assumption, we may find $b_{n+1}\in L\cap \bigcap_{i\leq n}(A-c_i)$ with $b_{n+1}>b_n$ and then we may take $c_{n+1}\in D$ such that $b_i+c_{n+1}\in A$ for $i=1,\ldots,n+1$.
\end{proof}

\section{A weak density version of Folkman's theorem}

At the beginning of this chapter, we discussed the fact that the density version of Hindman's theorem is false.  In fact, the odd numbers also show that the density version of Folkman's theorem is also false.  (Recall that Folkman's theorem stated that for any finite coloring\index{coloring} of $\N$, there are arbitrarily large finite sets $G$ such that $\FS(G)$ are monochromatic.)
However, we can use Lemma \ref{Lemma:ergodic} to prove a weak density version of Folkman's theorem.  Indeed, the proof of Lemma \ref%
{Lemma:ergodic} yields the following:

\begin{lemma}\label{Lemma:ergodic2}
Suppose that $A\subseteq \N$ is such that $\BD(A)\geq r$.  Then there is $\alpha\in \starA\setminus A$ such that $\BD(A-\alpha)\geq r$.
\end{lemma}

One should compare the previous lemma with Beiglbock's Lemma \ref{Blemma}.  Indeed, a special case of (the nonstandard formulation of) Lemma \ref{Blemma} yields $\alpha\in \starN\setminus \N$ such that $\BD(\starA-\alpha)\geq \BD(A)$; the previous lemma is stronger in that it allows us to find $\alpha\in \starA$.  We can now prove the aformentioned weak version of a density Folkman theorem.


\begin{theorem}\label{weakfolkman}
Fix $k\in \mathbb{N}$ and suppose $%
A\subseteq \N$ is such that $\BD(A)>0$. Then there exist
increasing sequences $(x_{n}^{(i)})$ for $i=0,1,2,\ldots ,k$ such that, for any $i$ and any $n_i\leq n_{i+1}\leq \cdots \leq n_k$, we have $%
x_{n_{i}}^{(i)}+x_{n_{i+1}}^{(i+1)}+\cdots +x_{n_{k}}^{(k)}\in A$.
\end{theorem}

The reason we think of the previous theorem as a weak density version of Folkman's theorem is that if all of the sequences were identical, then we would in particular have a set of size $k$ all of whose finite sums belong to $A$.  

\begin{proof}[of Theorem \ref{weakfolkman}]
Set $A=A^{(k)}$. Repeatedly applying Lemma \ref{Lemma:ergodic2}, one can define, for $%
i=0,1,\ldots,k$, subsets $A^{(i)}$ of $\mathbb{N}$ and $\alpha _{i}\in
{}^{\ast }A^{(i)}$ such that $A^{(i)}+\alpha _{i+1}\subseteq {}^{\ast }A^{(i+1)}$%
for all $i<k$.  We then define the sequences $(x_{n}^{(i)})$ for $i=0,1,2,\ldots ,k$ and finite subsets $A_n^{(i)}$ of $A^{(i)}$ so that:
\begin{itemize}
\item for $i=0,1,\ldots,k$ and any $n$, we have $x^{(i)}_n\in A_n^{(i)}$,
\item for $i=0,1,\ldots,k$ and any $n\leq m$, we have $A^{(i)}_n\subseteq A^{(i)}_m$, and
\item for $i=0,1,\ldots,k-1$ and any $n\leq m$, we have $A_n^{(i)}+x^{(i+1)}_m\subseteq A_m^{(i+1)}$.
\end{itemize}
It is clear that the sequences $(x_n^{(i)})$ defined in this manner satisfy the conclusion of the theorem.  Suppose that the sequences $(x_n^{(i)})$ and $A^{(i)}_n$ have been defined for $n<m$.  We now define $x^{(i)}_m$ and $A^{(i)}_m$ by recursion for $i=0,1,\ldots,k$.  We set $x^{(0)}_m$ to be any member of $A^{(0)}$ larger than $x^{(0)}_{m-1}$ and set $A^{(0)}_m:=A^{(0)}_{m-1}\cup \{x^{(0)}_{m-1}\}$.  Supposing that the construction has been carried out up through $i<k$, by transfer of the fact that $A^{(i)}_m+\alpha^{(i+1)}\subseteq \starA^{(i+1)}$, we can find $x^{(i+1)}_m\in A^{(i+1)}$ larger than $x^{(i+1)}_{m-1}$ such that $A^{(i)}_m+x^{(i+1)}_m\subseteq A^{(i+1)}$.  We then define $A^{(i+1)}_m:=A^{(i+1)}_{m-1}\cup (A^{(i)}_m+x^{(i+1)}_m)$.  This completes the recursive construction and the proof of the theorem.
\end{proof}

The usual compactness argument gives a finitary version:

\begin{corollary}
Suppose that $k\in \mathbb{N}$ and $%
\varepsilon >0$ are given. Then there exists $m$ such that for any
interval $I$ of length at least $m$ and any subset $A$ of $I$ such that $%
\left\vert A\right\vert >\varepsilon \left\vert I\right\vert $, there exist $%
(x_{n}^{(i)})$ for $i,n\in \left\{ 0,1,\ldots ,k\right\} $ such that $%
x_{n_{i}}^{(i)}+x_{n_{i+1}}^{(i+1)}+\cdots +x_{n_{\ell
-1}}^{(k)}\in A$ for any $i=0,1,\ldots ,k $ and any $0\leq n_{i}\leq n_{i+1}\leq \cdots
\leq n_{\ell -1}\leq k$.
\end{corollary}

\section*{Notes and references}  The proof of Corollary \ref{oneshift} from Theorem \ref{highdensityerdos} given in \cite{di_nasso_sumset_2015} proceeds via Ramsey's theorem.  The ultrafilter reformulation of the sumset property was first observed by Di Nasso and was used to give this alternate derivation of Corollary \ref{oneshift} from Theorem \ref{highdensityerdos}.  The paper \cite{di_nasso_sumset_2015} also presents a version of Theorem \ref{highdensityerdos} and Corollary \ref{oneshift} for countable amenable groups.  Likewise, the paper \cite{moreira_proof_2018} presents a version of the solution to Conjecture \ref{Erdosconj} for countable amenable groups.

\chapter{Near arithmetic progressions\index{arithmetic progression} in sparse sets}

\section{The main theorem}

Szemeredi's theorem says that relatively dense sets contain arithmetic progressions.  The purpose of this chapter is to present a result of Leth from \cite{leth_near_2006} which shows that certain sparse sets contain ``near'' arithmetic progressions.  Our first task is to make precise what ``near'' means in the previous sentence.

\begin{definition}
Fix $w\in \N_0$ and $t,d\in \N$.\footnote{In this chapter, we deviate somewhat from our conventions so as to match up with the notation from \cite{leth_near_2006}.}   A \emph{$(t,d,w)$-progression} is a set of the form 
$$\B(b,t,d,w):=\bigcup_{i=0}^{t-1}[b+id,b+id+w].$$  By a \emph{block progression} we mean a $(t,d,w)$-progression for some $t,d,w$. 
\end{definition}
Note that a $(t,d,0)$-progression is the same thing as a $t$-term arithmetic progression with difference $d$.

\begin{definition}
If $A\subseteq \N$, we say that $A$ \emph{nearly contains a $(t,d,w)$-progression}\index{$(t,d,w)$-progression} if there is a $(t,d,w)$-progression $\B(b,t,d,w)$ such that $A\cap [b+id,b+id+w]\not=\emptyset$ for each $i=1,\ldots,t-1$. 
\end{definition} 

Thus, if $A$ nearly contains a $(t,d,0)$-progression, then $A$ actually contains a $t$-term arithmetic progression.  Consequently, when $A$ nearly contains a $(t,d,w)$-progression with ``small'' $w$, then this says that $A$ is ``close'' to containing an arithmetic progression.  The main result of this chapter allows us to conclude that even relatively sparse sets with a certain amount of density regularity nearly contain block progressions satisfying a further homogeneity assumption that we now describe.

\begin{definition}\label{homog}
Suppose that $A\subseteq \N$, $I$ is an interval in $\N$, and $0<s<1$.  We say that $A$ \emph{nearly contains a $(t,d,w)$-progression in $I$ with homogeneity $s$}\index{$(t,d,w)$-progression} if there is some $\B(b,t,d,w)$ contained in $I$ such that the following two conditions hold for all $i,j=0,1,\ldots,t-1$:
\begin{enumerate}
\item[(i)] $\delta(A,[b+id,b+id+w])\geq (1-s)\delta(A,I)$
\item[(ii)] $\delta(A,[b+id,b+id+w])\geq (1-s)\delta(A,[b+jd,b+jd+w])$.
\end{enumerate}
\end{definition}

Thus, for small $s$, we see that $A$ meets each block in a density that is roughly the same throughout and that is roughly the same as on the entire interval.  

The density regularity condition roughly requires that on sufficiently large subintervals of $I$, the density does not increase too rapidly.  Here is the precise formulation:

\begin{definition}
Suppose that $I\subseteq \N$ is an interval, $r\in \mathbb{R}^{>1}$, and $m\in \N$.  We say that $A\subseteq I$ has the \emph{$(m,r)$-density property on $I$} if, whenever $J\subseteq I$ is an interval with $|J|/|I|\geq 1/m$, then $\delta(A,J)\leq r\delta(A,I)$.
\end{definition}

Of course, given any $m\in \N$ and $A\subseteq I$, there is $r\in \mathbb{R}^{>1}$ such that $A$ has the $(m,r)$-density property on $I$.  The notion becomes interesting when we think of $r$ as fixed.

Given a hyperfinite interval $I\subseteq \starN$, $r\in {}^{\ast}\mathbb{R}^{>1}$ and $M\in \starN$, we say that an internal set $A\subseteq I$ has the \emph{internal $(M,r)$-density property on $I$} if the conclusion of the definition above holds for internal subintervals $J$ of $I$.

\begin{lemma}\label{Lipshitz}
Suppose that $A\subseteq [1,N]$ is an internal set with the internal $(M,r)$-density property for some $M>\N$.  Let $f:[0,1]\to [0,1]$ be the (standard) function given by $$f(x):=\st\left(\frac{|A\cap [1,xN]|}{|A\cap [1,N]|}\right).$$  Then $f$ is a Lipschitz function with Lipschitz constant $r$.
\end{lemma}

\begin{proof}
Fix $x<y$ in $[0,1]$.  Write $x:=\st(K/N)$ and $y:=\st(L/N)$.  Since $y-x\not=0$, we have that $\frac{L-K}{N}$ is not infinitesimal; in particular, $\frac{L-K}{N}>1/M$.  Since $A$ has the $(M,r)$-density property on $[1,N]$, we have that $\delta(A,[K,L])\leq r\delta(A,[1,N])$.  Thus, it follows that
$$f(y)-f(x)=\st\left(\frac{|A\cap [K,L]|}{|A\cap [1,N]|}\right)=\st\left(\delta(A,[K,L]\frac{L-K}{|A\cap [1,N]|}\right)\leq r\st\left(\frac{L-K}{N}\right)=r(y-x).$$ 
\end{proof}

Here is the main result of this section:

\begin{theorem}[Leth]\label{leth}
Fix functions $g,h:\mathbb{R}^{>0}\to \mathbb{R}^{>0}$ such that $h$ is increasing and $g(x)\to \infty$ as $x\to \infty$.  Fix also $s>0$, $r>1$, and $j,t\in \N$.  Then there is $m=m(g,h,s,r,t,j)\in \N$ such that, for all $n>m$, whenever $I$ is an interval of length $n$ and $A\subseteq I$ is nonempty and has the $(m,r)$-density property on $I$, then $A$ contains a $(t,d,w)$-almost progression with homogeneity $s$ such that $w/d<h(d/n)$ and $1/g(m)<d/n<1/j$.
\end{theorem} 

Roughly speaking, if $A$ has sufficient density regularity, then $A$ contains an almost-progression with ``small'' $w$ (small compared to the distance of the progression).

The proof of the theorem relies on the following standard lemma; see \cite[Lemma 1]{leth_near_2006}.

\begin{lemma}\label{LebesgueAP}
Suppose that $E\subseteq \mathbb{R}$ has positive Lebesgue measure and $t\in \N$.  Then there is $v>0$ such that, for all $0<u<v$, there is an arithmetic progression in $E$ of length $t$ and difference $u$.
\end{lemma}

We stress that in the previous lemma, $u$ and $v$ are real numbers.


\begin{proof}[of Theorem \ref{leth}]
Fix $g,h,s,r,j,t$ as in the statement of Theorem \ref{leth}.  We show that the conclusion holds for all infinite $M$, whence by underflow there exists $m\in \N$ as desired.  Thus, we fix $M>\N$ and consider $N>M$, an interval $I\subseteq \starN$ of length $N$, and a hyperfinite subset $A\subseteq I$ that has the internal $(M,r)$-density property on $I$.  Without loss of generality, we may assume that $I=[1,N]$.  
Suppose that we can find $B,D,W\in \starN$ and standard $c>0$ such that $[B,B+(t-1)D+W]\subseteq [1,N]$ and, for all $i=0,1,\ldots,t-1$, we have:
$$\delta(A,[1,N])(c-\frac{s}{2})\leq \delta(A,[B+iD,B+iD+W])\leq \delta(A,[1,N])(c+\frac{s}{4}).  \quad (\dagger)$$
We claim that $A$ nearly contains the internal $(t,D,W)$- progression $\B(B,t,D,W)$ with homogeneity $s$.  Indeed, item (i) of Definition \ref{homog} is clear.  For item (ii), observe that
$$\delta(A,[B+iD,B+iD+W])\geq \delta(A,[1,N])(c-\frac{s}{2})\geq \delta(A,[B+jD,B+jD+W])(\frac{c-\frac{s}{2}}{c+\frac{s}{4}})$$ and note that $\frac{c-\frac{s}{2}}{c+\frac{s}{4}}>1-s$.  Thus, it suffices to find $B,D,W,c$ satisfying $(\dagger)$ and for which $W/D<h(D/N)$ and $1/g(M)<D/N<1/j$.

Let $f$ be defined as in the statement of Lemma \ref{Lipshitz}.  Set $b:=\st(B/N)$, $d:=\st(D/N)$, and $w:=\st(W/N)$.  Assume that $w\not=0$.  Then we have that
$$\st\left(\frac{\delta(A,[B+iD,B+iD+W])}{\delta(A,[1,N])}\right)=\frac{f(b+id+w)-f(b+id)}{w}.$$  We thus want to find $B,D,W$ and $c$ satisfying
$$c-\frac{s}{2}< \frac{f(b+id+w)-f(b+id)}{w}<c+\frac{s}{4}. \quad (\dagger\dagger)$$

Now the middle term in $(\dagger \dagger)$ looks like a difference quotient and the idea is to show that one can bound $f'(b+id)$ for $i=0,1,\ldots,t-1$.
%
%
%
%
Indeed, by Lemma \ref{Lipshitz}, $f$ is Lipshitz, whence it is absolutely continuous.  In particular, by the Fundamental Theorem of Calculus, $f$ is differentiable almost everywhere and $f(x)=\int_0^x f'(u)du$.  Since $f(0)=0$ and $f(1)=1$, it follows that $\{x\in [0,1] \ : \ f'(x)\geq (1-\frac{s}{4})\}$ has positive measure.  In particular, there is $c>1$ such that $$E:=\{x\in [0,1] \ : \ c-\frac{s}{4}\leq f'(x)\leq c\}$$ has positive measure.  By Lemma \ref{LebesgueAP}, there is $b\in E$ and $0<u<1/j$ such that $b,b+u,b+2u,\ldots,b+(t-1)u\in E$.  Take $B,D\in [1,N]$ such that $b=\st(B/N)$ and $u=\st(D/N)$.  Note that $g(M)$ is infinite and $D/N$ is noninfinitesimal, so $1/g(M)<D/N<1/j$.  It remains to choose $W$.  Since $f$ is differentiable on $E$, there is $w>0$ sufficiently small so that for all $i=0,1,\ldots,t-1$, we have $|f'(b+id)-\frac{f(b+id+w)-f(b+id)}{w}|<\frac{s}{4}$.  For this $w$, $(\dagger\dagger)$ clearly holds; we now take $W$ such that $w=\st(W/N)$.  Since $h(D/N)$ is nonfinitesimal (as $D/N$ is noninfinitesimal), if $w$ is chosen sufficiently small, then $W/D<h(D/N)$.
\end{proof}

Theorem \ref{leth} implies a very weak form of Szemeredi's theorem.

\begin{corollary}
Suppose that $\BD(A)>0$.  Suppose that $g,h,s,t,j$ are as in the hypothesis of Theorem \ref{leth}.  Then for $n$ sufficiently large, there is an interval $I$ of length $n$ such that $A\cap I$ contains a $(t,s,d)$-almost progression in $I$ with $w/d<h(d/n)$ and $1/g(m)<d/n<1/j$.
\end{corollary}

\begin{proof}
Let $r\in \mathbb{R}^{>1}$ be such that $\BD(A)>1/r$.  Let $m:=m(g,h,s,r,t,j)$ as in the conclusion of Theorem \ref{leth}.  Let $n>m$ and take an interval $I$ of length $n$ such that $\delta(A,I)>1/r$.  It remains to observe that $A\cap I$ has the $(m,r)$-density property on $I$.
\end{proof}

\section{Connection to the Erd\H{o}s-Tur\'{a}n conjecture}
Leth's original motivation was the following conjecture of Erd\H{o}s and Tur\'{a}n\index{Erd\H{o}s-Tur\'{a}n's Conjecture} from \cite{erdos_sequences_1936}:

\begin{conjecturenew}[Erd\H{o}s-Tur\'{a}n]
Suppose that $A=(a_n)$ is a subset of $\N$ such that $\sum 1/a_n$ diverges.  Then $A$ contains arbitrarily long arithmetic progressions.
\end{conjecturenew}

Leth first observed the following standard fact about the densities of sequences satisfying the hypotheses of the Erd\H{o}s-Tur\'{a}n conjecture.

\begin{lemma}
Suppose that $A=(a_n)$ is enumerated in increasing order and is such that $\sum 1/a_n$ diverges.  Then, for arbitrarily large $n$, one has $\delta(A,n)>1/(\log n)^2$.
\end{lemma}

\begin{proof}
We argue by contrapositive.  Suppose that $\delta(A,n])\leq 1/(\log n)^2$ for all $n\geq n_0\geq 4$.  We first show that this implies that $a_n\geq \frac{1}{2}n(\log n)^2$ for all $n>n_0$.  Suppose otherwise and fix $n\geq n_0$.  Then $|A\cap [1,\frac{1}{2}n(\log n)^2]|\geq n$.  On the other hand, by our standing assumption, we have that 
$$|A\cap [1,\frac{1}{2}n(\log n)^2])\leq \frac{1/2n(\log n)^2}{(\log((1/2n(\log n))^2}\leq \frac{1}{2}n,$$ yielding the desired contradiction.

Since $a_n\geq \frac{1}{2}n(\log n)^2$ eventually, we have that
$$\sum \frac{1}{a_n}\leq \sum \frac{2}{n(\log n)^2},$$ whence $\sum \frac{1}{a_n}$, converges.
\end{proof}

The truth of the following conjecture, together with the theorem that follows it, would imply that, for sets satisfying the density condition in the previous lemma, the existence of almost arithmetic progressions implies the existence of arithmetic progressions.

\begin{conjecturenew}[Leth]\label{lethconjecture}
Fix $t\in \N$ and $c>0$.  Then there is $n_0:=n_0(t,c)$ such that, for all $n\geq n_0$, whenever $A\subseteq \N$ is such that $\delta(A,n)>1/(c\log n)^{2\log \log n}$, then $A$ nearly contains a $(t,d,w)$-progression on $[1,n]$ with $w/d<d/n$ where $d$ is a power of $2$.\index{Leth's Conjecture}
\end{conjecturenew}

We should remark that requiring that $d$ be a power of $2$ is not much of an extra requirement.  Indeed, our proof of Theorem \ref{leth} shows that one can take  $d$ there to be a power of $2$.  For any $t$ and $c$, we let $L(t,c)$ be the statement that the conclusion of the previous conjecture holds for the given $t$ and $c$.  We let $L(t)$ be the statement that $L(t,c)$ holds for all $c>0$.  

\begin{theorem}\label{AP}
Suppose that $L(t)$ is true for a given $t\in \N$.  Further suppose that $A\subseteq \N$ is such that there is $c>0$ for which, for arbitrarily large $n$, one has $\delta(A,n)>c/(\log n)^2$.  Then $A$ contains an arithmetic progression of length $t$.
\end{theorem}

Before we prove this theorem, we state the following standard combinatorial fact, whose proof we leave as an exercise to the reader (alternatively, this is proven in \cite[Proposition 1]{leth_near_2006}).

\begin{proposition}\label{subinterval}
Let $m,n\in \N$ be such that $m<n$, let $A\subseteq \N$, and let $I$ be an interval of length $n$.  Then there is an interval $J\subseteq I$ of length $m$ such that $\delta(A,J)>\delta(A,I)/2$.
\end{proposition}

\begin{proof}[of Theorem \ref{AP}]
For reasons that will become apparent later in the proof, we will need to work with the set $2A$ rather than $A$.  Note that $2A$ satisfies the hypothesis of the theorem for a different constant $c'>0$.

By overflow \index{overflow principle}, we may find $M>\N$ such that $\delta({}^{\ast}(2A),M)>\frac{c'}{(\log M)^2}$.  Take $L>\N$ such that $2^{2^L}\leq M<2^{2^{L+1}}$ and set $N:=2^{2^{L}}$.  If we apply Proposition \ref{subinterval} to any $n\leq N$ and $I=[1,N]$, we can find an interval $I_n\subseteq [1,M]$ of length $n$ such that $$|{}^{\ast}(2A)\cap I_n|>\frac{c'M}{2(\log M)^2}\geq \frac{c'M}{2(\log 2^{2^{L+1}})^2}=\frac{c'/8}{(\log N)^2}.$$  For $1\leq k\leq L$, write $I_{2^{2^k}}=[x_k,y_k]$.  

We will now construct an internal set $B\subseteq [1,N]$ such that $\delta(B,N)>\frac{1}{(c''\log N)^{2\log \log N}}$, where $c'':=\sqrt{8/c'}$.  Since we are assuming that $L(t)$ holds, by transfer we will be able to find an internal $(t,d,w)$-progression nearly inside of $B$ with $w/d<d/N$ and $w$ and $d$ both powers of $2$.  The construction of $B$ will allow us to conclude that ${}^{\ast}(2A)$ contains a $t$-termed arithmetic progression of difference $d$, whence so does $2A$ by transfer, and thus so does $A$.


Set $B_0:=[1,N]$ and, for the sake of describing the following recursive construction, view $B_0$ as the union of two subintervals of length $N/2=2^{2^L-1}=2^{2^L-2^0}$; we refer to these subintervals of $B_0$ as \emph{blocks}.  Now divide each block in $B_0$ into $2=2^{2^0}$ intervals of length $2^{2^L-{2^0}}/2^{{2^0}}=2^{2^L-2^1}$ and, for each $0\leq j< 2^{2^0}$, we place the $j^{\text{th}}$ subblock of each block in $B_0$ into $B_1$ if and only if $x_0+j\in {}^{\ast}2A$.  

Now divide each block in $B_1$ into $2^{2^1}$ intervals of length $2^{2^L-2^1}/2^{2^1}=2^{2^L-2^2}$ and, for each $0\leq j<2^{2^1}$, we place the $j^{\text{th}}$ subblock of each block in $B_1$ into $B_2$ if and only if $x_1+j\in {}^{\ast}2A$. 

We continue recursively in this manner.  Thus, having constructed the hyperfinite set $B_k$, which is a union of blocks of length $2^{2^L-2^k}$, we break each block of $B_k$ into $2^{2^k}$ many intervals of length $2^{2^L-2^k}/2^{2^k}=2^{2^L-2^{k+1}}$ and we place the $j^{\text{th}}$ subblock of each block in $B_k$ into $B_{k+1}$ if and only if $x_k+j\in {}^{\ast}2A$.

We set $B:=B_L$.  Since $|B_{k+1}|/|B_k|>\frac{c'/8}{(\log N)^2}$ for each $0\leq k<L$, it follows that $$|B|>\frac{(c'/8)^LN}{(\log N)^{2L}}=\frac{N}{(c''\log N)^{2\log\log N}}.$$

By applying transfer to $L(t)$, we have that $B$ nearly contains an internal $(t,d,w)$-progression $\B(b,t,d,w)$ contained in $[1,N]$ such that $w/d<d/N$ and $d$ is a power of $2$.  Take $k$ such that $2^{2^L-2^{k+1}}\leq d<2^{2^L-2^k}$.  Note that this implies that $2^{2^L-2^{k+1}}\mid d$.  Also, we have
$$w<(d/N)\cdot d<(2^{-2^k})2^{2^L-2^k}=2^{2^L-2^{k+1}}.$$

We now note that $\B(b,t,d,w)$ must be contained in a single block $C$ of $B_k$.  Indeed, since $d\mid 2^{2^L-2^k}$ and $w\mid 2^{2^L-2^{k+1}}$, we have $d+w<(\frac{1}{2}+\frac{1}{2^{2^k}})(2^{2^L-2^k})$, whence the fact that $[b,b+w]$ and $[b+d,b+d+w]$ both intersect $B_k$ would imply that $[x_{k-1},y_{k-1}]$ contains consecutive elements of ${}^{\ast}2A$, which is clearly a contradiction.

Now write $d=m\cdot 2^{2^L-2^{k+1}}$.  Take $0\leq j<2^{2^k}$ so that $[b,b+w]$ intersects $B_{k+1}$ in the $j^{\text{th}}$ subblock of $C$ so $x_k+j\in {}^{\ast}2A$.  Since $[b+d,b+d+w]\cap B_{k+1}\not=\emptyset$, we have that at least one of $x_k+j+(m-1)$, $x_k+j+m$, or $x_k+j+(m+1)$ belong to ${}^{\ast}(2A)$.  However, since $x_k+j$ and $m$ are both even, it follows that we must have $x_k+j+m\in {}^{\ast}(2A)$.  Continuing in this matter, we see that $x_k+j+im\in {}^{\ast}2A$ for all $i=0,1,\ldots,t-1$.  It follows by transfer that $2A$ contains a $t$-term arithmetic progression, whence so does $A$.    
\end{proof}

Putting everything together, we have:

\begin{corollary}
The Erd\H{o}s-Tur\'{a}n conjecture\index{Erd\H{o}s-Tur\'{a}n's Conjecture} follows from Leth's Conjecture\index{Leth's Conjecture}.
\end{corollary}

Leth used Theorem \ref{leth} to prove the following theorem, which is similar in spirit to Conjecture \ref{leth}, except that it allows sparser sequences but in turn obtains almost progressions with weaker smallness properties relating $d$ and $w$.

\begin{theorem}
Suppose that $s>0$ and $t\in \N^{>2}$ are gven.  Further suppose that $h$ is as in Theorem \ref{leth}.  Let $A\subseteq \N$ be such that, for all $\epsilon>0$, we have $\delta(A,n)>1/n^{\epsilon}$ for sufficiently large $n$.  Then for sufficiently large $n$, $A$ nearly contains an $(t,d,w)$-progression on $[1,n]$ of homogeneity $s$ with $w/d<h(\log d/\log n)$, where $d$ is a power of $2$.
\end{theorem}

\begin{proof}
Suppose that the conclusion is false.  Then there is $N$ such that $\starA$ does not nearly contain any internal $(t,d,w)$-progression on $[1,N]$ of homogeneity $s$ with $w/d<h(\log d/\log N)$.  It suffices to show that there is $\epsilon>0$ such that $\delta(\starA,N)<1/N^{\epsilon}$.  Let $m$ be as in the conclusion of Theorem \ref{leth} with $r=2$ and $g(x)=x$ (and $h$ as given in the assumptions of the current theorem).

\

\noindent \textbf{Claim:}  If $I\subseteq [1,N]$ is a hyperfinite interval with $|I|>\sqrt{N}$, then $\starA$ does \emph{not} have the $(m,2)$-density property on $I$.

\

We will return to the proof of the claim in a moment.  We first see how the claim allows us to complete the proof of the theorem.  Let $K>\N$ be the maximal $k\in \starN$ such that $m^{2k}\leq N$, so $m^{2K}\leq N<m^{2K+2}$.  We construct, by internal induction, for $i=0,1,\ldots,K$, a descending chain of hyperfinite subintervals $(I_i)$ of $I$ of length $m^{2K-i}$ as follows.  By Proposition \ref{subinterval}, we may take $I_0$ to be any hyperfinite subinterval of $I$ of length $m^{2K}$ such that $\delta(\starA, I_0)\geq \delta(\starA, N)/2$.  Suppose that $i<K$ and $I_i$ has been constructed such that $|I_i|=m^{2k-i}$.  Since $\starA$ does not have the $(m,2)$ density property on $I_i$, there is a subinterval $I_{i+1}$ of length $|I_i|/m^{2k-i-1}$ with $\delta(\starA, I_{i+1})\geq 2\delta(\starA,I_i)$.  Notice now that $I_K$ is a hyperfinite interval of length $m^K\leq \sqrt{N}<m^{K+1}$ and $\delta(\starA,I_K)\geq 2^K\delta(\starA,I_0)$.  It follows that
$$\delta(\starA,N)\leq 2\delta(A,I_0)\leq 2^{-(K-1)}\delta(A,I_K)\leq 2^{-(K-1)}.$$  It follows that
$$|A\cap [1,N]|\leq 2^{-(K-1)}N\leq 2^{-(K-1)}m^{2K+2}=m^{2K+2-(K-1)\frac{\log 2}{\log m}}=(m^{2K})^{1-z}.$$  if we set $z:=\frac{(K-1)\log 2}{2K\log m}-\frac{1}{K}$.  If we set $\epsilon:=\st(z/2)=\frac{\log 2}{4\log m}$, then it follows that $|A\cap [1,N]|\leq N^{1-\epsilon}$, whence this $\epsilon$ is as desired.

We now prove the claim.  Suppose, towards a contradiction, that $I\subseteq [1,N]$ is a hyperfinite interval with $|I|>\sqrt{N}$ and is such that $\starA$ does have the $(m,2)$-density property on $I$.  By the choice of $m$, $\starA$ nearly contains an internal $(t,d,w)$-almost progression of homogeneity $s$ with $w/d<h(d/|I|)$ and $d>|I|/m>\sqrt{N}/m$.  Notice now that $\st\left(\frac{\log d}{\log N}\right)\geq \st\left(\frac{1/2\log N-\log m}{\log N}\right)=\frac{1}{2}$.  Note that we trivially have that $d/|I|<1/t$, whence $d/|I|<\log d/\log N$; since $h$ is increasing, we have that $w/d<h(\log d/\log N)$, contradicting the choice of $N$.  This proves the claim and the theorem.
\end{proof}

In \cite[Theorem 3]{leth_nonstandard_2007}, Leth shows that one cannot replace $(\log d)/(\log n)$ with $d/n$ in the previous theorem.

\section*{Notes and references}  There are other generalizations of arithmetic progressions appearing in the literature, e.g. the notion of \emph{quasi-progression} appearing in \cite{}.  It should be noted that they use the term $(t,d,w)$-progression in a related, but different, manner than it is used in this chapter.  The Erd\H os-Turan conjecture, first formulated in \cite{}, is one of the most important open problems in combinatorial number theory.  A positive solution would immediately generalize both Szemeredi's Theorem and the Green-Tao theorem on the existence of arbitrarily long arithmetic progressions in the primes \cite{}.
%
%
%

\chapter{The interval measure property}

\section{IM sets}
Let $I:=[y,z]$ be an infinite, hyperfinite interval.  Set $\st_I:=\st_{[y,z]}:I\to [0,1]$ to be the map $\st_I(a):=\st(\frac{a-y}{z-y})$.  For $A\subseteq \starN$ internal, we set $\st_I(A):=\st_I(A\cap I)$.  We recall that $\st_I(A)$ is a closed subset of $[0,1]$ and we may thus consider $\lambda_I(A):=\lambda(\st_I(A))$, where $\lambda$ is Lebesgue measure on $[0,1]$.  

We also consider the quantity $g_A(I):=\frac{d-c}{|I|}$, where $[c,d]\subseteq I$ is maximal so that $[c,d]\cap A=\emptyset$.

The main concern of this subsection is to compare the notions of making $g_A(I)$ small (an internal notion) and making $\lambda_I(A)$ large (an external notion).  There is always a connection in one direction:

\begin{lemma}\label{onedirection}
If $\lambda_I(A)> 1-\epsilon$, then $g_A(I)<\epsilon$.
\end{lemma}

\begin{proof}
Suppose that $g_A(I)\geq \epsilon$, whence there is $[c,d]\subseteq I$ such that $[c,d]\cap A=\emptyset$ and $\frac{d-c}{|I|}\geq \epsilon$.  It follows that, for any $\delta>0$, we have $(\st_I(c)+\delta,\st_I(d)-\delta)\cap \st_I(A)=\emptyset$, whence $$\lambda_I(A)\leq 1-\left(\st\left(\frac{d-c}{|I|}\right)-2\delta\right)\leq 1-\epsilon+2\delta.$$  Letting $\delta\to 0$ yields the desired result.
\end{proof}

We now consider sets where there is also a relationship in the other direction.

\begin{definition}
We say that $A$ has the \emph{interval-measure property}\index{interval-measure property} (or \emph{IM property}) on $I$ if for every $\epsilon>0$, there is $\delta>0$ such that, for all infinite $J\subseteq I$ with $g_{A}(J)\leq \delta$, we have $\lambda_{J}(A)\geq 1-\epsilon$.
\end{definition}

If $A$ has the IM property on $I$, we let $\delta(A,I,\epsilon)$ denote the supremum of the $\delta$'s that witness the conclusion of the definition for the given $\epsilon$.

It is clear from the definition that if $A$ has the IM property on an interval, then it has the IM property on every infinite subinterval.  Also note that it is possible that $A$ has the IM property on $I$ for a trivial reason, namely that there is $\delta>0$ such that $g_A(J)>\delta$ for every infinite $J\subseteq I$.  Let us temporarily say that $A$ has the \emph{nontrivial IM property} on $I$ if this does \emph{not} happen, that is, for every $\delta>0$, there is an infinite interval $J\subseteq I$ such that $g_A(J)\leq \delta$.  It will be useful to reformulate this in different terms.  In order to do that, we recall an important standard tool that is often employed in the study of sets with the IM property, namely the \emph{Lebesgue density theorem}.  Recall that for a measurable set $E\subseteq [0,1]$, a point $r\in E$ is a \emph{(one-sided) point of density of $E$} if
$$\lim_{s\to r^+}\frac{\mu(E\cap[r,s])}{s-r}=1.$$  The Lebesgue density theorem asserts that almost every point of $E$ is a density point of $E$.

\begin{proposition}
Suppose that $A\subseteq \starN$ is internal and $I$ is an infinite, hyperfinite interval such that $A$ has the IM property on $I$.  Then the following are equivalent:
\begin{enumerate}
\item There is an infinite subinterval $J$ of $I$ such that $A$ has the nontrivial IM property on $J$.
\item There is an infinite subinterval $J$ of $I$ such that $\lambda_J(A)>0$.
\end{enumerate}
\end{proposition}

\begin{proof}
First suppose that $J$ is an infinite subinterval of $I$ such that $A$ has the nontrivial IM property on $J$.  Let $J'$ be an infinite subinterval of $J$ such that $g_A(J')\leq \delta(A,J,\frac{1}{2})$.  It follows that $\lambda_{J'}(A)\geq \frac{1}{2}$.  

Now suppose that $J$ is an infinite subinterval of $I$ such that $\lambda_J(A)>0$.  By the Lebesgue density theorem, there is an infinite subinterval $J'$ of $J$ such that $\lambda_{J'}(A)>1-\delta$.  By Lemma \ref{onedirection}, we have that $g_{A}(J')<\delta$, whence $g_A(J)<\delta$.  It follows that $A $ has the nontrivial IM property on $J$. 
\end{proof}

In practice, the latter property in the previous proposition is easier to work with.  Consequently, let us say that $A$ has the \emph{enhanced IM property on $I$} if it has the IM property on $I$ and $\lambda_I(A)>0$.\footnote{This terminology does not appear in the original article of Leth.}

We now seek to establish nice properties of sets with the IM property.  We first establish a kind of partition regularity theorem.  

\begin{theorem}\label{IMpartreg}
Suppose that $A$ has the enhanced IM property on $I$.  Further suppose that $A\cap I=B_1\cup\cdots \cup B_n$ with each $B_i$ internal.  Then there is $i$ and infinite $J\subseteq I$ such that $B_i$ has the enhanced IM property on $J$.
\end{theorem}

\begin{proof}
We prove the theorem by induction on $n$.  The result is clear for $n=1$.  Now suppose that the result is true for $n-1$ and suppose $A\cap I=B_1\cup\cdots \cup B_n$ with each $B_i$ internal.  If there is an $i$ and infinite $J\subseteq I$ such that $B_i\cap J=\emptyset$ and $\lambda_J(A)>0$, then we are done by induction.  We may thus assume that whenever $\lambda_J(A)>0$, then each $B_i\cap J\not=\emptyset$.  We claim that this implies that each of the $B_i$ have the IM property on $I$.  Since there must be an $i$ such that $\lambda_I(B_i)>0$, for such an $i$ it follows that $B_i$ has the enhanced IM property on $I$.  

Fix $i$ and set $B:=B_i$.  Suppose that $J\subseteq I$ is infinite, $\epsilon>0$, and $g_B(J)\leq \delta(A,I,\epsilon)$; we show that $\lambda_J(B)\geq 1-\epsilon$.  Since $g_A(J)\leq g_B(J)\leq \delta(A,I,\epsilon)$, we have that $\lambda_J(A)\geq 1-\epsilon$.  Suppose that $[r,s]\subseteq [0,1]\setminus \st_J(B)$.  Then $r=\st_J(x)$ and $s=\st_J(y)$ with $\frac{y-x}{|J|}\approx s-r$ and $B\cap[x,y]=\emptyset$.  By our standing assumption, this implies that $\lambda_{[x,y]}(A)=0$, whence it follows that $\lambda_J(A\cap [x,y])=0$.  It follows that $\lambda_J(B)=\lambda_J(A)\geq 1-\epsilon$, as desired.
\end{proof}

If $A$ has the IM property on an interval $I$ and we have a subinterval of $I$ on which $A$ has small gap ratio, then by applying the IM property, the Lebesgue density theorem, and Lemma \ref{onedirection}, we can find a smaller, but appreciably sized, subinterval on which $A$ once again has small gap ratio.  Roughly speaking, one can iterate this procedure until one finds a \emph{finite} subinterval of $I$ on which $A$ has small gap ratio; the finiteness of the subinterval will be crucial for applications.   We now give a precise formulation.

 Fix internal sets $A_1,\ldots,A_n$ and intervals $I_1,\ldots,I_n$.  Fix also $\delta>0$.  A \emph{$\delta$-configuration} (with respect to $A_1,\ldots,A_n,I_1,\ldots,I_n)$ is a sequence of subintervals $J_1,\ldots,J_n$ of $I_1,\ldots,I_n$ respectively such that each $|J_i|$ has the same length and such that $g_{A_i}(J_i)\leq \delta$ for all $i$.  We call the common length of the $J_i$'s the \emph{length} of the configuration.  There is an obvious notion of \emph{$\delta$-subconfiguration}, although, for our purposes, we will need a stronger notion of subconfiguration.  Indeed, we say that a $\delta$-subconfiguration $J_1',\ldots,J_n'$ of $J_1,\ldots,J_n$ is a \emph{strong $\delta$-subconfiguration} if there is some $c\in \starN$ such that, writing $a_i$ for the left endpoint of $J_i$, we have that $a_i+c$ is the left endpoint of $J_i'$.  Note that the strong $\delta$-subconfiguration relation is transitive.
 
\begin{theorem}\label{mainIMtheorem}
Suppose that $A_1,\ldots,A_n$ are internal sets that satisfy the IM property on $I_1,\ldots,I_n$ respectively.  Fix $\epsilon>0$ such that $\epsilon<\frac{1}{n}$.  Take $\delta>0$ with $\delta<\min_{i=1,\ldots,n}\delta(A_i,I_i,\epsilon)$.  Then there is $w\in \N$ such that any $\delta$-configuration has a strong $\delta$-subconfiguration of length at most $w$.

\end{theorem}

\begin{proof}
Let $A_1,\ldots,A_n$, $I_1,\ldots,I_n$, $\epsilon$ and $\delta$ be as in the statement of the theorem.  The entire proof rests on the following:

\noindent \textbf{Claim:}  Any $\delta$-configuration of infinite length has a proper strong $\delta$-subconfiguration.

\

Given the claim, the proof of the theorem proceeds as follows:  let $\mathcal{C}$ denote the internal set of $\delta$-configurations.  Let $f:\mathcal{C}\to \starN$ be the internal function given by $f(J_1,\ldots,J_n)=$ the minimal length of a minimal strong $\delta$-subconfiguration of $J_1,\ldots,J_n$.  By the claim, the range of $f$ is contained in $\mathbb{N}$.  Thus, there is $w\in \mathbb{N}$ such that the range of $f$ is contained in $[1,w]$, as desired.
%
%
%
%

\

\noindent Thus, to finish the proof of the theorem, it suffices to prove the claim.

\

\noindent \textbf{Proof of Claim:}  
Write $J_i:=[a_i,a_i+b]$ for $i=1,\ldots,n$.  By assumption, $\lambda_{J_i}(A_i)\geq 1-\frac{1}{n}$, whence $ \lambda(\bigcap_{i=1}^n\st_{J_i}(A_i))>0$.  Let $r$ be a point of density for $\bigcap_{i=1}^n\st_{J_i}(A_i)$.  Thus, there is $s<1-r$ such that $$\lambda\left(\left(\bigcap_{i=1}^n \st_{J_i}(A_i)\right)\cap [r,r+s]\right)\geq (1-\delta)s.$$  Set $c:=\lfloor r\cdot b\rfloor$ and $b':=\lfloor s\cdot b\rfloor$.  Then $c+b'\leq b$ and, by Lemma \ref{onedirection}, we have $$g_{A_i}([a_i+c,a_i+c+b'])\leq \delta \text{ for all }i=1,\ldots,n.$$  Thus, the $[a_i+c,a_i+c+b']$ form the desired proper strong $\delta$-subconfiguration.
\end{proof}

A special case of Theorem \ref{mainIMtheorem} is worth singling out:

\begin{corollary}\label{corIMmain}
Let $A_1,\ldots,A_n$, $I_1,\ldots,I_n$, $\epsilon$, and $\delta$ be as in Theorem \ref{mainIMtheorem}.  Then there is $w\in \N$ such that, whenever $[a_i,a_i+b]$ is a $\delta$-configuration, then there is $c\in \starN$ such that
$$A_i\cap [a_i+c,a_i+c+w]\not=\emptyset \text{ for all }i=1,\ldots,n.$$
\end{corollary}

By refining the proof of Theorem \ref{mainIMtheorem}, we obtain the following:

\begin{corollary}\label{corofproof}
If $A$ has the IM property on $I$, then there is $w\in \N$ and a descending hyperfinite sequence $I=I_0,I_1,\ldots,I_K$ of hyperfinite subintervals of $I$ such that:
\begin{itemize}
\item $|I_K|\leq w$;
\item $\frac{|I_{k+1}|}{|I_k|}\geq \frac{1}{w}$;
\item whenever $I_k$ is infinite, we have $\lambda_{I_k}(A)>0$.
\end{itemize}
\end{corollary}

\begin{proof}
First note that the proof of the Claim in Theorem \ref{mainIMtheorem} actually yields that every $\delta$-configuration has a proper strong $\delta$-subconfiguration where the ratio of lengths is non-infinitesimal.  Thus, a saturation argument yields $\epsilon>0$ such that every $\delta$-configuration has a proper strong $\delta$-subconfiguration with ratio of lengths at least $\epsilon$.  The corollary follows easily from this, specializing to the case of a single internal set on a single interval.
\end{proof}

%
%
%

\begin{definition}
For any (not necessarily internal) $A\subseteq \starN$, we set $$D(A):=\{n\in \N \ : \ n=a-a' \text{ for infinitely many pairs }a,a'\in A\}.$$  
\end{definition}

The following corollary will be important for our standard application in the next section.

\begin{corollary}\label{DAsyndetic}
Suppose that $A$ has the enchanced IM property on $I$.  Then $D(A)$ is syndetic.
\end{corollary}

\begin{proof}
Let $w\in \N$ be as in Corollary \ref{corIMmain} for $A_1=A_2=A$ and $I_1=I_2=I$.  It suffices to show that for all $m\in \N$, there are infinitely many pairs $(x,y)\in A^2$ such that $y-x\in [m-w,m+w]$ (as then $[m-w,m+w]\cap D(A)\not=\emptyset$).
  

By considering countably many distinct points of density of $\st_I(A)$ and using Lemma \ref{onedirection} and overflow, we may find pairwise disjoint infinite subintervals $J_n:=[a_n,b_n]\subseteq I$ such that $g_A(J_n)\approx 0$.  Note also that $g_A(J_n+m)\approx 0$.  Thus, by the choice of $w$, for each $n$, there is $c_n\in \starN$ such that $$A\cap [a_n+c_n,a_n+c_n+w],A\cap [a_n+m+c_n,a_n+m+c_n+w]\not=\emptyset.$$  If $x_n\in A\cap [a_n+c_n,a_n+c_n+w]$ and $y_n\in A\cap [a_n+m+c_n,a_n+m+c_n+w]$, then $y_n-x_n\in (A-A)\cap [m-w,m+w]$.  By construction, the pairs $(x_n,y_n)$ are all distinct. 
\end{proof}

\section{SIM sets}

We now seek to extract the standard content of the previous section.

\begin{definition}
$A\subseteq \N$ has the \emph{standard interval-measure property}\index{standard interval-measure property} (or \emph{SIM property}) if:
\begin{itemize}
\item $\starA$ has the IM property on every infinite hyperfinite interval; 
\item $\starA$ has the enhanced IM property on some infinite hyperfinite interval.
\end{itemize}
\end{definition}

\begin{example}
Let $A=\bigcup_n I_n$, where each $I_n$ is an interval, $|I_n|\to \infty$ as $n\to \infty$, and there is $k\in \N$ such that the distance between consecutive $I_n$'s is at most $k$.  Then $A$ has the SIM property.
\end{example}

We now reformulate the definition of SIM set using only standard notions.  (Although recasting the SIM property in completely standard terms is not terribly illuminating, it is the polite thing to do.)  First, note that one can define $g_A(I)$ for standard $A\subseteq \N$ and standard finite intervals $I\subseteq \N$ in the exact same manner.  Now, for $A\subseteq \N$ and $0<\delta<\epsilon<1$, define the function $F_{\delta,\epsilon,A}:\N\to \N$ as follows.  First, if $g_A(I)>\delta$ for every $I\subseteq \N$ of length $\geq n$, set $F_{\delta,\epsilon,A}(n)=0$.  Otherwise, set $F_{\delta,\epsilon,A}(n)=$ the minimum $k$ such that there is an interval $I\subseteq \N$ of length $\geq n$ such that $g_A(I)\leq \delta$ and there are subintervals $I_1,\ldots,I_k\subseteq I$ with $I_i\cap A=\emptyset$ for all $i=1,\ldots,k$ and $\sum_{i=1}^k |I_i|\geq \epsilon |I|$.

\begin{theorem}
$A$ has the SIM property if and only if:  for all $\epsilon>0$, there is $\delta>0$ such that $\lim_{n\to \infty} F_{\delta,\epsilon,A}(n)=\infty$.
\end{theorem}

\begin{proof}
First suppose that there is $\epsilon>0$ such that $\liminf_{n\to \infty}F_{\delta,\epsilon,A}(n)<\infty$ for all $\delta>0$; we show that $A$ does not have the SIM property.  Towards this end, we may suppose that $\lambda_I(\starA)>0$ for some infinite hyperfinite interval $I$ and show that $\starA$ does not have the IM property on some infinite interval.  Fix $0<\delta<\epsilon$.  By the Lebesgue density theorem and Lemma \ref{onedirection}, we have that $g_{\starA}(J)\leq \delta$ for some infinite subinterval $J\subseteq I$.  By transfer, there are intervals $J_n\subseteq \N$ of length $\geq n$ such that $g_A(J_n)\leq \delta$, whence $0<\liminf_{n\to \infty}F_{\delta,\epsilon,A}(n)$ for all $0<\delta<\epsilon$.  For every $k\geq 1$, set $m_k:=1+\liminf_{n\to \infty}F_{\frac{1}{k},\epsilon,A}(n)$.  Consequently, for every $n\in \N$, there are intervals $I_{1,n},\ldots,I_{n,n}$ of length $\geq n$ such that, for each $k=1,\ldots,n$, $g_{A}(I_{k,n})\leq\frac{1}{k}$ and the sum of the lengths of $m_{k}$ many gaps of $A$ in $I_{k,n}$ is at least $\epsilon\cdot |I_{k,n}|$.  Set $I_n:=I_{1,n}\cup\cdots\cup I_{n,n}$.  By overflow, there is an infinite, hyperfinite interval $\tilde{I}$ that contains infinite subintervals $I_k$ such that $g_{\starA}(I_k)\leq \frac{1}{k}$ and yet the sum the lengths of of $m_{k}$ many gaps of $A$ on $I_k$ have size at least $\epsilon |I_k|$.  It follows that $\starA$ does not have the IM property on $\tilde{I}$.

Now suppose that for all $\epsilon>0$, there is $\delta>0$ such that $\lim_{n\to \infty} F_{\delta,\epsilon,A}(n)=\infty$ and that $I$ is an infinite, hyperfinite interval such that $g_{\starA}(I)\leq \delta$.  By transfer, it follows that no finite number of gaps of $\starA$ on $I$ have size at least $\epsilon\cdot |I|$.  Since $\st_I(\starA)$ is closed, we have that $\lambda_I(\starA)\geq 1-\epsilon$.  Consequently, $A$ has the IM property on any infinite, hyperfinite interval.  Since, by transfer, there is an infinite, hyperfinite interval $I$ with $g_{\starA}(I)\leq \delta$, this also shows that $\starA$ has the enhanced IM property on this $I$.  Consequently, $A$ has the SIM property.

\end{proof}

\begin{exercise}\label{singleSIM}
Suppose that $A\subseteq \mathbb{N}$ has the SIM property.  Show that, for each $\epsilon>0$, there is a $\delta>0$ such that $\delta\leq \delta(\starA,I,\epsilon)$ \emph{for every} infinite hyperfinite intervals $I\subseteq \starN$.
\end{exercise}

The next lemma shows that the SIM property is not simply a measure of ``largeness'' as this property is not preserved by taking supersets.

\begin{lemma}
Suppose that $A\subseteq \N$ is not syndetic.  Then there is $B\supseteq A$ such that $B$ does not have the SIM property.
\end{lemma}

\begin{proof}
For each $n$, let $x_n\in \N$ be such that $[x_n,x_n+n^2]\cap A=\emptyset$.  Let $$B:=A\cup\bigcup_n \{x_n+kn \ : \ k=0,1,\ldots n\}.$$
Fix $\epsilon>0$.  Take $m\in \N$ such that $m>\frac{1}{\epsilon}$ and take $N>\N$.  Set $I:=[x_N,x_N+mN]$.  Indeed, $g_{{}^{\ast}B}(I)=\frac{N}{mN}<\frac{1}{\epsilon}$ while $$\st_I({}^{\ast}B)=\left\{\st\left(\frac{kN}{mN}\right) \ : \ k=0,\ldots,m\right\}=\left\{\frac{k}{m} \ : \ k=0,\ldots,m\right\}$$ is finite and thus has measure $0$.  It follows that ${}^{\ast}B$ does not have the IM property on $I$, whence $B$ does not have the SIM property. 
\end{proof}

The previous lemma also demonstrates that one should seek structural properties of a set which ensure that it contains a set with the SIM property.  Here is an example:

\begin{lemma}
If $B$ is piecewise syndetic, then there is $A\subseteq B$ with the SIM property.
\end{lemma}

\begin{proof}
For simplicity, assume that $B$ is thick; the argument in general is similar, just notationally more messy.  Let $A:=\bigcup_n I_n$, with $I_n$ intervals contained in $B$, $|I_n|\to \infty$ as $n\to \infty$, and such that, setting $g_n$ to be the length in between $I_n$ and $I_{n+1}$, we have $g_{n+1}\geq ng_n$ for all $n$.  We claim that $A$ has the SIM property.  It is clear that $\lambda_I(\starA)>0$ for some infinite hyperfinite interval $I$; indeed, $\lambda_{I_N}(\starA)=1$ for $N>\N$.    Now suppose that $I$ is an infinite hyperfinite interval; we claim that $\starA$ has the IM property on $I$ as witnessed by $\delta=\epsilon$.  Suppose that $J$ is an infinite subinterval of $I$ such that $g_{\starA}(J)\leq \epsilon$.  Suppose that $I_n,\ldots,I_{M+1}$ is a maximal collection of intervals from $\starA$ intersecting $J$.  Since $\frac{g_M}{|J|}\leq \epsilon$, for $k=N,\ldots,M-1$, we have $\frac{g_k}{|J|}=\frac{g_k}{g_M}\cdot \frac{g_M}{|J|}\approx 0$, whence the intervals $I_n,\ldots,I_M$ merge when one applies $\st_J$.  It follows that $\lambda_J(\starA)\geq 1-\epsilon$.
\end{proof}

In connection with the previous result, the following question seems to be the most lingering open question about sets that contain subsets with the SIM property:

\begin{question}
Does every set of positive Banach density contain a subset with the SIM property?
\end{question}

The next result shows that many sets do \emph{not} have the SIM property.

\begin{proposition}
Suppose that $A=(a_n)$ is a subset of $\N$ written in increasing order.  Suppose that $\lim_{n\to \infty} (a_{n+1}-a_n)=\infty$.  Then $A$ does not have the SIM property.
\end{proposition}

\begin{proof}
Suppose that $A$ has the SIM property.  Take $I$ such that $\lambda_I(\starA)>0$.  Then by the proof of Corollary \ref{DAsyndetic}, we can find $x,y\in \starA\setminus A$ such that $x<y$ and $y-x\leq 2w$.  Then, by transfer, there are arbitrarily large $m,n\in \N$ with $0<m-n\leq 2w$.  It follows that $\lim_{n\to \infty}(a_{n+1}-a_n)\not=\infty$.
\end{proof}

The following theorem provides a connection between the current chapter and the previous one.  The proof follows immediately from Theorem \ref{mainIMtheorem}.

\begin{theorem}
Suppose that $A\subseteq \mathbb{N}$ is a SIM set.  Fix $t\in \mathbb{N}$ and $0<\epsilon<\frac{1}{t}$.  Let $\delta>0$ be as in Exercise \ref{singleSIM} for $\epsilon$. Then there is $j\in \mathbb{N}$ such that whenever $A$ nearly contains a $(t,d,w)$-progression $\B(b,t,d,w)$, then $A$ nearly contains a \emph{subprogression}\footnote{Here, subprogression means that every block $[b'+id,b'+id+j]$ is contained in the corresponding block $[b+id,b+id+w]$.} $\B(b',t,d,j)$ of $\B(b,t,d,w)$.
\end{theorem}

We end this section with a result concerning a structural property of sets with the SIM property.  A direct consequence of Corollary \ref{DAsyndetic} is the following:

\begin{corollary}\label{standardDAsyndetic}
If $A$ has the SIM property, then $D(A)$ is syndetic.
\end{corollary}

Leth's original main motivation for studying the IM property was a generalization of the previous corollary.  Stewart and Tijdeman \cite{stewart_infinite_1979} proved that, given $A_1,\ldots,A_n\subseteq \N$ with $\BD(A_i)>0$ for all $i=1,\ldots,n$, one has $D(A_1)\cap\cdots \cap D(A_n)$ is syndetic.  Leth proved the corresponding statement for sets with the SIM property:

\begin{theorem}
If $A_1,\ldots,A_n\subseteq \N$ all have the SIM property, then $D(A_1)\cap \cdots \cap D(A_n)$ is syndetic.
\end{theorem}

\begin{proof}
We break the proof up into pieces.

\

\noindent \textbf{Claim 1:}  There are infinite hyperfinite intervals $I_1,\ldots,I_n$, all of which have the same length, such that $\lambda_{I_i}(\starA_i)=1$ for all $i=1,\ldots,n$. 

\noindent \textbf{Proof of Claim 1:}  By the definition of the SIM property and Corollary \ref{corofproof}, we may find infinite, hyperfinite intervals $J_1,\ldots,J_n$ whose length ratios are all finite and for which $\st_{J_i}(\starA_i)>0$ for $i=1,\ldots,n$.  By taking points of density in each of these intervals, for any $\epsilon>0$, we may find \emph{equally sized} subintervals $J_i'$ of $J_i$ such that $\lambda_{J_i'}(\starA_i)\geq 1-\epsilon$, whence $g_{\starA_i}(J_i')\leq \epsilon$.  Since this latter condition is internal, by saturation, we may find equally sized subintervals $I_i$ of $J_i$ such that each $g_{\starA_i}(I_i)\approx 0$, whence, by the fact that $\starA_i$ has the IM property on $J_i$, we have $\lambda_{I_i}(\starA_i)=1$.

\

\noindent We now apply Corollary \ref{corIMmain} to $A_1,\ldots,A_n$, $I_1,\ldots,I_n$ and $\epsilon:=\frac{1}{n+1}$.  Let $w\in \N$ be as in the conclusion of that corollary.  Write $I_i:=[x_i,y_i]$ and for $i=1,\ldots,n$, set $d_i:=x_i-x_1$.  We then set
$$B:=\{a\in \starA_1\cap I_1 \ : \ \starA_i \cap [a+d_i-w,a+d_i+2w]\not=\emptyset \text{ for all }i=1,\ldots,n\}.$$

\

\noindent \textbf{Claim 2:}  Suppose that $J\subseteq I_1$ is infinite and $r$ is a point of density of 
$$\bigcap_{i=1}^n \st_{J+d_i}(\starA_i).$$  Then $r\in \st_J(B)$.

\noindent \textbf{Proof of Claim 2:}  By a (hopefully) by now familiar Lebesgue density and overflow argument, there is an infinite hyperfinite interval $[u,v]\subseteq J$ such that $\st_J(u)=\st_J(v)=r$ and $$g_{\starA_i}([u+d_i,v+d_i])\approx 0 \text{ for all } i=1,\ldots,n.$$  This allows us to find $c\in \starN$ such that $u+d_i+c+w\leq v_i$ and $\starA_i\cap [u+d_i+c,u+d_i+c+w]\not=\emptyset$ for $i=1,\ldots,n$.  Take $a\in \starA_1\cap [u+c,u+c+w]$, say $a=u+c+j$ for $j\in [0,w]$.  It follows that $$\starA_i\cap [a+d_i-j,a+d_i+j+w]\not=\emptyset \text{ for all }i=1,\ldots,n$$ whence $a\in B$.  Since $u\leq u+c\leq a\leq u+c+w\leq v$, we have that $\st_J(a)=\st_J(v)=r$, whence $r\in \st_J(B)$, as desired.      

\

\noindent \textbf{Claim 3:}  $B$ has the enhanced IM property on $I_1$.

\noindent \textbf{Proof of Claim 3:}  Taking $J=I_1$ in Claim 2 shows that $\lambda_{I_1}(B)=1$.  We now show that $B$ has the IM property on $I_1$.  Fix $\epsilon>0$.  Let $\delta=\min_{i=1,\ldots,n}\delta(\starA_i,I_i,\frac{\epsilon}{n})$.  Suppose $J\subseteq I_1$ is such that $g_B(J)\leq \delta$.  Then $g_{\starA_i}(J+d_i)\leq \delta$, whence $$\lambda\left(\bigcap_{i=1}^n \st_{J+d_i}(\starA_i)\right)\geq 1-\epsilon.$$  By Claim 2, we have $\lambda_J(B)\geq 1-\epsilon$, as desired.  

\

\noindent For $-w\leq k_1,\ldots,k_n\leq 2w$, set
$$B_{(k_1,\ldots,k_n)}:=\{b\in B \ : \ b+d_i+k_i\in \starA_i \text{ for all }i=1,\ldots,n\}.$$  By the definition of $B$, we have that $B$ is the union of these sets.  Since $B$ has the enhanced IM property on $I_1$, by Theorem \ref{IMpartreg}, there is such a tuple $(k_1,\ldots,k_n)$ and an infinite $J\subseteq I_1$ such that $B':=B_{(k_1,\ldots,k_n)}$ has the enhanced IM property on $J$.  By Corollary \ref{DAsyndetic}, $D(B')$ is syndetic.  Since $B'-B'\subseteq \bigcap_{i=1}^n (\starA_i-\starA_i)$, by transfer we have that $D(A_1)\cap \cdots \cap D(A_n)$ is syndetic.
\end{proof}

\section*{Notes and references}  The material in this chapter comes from the paper \cite{leth_some_1988}, although many of the proofs appearing above were communicated to us by Leth and are simpler than those appearing in the aforementioned article.  In a recent preprint \cite{}, Goldbring and Leth study the notion of a \emph{supra-SIM} set, which is simply a set that contains a SIM set.  They show that these sets have very nice properties such as being partition regular and closed under finite-embeddability.  They also show that SIM sets satisfy the conclusions of the Sumset Theorem \ref{jinsumset} and Nathansons' Theorem \ref{nathansontheorem}.

\part{Other topics}

\chapter{Triangle removal and Szemeredi regularity}\label{trianglechapter}

\section{Triangle removal lemma}

The material in this section was not proven first by nonstandard methods.  However, the nonstandard perspective makes the proofs quite elegant.  We closely follow \cite{tao_hilberts_2014}.\index{Triangle Removal Lemma}

Suppose that $G=(V,E)$ is a finite graph.  We define the \emph{edge density} of $G$ to be the quantity
$$e(G):=\frac{|E|}{|V\times V|}$$ and the \emph{triangle density} of $G$ to be the quantity
$$t(G):=\frac{|\{(x,y,z)\in V\times V\times V \ : \ (x,y), (y,z), (x,z)\in E\}|}{|V\times V \times V|}.$$

 
\begin{theorem}[Triangle removal lemma]
For every $\epsilon>0$, there is a $\delta>0$ such that, whenever $G=(V,E)$ is a finite graph with $t(G)\leq \delta$, then there is a subgraph $G'=(V,E')$ of $G$ that is triangle-free (so $t(G')=0$) and such that $e(G\setminus G')\leq \epsilon$.
\end{theorem}

In short, the triangle removal lemma\index{Triangle Removal Lemma} says that if the triangle density of a graph is small, then one can remove a few number of edges to get one that is actually triangle-free.  We first show how the Triangle Removal Lemma can be used to prove Roth's theorem\index{Roth's theorem}, which was a precursor to Szemeredi's theorem.

\begin{theorem}[Roth's theorem]\index{Roth's theorem}
For all $\epsilon>0$, there is $n_0\in \N$ such that, for all $n\geq n_0$ and all $A\subseteq [1,n]$, if $\delta(A,n)\geq \epsilon$, then $A$ contains a 3-term arithmetic progression.
\end{theorem}

\begin{proof}
Fix $n$ and form a tripartite graph $G=G(A,n)$ with vertex set $V=V_1\cup V_2\cup V_3$, where each $V_i$ is a disjoint copy of $[1,3n]$.  If $(v,w)\in (V_1\times V_2)\cup (V_2\times V_3)$, we declare $(v,w)\in E \Leftrightarrow w-v\in A$.  If $(v,w)\in V_1\times V_3$, then we declare $(v,w)\in E\Leftrightarrow (w-v)\in 2A$.  Note then that if $(v_1,v_2,v_3)$ is a triangle in $G$, then setting $a:=v_2-v_1$, $b:=v_3-v_2$, and $c:=\frac{1}{2}(v_3-v_1)$, we have that $a,b,c\in A$ and $a-c=c-b$.  If this latter quantity is nonzero, then $\{a,b,c\}$ forms a 3-term arithmetic progression in $A$.  

Motivated by the discussion in the previous paragraph, let us call a triangle $\{v_1,v_2,v_3\}$ in $G$ \emph{trivial} if $v_2-v_1=v_3-v_2=\frac{1}{2}(v_3-v_1)$.  Thus, we aim to show that, for $n$ sufficiently large, if $\delta(A,n)\geq 1-\epsilon$, then $G(A,n)$ has a nontrivial triangle.  If $a\in A$ and $k\in [1,n]$, then $(k,k+a,k+2a)$ is a trivial triangle in $G$.  Since trivial triangles clearly do not share any edges, one would have to remove at least $3\cdot |A|\cdot n\geq 3\epsilon n^2$ many edges of $G$ in order to obtain a triangle-free subgraph of $G$.  Thus, if $\delta>0$ corresponds to $3\epsilon$ in the triangle removal lemma\index{Triangle Removal Lemma}, then we can conclude that $t(G)\geq \delta$, that is, there are at least $27\delta n^3$ many triangles in $G$.  Since the number of trivial triangles is at most $|A|\cdot (3n)\leq 3n^2$, we see that $G$ must have a nontrivial triangle fi $n$ is sufficiently large.  
\end{proof}

We now turn to the proof of the triangle removal lemma\index{Triangle Removal Lemma}.  The basic idea is that if the triangle removal lemma were false, then by a now familiar compactness/overflow argument, we will get a contradiction to some nonstandard triangle removal lemma\index{Triangle Removal Lemma}.  Here is the precise version of such a lemma:

\begin{theorem}[Nonstandard triangle removal lemma]\label{nstrl}\index{Triangle Removal Lemma!nonstandard}
Suppose that $V$ is a nonempty hyperfinite set and $E_{12}, E_{23}, E_{13}\in \l_{V\times V}$ are such that
$$\int_{V\times V\times V}1_{E_{12}}(u,v)1_{E_{23}}(v,w)1_{E_{13}}(u,v)d\mu(u,v,w)=0. \quad (\dagger)$$  Then for every $\epsilon>0$ and $(i,j)\in \{(1,2),(2,3),(1,3)\}$, there are hyperinite $F_{ij}\subseteq V\times V$ such that $\mu_{V\times V}(E_{ij}\setminus F_{ij})<\epsilon$ and
$$1_{F_{12}}(u,v)1_{F_{23}}(v,w)1_{F_{13}}(u,v)=0 \text{ for all }(u,v,w)\in V\times V\times V. \quad (\dagger\dagger)$$
\end{theorem}

\begin{proposition}
The nonstandard triangle removal lemma\index{Triangle Removal Lemma!nonstandard} implies the triangle removal lemma.
\end{proposition}

\begin{proof}
Suppose that the triangle removal lemma is false.  Then there is $\epsilon>0$ such that, for all $n\in \N$, there is a finite graph $G_n=(V_n,E_n)$ for which $t(G_n)\leq \frac{1}{n}$ and yet there does not exist a triangle-free subgraph $G'=(V_n,E_n')$ with $|E_n\setminus E_n'|\leq \epsilon |V_n|^2$.  Note that it follows that $|V_n|\to \infty$ as $n\to \infty$.  By \index{overflow principle}, there is an infinite hyperfinite graph $G=(V,E)$ such that $t(G)\approx 0$, whence $(\dagger)$ holds, and yet there does not exist a triangle-free hyperfinite subgraph $G'=(V,E')$ with $|E\setminus E'|\leq \epsilon |V|^2$.  We claim that this latter statement yields a counterexample to the nonstandard triangle removal lemma.  Indeed, if the nonstandard triangle removal held, then there would be hyperfinite $F_{ij}\subseteq V\times V$ such that $\mu_{V\times V}(E\setminus F_{ij})<\frac{\epsilon}{6}$ and for which $(\dagger \dagger)$ held.  If one then sets $E':=E\cap \bigcap_{ij}(F_{ij}\cap F_{ij}^{-1})$, then $G'=(V,E')$ is a hyperfinite subgraph of $G$ that is triangle-free and $\mu(E\setminus E')<\epsilon$, yielding the desired contradiction.\footnote{Given a binary relation $R$ on a set $X$, we write $R^{-1}$ for the binary relation on $X$ given by $(x,y)\in R^{-1}$ if and only if $(y,x)\in R$.}
\end{proof}

It might look like the nonstandard triangle removal lemma\index{Triangle Removal Lemma!nonstandard} is stated in a level of generality that is more than what is needed for we have $E_{12}=E_{23}=E_{13}=E$.  However, in the course of proving the lemma, we will come to appreciate this added level of generality of the statement.

\begin{lemma}
Suppose that $f\in L^2(\l_{V\times V})$ is orthogonal to $L^2(\l_V\otimes \l_V)$.  Then for any $g,h\in L^2(\l_{V\times V})$, we have
$$\int_{V\times V\times V} f(x,y)g(y,z)h(x,z)d\mu_{V\times V\times V}(x,y,z)=0.$$
\end{lemma}

\begin{proof}
Fix $z\in V$.  Let $g_z:V\to \mathbb{R}$ be given by $g_z(y):=g(y,z)$.  Likewise, define $h_z(x):=h(x,z)$.  Note then that $g_z\cdot h_z\in L^2(\l_V\otimes \l_V)$.  It follows that
$$\int_{V\times V} f(x,y)g(y,z)h(x,z)d\mu_{V\times V}(x,y)=\int_{V\times V} f(x,y)g_z(y)h_z(x)d\mu_{V\times V}(x,y)=0.$$
By Theorem \ref{fubiniloeb}, we have that
$$\int_{V\times V\times V} f(x,y)g(y,z)h(x,z)d\mu_{V\times V\times V}(x,y,z)=\int_V \left[\int_{V\times V} f(x,y)g(y,z)h(x,z)d\mu_{V\times V}(x,y)\right]d\mu_V(z)=0.$$
\end{proof}

\begin{proof}[of Theorem \ref{nstrl}]
We first show that we can assume that each $E_{ij}$ belongs to $\l_V\otimes \l_V$.  Indeed, let $f_{ij}:=\mathbb{E}[1_{E_{ij}}|\l_V\otimes \l_V]$.\footnote{Here, for $f\in L^2(\l_{V\times V})$, $\mathbb{E}[f|\l_V\otimes \l_V]$ denotes the conditional expectation of $f$ onto the subspace $L^2(\l_V\otimes \l_V)$.}   Then by three applications of the previous lemma, we have
$$\int_{V\times V\times V} f_{12}f_{23}f_{13}d\mu_{V\times V\times V}=\int_{V\times V\times V} f_{12}f_{23}1_{13}d\mu_{V\times V\times V}=\int_{V\times V\times V} f_{12}1_{23}1_{13}d\mu_{V\times V\times V}=\int_{V\times V\times V} 1_{12}1_{23}1_{13}d\mu_{V\times V\times V}=0. \quad (*)$$  Let $G_{ij}:=\{(u,v)\in V\times V \ : \ f_{ij}(u,v)\geq \frac{\epsilon}{2}\}$.  Observe that each $G_{ij}$ belongs to $\l_V\otimes \l_V$ and
$$\mu(E_{ij}\setminus G_{ij})=\int_{V\times V} 1_{E_{ij}}(1-1_{G_{ij}})d\mu_{V\times V}=\int_{V\times V} f_{ij}(1-1_{G_{ij}})d\mu_{V\times V}\leq \frac{\epsilon}{2}.$$  By $(*)$ we have
$$\int_{V\times V} 1_{G_{12}}1_{G_{23}}1_{G_{13}}d\mu_{V\times V}=0.$$   Thus, if the nonstandard triangle removal lemma is true for sets belonging to $\l_V\otimes \l_V$, we can find hyperfinite $F_{ij}\subseteq V\times V$ such that $\mu(G_{ij}\setminus F_{ij})<\frac{\epsilon}{2}$ and such that $(\dagger \dagger)$ holds.  Since $\mu(E_{ij}\setminus F_{ij})<\epsilon$, the $F_{ij}$ are as desired.

Thus, we may now assume that each $E_{ij}$ belongs to $\l_V\otimes \l_V$.  Consequently, there are elementary sets $H_{ij}$ such that $\mu(E_{ij}\triangle H_{ij})<\frac{\epsilon}{6}$.  By considering the boolean algebra generated by the sides of the boxes appearing in the description of $H_{ij}$, we obtain a partition $V=V_1\sqcup \cdots \sqcup V_n$ of $V$ into finitely many hyperfinite subsets of $V$ such that each $H_{ij}$ is a union of boxes of the form $V_k\times V_l$ for $k,l\in \{1,\ldots,n\}$.  Let
$$F_{ij}:=\bigcup\{V_k\times V_l \ : \ V_k\times V_l\subseteq H_{ij}, \ \mu(V_k\times V_l)>0, \text{ and } \mu(E_{ij}\cap (V_k\times V_l))>\frac{2}{3}\mu(V_k\times V_l)\}.$$  Clearly each $F_{ij}$ is hyperfinite.  Note that $$\mu(H_{ij}\setminus F_{ij})=\mu((H_{ij}\setminus F_{ij})\cap E_{ij})+\mu((H_{ij}\setminus F_{ij})\setminus E_{ij})\leq \frac{2}{3}\mu(H_{ij}\setminus F_{ij})+\frac{\epsilon}{6},$$ whence $\mu(H_{ij}\setminus F_{ij})\leq \frac{\epsilon}{2}$ and thus $\mu(E_{ij}\setminus F_{ij})\leq \frac{\epsilon}{6}+\frac{\epsilon}{2}<\epsilon$.  It remains to show that $(\dagger \dagger)$ holds.  Towards a contradiction, suppose that $(u,v,w)$ witnesses that $(\dagger \dagger)$ is false.  Take $k,l,m\in \{1,\ldots,n\}$ such that $u\in V_k$, $v\in V_l$, and $w\in V_m$.  Since $(u,v)\in F_{12}$, we have that $\mu(E_{12}\cap (V_k\times V_l))>\frac{2}{3}\mu(V_k\times V_l)$.  Consequently, $\mu(E_{12}\times V_m)>\frac{2}{3}\mu(V_k\times V_l\times V_m)$.  Similarly, we have that $\mu(E_{23}\times V_k),\mu(E_{13}\times V_l)>\frac{2}{3}\mu(V_k\times V_l\times V_m)$.  Thus, by elementary probability considerations, it follows that $$\int_{V\times V} 1_{E_{12}}1_{E_{23}}1_{E_{13}}d\mu_{V\times V}>0,$$ contradicting $(\dagger)$. 
\end{proof}

\section{Szemer\'edi Regularity Lemma}

Suppose that $(V,E)$ is a finite graph.  For two nonempty subsets $X$, $Y$ of $V$, we define the \emph{density of arrows between $X$ and $Y$} to be the quantity $$d(X,Y):=\delta(E,X\times Y)=\frac{|E\cap (X\times Y)|}{|X||Y|}.$$  For example, if every element of $X$ is connected to every element of $Y$ by an edge, then $d(X,Y)=1$.  Fix $\epsilon\in \mathbb{R}^{>0}$.  We say that $X$ and $Y$ as above are \emph{$\epsilon$-pseudorandom} if whenever $A\subseteq X$ and $B\subseteq Y$ are such that $|A|\geq \epsilon |X|$ and $|B|\geq \epsilon |Y|$, then $|d(A,B)-d(X,Y)|<\epsilon$.  In other words, as long as $A$ and $B$ contain at least an $\epsilon$ proportion of the elements of $X$ and $Y$ respectively, then $d(A,B)$ is essentially the same as $d(X,Y)$, so the edges between $X$ and $Y$ are distributed in a sort of random fashion.  

If $X=\{x\}$ and $Y=\{y\}$ are singletons, then clearly $X$ and $Y$ are $\epsilon$-pseudorandom for any $\epsilon$.  Thus, any finite graph can trivially be partitioned into a finite number of $\epsilon$-pseudorandom pairs by partitioning the graph into singletons.  Szemer\'edi's Regularity Lemma essentially says that one can do much better in the sense that there is a constant $C(\epsilon)$ such that any finite graph has an ``$\epsilon$-pseudorandom partition'' into at most $C(\epsilon)$ pieces.  Unfortunately, the previous sentence is not entirely accurate as there is a bit of error that we need to account for.

Suppose that $V_1,\ldots,V_m$ is a partition of $V$ into $m$ pieces.  Set $$R:=\{(i,j) \ | \ 1\leq i,j\leq m, \ \ V_i \text{ and }V_j \text{ are }\epsilon\text{-pseudorandom}\}.$$  We say that the partition is \emph{$\epsilon$-regular} if $\sum_{(i,j)\in R}\frac{|V_i||V_j|}{|V|^2}>(1-\epsilon).$  This says that, in some sense, almost all of the pairs of points are in $\epsilon$-pseudorandom pairs.  We can now state:

\begin{theorem}[Szemer\'edi's Regularity Lemma]\index{Szemer\'edi's Regularity Lemma}
For any $\epsilon \in \mathbb{R}^{>0}$, there is a constant $C(\epsilon)$ such that any graph $(V,E)$ admits an $\epsilon$-regular partition into $m\leq C(\epsilon)$ pieces.
\end{theorem}

As in the previous section, the regularity lemma is equivalent to a nonstandard version of the lemma.  We leave the proof of the equivalence as an exercise to the reader.

\begin{proposition}\index{Szemer\'edi's Regularity Lemma!nonstandard}
Szemer\'edi's Regularity Lemma is equivalent to the following statement:  for any $\epsilon$ and any hyperfinite graph $(V,E)$, there is a \emph{finite} partition $V_1,\ldots,V_m$ of $V$ into internal sets and a subset $R\subseteq \{1,\ldots,m\}^2$ such that:
\begin{itemize}
\item for $(i,j)\in R$, $V_i$ and $V_j$ are \emph{internally} $\epsilon$-pseudorandom:  for all \emph{internal} $A\subseteq V_i$ and $B\subseteq V_j$ with $|A|\geq \epsilon |V_i|$ and $|B|\geq \epsilon |V_j|$, we have $|d(A,B)-d(V_i,V_j)|<\epsilon$; and
\item $\sum_{(i,j)\in R}\frac{|V_i||V_j|}{|V|^2}>(1-\epsilon).$
\end{itemize}
\end{proposition}

We will now prove the above nonstandard equivalent of the Szemer\'edi Regularity Lemma.  Fix $\epsilon$ and a hyperfinite graph $(V,E)$.  
Set $f:=\mathbb E[1_E |\l_V\otimes \l _V]$.  The following calculation will prove useful:  Suppose that $A,B\subseteq V$ are internal and $\frac{|A|}{|V|}$ and $\frac{|B|}{|V|}$ are noninfinitesimal.  Then $(\clubsuit)$:
\begin{alignat}{2}
\int_{A\times B}fd(\mu_V\otimes \mu_V)&=\int_{A\times B} 1_Ed\mu_{V\times V} \quad \text{ by the definition of }f\notag \\ \notag
							     &=\st\left(\frac{|E\cap (A\times B)|}{|V|^2}\right)\notag \\
							     &=\st\left(\frac{|E\cap (A\times B)|}{|A||B|}\right)\st\left(\frac{|A||B|}{|V|^2}\right)\notag \\
							     &=\st(d(A,B))\st\left(\frac{|A||B|}{|V|^2}\right).\notag
\end{alignat}

Fix $r\in \mathbb{R}^{>0}$, to be determined later.  Now, since $f$ is $\mu_V\otimes \mu_V$-integrable, there is a $\mu_V\otimes \mu_V$-simple function $g\leq f$ such that $\int (f-g)d(\mu_V\otimes \mu_V)<r$.  Set $C:=\{\omega \in V\times V \ | \ f(\omega)-g(\omega)\geq \sqrt{r}\}\in s_V\otimes s_V$.  Then $(\mu_V\otimes \mu_V)(C)<\sqrt{r}$, for otherwise
$$\int(f-g)d(\mu_V\otimes \mu_V)\geq \int_C (f-g)d(\mu_V\otimes \mu_V)\geq \int_C \sqrt{r}d(\mu_V\otimes \mu_V)\geq \sqrt{r}\sqrt{r}=r.$$  By Fact \ref{approx}, there is an elementary set $D\in s_V\otimes s_V$ that is a finite, disjoint union of rectangles of the form $V'\times V''$, with $V',V''\subseteq V$ internal sets, such that $C\subseteq D$ and $(\mu_V\otimes \mu_V)(D)<\sqrt{r}$.  In a similar way, we may assume that the level sets of $g$ (that is, the sets on which $g$ takes constant values) are elementary sets (Exercise).  We now take a finite partition $V_1,\ldots,V_m$ of $V$ into internal sets such that $g$ and $1_D$ are constant on each rectangle $V_i\times V_j$.  For ease of notation, set $d_{ij}$ to be the constant value of $g$ on $V_i\times V_j$.

\

\noindent \textbf{Claim:}  If $\mu_V(V_i),\mu_V(V_j)\not=0$ and $(V_i\times V_j)\cap D=\emptyset$, then $V_i$ and $V_j$ are internally $2\sqrt{r}$-pseudorandom.

\

\noindent \textbf{Proof of Claim:}  Since $C\subseteq D$, we have that $(V_i\times V_j)\cap C=\emptyset$, whence 
$$d_{ij}\leq f(\omega)<d_{ij}+\sqrt{r}\text{ for }\omega \in V_i\times V_j. \quad (\clubsuit \clubsuit).$$  Now suppose that $A\subseteq V_i$ and $B\subseteq V_j$ are such that $|A|\geq 2\sqrt{r}|V_i|$ and $|B|\geq 2\sqrt{r}|V_j|$.  In particular, $\frac{|A|}{|V_i|}$ and $\frac{|B|}{|V_j|}$ are noninfinitesimal.  Since $\mu_V(V_i),\mu_V(V_j)>0$, it follows that $\frac{|A|}{|V|}$ and $\frac{|B|}{|V|}$ are noninfinitesimal and the calculation $(\clubsuit)$ applies.  Integrating the inequalities $(\clubsuit \clubsuit)$ on $A\times B$ yields:
$$d_{ij}\st\left(\frac{|A||B|}{|V|^2}\right)\leq \st(d(A,B))\st\left(\frac{|A||B|}{|V|^2}\right)< (d_{ij}+\sqrt{r})\st\left(\frac{|A||B|}{|V|^2}\right).$$  We thus get:
$$|d(A,B)-d(V_i,V_j)|\leq |d(A,B)-d_{ij}|+|d(V_i,V_j)-d_{ij}|< 2\sqrt{r}.$$

\

By the Claim, we see that we should choose $r<(\frac{\epsilon}{2})^2$, so $V_i$ and $V_j$ are internally $\epsilon$-pseudorandom when $V_i$ and $V_j$ are non-null and satisfy $(V_i\times V_j)\cap D=\emptyset$.  It remains to observe that the $\epsilon$-pseudorandom pairs almost cover all pairs of vertices.  Let $R:=\{(i,j) \ | \ V_i\text{ and }V_j \text{ are }\epsilon\text{-pseudorandom}\}$.  Then
\begin{alignat}{2}
\st\left(\sum_{(i,j)\in R} \frac{|V_i||V_j|}{|V|^2}\right)&=\mu_{V\times V}\left(\bigcup_{(i,j)\in R}(V_i\times V_j)\right)\notag \\ \notag
						         &\geq \mu_{V\times V}((V\times V)\setminus D)\notag \\
						         &> 1-\sqrt{r} \notag \\
						         &> 1-\epsilon.\notag
\end{alignat}

This finishes the proof of the Claim and the proof of the Szemer\'edi Regularity Lemma.

\section*{Notes and references}  The triangle removal lemma was originally proven by Rusza and Szemer\'edi in \cite{} and their proof used several applications of an early version of the Szemer\'edi regularity lemma.  Nowadays, the most common standard proof of the triangle removal lemma goes through a combination of the Szemer\'edi regularity lemma and the so-called \emph{Counting lemma}; see, for example, \cite{}.  Szemer\'edi's regularity lemma was a key ingredient in his proof in \cite{szemeredi_sets_1975} that sets of postiive density contain arbitrarily long arithmetic progressions (which we now of course call Szemer\'edi's theorem).  Analogous to the above proof of Roth's Theorem from the Triangle Removal Lemma, one can prove Szemeredi's theorem by first proving an appropriate removal lemma called the \emph{Hypergraph removal lemma} and then coding arithmetic progressions by an appropriate hypergraph generalization of the argument given above.  For more details, see \cite{gowers_hypergraph_2007} for the original standard proof and \cite{tao_correspondence_2007} and \cite{goldbring_approximate_2014} for simplified nonstandard proofs.

\chapter{Approximate groups}
\index{approximate group}
In this chapter, we describe a recent application of nonstandard methods to multiplicative combinatorics, namely to the structure theorem for finite approximate groups.  The general story is much more complicated than the rest of the material in this book and there are already several good sources for the complete story (see \cite{Breuillard_Structure_2012} or \cite{van_den_dries_approximate_2015}), so we content ourselves to a summary of some of the main ideas.  Our presentation will be similar to the presentation from \cite{van_den_dries_approximate_2015}.

\section{Statement of definitions and the main theorem}

In this chapter, $(G,\cdot)$ denotes an arbitrary group and $K\in \mathbb{R}^{\geq 1}$.  (Although using $K$ for a real number clashes with the notation used throughout the rest of this book, it is standard in the area.)  One important convention will be important to keep in mind:  for $X$ a subset of $G$ and $n\in \N$, we set $X^n:=\{x_1\cdots x_n \ : \ x_1,\ldots,x_n\in X\}$ (so $X^n$ does \emph{not} mean the $n$-fold Cartesian power of $X$).

By a \emph{symmetric} subset of $G$, we mean a set that contains the identity of $G$ and is closed under taking inverse. 

\begin{definition}
$X\subseteq G$ is a \emph{$K$-approximate group}\index{approximate group} if $X$ is symmetric and $X^2$ can be covered by at most $K$ left translates of $X$, that is, there are $g_1,\ldots,g_m\in G$ with $m\leq K$ such that $X^2\subseteq \bigcup_{i=1}^m g_iX$.
\end{definition}

\begin{example}

\

\begin{enumerate}
\item A $1$-approximate subgroup of $G$ is simply a subgroup of $G$.
\item If $X\subseteq G$ is finite, then $X$ is a $|X|$-approximate subgroup of $G$.
\end{enumerate}

\end{example}

The second example highlights that, in order to try to study the general structure of finite $K$-approximate groups, one should think of $K$ as fixed and ``small'' and then try to classify the finite $K$-approximate groups $X$, where $X$ has cardinality much larger than $K$.

\begin{exercise}
Suppose that $(G,+)$ is an abelian group.  For distinct $v_1\ldots,v_r\in G$ and (not necessarily distinct) $N_1,\ldots,N_r\in \N$, set 
$$P(\vec v, \vec N):=\{a_1v_1+\cdots+a_rv_r \ : \ a_i\in \mathbb{Z}, \ |a_i|\leq N_i\}.$$  Show that $P(\vec v,\vec N)$ is a $2^r$-approximate subgroup of $G$.
\end{exercise}

The approximate subgroups appearing in the previous exercise are called \emph{symmetric generalized arithmetic progressions} and the number $r$ of generators is called the \emph{rank} of the progression.  The \emph{Freiman Theorem for abelian groups}\index{Freiman's Theorem} (due to due to Freiman \cite{freiman_foundations_1973} for $\Z$ and to Green and Rusza \cite{green_freimans_2007} for a general abelian group) says that approximate subgroups of abelian groups are ``controlled'' by symmetric generalized arithmetic progressions:

\begin{theorem}\label{nonabelianfreiman}
There are constants $r_K,C_K$ such that the following hold:  Suppose that $G$ is an abelian group and $A\subseteq G$ is a finite $K$-approximate group.  Then there is a finite subgroup $H$ of $G$ and a symmetric generalized arithmetic progression $P\subseteq G/H$ such that $P$ has rank at most $r_K$, $\pi^{-1}(P)\subseteq \Sigma_4(A)$, and $|P|\geq C_K \cdot \frac{|A|}{|H|}$.
\end{theorem}

Here, $\pi:G\to G/H$ is the quotient map.  For a while it was an open question as to whether there was a version of the Freiman theorem that held for finite approximate subgroups of arbitrary groups.  Following a breakthrough by Hrushovski \cite{Hrushovski_Stable_2012}, Breuillard, Green, and Tao \cite{Breuillard_Structure_2012} were able to prove the following general structure theorem for approximate groups.

\begin{theorem}
There are constants $r_K,s_K,C_K$ such that the following hold:  Suppose that $G$ is a group and $A\subseteq G$ is a finite $K$-approximate group.  Then there is a finite subgroup $H\subseteq G$, a \emph{noncommutative progression} of rank at most $r_K$ whose generators generate a nilpotent group of step at most $s_K$ such that $\pi^{-1}(P)\subseteq A^4$ and $|P|\geq C_K \cdot \frac{|A|}{|H|}$.
\end{theorem}

Here, $\pi:G\to G/H$ is once again the quotient map.  To understand this theorem, we should explain the notion of noncommutative progression.

Suppose that $G$ is a group, $v_1,\ldots, v_r\in G$ are distinct, and $N_1,\ldots,N_r>0$ are (not necessarily distinct) natural numbers.  The noncommutative progression generated by $v_1,\ldots,v_r$ with dimensions $N_1,\ldots,N_r$ is the set of words on the alphabet $\{v_1,v_1^{-1},\ldots,v_r,v_r^{-1}\}$ such that the total number of occurrences of $v_i$ and $v_i^{-1}$ is at most $N_i$ for each $i=1,\ldots,r$; as before, $r$ is called the rank of the progression.  In general, noncommutative progressions need not be approximate groups (think free groups).  However, if $v_1,\ldots,v_r$ generate a nilpotent subgroup of $G$ of step $s$, then for $N_1,\ldots,N_r$ sufficiently large, the noncommutative progression is in fact a $K$-approximate group for $K$ depending only on $r$ and $s$.  (See, for example, \cite[Chapter 12]{tao_hilberts_2014}.)

\section{A special case:  approximate groups of finite exponent}
To illustrate some of the main ideas of the proof of the Breuillard-Green-Tao theorem, we prove a special case due to Hrushovski \cite{Hrushovski_Stable_2012}:

\begin{theorem}\label{udi}
Suppose that $X\subseteq G$ is a finite $K$-approximate group.  Assume that $X^2$ has exponent $e$, that is, for every $x\in X^2$, we have $x^e=1$.  Then $X^4$ contains a subgroup $H$ of $\langle X\rangle$ such that $X$ can be covered by $L$ left cosets of $H$, where $L$ is a constant depending only on $K$ and $e$.
\end{theorem}

Here, $\langle X\rangle$ denotes the subgroup of $G$ generated by $X$.  Surprisingly, this theorem follows from the simple observation that the only connected Lie group which has an identity neighborhood of finite exponent is the trivial Lie group consisting of a single point.  But how do continuous objects such as Lie groups arise in proving a theorem about finite objects like finite approximaite groups?  The key insight of Hrushovski is that ultraproducts of finite $K$-approximate groups are naturally ``modeled'' in a precise sense by second countable, locally compact groups and that, using a classical theorem of Yamabe, this model can be perturbed to a Lie model.

More precisely, for each $i\in \N$, suppose that $X_i\subseteq G_i$ is a finite $K$-approximate group.  We set $X:=\prod_\u X_i$, which, by transfer, is a hyperfinite $K$-approximate subgroup of $G:=\prod_\u G_i$.  In the rest of this chapter, unless specified otherwise, $X$ and $G$ will denote these aforementioned ultraproducts.  By a \emph{monadic} subset of $G$ we mean a countable intersection of internal subsets of $G$.  Also, $\langle X\rangle$ denotes the subgroup of $G$ generated by $X$.

\begin{theorem}\label{lcmodel}
There is a monadic subset $o(X)$ of $X^4$ such that $o(X)$ is a normal subgroup of $\langle X\rangle$ such that the quotient $\G:=\langle X\rangle/o(X)$ has the structure of a second countable, locally compact group.  Moreover, letting $\pi:\langle X\rangle \to \G$ denote the quotient map, we have:
\begin{enumerate}
\item The quotient $\langle X\rangle/o(X)$ is \emph{bounded}, meaning that for all internal sets $A,B\subseteq \langle X\rangle$ with $o(X)\subseteq A$, finitely many left translates of $A$ cover $B$.
\item $Y\subseteq \G$ is compact if and only if $\pi^{-1}(Y)$ is monadic; in particular, $\pi(X)$ is compact.
\item If $Y\subseteq G$ is internal and contains $o(X)$, then $Y$ contains $\pi^{-1}(U)$ for some open neighborhood of the identity in $\G$.
\item $\pi(X^2)$ is a compact neighborhoods of the identity in $\G$.
\end{enumerate}
\end{theorem}


%



Let us momentarily assume that Theorem \ref{lcmodel} holds and see how it is used to prove Theorem \ref{udi}.  As usual, we first prove a nonstandard version of the desired result.

\begin{theorem}\label{nsudi}
Suppose that $X\subseteq G$ is a hyperfinite $K$-approximate group such that $X^2$ has exponent $e$.  Then $X^4$ contains an internal subgroup $H$ of $G$ such that $o(X)\subseteq H$.  
\end{theorem}

\begin{proof}
Let $U$ be an open neighborhood of the identity in $\G$ with $\pi^{-1}(U)\subseteq X^4$ such that $U$ is contained in $\pi(X^2)$, whence $U$ has exponent $e$.  By the Gleason-Yamabe theorem \cite{yamabe_generalization_1953}, there is an open subgroup  $\G'$ of $\G$ and normal $N\unlhd \G'$ with $N\subseteq U$ such that $\mathcal{H}:=\G'/N$ is a connected Lie group.  Let $Y:=X\cap \pi^{-1}(\G')$ and let $\rho:\langle Y\rangle \to \mathcal{H}$ be the composition of $\pi$ with the quotient map $\G'\to \mathcal{H}$.  Since $\G'$ is clopen in $\G$, $\pi^{-1}(\G')$ is both monadic and co-monadic (the complement of a monadic, also known as \emph{galactic}), whence internal by saturation; it follows that $Y$ is also internal.  Since the image of $U\cap \G'$ in $\mathcal{H}$ is also open, it follows that $\mathcal{H}$ is a connected Lie group with an identity  neighborhood of finite exponent.  We conclude that $\mathcal{H}$ is trivial, whence $\ker(\rho)=Y=\langle Y\rangle$ is the desired internal subgroup of $G$ contained in $X^4$.
\end{proof}

\begin{remark}
The passage from $\G$ to the Lie subquotient $\G'/N$ is called the \emph{Hrushovski Lie Model Theorem}.  More precisely, \cite{Breuillard_Structure_2012} abstracts the important properties of the quotient map $\pi:\langle X \rangle\to \G$ and calls any group morphism onto a second countable, locally compact group satisfying these properties a \emph{good model}.  In the proof of Theorem \ref{nsudi}, we actually showed that the good model $\pi:\langle X\rangle\to \G$ can be replaced by a good model $\rho:\langle Y\rangle \to \mathcal{H}$ onto a connected Lie group.  One can show that $Y$ is also an approximate group (in fact, it is a $K^6$-approximate group) that is closely related to the original approximate group $X$, whence the Hrushovski Lie model theorem allows one to study ultraproducts of $K$-approximate groups by working with the connected Lie groups that model them.  For example, the proof of Theorem \ref{nonabelianfreiman} actually proceeds by induction on the dimension of the corresponding Lie model.  To be fair, the proof of Theorem \ref{nonabelianfreiman} actually requires the use of \emph{local Lie groups} and, in particular, uses the local version of Yamabe's theorem, whose first proof used nonstandard analysis \cite{goldbring_hilberts_2010}.
\end{remark}

\begin{proof}[of Theorem \ref{udi}]
Suppose, towards a contradiction, that the theorem is false.  For each $L$, let $G_L$ be a group and $X_L\subseteq G_L$ a finite $K$-approximate group such that $X_L^2$ has exponent $e$ and yet, for any finite subgroup $H$ of $\langle X_L\rangle$ contained in $X_L^4$, we have that $X_L$ is not covered by $L$ cosets of $H$.  Let $X:=\prod_\u X_L$ and $G:=\prod_\u G_L$.  By transfer, $X$ is a $K$-approximate subgroup of $G$ such that $X^2$ has exponent $e$.  By Theorem \ref{nsudi}, $X^4$ contains an internal subgroup $H\supseteq o(X)$ of $\langle X\rangle$.  Without loss of generality, we may write $H:=\prod_\u H_L$ with $H_L$ a subgroup of $G_L$ contained in $X_L^4$.  Since the quotient is bounded by Theorem \ref{lcmodel}, there is $M\in \N$ such that $M$ left translates of $H$ cover $X^4$.  Thus, for $\u$-almost all $L$, $M$ left translates of $H_L$ cover $X_L$; taking $L>M$ yields the desired contradiction.  
\end{proof}

We now turn to the proof of Theorem \ref{lcmodel}.  Hrushovski's original proof used some fairly sophisticated model theory.  A key insight of Breuillard-Green-Tao was that a proof that relied only on fairly elementary combinatorics and nonstandard methods could be given.  The following result is the combinatorial core of their proof.  It, and the easy lemma after it, do not follow the convention that $X$ is a hyperfinite $K$-approximate group.

\begin{theorem}[Sanders-Croot-Sisask]\label{SCS}
Given $K$ and $\delta>0$, there is $\epsilon>0$ so that the following holds:  Suppose that $X$ is a finite $K$-approximate subgroup of $G$.  Suppose that $Y\subseteq X$ is symmetric and $|Y|\geq \delta |X|$.  Then there is a symmetric $E\subseteq G$ such that $|E|\geq \epsilon |X|$ and $(E^{16})^X\subseteq Y^4$.
\end{theorem}

\begin{lemma}
Let $X\subseteq G$ be a finite $K$-approximate group and $S\subseteq G$ symmetric such that $S^4\subseteq X^4$ and $|S|\geq c|X|$ for some $c>0$.  Then $X^4$ can be covered by $K^7/c$ left cosets of $S^2$. 
\end{lemma}

We now return to our assumption that $X$ is a hyperfinite $K$-approximate subgroup of $G$.

\begin{proposition}\label{grouptop}
There is a descending sequence
$$X^4=:X_0\supseteq X_1\supseteq X_2\supseteq \cdots \supseteq X_n\supseteq \cdots$$ of internal, symmetric subsets of $G$ such that:
\begin{enumerate}
\item[(i)] $X_{n+1}^2\subseteq X_n$;
\item[(ii)] $X_{n+1}^X\subseteq X_n$;
\item[(iii)] $X^4$ is covered by finitely many left cosets of $X_n$.
\end{enumerate}
\end{proposition}

\begin{proof}
Suppose that $Y\subseteq G$ is internal, symmetric, $Y^4\subseteq X^4$, and $X^4$ can be covered by finitely many left cosets of $Y$.  We define a new set $\tilde{Y}$ with these same properties.  First, take $\delta>0$ such that $|Y|\geq \delta |X^4|$; such $\delta$ exists since $X^4$ can be covered by finitely many left cosets of $Y$.  By the transfer of Theorem \ref{SCS}, there is an internal, symmetric $S\subseteq Y^4$ such that $|S|\geq \epsilon |X^4|$ and $(S^{16})^X\subseteq Y^4$.  Let $\tilde{Y}:=S^2$.  Note that $\tilde{Y}$ has the desired properties, the last of which follows from the preceding lemma.

We now define a sequence $Y_0,Y_1,Y_2,\ldots,$ of internal subsets of $X^4$ satisfying the above properties by setting $Y_0:=X$ and $Y_{n+1}:=\tilde{Y_n}$.  Finally, setting $X_n:=Y_n^4$ yields the desired sequence.  
\end{proof}

\begin{proof}[of Theorem \ref{lcmodel}]  
Take $(X_n)$ as guaranteed by Proposition \ref{grouptop}.  We set $o(X):=\bigcap_n X_n$, a monadic subset of $X^4$.  It is clear from (i) and (ii) that $o(X)$ is a normal subgroup of $\langle X\rangle$.
We can topologize $\langle X\rangle$ by declaring, for $a\in \langle X\rangle$, $\{aX_n \ : \ n\in \N\}$ to be a neighborhood base for $a$.  The resulting space is not Hausdorff, but it is clear that the quotient space $\langle X \rangle/o(X)$ is precisely the separation of $\langle X\rangle$.  It is straightforward to check that the resulting space is separable and yields a group topology on $\G$.  Now one uses the boundedness property (proven in the next paragraph) to show that $\G$ is locally compact; see \cite{van_den_dries_approximate_2015} for details.

To show that it is bounded, suppose that $A,B\subseteq \langle X\rangle$ are such that $o(X)\subseteq A$.  We need finitely many left cosets of $A$ to cover $B$.    Take $n$ such that $X_n\subseteq A$ and take $m$ such that $B\subseteq (X^4)^m$.  Since $X^4$ is a $K^4$-approximate group, $(X^4)^m\subseteq E\cdot X^4$ for some finite $E$.  By (iii), we have that $X^4\subseteq F\cdot X_n$ for some finite $F$.  It follows that $B\subseteq EFA$, as desired.

The proof that $Y\subseteq \G$ is compact if and only if $\pi^{-1}(Y)$ is monadic is an exercise left to the reader (or, once again, one can consult \cite{van_den_dries_approximate_2015}).  To see the moreover part, note that $$\pi^{-1}(\pi(X))=\{x\in \langle X\rangle \ : \ \text {there is }y\in X \text{ such that }x^{-1}y\in \bigcap_n X_n\}.$$  In particular, $\pi^{-1}(\pi(X))\subseteq X^5$ and, by saturation, we actually have
$$\pi^{-1}(\pi(X))=\{x\in X^5 \ : \text{ for all $n$ there is }y\in X \text{ such that }x^{-1}y\in X_n\}.$$  From this description of $\pi^{-1}(\pi(X))$, we see that it is monadic, whence $\pi(X)$ is compact.

To prove (3), suppose that $Y$ is an internal subset of $G$ containing $o(X)$.  Take $n$ such that $X_n\subseteq Y$.  Thus, $\pi^{-1}(\pi(X_{n+1}))\subseteq X_n\subseteq Y$ and $\pi(X_{n+1})$ is open in $\G$.

Finally, to see that $\pi(X^2)$ is a neighborhood of the identity in $\G$, first observe that since $X^4$ is covered by finitely many left cosets of $X$, the neighborhood $\pi(X^4)$ of the identity is covered by finitely many left cosets of the compact set $\pi(X)$, whence $\pi(X)$ has nonempty interior and thus $\pi(X^2)=\pi(X)\cdot \pi(X)^{-1}$ is a neighborhood of the identity in $\G$.
\end{proof}

\appendix

\part{Appendix}

\chapter{Foundations of nonstandard analysis}
\section{Foundations}

In this appendix we will revise all the basic notions and principles that
we presented in Chapter 2 and put them on firm foundations.
As it is customary in the foundations of mathematics, 
we will work in a set-theoretic framework as formalized
by Zermelo-Fraenkel set theory with choice \zfc.
Since the purpose of this book is not a foundational one,
we will only outline the main arguments, and
then give precise bibliographic references where the interested reader 
can find all proofs worked out in detail.

\subsection{Mathematical universes and superstructures}

Let us start with the notion of a mathematical universe, which
formalizes the idea of a sufficiently large collection of
mathematical objects that contains all that one needs
when applying nonstandard methods.

\begin{definition}\label{def-universe}
A \emph{universe} $\UU$ is a nonempty collection of
``mathematical objects'' that satisfies the following properties:
\begin{enumerate}
\item
The numerical sets $\N, \Z, \Q, \R, \C\in\UU$;
\item
If $a_1,\ldots,a_k\in\UU$ then also the tuple
$\{a_1,\ldots,a_k\}$ and the ordered tuple $(a_1,\ldots,a_k)$
belong to $\UU$;
\item
If the family of sets $\F\in\UU$ then also its
union $\bigcup\F=\bigcup_{F\in\F}F\in\UU$;
\item
If the sets $A, B\in\UU$ then also
the \emph{Cartesian product} $A\times B$, 
the \emph{powerset} $\Pow(A)=\{A'\mid A'\subseteq A\}$,
and the \emph{function set} $\Fun(A,B)=\{f\mid f:A\to B\}$
belong to $\UU$;
\item
$\UU$ is \emph{transitive}, that is, $a\in A\in\UU\Rightarrow a\in\UU$.
\end{enumerate}
\end{definition}

Notice that a universe $\UU$ is necessarily closed under subsets; indeed 
if $A'\subseteq A\in\UU$, then $A'\in\Pow(A)\in\UU$, and hence $A'\in\UU$,
by transitivity. Thus, if the sets $A,B\in\UU$
then also the the \emph{intersection} $A\cap B$ and
the \emph{set-difference} $A\setminus B$ belong to $\UU$;
moreover, by combining properties 2 and 3, one
obtains that also the \emph{union} $A\cup B=\bigcup\{A,B\}\in\UU$.

\begin{remark}\label{reductionism}
It is a well-known fact that all ``mathematical objects'' used in 
the ordinary practice of mathematics, including
numbers, sets, functions, relations, ordered tuples, and
Cartesian products, can all be coded as sets.
Recall that, in \zfc, an ordered pair $(a,b)$ is defined as the
\emph{Kuratowski pair} $\{\{a\},\{a, b\}\}$;
in fact, it is easily shown that by adopting that definition
one has the characterizing property that
$(a,b)=(a',b')$ if and only if $a=a'$ and $b=b'$.
Ordered tuples are defined inductively by
letting $(a_1,\ldots,a_k,a_{k+1})=((a_1,\ldots,a_k),a_{k+1})$.
A binary relation $R$ is defined as a set of ordered pairs;
so, the notion of a relation is identified with the set of pairs that
satisfy it. A function $f$ is a relation such that every
element $a$ in the domain is in relation with a unique element $b$ 
of the range, denoted $b=f(a)$; so, the notion of a function is identified 
with its graph. 
As for numbers, the natural numbers $\N_0$ of \zfc\ are defined as the set 
of \emph{von Neumann naturals}: $0=\emptyset$ and, recursively,
$n+1=n\cup\{n\}$,
so that each natural number $n=\{0, 1, \ldots, n-1\}$
is identified with the set of its predecessors;
the integers $\Z$
are then defined as a suitable quotient of $\N\times\N$,
and the rationals $\Q$ as a suitable
quotient of $\Z\times\Z$; the real numbers $\R$
are defined as suitable sets of rational
numbers, namely the \emph{Dedekind cuts};
the complex numbers $\C=\R\times\R$
are defined as ordered pairs of
real numbers, where the pair $(a,b)$ is denoted $a+i b$.
(See, \emph{e.g.}, \cite{hrbacek_introduction_1999}.)

We remark that the above definitions are instrumental
if one works within axiomatic set theory, where all notions must be reduced to the
sole notion of a set; however, in the ordinary 
practice of mathematics, one can safely take the
ordered tuples, the relations, the functions, and the natural numbers
as primitive objects of a different nature with respect to sets.
\end{remark}

For convenience, in the following we will consider \emph{atoms},\index{atom}
that is, primitive objects that are not sets.\footnote
{~The existence of atoms is disproved by the axioms of \zfc, where
all existing objects are sets;
however, axiomatic theories are easily formalized that allow
a proper class of atoms. For instance, one can consider
a suitably modified versions of \zfc\ where a unary predicate
$A(x)$ for ``$x$ is an atom" is added to the language,
and where the axiom of extensionality is restricted to non-atoms.}
A notion of a universe that is convenient to our purposes is the following.

\begin{definition}
Let $X$ be a set of atoms.
The \emph{superstructure} over $X$ is the
union $\VV(X):=\bigcup_{n\in\N_0}V_n(X)$, where
$V_0(X)=X$, and, recursively, $V_{n+1}(X)=V_n(X)\cup\Pow(V_n(X))$.
\end{definition}

\begin{proposition}
Let $X$ be a set of atoms that includes (a copy of) $\N$. 
Then the superstructure $\VV(X)$ is a universe in the sense 
of Definition \ref{def-universe}.\footnote
{~Clearly, the transitivity property ``$a\in A\in\VV(X)\Rightarrow a\in\VV(X)$"
applies provided $A\notin X$.}
\end{proposition}

\begin{proof}
See \cite[\S 4.4]{chang_model_1977}.
\end{proof}


\begin{remark}
In set theory, one considers the universe
$\mathbf{V}=\bigcup_\gamma V_\gamma$
given by the union of all levels of the so-called \emph{von Neumann cumulative hierarchy}, 
which is defined by transfinite recursion on the class of all
ordinals by letting $V_0=\emptyset$, $V_{\gamma+1}=\Pow(V_\gamma)$, and 
$V_\lambda=\bigcup_{\gamma<\lambda}V_\gamma$ if $\lambda$ is a limit ordinal.
Basically, the Regularity axiom was introduced in set theory
to show that the above class $\mathbf{V}$ is the universal class of all sets.

Instead, the superstructures are defined by only taking the finite
levels $V_n(X)$ constructed over a given set of atoms $X$.
The main motivation for that restriction is that if one goes
beyond the finite levels and allows the first infinite ordinal $\omega$
to belong to the domain of the star map, then 
${}^*\omega$ would contain $\in$-descending chains
$\xi\ni\xi-1\ni\xi-2\ni\ldots$ for every $\xi\in{}^*\omega\setminus\omega$,
contradicting the Regularity axiom. 
Since $V_\omega=\bigcup_{n\in\omega} V_n$ would not be suitable, as it only
contains finite sets, one takes an infinite set of atoms $X$ as
the ground level $V_0(X)=X$, so as to
enclose (a copy of the) natural numbers in the universe. 

However, we remark that if one drops the Regularity Axiom from the
axiomatics \textsf{ZFC}, and replace it with a suitable Anti-Foundation
Axiom (such as Boffa's \emph{superuniversality axiom}),
then one can construct star maps $*:\mathbf{V}\to\mathbf{V}$ from
the universe all sets into itself that satisfies the \emph{transfer principle}
and $\kappa$-saturation for any given cardinal $\kappa$.
(This is to be contrasted with the well-known result by Kunen
about the impossibility in \textsf{ZFC} of non-trivial elementary extensions 
$j:\mathbf{V}\to\mathbf{V}$.)
This kind of foundational issues are the subject matter of the so-called 
\emph{nonstandard set theory} (see Remark \ref{remark-nstheories}).
\end{remark}

\subsection{Bounded quantifier formulas}

In this section we formalize the notion 
of ``elementary property"\index{elementary property}
by means of suitable formulas.
It is a well-known fact that virtually all properties
of mathematical objects can be described within first-order 
logic; in particular, one can reduce to the language
of set theory grounded on the usual logic symbols plus
the sole membership relation symbol.
Here is the ``alphabet" of our language.\footnote
{~To be precise, also parentheses ``\,$($\," and ``\,$)$\,"
should be included among the symbols of our alphabet.}
\begin{itemize}
\item
\emph{Variables}:\index{variable} $x, y, z, \ldots, x_1, x_2, \ldots$\index{variable};
\item
\emph{Logical Connectives}\index{logical connective}:
$\lnot$ (negation ``not");
$\land$ (conjunction ``and");
$\lor$ (disjunction ``or");
$\Rightarrow$ (implication ``if \ldots then");
$\Leftrightarrow$ (double implication ``if and only if");
\item
\emph{Quantifiers}\index{quantifier}\index{quantifier}:
$\exists$ (existential quantifier ``there exists");
$\forall$ (universal quantifier  ``for all");
\item
\emph{Equality symbol} $=$;
\item
\emph{Membership symbol}\index{membership} $\in$.
\end{itemize}

\begin{definition}\label{elementaryformula}
An \emph{elementary formula}\index{formula!elementary} $\sigma$
is a finite string of symbols in the above alphabet where
it is specified a set of \emph{free variables}\index{variable!free} $FV(\sigma)$
and a set of \emph{bound variables}\index{variable!bound} $BV(\sigma)$,
according to the following rules.
\begin{itemize}
\item
\emph{Atomic formulas}\index{formula!atomic}.
If $x$ and $y$ are variables then ``$(x=y)$" and ``$(x\in y)$"
are elementary formulas, named \emph{atomic formulas},
where $FV(x=y)=FV(x\in y)=\{x,y\}$ and $BV(x=y)=BV(x\in y)=\emptyset$;
\item
\emph{Restricted quantifiers}\index{quantifier!restricted}.
If $\sigma$ is an elementary formula, 
$x\in FV(\sigma)$ and $y\notin BV(\sigma)$,
then $``(\forall x\in y)\ \sigma"$ 
is an elementary formula where
$FV((\forall x\in y)\ \sigma)=(FV(\sigma)\setminus\{x\})\cup\{y\}$ and
$BV((\forall x\in y)\ \sigma)=BV(\sigma)\cup\{y\}$. Similarly 
with the elementary formula $``(\exists x\in y)\ \sigma"$ 
obtained by applying the existential quantifier;
\item
\emph{Negation}.
If $\sigma$ is an elementary formula then $(\lnot\,\sigma)$
is an elementary formula where $FV(\lnot\sigma)=FV(\sigma)$
and $BV(\lnot\sigma)=BV(\sigma)$;
\item
\emph{Binary connectives}.
If $\sigma$ and $\tau$
are elementary formulas where 
$FV(\sigma)\cap BV(\tau)=FV(\tau)\cap BV(\sigma)=\emptyset$,
then $``(\sigma\land\tau)"$ is an elementary formula
where $FV(\sigma\land\tau)=FV(\sigma)\cup FV(\tau)$
and $BV(\sigma\land\tau)=BV(\sigma)\cup BV(\tau)$;
and similarly with the elementary formulas
$(\sigma\lor\tau)$, $(\sigma\Rightarrow\tau)$,
and $(\sigma\Leftrightarrow\tau)$ obtained
by applying the connectives $\lor$, $\Rightarrow$, and 
$\Leftrightarrow$, respectively.
\end{itemize}
\end{definition}

According to the above,
every elementary formula is built from \emph{atomic formulas}
(and this justifies the name ``atomic").
in that an arbitrary elementary formula is obtained from
atomic formulas by finitely many iterations of 
restricted quantifiers, negations, and binary connectives, 
in whatever order.
Only quantifiers produces bound variables, and in fact
the bound variables are those that are quantified.
Notice that a variable can be quantified only if it is free
in the given formula, that is, it actually appears
and it has been not quantified already.

It is worth stressing that quantifications are only permitted
in the \emph{restricted forms} $(\forall x\in y)$ or $(\exists x\in y)$,
where the ``scope" of the quantified variable $x$ is ``restricted" by
another variable $y$. To avoid potential ambiguities,
we required that the ``bounding" variable $y$ does not appear
bound itself in the given formula.

As it is customary in the practice, to simplify notation
we will adopt natural short-hands. For instance,
we will write ``$x\ne y$" to mean ``$\lnot(x=y)$" and
``$x\notin y$" to mean ``$\lnot(x\in y)$";
we will write ``$\forall x_1,\ldots,x_k\in y\ \sigma$" to mean
``$(\forall x_1\in y)\ldots(\forall x_k\in y)\ \sigma$",
and similarly with existential quantifiers.
Moreover, we will use parentheses informally,
and omit some of them whenever confusion is unlikely.
So, we may write ``$\forall x\in y\ \sigma$" instead
of ``$(\forall x\in y)\ \sigma$; or ``$\sigma\land\tau$"
instead of ``$(\sigma\land\tau)$"; and so forth.

Another usual agreement is that negation $\lnot$ binds more
strongly than conjunctions $\land$ and
disjunctions $\lor$, which in turn bind more strongly
than implications $\Rightarrow$ and double implications $\Leftrightarrow$.
So, we may write ``$\lnot\sigma\land\tau$" to mean
``$((\lnot\sigma)\land\tau)$"; or
``$\lnot\sigma\lor\tau\Rightarrow\upsilon$"
to mean ``$(((\lnot\sigma)\lor\tau)\Rightarrow\upsilon)$"; or
``$\sigma\Rightarrow \tau\lor\upsilon$" to mean
``$(\sigma\Rightarrow(\tau\lor\upsilon))$".

When writing $\sigma(x_1,\ldots,x_k)$
we will mean that
$x_1,\ldots,x_k$ are all and only the free variables that
appear in the formula $\sigma$.
The intuition is that the truth or falsity of a formula
depends only on the values given to its \emph{free variables},
whereas \emph{bound variables} can be renamed
without changing the meaning of a formula.

\begin{definition}
A property of mathematical objects $A_1,\ldots,A_k$
is \emph{expressed in elementary form}
if it is written down by taking an elementary formula
$\sigma(x_1,\ldots,x_k)$, and by replacing all occurrences of
each free variable $x_i$ by $A_i$. In this case we denote
$$\sigma(A_1,\ldots,A_k),$$
and we will refer to objects $A_1,\ldots,A_k$ 
as \emph{constants} or \emph{parameters}.\footnote
{~In order to make sense, it is implicitly assumed
that in every quantification
$(\forall x\in A_i)$ and $(\exists x\in A_i)$,
the object $A_i$ is a \emph{set} (not an atom).}
By a slight abuse, sometimes we will simply say
\emph{elementary property}
to mean ``property expressed in elementary form".
\end{definition}

The motivation of our definition is the well-known fact
that virtually all properties considered in mathematics
can be formulated in elementary form.
Below is a list of examples that include the fundamental ones.
As an exercise, the reader can easily write down
by him- or herself any other mathematical property that comes
to his or her mind, in elementary form.

\medskip
\begin{example}\label{exampleelementary}
Each property is followed by one of its
possible expressions in elementary form.\footnote
{~For simplicity, in each item we use short-hands for properties
that have been already considered in previous items.}

\smallskip
\begin{enumerate}
\item
``$A\subseteq B$": $(\forall x\in A) (x\in B)$;
\item
$C=A\cup B$: $(A\subseteq C) \land (B\subseteq C) \land
(\forall x\in C) (x\in A \lor x\in B)$;
\item
$C=A\cap B$:
$(C\subseteq A)\land (\forall x\in A)(x\in B\Leftrightarrow x\in C)$;
\item
$C=A\setminus B$:
$(C\subseteq A)\land (\forall x\in A)(x\in C\Leftrightarrow x\notin B)$;
\item
$C=\{a_1,\ldots,a_k\}$:
$(a_1\in C)\land\ldots\land(a_k\in C)\land(\forall x\in C)(x=a_1\lor\ldots\lor x=a_k)$;
\item
$\{a_1,\ldots,a_k\}\in C$:
$(\exists x\in C)(x=\{a_1,\ldots,a_k\})$;
\item
$C=(a,b)$:
$C=\{\{a\},\{a,b\}\}$;\footnote
{~Recall that ordered pairs $(a,b)=\{\{a\},\{a,b\}\}$ were defined as
\emph{Kuratowski pairs}.}
\item
$C=(a_1,\ldots,a_k)$ with $k\ge 3$: Inductively, $C=\left((a_1,\ldots,a_{k-1}),a_k\right)$;
\item
$(a_1,\ldots,a_k)\in C$:
$(\exists x\in C)(x=(a_1,\ldots,a_k))$;
\item
$C=A_1\times\ldots\times A_k$:
$(\forall x_1\in A_1)\ldots(\forall x_k\in A_k)((a_1,\ldots,a_k)\in C)\land
(\forall z\in C)(\exists x_1\in A_1)\ldots(\exists x_k\in A_k)(z=(x_1,\ldots,x_k))$;
\item
$R$ is a $k$-place relation on $A$:
$(\forall z\in R)(\exists x_1,\ldots,x_k\in A)(z=(x_1,\ldots,x_k))$;
\item
$f:A\to B$: $(f\subseteq A\times B)\land
(\forall a\in A)(\exists b\in B)((a,b)\in f)\land
(\forall a,a'\in A)(\forall b\in B)((a,b),(a',b)\in f\Rightarrow a=a')$;
\item
$f(a_1,\ldots,a_k)=b$:
$((a_1,\ldots,a_k),b)=(a_1,\ldots,a_k,b)\in f$;
\item
$x<y$ in $\R$:
$(x,y)\in R$, where $R\subset\R\times\R$ is the order relation on $\R$.
\end{enumerate}
\end{example}

It is worth remarking that a same property may be
expressed both in an elementary form and in a non-elementary form.
The typical examples involve the powerset operation.

\begin{example}
``$\Pow(A)=B$" is trivially an elementary property of constants
$\Pow(A)$ and $B$, but \emph{cannot} be formulated as an
elementary property of constants $A$ and $B$.
In fact, while the inclusion ``$B\subseteq\Pow(A)$" is
formalized in elementary form by
``$(\forall x\in B)(\forall y\in x)(y\in A)$",
the other inclusion $\Pow(A)\subseteq B$ does not
admit any elementary formulation with $A$ and $B$ as constants.
The point here is that quantifications over subsets
``$(\forall x\subseteq A)(x\in B)$" are not
allowed by our rules.
\end{example}

\subsection{\L os' Theorem}

The ultrapower construction of the hyperreals is
naturally extended to the whole superstructure.

\begin{definition}
Let $\U$ be an ultrafilter on the set of indexes $I$.
The \emph{bounded ultrapower} of the superstructure $\VV(X)$
modulo $\U$ is the union
$$\VV(X)_b^I/\U\ :=\ \bigcup_n V_n(X)^I/\U$$
where $V_n(X)^I/\U=\{[f]\mid f:I\to V_n(X)\}$
contains the equivalence classes modulo $\U$ of 
the $I$-sequences $f$ that take values in the finite level $V_n(X)$.
The \emph{pseudo-membership relation} $\in_\U$ on $\VV(X)_b^I/\U$
is defined by setting:
$$[f]\in_\U[g]\ \Longleftrightarrow\ \{i\in I\mid f(i)\in g(i)\}\in\U.$$
\end{definition}

So, the \emph{bounded} ultrapower consists
of the equivalence classes modulo $\U$ of 
the ``bounded" $I$-sequences (that is,
of those sequences $f:I\to\VV(X)$
whose range is included in some finite level $V_n(X)$);
and the pseudo-membership holds when the
actual membership holds pointwise for $\U$-almost all indexes.

In bounded ultrapowers, properties expressed in elementary form
can be interpreted in a natural way.

\begin{definition}
Let $P$ be a property expressed in elementary form with constant
parameters in $\VV_b(X)^I/\U$.
The \emph{satisfaction relation} ``$\VV(X)_b^I/\U\models P$"
(read: ``the property $P$ holds in $\VV(X)_b^I/\U$")
is defined according to the following rules:\footnote
{~Following classic logic, we agree that the disjunction ``or" is \emph{inclusive}, that is,
``$A$ or $B$" is always true except when both A and B are false;
and the implication ``$A\Rightarrow B$" is true except when $A$ is true and $B$ is false.}
\begin{itemize}
\item
``$\VV(X)_b^I/\U\models [f]=[g]$" when $[f]=[g]$, that is, when
$\{i\in I\mid f(i)=g(i)\}\in\U$.
\item
``$\VV(X)_b^I/\U\models[f]\in[g]$" when $[f]\in_\U[g]$, that is, when
$\{i\in I\mid f(i)\in g(i)\}\in\U$.
\item
``$\VV(X)_b^I/\U\models\lnot P$" when ``$\VV(X)_b^I/\U\not\models P$."
\item
``$\VV(X)_b^I/\U\models(P_1\land P_2)$" when 
both ``$\VV(X)_b^I/\U\models P_1$" and ``$\VV(X)_b^I/\U\models P_2$."
\item
``$\VV(X)_b^I/\U\models(P_1\lor P_2)$" when 
``$\VV(X)_b^I/\U\models P_1$" or ``$\VV(X)_b^I/\U\models P_2$."
\item
``$\VV(X)_b^I/\U\models(P_1\Rightarrow P_2)$" when 
if ``$\VV(X)_b^I/\U\models P_1$" then also
``$\VV(X)_b^I/\U\models P_2$."
\item
``$\VV(X)_b^I/\U\models(P_1\Leftrightarrow P_2)$" when 
both ``$\VV(X)_b^I/\U\models P_1\Rightarrow P_2$" and
``$\VV(X)_b^I/\U\models P_2\Rightarrow P_1$."
\item
``$\VV(X)_b^I/\U\models(\exists x\in [g])\ \sigma(x,[f_1],\ldots,[f_n])$"
when ``$\VV(X)_b^I/\U\models\ \sigma([h],[f_1],\ldots,[f_n])$"
for some $[h]\in_\U[g]$.
\item
``$\VV(X)_b^I/\U\models(\forall x\in [g])\ \sigma(x,[f_1],\ldots,[f_n])$"
when ``$\VV(X)_b^I/\U\models\ \sigma([h],[f_1],\ldots,[f_n])$"
for every $[h]\in_\U[g]$.
\end{itemize}
\end{definition}

\L os' Theorem \index{\L os' Theorem}
is a fundamental result in model theory
stating that an ultrapower satisfies the
same elementary properties as the initial structure.
In the case of bounded ultrapowers of superstructures, 
one has the following formulation.

\begin{theorem}[\L os]
Let $\VV(X)_b^I/\U$ be a bounded ultrapower
of the superstructure $\VV(X)$ and let
$\sigma([f_1],\ldots,[f_n])$ be a property expressed in elementary form
with constant parameters from $\VV(X)_b^I/\U$. Then
$$\VV(X)_b^I/\U\models\sigma([f_1],\ldots,[f_n])\ \Longleftrightarrow\ 
\{i\in I\mid \sigma(f_1(i),\ldots,f_n(i))\ \text{holds}\}\in\U.$$
\end{theorem}

\begin{corollary}
Let $d:\VV(X)\to\VV(X)_b^I/\U$ be the diagonal
embedding $A\mapsto[\langle A\mid i\in I\rangle]$
of a superstructure into its bounded ultrapower. Then
for every property $\sigma(A_1,\ldots,A_n)$ expressed in elementary form
with constant parameters $A_j\in\VV(X)$ one has
$$\sigma(A_1,\ldots,A_n)\ \Longleftrightarrow\ 
\VV(X)_b^I/\U\models\sigma(d(A_1),\ldots,d(A_n)).$$
\end{corollary}

Usually, in nonstandard analysis one considers
a superstructure $\VV(X)$, named the \emph{standard universe},
takes a bounded ultrapower $\VV(X)_b^I/\U$ of it,
and then defines an injective map $\pi:\VV(X)_b^I/\U\to\VV(Y)$,
where $\VV(Y)$ is a suitable superstructure called the \emph{nonstandard universe}.
Such a map $\pi$, called the \emph{Mostowski collapse},
has the important property that it transforms the pseudo-membership $\in_\u$ into 
actual membership, that is, $[f]\in_\U[g]\Leftrightarrow \pi([f])\in\pi([g])$.
As a result, the star map $*=\pi\circ d:\VV(X)\to\VV(Y)$ obtained
by composing the diagonal embedding with the Mostowski collapse
satisfies the \emph{transfer} principle. All details of the construction can be found in
\S 4.4 of \cite{chang_model_1977}.

Triples $\langle *, \VV(X),\VV(Y)\rangle$ where the map
$*:\VV(X)\to\VV(Y)$ satisfies the \emph{transfer principle}\index{transfer principle}
and ${}^*X=Y$ are called \emph{superstructure models of nonstandard analysis}.\footnote
{~Typically, one takes (a copy of) the real numbers $\R$ as $X$.}

\subsection{Models that allow iterated hyper-extensions}\label{allowiterated}

In applications, we needed
iterated hyper-extensions, but in the usual superstructure
approach to nonstandard analysis (recalled in the previous section),
such extensions cannot be accommodated directly.
To this end, one would need to
construct a different standard universe each time,
which contains the previous nonstandard universe.
A neat way to overcome this problem is to consider
a superstructure model of nonstandard analysis 
$\langle *, \VV(X),\VV(X)\rangle$ where 
the standard and the nonstandard universe coincide.
Clearly, in this case a hyper-extension also belongs to
the standard universe, and so one can apply
the star map to it.\footnote
{~We remark that the notion
of ``iterated hyper-image" does not make
sense in Nelson's \emph{Internal Set Theory} \textsf{IST}, 
as well as in other axiomatic theories elaborated upon that
approach.}

In the following we will assume that
arbitrarily large sets of atoms are available.\footnote
{~We remark that this is just a simplifying assumption; indeed,
in \textsf{ZFC} one can easily construct 
sets $X$ of arbitrarily large cardinality that
behaves like sets of atoms
with respect to the corresponding superstructures $\VV(X)$,
that is, such that $\emptyset\notin X$ and 
$x\cap\VV(X)=\emptyset$ for every $x\in X$.
Such sets are named \emph{base sets} in 
\cite[\S 4.4]{chang_model_1977}.}

\begin{theorem}
Let $\kappa,\mu$ be infinite cardinals.
Then there exist sets of atoms $X_0\subset X$ of cardinality 
$|X_0|=|X|=\mu^\kappa$
and star maps $*:\VV(X)\to\VV(X)$ such that:
\begin{enumerate}
\item
(a copy of) the real numbers $\R\subset Y$;
\item
${}^*x=x$ for every $x\in X_0$, and hence ${}^*r=r$ for every $r\in\R$;
\item
$\starX=X$;
\item
\emph{transfer principle}\index{transfer principle}. 
For every bounded quantifier formula $\varphi(x_1,\ldots,x_n)$
and for every $a_1,\ldots,a_n\in\VV(X)$:
$$\varphi(a_1,\ldots,a_n)\Longleftrightarrow
\varphi({}^*a_1,\ldots,{}^*a_n);$$
\item
The $\kappa^+$-saturation\index{saturation!$\kappa $-} principle holds.
\end{enumerate}
\end{theorem}

\begin{proof}
Since $\mu^\kappa\ge\mathfrak{c}$ 
has at least the size of the \emph{continuum}, we can
pick a set of atoms $X$ of cardinality $\mu^\kappa$
that contains (a copy of) the real numbers $\R$, and such that
the relative complement $X_0=X\setminus\R$ has cardinality
$\mu^\kappa$. For every $x\in\kappa$, let 
$\langle x\rangle=\{a\in\Fin(\kappa)\mid x\in a\}$
be the set of all finite parts of $\kappa$ that contains $x$.
It is readily seen that the family $\{\langle x\rangle\mid x\in\kappa\}$
has the \emph{finite intersection property}, and so it can be extended
to an ultrafilter $\U$ on $I=\Fin(\kappa)$\index{ultrafilter}. We now inductively define 
maps $\Psi_n:V_n(X)^I/\U\to V_n(X)$ as follows.

Since $\mu^\kappa=|X|\le |X^I/\U|\le|X|^{|I|}=(\mu^\kappa)^\kappa=\mu^\kappa$,
we have $|X|=|X^I/\U|$ and
we can pick a bijection $\Psi_0:X\to X^I/\U$ with the
property that $\Psi_0(x)=[c_x]_\U$ for every $x\in X_0$.
At the inductive step, let $f:I\to V_{n+1}(X)$ be given.
If $f(i)\in V_n(X)$ $\U$-a.e., let
$\Psi_{n+1}([f]_\U)=\Psi_n([f]_\U)$; and
if $f(i)\notin V_n(X)$ $\U$-a.e., that is, if $f(i)\in\Pow(V_n(X))$ $\U$-a.e., define
$$\Psi_{n+1}([f]_\U)\ =\ \left\{\Psi_n([g]_\U)\mid g(i)\in f(i)\ \U\text{-a.e.}\right\}.$$ 
By gluing together the above functions $\Psi_n$, we obtain a map 
$\Psi:\VV(X)^I_b/\U\to \VV(X)$ from the bounded
ultrapower\index{ultrapower} of our superstructure into the superstructure itself.
Finally, define the star map $*:\VV(X)\to\VV(X)$ as the
composition $\Psi\circ d$, where $d$ is the diagonal embedding:

\begin{center}
\begin{picture}(100,50)(0,10)
   \put(-10,0){\makebox(0,0){$\VV(X)$}}
   \put(100,0){\makebox(0,0){$\VV(X)$}}
   \put(50,50){\makebox(0,0){$\VV(X)^I_b/\U$}}
   \put(18,30){\makebox(0,0){$d$}}
   \put(82,30){\makebox(0,0){$\Psi$}}
   \put(50,4){\makebox{$*$}}
   \put(6,8){\vector(1,1){30}}
   \put(62,40){\vector(1,-1){30}}
   \put(10,0){\vector(1,0){70}}
\end{picture}
\end{center}

\bigskip
\medskip
By the definition of $\Psi_0$, for every $x\in X_0$ we have
that ${}^*x=\Psi(d(x))=\Psi_0([c_x]_\U)=x$.
Moreover, the map $*$ satisfies the \emph{transfer principle}\index{transfer principle}
for bounded quantifier formulas, as one can show by using
the same arguments as in \cite[Theorem 4.4.5]{chang_model_1977}.
In brief, the diagonal embedding $d$ preserves the bounded quantifier
formulas by \L os' Theorem; moreover it is easily verified
from the definition that also $\Psi$ preserves
the bounded quantifier formulas. Finally,
the range of $\Psi$ is a transitive subset of $\VV(X)$,
and bounded quantifier formulas are preserved under 
transitive submodels. 
\end{proof}

\begin{remark}\label{remark-nstheories}
The so-called \emph{nonstandard set theories} study
suitable adjustments of the usual axiomatic set theory where also
the methods of nonstandard analysis are incorporated
in their full generality. 
The most common approach in nonstandard set theories
is the so-called \emph{internal viewpoint} as initially proposed
independently by E. Nelson \cite{nelson_internal_1977} and K. Hrbacek \cite{hrbacek_axiomatic_1978} where one includes in the language 
a unary relation symbol $\texttt{st}$ for ``standard object''.
The underlying universe is then given by the internal sets, 
and the standard objects are those internal elements that are hyper-extensions. 
As a consequence, external sets do not belong to the universe, and 
can only be considered indirectly, similarly as
proper classes are treated in \zfc\ as extensions of formulas.

An alternative \emph{external viewpoint}, closer to
the superstructure approach, is to  
postulate a suitably modified version of Zermelo-Fraenkel theory \zfc,
plus the properties of an \emph{elementary embedding} 
for a star map $*:\mathbb{S}\to\mathbb{I}$ 
from the sub-universe $\mathbb{S}$
of ``standard" objects into the sub-universe $\mathbb{I}$ of 
``internal" objects. Of course, to this end one needs to
include in the language a new function symbol $*$ for the star map.
We remark that if one replaces the \emph{regularity axiom} by a suitable
\emph{anti-foundation principle}, then one can actually
construct \emph{bounded elementary embeddings} 
$*:\VV\to\VV$ defined on the whole
universe into itself, thus providing a foundational
framework for iterated hyper-extensions that
generalizes the superstructure models that we have
seen in this section; see \cite{ballard_standard_1992,di_nasso_axiomatic_2002}.

A simple axiomatic presentation to nonstandard analysis that
naturally accommodates iterated hyper-extensions is the Alpha-Theory
proposed by V.\ Benci and M.\ Di Nasso 
(see the book \cite{benci_measure_2018}).

\end{remark}

\backmatter

\bibliographystyle{plain}
\bibliography{bibliography}
\printindex


\end{document}